\documentclass[english,leqno]{article}

\usepackage{latexsym}
\usepackage{linearb}
\usepackage[LGR, OT1]{fontenc}
\usepackage{amsmath,amsthm}
\usepackage{amssymb}
\usepackage{MnSymbol}
\usepackage{graphicx}
\usepackage{marvosym}
\usepackage{eufrak}
\usepackage[margin=1.25in,dvips]{geometry}
\usepackage{color}
\usepackage{comment}
\usepackage{mathrsfs}\usepackage[LGR,T1]{fontenc}
\usepackage[latin9]{inputenc}
\usepackage{geometry}
\usepackage{babel}
\usepackage{enumitem}
\usepackage{amsthm}
\usepackage{amstext}
\usepackage{amssymb}
\usepackage{esint}
\usepackage{float}
\usepackage[unicode=true,pdfusetitle,
 bookmarks=true,bookmarksnumbered=false,bookmarksopen=false,
 breaklinks=false,pdfborder={0 0 1},backref=false,colorlinks=false]
 {hyperref}
\makeatletter

\DeclareRobustCommand{\greektext}{%
  \fontencoding{LGR}\selectfont\def\encodingdefault{LGR}}
\DeclareRobustCommand{\textgreek}[1]{\leavevmode{\greektext #1}}
\ProvideTextCommand{\~}{LGR}[1]{\char126#1}

\numberwithin{equation}{section}
\numberwithin{figure}{section}
\newcommand{\lyxaddress}[1]{
\par {\raggedright #1
\vspace{1.4em}
\noindent\par}
}
  \theoremstyle{plain}
  \newtheorem{thm}{\protect\theoremname}[section]
  \theoremstyle{plain}
  \newtheorem{cor}{\protect\corollaryname}[section]
 \theoremstyle{definition}
 \newtheorem*{defn*}{\protect\definitionname}
  \theoremstyle{remark}
  \newtheorem*{rem*}{\protect\remarkname}
  \theoremstyle{definition}
  \newtheorem{defn}{\protect\definitionname}[section]
  \theoremstyle{plain}
  \newtheorem{lem}{\protect\lemmaname}[section]
  \theoremstyle{plain}
  \newtheorem{prop}{\protect\propositionname}[section]

\newtheorem{innercustomprop}{Proposition}
\newenvironment{customprop}[1]
  {\renewcommand\theinnercustomprop{#1}\innercustomprop}
  {\endinnercustomprop}

\makeatother

\usepackage{babel}
  \providecommand{\definitionname}{Definition}
  \providecommand{\lemmaname}{Lemma}
  \providecommand{\propositionname}{Proposition}
  \providecommand{\remarkname}{Remark}
\providecommand{\corollaryname}{Corollary}
\providecommand{\theoremname}{Theorem}

\begin{document}
\title{The characteristic initial-boundary value problem for the\\ Einstein--massless Vlasov system in spherical symmetry}

\author{Georgios Moschidis}
\maketitle

\lyxaddress{University of California Berkeley, Department of Mathematics, Evans
Hall, Berkeley, CA 94720-3840, United States, \tt gmoschidis@berkeley.edu}
\begin{abstract}
In this paper, we initiate the study of the asymptotically AdS initial-boundary
value problem for the Einstein\textendash massless Vlasov system with
$\Lambda<0$ in spherical symmetry. We will establish the existence
and uniqueness of a maximal future development for the characteristic
initial-boundary value problem in the case when smooth initial data
are prescribed on a future light cone $\mathcal{C}^{+}$ emanating
from a point on the center of symmetry $\{r=0\}$ and a reflecting
boundary condition is imposed on conformal infinity $\mathcal{I}$.
We will then prove a number of continuation criteria for smooth solutions
of the spherically symmetric Einstein\textendash massless Vlasov system,
under the condition that the ratio $2m/r$ remains small in a neighborhood
of $\{r=0\}$. Finally, we will establish a Cauchy stability statement
for Anti-de~Sitter spacetime as a solution of the spherically symmetric
Einstein\textendash massless Vlasov system under initial perturbations
which are small only with respect to a low regularity, scale invariant
norm $||\cdot||$. This result will imply, in particular, a long time
of existence statement for $||\cdot||$-small initial data. 

This paper provides the necessary tools for addressing the AdS instability
conjecture in the setting of the spherically symmetric Einstein\textendash massless
Vlasov system, a task which is carried out in our companion paper
\cite{MoschidisVlasov}. However, the results of this paper are also
of independent interest. 
\end{abstract}
\tableofcontents{}

\section{Introduction}

In recent years, the study of the geometry and dynamics of asymptotically
Anti-de~Sitter solutions $(\mathcal{M}^{3+1},g)$ to the vacuum Einstein
equations 
\begin{equation}
Ric_{\mu\nu}-\frac{1}{2}Rg_{\mu\nu}+\Lambda g_{\mu\nu}=0\label{eq:VacuumIntro}
\end{equation}
with a \emph{negative} cosmological constant $\Lambda$ has been the
subject of intense ongoing research. In the high energy physics literature,
this surge of interest was mainly motivated by the \emph{AdS/CFT correspondence}
conjecture proposed by Maldacena, Gubser\textendash Klebanov\textendash Polyakov
and Witten \cite{Maldacena,gubser1998gauge,witten1998anti}; see \cite{AGMOO2000,Hartnoll2009,AmmonErdmenger}
and references therein. 

A distinctive feature of any asymptotically AdS spacetime $(\mathcal{M},g)$
is the presence of a \emph{conformal boundary} $\mathcal{I}$ at infinity,
which has the conformal structure of a timelike hypersurface (see
\cite{HawkingEllis1973}). In view of the timelike character of $\mathcal{I}$,
the appropriate framework for the study of asymptotically AdS solutions
$(\mathcal{M},g)$ of (\ref{eq:VacuumIntro}) is that of an initial-boundary
value problem, roughly formulated as follows:
\begin{itemize}
\item Initial data for $g$ are prescribed in the form of \emph{Cauchy}
or \emph{characteristic} data, satisfying, in each case, the associated
constraint equations for (\ref{eq:VacuumIntro}).
\item Boundary conditions are imposed asymptotically on $\mathcal{I}$,
with the requirement that the initial data and the boundary conditions
satisfy a certain set of compatibility conditions asymptotically on
the initial data hypersurface. 
\end{itemize}
The well-posedness of the asymptotically AdS initial-boundary value
problem for (\ref{eq:VacuumIntro}) was first addressed by Friedrich
\cite{Friedrich1995}. In particular, \cite{Friedrich1995} showed
that, for any suitably regular, asymptotically AdS Cauchy data set
$(\text{\textgreek{S}}^{3};\bar{g},k)$ and any smooth Lorentzian
conformal structure on $\mathcal{I}$ (such that a set of compatibility
conditions are satisfied, in an adapted gauge, at $\mathcal{I}\cap\text{\textgreek{S}}$),
there exists, at least locally in time, a unique corresponding solution
$g$ of (\ref{eq:VacuumIntro}), which is conformally regular up to
$\mathcal{I}$. The broad class of boundary conditions on $\mathcal{I}$
encoded in terms of the prescribed conformal structure on $\mathcal{I}$
contains examples both of \emph{reflecting} and of \emph{dissipative}
conditions (see also the discussion in \cite{Friedrich2014,HolzegelLukSmuleviciWarnick}).
For an extension of this result in higher dimensions (providing also
an alternative proof using wave coordinates), see \cite{EncisoKamran2014}.

In the presence of matter fields that are not conformally regular
up to $\mathcal{I}$, analogous well-posedness results for the associated
Einstein\textendash matter systems are only known under symmetry assumptions.
The study of the well-posedness of the initial-boundary value problem
for the \emph{spherically symmetric} Einstein\textendash Klein Gordon
system 
\begin{equation}
\begin{cases}
Ric[g]_{\mu\nu}-\frac{1}{2}R[g]g_{\mu\nu}+\Lambda g_{\mu\nu}=8\pi T_{\mu\nu}[\varphi],\\
\square_{g}\varphi-\mu\varphi=0,\\
T_{\mu\nu}[\varphi]\doteq\partial_{\mu}\varphi\partial_{\nu}\varphi-\frac{1}{2}g_{\mu\nu}\partial^{\alpha}\varphi\partial_{\alpha}\varphi,
\end{cases}\label{eq:EinsteinKlein--Gordon}
\end{equation}
 was initiated by Holzegel\textendash Smulevici \cite{HolzSmul2012},
who showed that, when the Klein\textendash Gordon mass $\mu$ satisfies
the Breitenlohner\textendash Freedman bound $\mu>-\frac{3}{4}|\Lambda|$,
the characterisitic initial-boundary value problem for (\ref{eq:EinsteinKlein--Gordon})
with \emph{Dirichlet} conditions on $\mathcal{I}$ is well-posed.
In \cite{HolzWarn2015}, this result was extended to a more general
class of boundary conditions on $\mathcal{I}$, including conditions
which require the initial energy of the scalar field $\varphi$ to
be \emph{infinite}.\footnote{In the case of the \emph{linear} Klein\textendash Gordon equation
on general asymptotically AdS backgrounds, well-posedness under Dirichlet
boundary conditions with no symmetry assumptions was established earlier
by Vasy \cite{VasyAdS}; well-posedness for a broader class of boundary
conditions (again, without any symmetry conditions) was shown Warnick
\cite{Warnick2013}.}

In this paper, we initiate the study of the asymptotically AdS initial-boundary
value problem for the \emph{Einstein\textendash massless Vlasov} system
in spherical symmetry. We will consider the case when initial data
are prescribed on a future light cone emanating from a point in the
center of symmetry and a reflecting boundary condition is imposed
on $\mathcal{I}$. In this setting, we will establish the well-posedness
of the initial-boundary value problem for \emph{smooth} initial data.
We will also obtain a number of extension principles for smooth developments,
presented in terms of scale invariant quantities associated to the
evolution. 

Our final result will be a Cauchy stability statement for Anti-de~Sitter
spacetime as a solution of the Einstein\textendash massless Vlasov
system, under spherically symmetric perturbations which are initially
small with respect to a low-regularity, scale-invariant norm. This
result provides the necessary first step in the direction of addressing
the AdS instability conjecture in the setting of the spherically symmetric
Einstein\textendash massless Vlasov system, with respect to a low
regularity initial data topology; this task is carried out in our
companion paper \cite{MoschidisVlasov}.

\subsection{\label{subsec:The-main-results}Statement of the main results}

The Einstein\textendash massless Vlasov system for a $3+1$ dimensional
Lorentzian manifold $(\mathcal{M},g)$ and a non-negative measure
$f$ on $T\mathcal{M}$ takes the form 

\begin{equation}
\begin{cases}
Ric[g]_{\mu\nu}-\frac{1}{2}R[g]g_{\mu\nu}+\Lambda g_{\mu\nu}=8\pi T_{\mu\nu}[f],\\
\mathcal{L}^{(g)}f=0,\\
supp(f)\subset\mathcal{N}^{+}
\end{cases}\label{eq:EinsteinSystemMasslessVlasov}
\end{equation}
where $T_{\mu\nu}[f]$ is the Vlasov energy momentum tensor (defined
by the relation (\ref{eq:EnergyMomentumTensor})), $\mathcal{L}^{(g)}\in\text{\textgreek{G}}(TT\mathcal{M})$
is the geodesic spray of $g$ and $\mathcal{N}^{+}\subset T\mathcal{M}$
is the set of future directed null vectors. The trivial solution of
(\ref{eq:EinsteinSystemMasslessVlasov}) is \emph{Anti-de~Sitter}
spacetime $(\mathcal{M}_{AdS},g_{AdS};0)$, where $\mathcal{M}_{AdS}\simeq\mathbb{R}^{3+1}$
and, in the standard double null coordinate chart $(u,v,\theta,\varphi)$
on $\mathcal{M}_{AdS}$: 
\[
g_{AdS}=-\Omega_{AdS}^{2}(u,v)dudv+r_{AdS}^{2}(u,v)\big(d\theta^{2}+\sin^{2}\theta d\varphi^{2}\big),
\]
where
\begin{align}
r_{AdS}(u,v) & =\sqrt{-\frac{3}{\Lambda}}\tan\Big(\frac{1}{2}\sqrt{-\frac{\Lambda}{3}}(v-u)\Big),\label{eq:AdSMetricValues-1}\\
\Omega_{AdS}^{2}(u,v) & =1-\frac{1}{3}\Lambda r_{AdS}^{2}(u,v).\nonumber 
\end{align}

In this paper, we will only consider the case when$\Lambda<0$ and
$(\mathcal{M},g;f)$ is \emph{spherically symmetric} and \emph{asymptotically
Anti-de~Sitter}, with regular axis of symmetry $\mathcal{Z}\neq\emptyset$
and regular conformal infinity $\mathcal{I}$; for the relevant definitions,
see Section \ref{sec:The-Einstein--Vlasov-system}. On this class
of spacetimes, a \emph{reflecting} boundary condition for (\ref{eq:EinsteinSystemMasslessVlasov})
can be naturally defined by the requirement that $f$ is conserved
along the reflection of null geodesics off $\mathcal{I}$ (see Section
\ref{subsec:Reflective-boundary-conditions}). 

We will establish a number of results related to the dynamics of spherically
symmetric and asymptotically AdS solutions of the system (\ref{eq:EinsteinSystemMasslessVlasov}),
including:
\begin{itemize}
\item A fundamental well-posedness result for the spherically symmetric,
asymptotically AdS characteristic initial-boundary value problem for
(\ref{eq:EinsteinSystemMasslessVlasov}), with initial data prescribed
on a future light cone,
\item A number of continuation criteria for smooth solutions of (\ref{eq:EinsteinSystemMasslessVlasov})
in spherical symmetry and
\item A Cauchy stability statement for the trivial solution of (\ref{eq:EinsteinSystemMasslessVlasov})
in a low regularity, scale invariant topology in spherical symmetry.
\end{itemize}
We will now proceed to briefly present the main results of this paper.

\subsubsection{\label{subsec:Well-posedness-for-smoothly}Well-posedness for smoothly
compatible characteristic initial data}

In this paper, we will consider the spherically symmetric characteristic
initial-boundary value problem for (\ref{eq:EinsteinSystemMasslessVlasov}),
obtained by prescribing characteristic initial data for $g$, $f$
on a future light cone $\mathcal{C}^{+}$ emanating from a point $p\in\mathcal{Z}$
and imposing the reflecting boundary condition on $\mathcal{I}$.
In terms of a spherically symetric, double null coordinate chart $(u,v,\theta,\varphi)$
on $(\mathcal{M},g)$, where the level sets of the optical functions
$u$ and $v$ are future light cones and past light cones, respectively,
of points lying on $\mathcal{Z}$, a characteristic initial data set
for $g$, $f$ on $\mathcal{C}^{+}$ consists of a triplet $(r,\Omega^{2};f)|_{u=0}$
defined along $u=0$ and satisfying the \emph{characteristic constraint
equations} for (\ref{eq:EinsteinSystemMasslessVlasov}), where $(r,\Omega^{2})|_{u=0}$
are the restrictions on $u=0$ of the components of a spherically
symmetric metric
\[
g=-\Omega^{2}(u,v)dudv+r^{2}(u,v)\big(d\theta^{2}+\sin^{2}\theta d\varphi^{2}\big).
\]

We will obtain the following fundamental well-posedness result:
\begin{thm}
\label{thm:Well-Posedness-Intro} Let $(r,\Omega^{2};f)|_{u=0}$ be
a regular, spherically symmetric and asymptotically AdS characteristic
initial data set for (\ref{eq:EinsteinSystemMasslessVlasov}) on $\{u=0\}$,
admitting a smooth expression in a smoothly compatible gauge. Assume,
in addition, $f|_{u=0}$ has bounded support in phase space. Then,
there exists a unique, maximal solution $(\mathcal{M},g;f)$ of (\ref{eq:EinsteinSystemMasslessVlasov})
inducing the given initial data on $\{u=0\}\subset\mathcal{M}$ and
satisfying the reflecting boundary condition on $\mathcal{I}$.
\end{thm}
For a more detailed statement of Theorem \ref{thm:Well-Posedness-Intro},
see Theorem \ref{thm:LocalExistenceUniqueness} and Corollary \ref{cor:MaximalDevelopment}.
For the definition of the smooth compatibility gauge condition for
$(r,\Omega^{2};f)|_{u=0}$, as well as for a detailed discussion on
the gauge choices involved in the statement of Theorem \ref{thm:Well-Posedness-Intro},
see Section \ref{subsec:Asymptotically-AdS-Initial-Data}.

\subsubsection{Continuation criteria for smooth solutions}

Having established the well-posedness of the characteristic initial-boundary
value problem for (\ref{eq:EinsteinSystemMasslessVlasov}), we will
proceed to obtain a number of continuation criteria for smooth solutions
$(\mathcal{M},g;f)$ of (\ref{eq:EinsteinSystemMasslessVlasov}).
Our most technically involved continuation criterion, providing an
extension principle in a neighborhood of the axis $\mathcal{Z}$,
will be the following:
\begin{thm}
\label{thm:ExtensionPrincipleAxisIntro} Let $(\mathcal{M},g;f)$
be a smooth spherically symmetric solution of (\ref{eq:EinsteinSystemMasslessVlasov})
with a non-empty axis $\mathcal{Z}$, and let $J^{-}(\mathcal{Z})\subset\mathcal{M}$
be the past domain of influence of $\mathcal{Z}$. Let also $(u,v)$
be a double null coordinate pair on $(\mathcal{M},g)$ as in Section
\ref{subsec:Well-posedness-for-smoothly}. Assume that there exists
a point $p\in\mathcal{Z}$ such that the following conditions are
satisfied on the light cone $\{u=u(p)\}$:

\begin{itemize}

\item $\{u=u(p)\}\cap J^{-}(\mathcal{Z})$ has compact closure in
$\{u=u(p)\}$,

\item $f|_{\{u=u(p)\}\cap J^{-}(\mathcal{Z})}$ has bounded support
in phase space

\end{itemize}

Assume, moreover, that $(\mathcal{M},g;f)$ satisfies the scale invariant
condition 
\[
\sup_{J^{-}(\mathcal{Z})\cap\{u\ge u(p)\}}\frac{2m}{r}\ll1,
\]
where the Hawking mass $m$ is defined by (\ref{eq:DefinitionHawkingMass}).
Then, there exists a smooth, spherically symmetric solution $(\widetilde{\mathcal{M}},\tilde{g};\tilde{f})$
of (\ref{eq:EinsteinSystemMasslessVlasov}) which strictly extends
$(\mathcal{M},g;f)$ to the future along $\mathcal{Z}$. 
\end{thm}
For a more detailed statement of Theorem \ref{thm:ExtensionPrincipleAxisIntro},
see Theorem \ref{thm:ExtensionPrinciple}.

As a Corollary of Proposition \ref{thm:ExtensionPrincipleAxisIntro}
(and a number of related extension principles), we will infer the
following general continuation criterion for the domain of outer communications
of $(\mathcal{M},g;f)$:
\begin{cor}
\label{cor:ExtensionPrincipleIntro} Let $(\mathcal{M},g;f)$ be a
smooth, spherically symmetric and asymptotically AdS solution of (\ref{eq:EinsteinSystemMasslessVlasov})
satisfying the reflecting boundary condition on conformal infinity
$\mathcal{I}$, such that $(\mathcal{M},g,f)$ is the maximal future
development of a characteristic initial data set $(r,\Omega^{2};f)|_{\{u=0\}}$
with $f|_{\{u=0\}}$ of bounded support in phase space, as in Theorem
\ref{thm:Well-Posedness-Intro}. Let us also fix a double null pair
$u,v$ satisfying the additional gauge condition
\begin{align}
u & =v\text{ on }\mathcal{Z},\label{eq:NormalizationIntro}\\
u & =v-v_{\mathcal{I}}\text{ on }\mathcal{I}\nonumber 
\end{align}
for some given $v_{\mathcal{I}}>0$.

Assume that, for some $u_{*}>0$, the projection of $\mathcal{M}$
in the $(u,v)$-plane contains the domain 
\[
\mathcal{U}_{u_{*};v_{\mathcal{I}}}\doteq\{0\le u<u_{*}\}\cap\{u\le v<u+v_{\mathcal{I}}\}.
\]
If $(\mathcal{M},g,f)$ satisfies the scale invariant conditions

\begin{equation}
\sup_{\mathcal{U}_{u_{*};v_{\mathcal{I}}}}\frac{2m}{r}<1\label{eq:NoTrappingContinuity-1}
\end{equation}
and
\begin{equation}
\limsup_{(u,v)\rightarrow(u_{*},u_{*})}\frac{2m}{r}\ll1,\label{eq:NoConcentrationContinuity-1}
\end{equation}
where the Hawking mass $m$ is defined by (\ref{eq:DefinitionHawkingMass}),
then there exists a $\bar{u}_{*}>u_{*}$ such that the projection
of $\mathcal{M}$ in the $(u,v)$-plane contains the larger domain
$\mathcal{U}_{\bar{u}_{*};v_{\mathcal{I}}}\supset\mathcal{U}_{u_{*};v_{\mathcal{I}}}$.
\end{cor}
For a more detailed statement of Corollary (\ref{cor:ExtensionPrincipleIntro}),
see Corollary (\ref{cor:GeneralContinuationCriterion}). 

\subsubsection{Cauchy stability of AdS in a scale invariant initial data topology}

Our final result will be a low-regularity Cauchy stability statement
for the trivial solution $(\mathcal{M}_{AdS},g_{AdS};0)$ of (\ref{eq:EinsteinSystemMasslessVlasov})
in spherical symmetry. In particular, we will consider the following
norm on the space of spherically symmetric characteristic initial
data $(r,\Omega^{2};f)|_{\{u=0\}}$ for (\ref{eq:EinsteinSystemMasslessVlasov})
satisfying the conditions of Theorem \ref{thm:Well-Posedness-Intro}:
\begin{defn*}
For any characteristic initial data set $(r,\Omega^{2};f)|_{\{u=0\}}$
as in Theorem \ref{thm:Well-Posedness-Intro}, let $f^{(AdS)}$ be
the solution of the (free) Vlasov field equation 
\[
\mathcal{L}^{(g_{AdS})}f^{(AdS)}=0
\]
 on Anti-de~Sitter spacetime $(\mathcal{M}_{AdS},g_{AdS})$ arising
from initial data $f^{(AdS)}|_{\mathcal{C}_{AdS}^{+}}$ on a future
light cone $\mathcal{C}_{AdS}^{+}\subset\mathcal{M}_{AdS}$ which
are obtained from $f|_{\{u=0\}}$ through the choice of a suitable
normalising gauge condition along $\{u=0\}$. We will define the ``norm''
$||(r,\Omega^{2};f)|_{\{u=0\}}||$ of $(r,\Omega^{2};f)|_{\{u=0\}}$
by the following relation measuring the concentration of energy occuring
in the evolution of the free Vlasov field $f^{(AdS)}$ in the region
$\{u\ge0\}$ of $(\mathcal{M}_{AdS},g_{AdS})$: 
\begin{align}
||(r,\Omega^{2};f)|_{\{u=0\}}||\doteq\sup_{U\ge0} & \int_{U}^{U+\sqrt{-\frac{3}{\Lambda}}\pi}\Big(\frac{r_{AdS}T_{vv}[f^{(AdS)}]}{\partial_{v}r_{AdS}}(U,v)+\frac{r_{AdS}T_{uv}[f^{(AdS)}]}{-\partial_{u}r_{AdS}}(U,v)\Big)\,dv+\label{eq:InitialDataNorm-1}\\
+ & \sup_{V_{*}\ge0}\int_{\max\{0,V-\sqrt{-\frac{3}{\Lambda}}\pi\}}^{V}\Big(\frac{r_{AdS}T_{uu}[f^{(AdS)}]}{-\partial_{u}r_{AdS}}(u,V)+\frac{r_{AdS}T_{uv}[f^{(AdS)}]}{\partial_{v}r_{AdS}}(u,V)\Big)\,du+\nonumber \\
 & +\sqrt{-\Lambda}\tilde{m}_{/}|_{\mathcal{I}},\nonumber 
\end{align}
where $T_{\mu\nu}[f^{(AdS)}]$ are the components of the energy momenum
tensor associated to the free Vlasov $f^{(AdS)}$ on $(\mathcal{M}_{AdS},g_{AdS})$
and $\tilde{m}_{/}|_{\mathcal{I}}$ is the value of the (renormalised)
Hawking mass of the initial data set $(r,\Omega^{2};f)|_{\{u=0\}}$
at infinity.
\end{defn*}
For the precise definition of the functional $||\cdot||$, as well
as for a discussion on the properties of the resulting initial data
topology, see Definition \ref{def:InitialDataNorm} in Section \ref{sec:Cauchy_Stability_Low_Regularity}. 

We will establish the following Cauchy stability statement of the
trivial solution with respect to the initial data topology defined
by the norm $||\cdot||$:
\begin{thm}
\label{thm:CauchyStabilityIntro} For any $U>0$, there exists an
$\varepsilon>0$, such that, for every characteristic initial data
set $(r,\Omega^{2};f)|_{\{u=0\}}$ as in Theorem \ref{thm:Well-Posedness-Intro}
satisfying the smallness condition 
\begin{equation}
||(r,\Omega^{2};f)|_{\{u=0\}}||<\varepsilon,
\end{equation}
the corresponding maximal future development $(\mathcal{M},g;f)$
solving (\ref{eq:EinsteinSystemMasslessVlasov}) with the reflecting
boundary condition on $\mathcal{I}$ satisfies the following conditions:

\begin{itemize}

\item \textbf{Long time of existence.} Fixing a double null coordinate
pair $(u,v)$ on $\mathcal{M}$ satisfying the gauge condition (\ref{eq:NormalizationIntro})
with $v_{\mathcal{I}}=\sqrt{-\frac{3}{\Lambda}}\pi$, the projection
of $\mathcal{M}$ in the $(u,v)$-plane contains the domain 
\[
\mathcal{U}_{U}\doteq\big\{0\le u\le U\big\}\cap\big\{ u<v<u+\sqrt{-\frac{3}{\Lambda}}\pi\big\}.
\]

\item \textbf{Cauchy stability estimates.} In the region $\{0\le u\le U\}$,
the solution $(\mathcal{M},g;f)$ satisfies the scale invariant bounds
\begin{equation}
\sup_{u_{*}\in(0,U)}||(r,\Omega^{2};f)|_{u=u_{*}}||\le C\varepsilon\label{eq:SmallnessInNormCauchy-1}
\end{equation}
and
\begin{equation}
\sup_{\{0\le u\le U\}}\Big(\frac{2\tilde{m}}{r}(u,v)\Big)<C\varepsilon,\label{eq:SmallnessTrappingCauchy-1}
\end{equation}
where $C>0$ is an absolute constant, $(r,\Omega^{2};f)|_{u=u_{*}}$
are the initial data induced by $(\mathcal{M},g;f)$ on the cone $\{u=u_{*}\}$
and $\tilde{m}(u,v)$ is the renormalised Hawking mass associated
to the sphere $\{u,v=\text{const}\}$. In particular, $(g;f)$ remains
close to $(g_{AdS};0)$ in the region $\{0\le u\le U\}$ with respect
to the topology defined by (\ref{eq:InitialDataNorm-1}) on the slices
$\{u=u_{*}\}$, $u_{*}\in[0,U]$.

\end{itemize}
\end{thm}
For a more detailed statement of Theorem \ref{thm:CauchyStabilityIntro},
see Theorem \ref{thm:CauchyStabilityAdS}.

\subsection{Outline of the paper}

This paper is organised as follows:
\begin{itemize}
\item In Section \ref{sec:The-Einstein--Vlasov-system}, we will introduce
the class of spherically symmetric and asymptotically AdS solutions
of the Einstein\textendash massless Vlasov system (\ref{eq:EinsteinSystemMasslessVlasov}),
expressed in a double null coordinate chart. We will also define the
notion of a reflecting boundary condition for (\ref{eq:EinsteinSystemMasslessVlasov})
at conformal infinity, and we will present some fundamental identities
for the null geodesic flow on spherically symmetric solutions of (\ref{eq:EinsteinSystemMasslessVlasov}).
\item In Section \ref{sec:InitialValueProblem}, we will introduce the notion
of a regular solution of (\ref{eq:EinsteinSystemMasslessVlasov})
with smooth axis and smooth conformal infinity. We will then proceed
to set up the asymptotically AdS, characteristic initial-boundary
value problem for (\ref{eq:EinsteinSystemMasslessVlasov}) with a
reflecting boundary condition at infinity. In particular, we will
define the class of smoothly compatible initial data sets for (\ref{eq:EinsteinSystemMasslessVlasov}),
and we will inspect the properties of gauge transformations mapping
smoothly compatible data to data satisfying a certain gauge normalisation
condition, which allows one to uniquely determine an initial data
set in terms of a freely prescribed initial datum for the Vlasov field;
the loss of smooth compatibility under gauge normalisation will be
also discussed.
\item In Section \ref{sec:Well-posedness-of-the}, we will establish the
well-posedness of the characterstic initial-boundary value problem
for (\ref{eq:EinsteinSystemMasslessVlasov}) in spherical symmetry,
in the class of smoothly compatible initial data. In particular, we
will prove Theorem \ref{thm:Well-Posedness-Intro}.
\item In Section \ref{sec:Extension-principles-for}, we will establish
a number of continuation criteria for smooth solutions of the system
(\ref{eq:EinsteinSystemMasslessVlasov}) in spherical symmetry. Among
the results obtained in Section \ref{sec:Extension-principles-for},
we will prove Theorem \ref{thm:ExtensionPrincipleAxisIntro} and Corollary
\ref{cor:ExtensionPrincipleIntro}.
\item In Section \ref{sec:Cauchy_Stability_Low_Regularity}, we will introduce
the scale invariant norm $||\cdot||$ on the space of smoothly comatible
initial data sets for (\ref{eq:EinsteinSystemMasslessVlasov}). In
the topology defined by $||\cdot||$, we will establish the Cauchy
stability statement of Theorem \ref{thm:CauchyStabilityIntro}.
\item In Section \ref{sec:Geodesic-flow-AdS} of the Appendix, we will review
some fundamental identities related to the geodesic flow on Anti-de~Sitter
spacetime.
\item Finally, in Section \ref{sec:CompleteI} of the Appendix, we will
establish that, in the presence of a black hole region, any maximally
extended, spherically symmetric and asymptotically AdS solution of
(\ref{eq:EinsteinSystemMasslessVlasov}) possesses a future complete
conformal infinity.
\end{itemize}

\subsection{Acknowledgements}

I would like to express my gratitude to Mihalis Dafermos for many
useful suggestions and insightful conversations. I would also like
to acknowledge support from the Miller Institute for Basic Research
in Science, University of California Berkeley.


\section{\label{sec:The-Einstein--Vlasov-system}Spherically symmetric spacetimes
and the Einstein\textendash massless Vlasov system}

This section will be mainly devoted to reviewing the properties of
the class of spherically symmetric and asymptotically AdS solutions
to the Einstein\textendash massless Vlasov system. We will follow
similar conventions as those adopted by Dafermos\textendash Rendall
in \cite{DafermosRendall}, expressing the Einstein\textendash massless
Vlasov system system in a double null coordinate chart. In paticular,
after reviewing the general properties of spherically symmetric and
asymptotically AdS spacetimes, we will proceed to fix our notations
regarding the Vlasov field equation and the Einstein\textendash Vlasov
system, listing a number of fundamental identities that will be useful
later in this paper. We will conclude this section by introducing
the reflecting boundary condition for a massless Vlasov field along
conformal infinity.

\subsection{\label{subsec:Spherically-symmetric-spacetimes}Spherically symmetric
spacetimes in double null coordinates and Anti-de~Sitter spacetime}

In this paper, we will follow similar conventions as \cite{MoschidisMaximalDevelopment,MoschidisNullDust}
regarding double null coordinate charts on spherically symmeric spacetimes.

Let $(\mathcal{M}^{3+1},g)$ be a connected, time oriented, smooth
Lorentzian manifold which is \emph{spherically symmetric}, i.\,e.~there
exists a smooth isometric action $\mathcal{A}:\mathcal{M}\times SO(3)\rightarrow\mathcal{M}$
of $SO(3)$ on $(\mathcal{M},g)$ such that, for each $p\in\mathcal{M}$,
$Orb(p)\simeq\mathbb{S}^{2}$ or $Orb(p)=\{p\}$, and at least one
point $p\in\mathcal{M}$ has a non-trivial orbit. We will define the
\emph{axis} $\mathcal{Z}$ of $(\mathcal{M},g)$ to be the set of
fixed points of $\mathcal{M}$, i.\,e.~.
\[
\mathcal{Z}\doteq\Big\{ p\in\mathcal{M}\,:\,\mathcal{A}(p,s)=p\,\forall s\in SO(3)\Big\}.
\]

Using the fact that $\mathcal{M}$ is connected, $g$ is a Lorentzian
metric and $\mathcal{A}$ maps geodesics of $(\mathcal{M},g)$ to
geodesics, it can be readily shown that $\mathcal{Z}$ consists of
a disjoint union of timelike geodesics. For the rest of this paper,
we will only consider the case when $\mathcal{Z}$ consists of a single
timelike geodesic and $\mathcal{M}\backslash\mathcal{Z}$ splits diffeomorphically
under the action $\mathcal{A}$ as
\begin{equation}
\mathcal{M}\backslash\mathcal{Z}\simeq\mathcal{U}\times\mathbb{S}^{2},\label{eq:SphericallySymmetricmanifold}
\end{equation}
where $\mathcal{U}$ is a smooth $2$ dimensional manifold. Furthermore,
we will restrict to spacetimes $(\mathcal{M},g)$ such that 
\begin{equation}
\bigcup_{p\in\mathcal{Z}}\mathcal{C}^{+}(p)\cup\mathcal{C}^{-}(p)=\mathcal{M},
\end{equation}
where $\mathcal{C}^{+}(p),\mathcal{C}^{-}(p)$ denote the future and
past light cones emanating from $p$, respectively.

Under the above assumptions on $(\mathcal{M},g)$, it can be readily
deduced that the two families of spherically symmetric null hypersurfaces
$\mathcal{H}=\big\{\mathcal{C}^{+}(p):\,p\in\mathcal{Z}\big\}$ and
$\overline{\mathcal{H}}=\big\{\mathcal{C}^{-}(p):\,p\in\mathcal{Z}\big\}$
foliate regularly the region $\mathcal{M}\backslash\mathcal{Z}$.
A pair of continuous functions $u,v:\mathcal{M}\rightarrow\mathbb{R}$
which are a smooth parametrization of the foliations $\mathcal{H},\overline{\mathcal{H}}$,
respectively, on $\mathcal{M}\backslash\mathcal{Z}$ will be called
a \emph{double null coordinate pair}. Note that any double null coordinate
pair $(u,v)$ on $(\mathcal{M},g)$ can be naturally viewed as a smooth
coordinate chart on $\mathcal{U}$.
\begin{rem*}
We will only consider double null coordinate pairs $(u,v)$ which
are compatible with the chosen time orientation of $(\mathcal{M},g)$,
i.\,e.~coordinate pairs such that both $u,v$ are increasing functions
along any future directed timelike curve. In this case, the vector
field $\partial_{u}+\partial_{v}$ on $\mathcal{M}\backslash\mathcal{Z}\simeq\mathcal{U}\times\mathbb{S}^{2}$
is timelike and future directed.
\end{rem*}
Given any double null coordinate pair on $(\mathcal{M},g)$, it readily
follows that the metric $g$ takes the following form on $\mathcal{M}\backslash\mathcal{Z}$:
\begin{equation}
g=-\text{\textgreek{W}}^{2}(u,v)dudv+r^{2}(u,v)g_{\mathbb{S}^{2}},\label{eq:SphericallySymmetricMetric}
\end{equation}
where $g_{\mathbb{S}^{2}}$ is the standard round metric on $\mathbb{S}^{2}$
and $\text{\textgreek{W}},r:\mathcal{U}\rightarrow(0,+\infty)$ are
smooth functions. 
\begin{rem*}
Viewed as a function on $\mathcal{M}\backslash\mathcal{Z}$, $r$
is expressed geometrically as 
\begin{equation}
r(p)=\sqrt{\frac{Area(Orb(p))}{4\pi}}.\label{eq:DefinitionR-1}
\end{equation}
As a result, $r$ extends continuously to $0$ on $\mathcal{Z}$.
\end{rem*}
For any pair of smooth functions $U,V:\mathbb{R}\rightarrow\mathbb{R}$
with $U^{\prime},V^{\prime}\neq0$, in the new double null coordinate
pair 
\begin{equation}
(\bar{u},\bar{v})=(U(u),V(v)),\label{eq:GeneralCoordinateTransformation}
\end{equation}
the metric $g$ takes the form 
\begin{equation}
g=-\bar{\text{\textgreek{W}}}^{2}(\bar{u},\bar{v})d\bar{u}d\bar{v}+r^{2}(\bar{u},\bar{v})g_{\mathbb{S}^{2}},\label{eq:SphericallySymmetricMetricNewGauge}
\end{equation}
where 
\begin{gather}
\bar{\text{\textgreek{W}}}^{2}(\bar{u},\bar{v})=\frac{1}{U^{\prime}V^{\prime}}\text{\textgreek{W}}^{2}(U^{-1}(\bar{u}),V^{-1}(\bar{v})),\label{eq:NewOmega}\\
r(\bar{u},\bar{v})=r(U^{-1}(\bar{u}),V^{-1}(\bar{v})).\label{eq:NewR}
\end{gather}

We will also define the \emph{Hawking mass} $m:\mathcal{M}\rightarrow\mathbb{R}$
by the expression 
\begin{equation}
m=\frac{r}{2}\big(1-g(\nabla r,\nabla r)\big).
\end{equation}
 Viewed as a function on $\mathcal{U}$, $m$ is related to $\Omega$
and $r$ by the following formula 
\begin{equation}
m=\frac{r}{2}\big(1+4\text{\textgreek{W}}^{-2}\partial_{u}r\partial_{v}r\big)\Leftrightarrow\Omega^{2}=\frac{4\partial_{v}r(-\partial_{u}r)}{1-\frac{2m}{r}}.\label{eq:DefinitionHawkingMass}
\end{equation}

In any local coordinate chart $(y^{1},y^{2})$ on $\mathbb{S}^{2}$,
the non-zero Christoffel symbols of (\ref{eq:SphericallySymmetricMetric})
in the $(u,v,y^{1},y^{2})$ local coordinate chart on $\mathcal{M}\backslash\mathcal{Z}$
are computed as follows: 
\begin{gather}
\text{\textgreek{G}}_{uu}^{u}=\partial_{u}\log(\text{\textgreek{W}}^{2}),\hphantom{A}\text{\textgreek{G}}_{vv}^{v}=\partial_{v}\log(\text{\textgreek{W}}^{2}),\label{eq:ChristoffelSymbols}\\
\text{\textgreek{G}}_{AB}^{u}=\text{\textgreek{W}}^{-2}\partial_{v}(r^{2})(g_{\mathbb{S}^{2}})_{AB},\hphantom{\,}\text{\textgreek{G}}_{AB}^{v}=\text{\textgreek{W}}^{-2}\partial_{u}(r^{2})(g_{\mathbb{S}^{2}})_{AB},\nonumber \\
\text{\textgreek{G}}_{uB}^{A}=r^{-1}\partial_{u}r\text{\textgreek{d}}_{B}^{A},\hphantom{\,}\text{\textgreek{G}}_{vB}^{A}=r^{-1}\partial_{v}r\text{\textgreek{d}}_{B}^{A},\nonumber \\
\text{\textgreek{G}}_{BC}^{A}=(\text{\textgreek{G}}_{\mathbb{S}^{2}})_{BC}^{A},\nonumber 
\end{gather}
where the latin indices $A,B,C$ are associated to the spherical coordinates
$y^{1},y^{2}$, $\text{\textgreek{d}}_{B}^{A}$ is Kronecker delta
and $\text{\textgreek{G}}_{\mathbb{S}^{2}}$ are the Christoffel symbols
of the round sphere in the $(y^{1},y^{2})$ coordinate chart.

A fundamental example of a family of spherically symmetric spacetimes
admitting a globally defined double null coordinate pair is the family
of Anti-de~Sitter spacetimes $(\mathcal{M}_{AdS},g_{AdS}^{(\Lambda)})$,
parametrised by the value $\Lambda<0$ of the cosmological constant.
In polar coordinates $(t,\bar{r},\theta,\varphi)$ on $\mathcal{M}_{AdS}\simeq\mathbb{R}^{3+1}$,
$g_{AdS}^{(\Lambda)}$ takes the form 
\begin{equation}
g_{AdS}^{(\Lambda)}=-\big(1-\frac{1}{3}\Lambda\bar{r}^{2}\big)dt^{2}+\big(1-\frac{1}{3}\Lambda\bar{r}^{2}\big)^{-1}d\bar{r}^{2}+\bar{r}^{2}\big(d\theta^{2}+\sin^{2}\theta d\varphi^{2}\big).\label{eq:AdSMetricPolarCoordinates}
\end{equation}
In this paper, we will usually drop the superscript $\Lambda$ in
(\ref{eq:AdSMetricPolarCoordinates}). The standard polar coordinate
pair $(u,v)$ on $(\mathcal{M}_{AdS},g_{AdS})$, defined by 
\begin{align*}
du & =dt-\frac{d\bar{r}}{1-\frac{1}{3}\Lambda\bar{r}^{2}},\\
dv & =dt+\frac{d\bar{r}}{1-\frac{1}{3}\Lambda\bar{r}^{2}},
\end{align*}
maps the manifold $\mathcal{M}_{AdS}\backslash\mathcal{Z}\simeq\mathbb{R}^{3+1}\backslash\{r=0\}$
to the planar domain 
\[
\mathcal{U}_{AdS}=\big\{ u<v<u+\sqrt{-\frac{3}{\Lambda}}\pi\big\}.
\]
In these coordinates, the metric coefficients (\ref{eq:SphericallySymmetricMetric})
associated to $g_{AdS}$ take the form (\ref{eq:AdSMetricValues-1}).

\subsection{\label{subsec:Asymptotically-AdS-spacetimes}Asymptotically AdS spacetimes }

Let $(\mathcal{M},g)$ be a spherically symmetric spacetime as in
Section \ref{subsec:Spherically-symmetric-spacetimes}. Recall that
$\mathcal{M}\backslash\mathcal{Z}$ splits topologically as the product
\[
\mathcal{M}\backslash\mathcal{Z}\simeq\mathcal{U}\times\mathbb{S}^{2}
\]
and, in any spherically symmetric double null coordinate chart on
$\mathcal{M}\backslash\mathcal{Z}$, the metric $g$ splits as (\ref{eq:SphericallySymmetricMetric}).
Note also that any choice of double null coordinate pair $(u,v)$
on $\mathcal{M}$ fixes a smooth embedding $(u,v):\mathcal{U}\rightarrow\mathbb{R}^{2}$.
In this section, we will identify $\mathcal{U}$ with its image in
$\mathbb{R}^{2}$ associated to a given null coordinate pair.

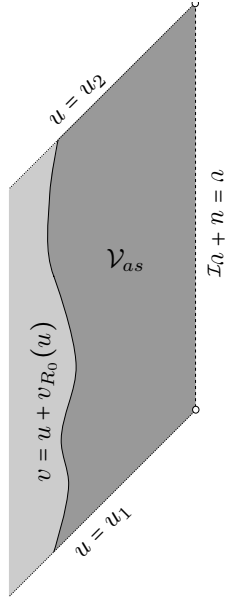
\begin{figure}[h] 
\centering 
\begingroup%
  \makeatletter%
  \providecommand\color[2][]{%
    \errmessage{(Inkscape) Color is used for the text in Inkscape, but the package 'color.sty' is not loaded}%
    \renewcommand\color[2][]{}%
  }%
  \providecommand\transparent[1]{%
    \errmessage{(Inkscape) Transparency is used (non-zero) for the text in Inkscape, but the package 'transparent.sty' is not loaded}%
    \renewcommand\transparent[1]{}%
  }%
  \providecommand\rotatebox[2]{#2}%
  \newcommand*\fsize{\dimexpr\f@size pt\relax}%
  \newcommand*\lineheight[1]{\fontsize{\fsize}{#1\fsize}\selectfont}%
  \ifx\svgwidth\undefined%
    \setlength{\unitlength}{150bp}%
    \ifx\svgscale\undefined%
      \relax%
    \else%
      \setlength{\unitlength}{\unitlength * \real{\svgscale}}%
    \fi%
  \else%
    \setlength{\unitlength}{\svgwidth}%
  \fi%
  \global\let\svgwidth\undefined%
  \global\let\svgscale\undefined%
  \makeatother%
  \begin{picture}(1,1.5)%
    \lineheight{1}%
    \setlength\tabcolsep{0pt}%
    \put(0,0){\includegraphics[width=\unitlength,page=1]{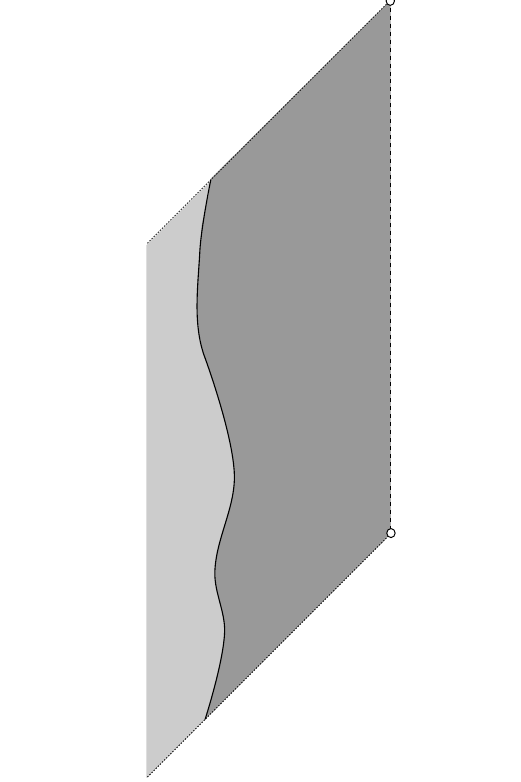}}%
    \put(0.46447705,0.10161798){\color[rgb]{0,0,0}\rotatebox{45}{\makebox(0,0)[lt]{\lineheight{1.25}\smash{\begin{tabular}[t]{l}$u=u_1$\end{tabular}}}}}%
    \put(0.40012062,1.1878184){\color[rgb]{0,0,0}\rotatebox{45}{\makebox(0,0)[lt]{\lineheight{1.25}\smash{\begin{tabular}[t]{l}$u=u_2$\end{tabular}}}}}%
    \put(0.52599007,0.83539608){\color[rgb]{0,0,0}\makebox(0,0)[lt]{\lineheight{1.25}\smash{\begin{tabular}[t]{l}$\mathcal{V}_{as}$\end{tabular}}}}%
    \put(0.79085451,1.08158867){\color[rgb]{0,0,0}\rotatebox{-90}{\makebox(0,0)[lt]{\lineheight{1.25}\smash{\begin{tabular}[t]{l}$v=u+v_{\mathcal{I}}$\end{tabular}}}}}%
    \put(0.38483024,0.29399715){\color[rgb]{0,0,0}\rotatebox{90}{\makebox(0,0)[lt]{\lineheight{1.25}\smash{\begin{tabular}[t]{l}$v=u+v_{R_0}(u)$\end{tabular}}}}}%
  \end{picture}%
\endgroup%
 
\caption{Schematic depiction of the asymptotic region $\mathcal{V}_{as}=\{r \ge R_0 \gg 1 \}$ of an asymptotically AdS spacetime. The function $r$ extends to $+\infty$ on conformal infinity $\mathcal{I}=\{ v=u+v_{\mathcal{I}}\}$. \label{fig:Asymptotic_Domain}}
\end{figure}
\begin{defn}
\label{def:AsADS}Let $(\mathcal{M},g)$ be a spacetime as above with
$\sup_{\mathcal{M}}r=+\infty$. We will call $(\mathcal{M},g)$ \emph{asymptotically
AdS} if, for some $R_{0}\gg1$, there exists a spherically symmetric
double null coordinate pair $(u,v)$ on $\mathcal{M}$ covering the
whole region $\mathcal{V}_{as}\doteq\{r\ge R_{0}\}$, with the following
properties: 

\begin{enumerate}

\item The region $\mathcal{V}_{as}$ has the form
\[
\mathcal{V}_{as}=\big\{ u_{1}<u<u_{2}\big\}\cap\big\{ u+v_{R_{0}}(u)\le v<u+v_{\mathcal{I}}\big\}
\]
for some $u_{1}<u_{2}\in\mathbb{R}\cup\{\pm\infty\}$, $v_{\mathcal{I}}\in\mathbb{R}$
and $v_{R_{0}}:(u_{1},u_{2})\rightarrow\mathbb{R}$ with $v(u)<v_{\mathcal{I}}$
(see Figure \ref{fig:Asymptotic_Domain}).

\item The function $\frac{1}{r}$ on $\mathcal{U}$ extends smoothly
on 
\begin{equation}
\mathcal{I}\doteq\big\{ u_{1}<u<u_{2}\big\}\cap\big\{ v=u+v_{\mathcal{I}}\big\}\subset clos(\mathcal{U})\label{eq:TimelikeInfinity}
\end{equation}
with 
\begin{equation}
\frac{1}{r}\Big|_{\mathcal{I}}=0.\label{eq:1/r0}
\end{equation}

\item The renormalised metric coefficient $r^{-2}\Omega^{2}$ extends
smoothly on $\mathcal{I}$, with 
\begin{equation}
r^{-2}\Omega^{2}\big|_{\mathcal{I}}\neq0.\label{eq:ConformalExtensionI}
\end{equation}

\end{enumerate}
\end{defn}
\begin{rem*}
The boundary condition (\ref{eq:1/r0}) implies that 
\begin{equation}
\partial_{v}(\frac{1}{r})\Big|_{\mathcal{I}}=-\partial_{u}(\frac{1}{r})\Big|_{\mathcal{I}}.\label{eq:BoundaryConditionRInfinity}
\end{equation}
\end{rem*}
In the class of spacetimes introduced in Section \ref{subsec:Spherically-symmetric-spacetimes},
Definition \ref{def:AsADS} coincides with the standard definition
of asymptotically AdS spacetimes (see e.\,g.~\cite{Friedrich1995}).
In particular, for an asymptotically AdS spacetime $(\mathcal{M},g)$
as above, the conformal metric 
\[
\tilde{g}=r^{-2}g
\]
in the region $\{r\ge R_{0}\}$ admits a smooth extension across $r=+\infty$,
with $\mathcal{I}^{(2+1)}=\{r=+\infty\}$ corresponding to a timelike
conformal boundary of $(\mathcal{M},g)$. The isometric action of
$SO(3)$ on $(\mathcal{M},g)$ by rotations extends to a smooth isometric
action on the manifold with boundary $(\widetilde{\mathcal{M}},\tilde{g})$,
where
\[
\widetilde{\mathcal{M}}=\mathcal{M}\cup\mathcal{I}^{(2+1)}.
\]
 In this extension, $\mathcal{I}$ defined by (\ref{eq:TimelikeInfinity})
corresponds to the spherical quotient of $\mathcal{I}^{(2+1)}$. We
will use the term \emph{conformal infinity} for both $\mathcal{I}^{(2+1)}$
and $\mathcal{I}$.

\subsection{\label{subsec:VlasovEquations}The massless Vlasov equation in spherical
symmetry}

Let $(\mathcal{M},g)$ be as in Section \ref{subsec:Spherically-symmetric-spacetimes}.
In any local coordinate chart $(x^{0},x^{1},x^{2},x^{3})$ on $\mathcal{M}$
with associated momentum coordinates $(p^{0},p^{1},p^{2},p^{3})$
on the fibers of $T\mathcal{M}$, the \emph{geodesic flow} takes the
form 
\begin{equation}
\begin{cases}
\frac{dx^{\alpha}}{ds}=p^{\alpha},\\
\frac{dp^{a}}{ds}+\text{\textgreek{G}}_{\beta\gamma}^{\alpha}p^{\beta}p^{\gamma}=0,
\end{cases}\label{eq:GeodesicFlow}
\end{equation}
 where $\text{\textgreek{G}}_{\text{\textgreek{b}\textgreek{g}}}^{\text{\textgreek{a}}}$
are the Christoffel symbols of $g$ in the coordinates $(x^{0},x^{1},x^{2},x^{3})$.
The set $\mathcal{P}^{+}\subset T\mathcal{M}$ of future directed
null vectors, i.\,e.~the subset of $T\mathcal{M}$ where
\begin{equation}
g_{\text{\textgreek{a}\textgreek{b}}}(x)p^{\text{\textgreek{a}}}p^{\text{\textgreek{b}}}=0\text{ and }g_{\text{\textgreek{a}\textgreek{b}}}(x)p^{\text{\textgreek{a}}}Q^{\text{\textgreek{b}}}(x)\le0,\label{eq:NullShell}
\end{equation}
where $Q$ is a fixed, non-vanishing future directed vector field
on $\mathcal{M}$, is preserved under (\ref{eq:GeodesicFlow}).

In view of the spherical symmetry of $(\mathcal{M},g)$, the \emph{angular
momentum} function $l:T\mathcal{M}\rightarrow[0,+\infty)$ defined
(in a local coordinate chart $(u,v,y^{1},y^{2})$ as in Section \ref{subsec:Spherically-symmetric-spacetimes})
by
\begin{equation}
l^{2}\doteq r^{2}g_{AB}p^{A}p^{B}=r^{4}(g_{\mathbb{S}^{2}})_{AB}p^{A}p^{B}\label{eq:AngularMomentum}
\end{equation}
is a constant of motion for the geodesic flow (\ref{eq:GeodesicFlow}).\footnote{Note that (\ref{eq:AngularMomentum}) is coordinate independent.}
Thus, the geodesic flow equations (\ref{eq:GeodesicFlow}) can be
reduced to a system depending only on the parameters $u$, $v$, $p^{u}$,
$p^{v}$ and $l$. In terms of these parameters, condition (\ref{eq:NullShell})
is expressed as 
\begin{equation}
\Omega^{2}p^{u}p^{v}=\frac{l^{2}}{r^{2}}\text{ and }\text{ }p^{u}\ge0\label{eq:NullShellAngularMomentum}
\end{equation}
 while (\ref{eq:GeodesicFlow}) for null geodesics is reduced to
\begin{equation}
\begin{cases}
\frac{du}{ds}=p^{u},\\
\frac{dv}{ds}=p^{v},\\
\frac{d}{ds}\big(\Omega^{2}p^{u}\big)=\Big(\partial_{v}\log(\Omega^{2})-2\frac{\partial_{v}r}{r}\Big)\frac{l^{2}}{r^{2}},\\
\frac{d}{ds}\big(\Omega^{2}p^{v}\big)=\Big(\partial_{u}\log(\Omega^{2})-2\frac{\partial_{u}r}{r}\Big)\frac{l^{2}}{r^{2}},\\
\frac{dl}{ds}=0.
\end{cases}\label{eq:NullGeodesicsSphericalSymmetry}
\end{equation}
Note that the relations (\ref{eq:NullShellAngularMomentum}) and (\ref{eq:NullGeodesicsSphericalSymmetry})
imply that, on a smooth spacetime $(\mathcal{M},g)$ as above, a geodesic
$\gamma$ with angular momentum $l>0$ cannot reach the axis $r=0$.
\begin{rem*}
We will frequently identify a geodesic $\gamma$ in $\mathcal{M}$
with its image in $\mathcal{U}$. For this reason, we will frequently
refer to (\ref{eq:NullGeodesicsSphericalSymmetry}) simply as the
equations of motion for a geodesic in $\mathcal{U}$.
\end{rem*}
A \emph{Vlasov field} $f$ on $(\mathcal{M},g)$ is a non-negative
measure on $T\mathcal{M}$ which is invariant under the geodesic flow
(\ref{eq:GeodesicFlow}). In particular, in any coordinate chart $(x^{\alpha};p^{\alpha})$
on $T\mathcal{M}$, a Vlasov field $f$ satisfies (as a distribution)
the equation 
\begin{equation}
p^{\alpha}\partial_{x^{\alpha}}f-\text{\textgreek{G}}_{\beta\gamma}^{\alpha}p^{\beta}p^{\gamma}\partial_{p^{\alpha}}f=0.\label{eq:Vlasov}
\end{equation}
A Vlasov field $f$ supported on (\ref{eq:NullShell}) will be called
a \emph{massless Vlasov field}.

Associated to any Vlasov field $f$ is a symmetric $(0,2)$-form on
$\mathcal{M}$, the \emph{energy momentum }tensor of $f$, given (formally)
by the expression 
\begin{equation}
T_{\alpha\beta}(x)\doteq\int_{\pi^{-1}(x)}p_{\alpha}p_{\beta}f\,\sqrt{-det(g(x))}dp^{0}\cdots dp^{3},\label{eq:EnergyMomentumTensor}
\end{equation}
where $\pi^{-1}(x)$ denotes the fiber of $T\mathcal{M}$ over $x\in\mathcal{M}$
and the indices of the momentum coordinates are lowered with the use
of the metric $g$, i.\,e.
\begin{equation}
p_{\gamma}=g_{\gamma\delta}(x)p^{\delta}.
\end{equation}
Equation (\ref{eq:Vlasov}) implies that $T_{\alpha\beta}$ is conserved,
i.\,e.: 
\begin{equation}
\nabla^{\alpha}T_{\alpha\beta}=0.\label{eq:ConservationEnergyMomentum}
\end{equation}
Furthermore, for any Vlasov field $f$ on $(\mathcal{M},g)$, associated
to each open set $\mathcal{V}\subseteq T\mathcal{M}$ is a 1-form
$N^{(\mathcal{V})}$ called the the \emph{particle current}, expressed
in any coordinate chart $(x^{\alpha};p^{\alpha})$ on $T\mathcal{M}$
as
\begin{equation}
N_{a}^{(\mathcal{V})}(x)\doteq\int_{T_{x}\mathcal{M}\cap\mathcal{V}}p_{a}f\,\sqrt{-det(g(x))}dp^{0}\cdots dp^{3}.\label{eq:ParticleCurrent}
\end{equation}
In the case when $\mathcal{V}$ is invariant under the geodesic flow,
(\ref{eq:ParticleCurrent}) is conserved: 
\begin{equation}
\nabla^{\alpha}N_{a}^{(\mathcal{V})}=0.\label{eq:ParticleConservation}
\end{equation}
We will denote
\begin{equation}
N\doteq N^{(T\mathcal{M})}.\label{eq:TotalParticleCurrent}
\end{equation}
\begin{rem*}
We will only consider smooth Vlasov fields $f$ which decay sufficiently
fast as $p^{a}\rightarrow+\infty$ (for any fixed $x$). In this case,
the expressions (\ref{eq:EnergyMomentumTensor}) and (\ref{eq:TotalParticleCurrent})
are finite and depend smoothly on $x\in\mathcal{M}$. 
\end{rem*}
It can be readily inferred that a \emph{spherically symmetric} Vlasov
field $f$, i.\,e.~a Vlasov field which is invariant under the induced
action of $SO(3)$ on $T\mathcal{M}$, only depends on $u$, $v$,
$p^{u}$, $p^{v}$ and $l$, and, in the massless case, is conserved
along the flow lines of the reduced system (\ref{eq:NullGeodesicsSphericalSymmetry}).
In particular, a \emph{smooth }spherically symmetric massless Vlasov
field $f$ will necessarily be of the form
\begin{equation}
f(u,v;p^{u},p^{v},l)=\bar{f}(u,v;p^{u},p^{v},l)\cdot\delta\big(\Omega^{2}p^{u}p^{v}-\frac{l^{2}}{r^{2}}\big),\label{eq:FAsASmoothDeltaFunction}
\end{equation}
where $\bar{f}$ is smooth in its variables and $\delta$ is Dirac's
delta function. For a spherically symmetric Vlasov field $f$, equation
(\ref{eq:Vlasov}) takes the form:
\begin{equation}
p^{u}\partial_{u}f+p^{v}\partial_{v}f=\Big(\partial_{u}\log(\Omega^{2})(p^{u})^{2}+\frac{2}{r}\Omega^{-2}\partial_{v}r\frac{l^{2}}{r^{2}}\Big)\partial_{p^{u}}f+\Big(\partial_{v}\log(\Omega^{2})(p^{v})^{2}+\frac{2}{r}\Omega^{-2}\partial_{u}r\frac{l^{2}}{r^{2}}\Big)\partial_{p^{v}}f.\label{eq:VlasovEquationSphericalSymmetry}
\end{equation}
\begin{rem*}
Given a smooth, spherically symmetric massless Vlasov field $f$,
we will frequently denote with $\bar{f}$ any smooth function for
which (\ref{eq:FAsASmoothDeltaFunction}) holds. Note that $\bar{f}$
is uniquely determined only along the future null set $\mathcal{P}^{+}$.
\end{rem*}
For a smooth, spherically symmetric massless Vlasov field, the energy-momentum
tensor (\ref{eq:EnergyMomentumTensor}) is of the form 
\begin{equation}
T=T_{uu}(u,v)du^{2}+2T_{uv}(u,v)dudv+T_{vv}(u,v)dv^{2}+T_{AB}(u,v)dy^{A}dy^{B}\label{eq:SphericallySymmetricTensor}
\end{equation}
and the components of (\ref{eq:SphericallySymmetricTensor}) can be
expressed in terms of the variables $p^{u},p^{v},l$ as
\begin{align}
T_{uu} & =\frac{\pi}{2}r^{-2}\int_{0}^{+\infty}\int_{0}^{+\infty}\big(\Omega^{2}p^{v}\big)^{2}\bar{f}(u,v;p^{u},p^{v},l)\Big|_{\Omega^{2}p^{u}p^{v}=\frac{l^{2}}{r^{2}}}\,\frac{dp^{u}}{p^{u}}ldl,\label{eq:SphericallySymmetricComponentsEnergyMomentum}\\
T_{vv} & =\frac{\pi}{2}r^{-2}\int_{0}^{+\infty}\int_{0}^{+\infty}\big(\Omega^{2}p^{u}\big)^{2}\bar{f}(u,v;p^{u},p^{v},l)\Big|_{\Omega^{2}p^{u}p^{v}=\frac{l^{2}}{r^{2}}}\,\frac{dp^{u}}{p^{u}}ldl,\nonumber \\
T_{uv}=\frac{1}{4}\Omega^{2}g^{AB}T_{AB} & =\frac{\pi}{2}r^{-2}\int_{0}^{+\infty}\int_{0}^{+\infty}\big(\Omega^{2}p^{u}\big)\cdot\big(\Omega^{2}p^{v}\big)\bar{f}(u,v;p^{u},p^{v},l)\Big|_{\Omega^{2}p^{u}p^{v}=\frac{l^{2}}{r^{2}}}\,\frac{dp^{u}}{p^{u}}ldl.\nonumber 
\end{align}
Similarly, in this case, provided $\mathcal{V}\subseteq T\mathcal{M}$
is invariant under the action of $SO(3)$ on $T\mathcal{M}$, (\ref{eq:ParticleCurrent})
takes the form 
\begin{equation}
N^{(\mathcal{V})}=N_{u}^{(\mathcal{V})}du+N_{v}^{(\mathcal{V})}dv,
\end{equation}
with 
\begin{align}
N_{u}^{(\mathcal{V})} & =\pi r^{-2}\int_{\mathcal{V}\cap T_{x}\mathcal{M}\cap\{\Omega^{2}p^{u}p^{v}=\frac{l^{2}}{r^{2}}\}\cap\{p^{u}\ge0\}}\Omega^{2}p^{v}\bar{f}(u,v;p^{u},p^{v},l)\Big|_{\Omega^{2}p^{u}p^{v}=\frac{l^{2}}{r^{2}}}\,\frac{dp^{u}}{p^{u}}ldl,\label{eq:SphericallySymmetricParticleCurrent}\\
N_{v}^{(\mathcal{V})} & =\pi r^{-2}\int_{\mathcal{V}\cap T_{x}\mathcal{M}\cap\{\Omega^{2}p^{u}p^{v}=\frac{l^{2}}{r^{2}}\}\cap\{p^{u}\ge0\}}\Omega^{2}p^{u}\bar{f}(u,v;p^{u},p^{v},l)\Big|_{\Omega^{2}p^{u}p^{v}=\frac{l^{2}}{r^{2}}}\,\frac{dp^{u}}{p^{u}}ldl.\nonumber 
\end{align}

For any smooth and spherically symmetric massless Vlasov field $f$
, any open set $\mathcal{V}\subseteq T\mathcal{M}$ which is invariant
under the action of $SO(3)$ on $T\mathcal{M}$ and any $l\ge0$,
we will also introduce the quantities
\begin{equation}
N_{\mu}^{(\mathcal{V};l)}=2\pi r^{-2}\int_{\mathcal{V}^{(l)}\cap T_{x}\mathcal{M}\cap\{\Omega^{2}p^{u}p^{v}=\frac{l^{2}}{r^{2}}\}}p_{\mu}\bar{f}(u,v;p^{u},p^{v},l)\Big|_{\Omega^{2}p^{u}p^{v}=\frac{l^{2}}{r^{2}}}\,\frac{dp^{u}}{p^{u}}\label{eq:ReducedParticleCurrent}
\end{equation}
(where $\mathcal{V}^{(l_{0})}\doteq\mathcal{V}\cap\{l=l_{0}\}$) and
\begin{equation}
T_{\mu\nu}^{(l)}=2\pi r^{-2}\int_{0}^{+\infty}p_{\mu}p_{\nu}\bar{f}(u,v;p^{u},p^{v},l)\Big|_{\Omega^{2}p^{u}p^{v}=\frac{l^{2}}{r^{2}}}\,\frac{dp^{u}}{p^{u}}.\label{eq:ReducedEnergyMomentum}
\end{equation}
Note that, when $\mathcal{V}$ is invariant under the geodesic flow,
(\ref{eq:ReducedParticleCurrent}) is conserved, i.\,e.
\begin{equation}
\nabla^{\mu}N_{\mu}^{(\mathcal{V};l)}=0.\label{eq:ReducedParticleConservation}
\end{equation}
Note also the relations 
\begin{equation}
N_{\mu}^{(\mathcal{V})}=\int_{0}^{+\infty}N_{\mu}^{(\mathcal{V};l)}\,ldl,\text{ }T_{\mu\nu}=\int_{0}^{+\infty}T_{\mu\nu}^{(l)}\,ldl.\label{eq:TotalCurrentsFromReduced}
\end{equation}
We will also denote 
\begin{equation}
N_{\mu}^{(l)}\doteq N_{\mu}^{(T\mathcal{M};l)}.\label{eq:TotalReducedParticleCurrent}
\end{equation}

Combining the expressions (\ref{eq:DefinitionHawkingMass}), (\ref{eq:ReducedParticleCurrent}),
(\ref{eq:ReducedEnergyMomentum}), we can readily estimate for any
$l>0$:
\begin{equation}
\frac{1-\frac{2m}{r}}{\partial_{v}r}T_{vv}^{(l)}(u,v)+\frac{1-\frac{2m}{r}}{-\partial_{u}r}T_{uv}^{(l)}(u,v)\le2\sup_{supp\big(f(u,v;\cdot,\cdot,l)\big)}\Big(\partial_{v}r(u,v)p^{v}-\partial_{u}r(u,v)p^{u}\Big)\cdot N_{v}^{(l)}(u,v)\label{eq:VDerivativeMassfromReducedParticleCurrent}
\end{equation}
and 
\begin{equation}
\frac{1-\frac{2m}{r}}{\partial_{v}r}T_{uv}^{(l)}(u,v)+\frac{1-\frac{2m}{r}}{-\partial_{u}r}T_{uu}^{(l)}(u,v)\le2\sup_{supp\big(f(u,v;\cdot,\cdot,l)\big)}\Big(\partial_{v}r(u,v)p^{v}-\partial_{u}r(u,v)p^{u}\Big)\cdot N_{u}^{(l)}(u,v).\label{eq:UDerivativeMassFromReducedParticleCurrent}
\end{equation}

\subsection{\label{subsec:The-Einstein-equations}The spherically symmetric Einstein\textendash massless
Vlasov system}

Let $(\mathcal{M}^{3+1},g)$ be a smooth Lorentzian manifold and let
$f$ be a non-negative measure on $T\mathcal{M}$. The \emph{Einstein\textendash Vlasov}
system for $(\mathcal{M},g;f)$ with a cosmological constant $\Lambda<0$
is 
\begin{equation}
\begin{cases}
Ric_{\mu\nu}(g)-\frac{1}{2}R(g)g_{\mu\nu}+\Lambda g_{\mu\nu}=8\pi T_{\mu\nu},\\
p^{\alpha}\partial_{x^{\alpha}}f-\text{\textgreek{G}}_{\beta\gamma}^{\alpha}p^{\beta}p^{\gamma}\partial_{p^{\alpha}}f=0,
\end{cases}\label{eq:EinsteinVlasovEquations}
\end{equation}
where $T_{\mu\nu}$ is expressed in terms of $f$ by (\ref{eq:EnergyMomentumTensor})
(see also \cite{DafermosRendall,MoschidisMaximalDevelopment,MoschidisNullDust}). 

In the case when $(\mathcal{M},g)$ is a spherically symmetric spacetime
as in Section \ref{subsec:Spherically-symmetric-spacetimes} and $f$
is a spherically symmetric massless Vlasov field (see Section \ref{subsec:VlasovEquations}),
(\ref{eq:EinsteinVlasovEquations}) is equivalent to the following
system for $(r,\Omega^{2},f)$:
\begin{align}
\partial_{u}\partial_{v}(r^{2})= & -\frac{1}{2}(1-\Lambda r^{2})\Omega^{2}+8\pi r^{2}T_{uv},\label{eq:RequationFinal}\\
\partial_{u}\partial_{v}\log(\Omega^{2})= & \frac{\Omega^{2}}{2r^{2}}\big(1+4\Omega^{-2}\partial_{u}r\partial_{v}r\big)-8\pi T_{uv}-2\pi\Omega^{2}g^{AB}T_{AB},\label{eq:OmegaEquationFinal}\\
\partial_{v}(\Omega^{-2}\partial_{v}r)= & -4\pi rT_{vv}\Omega^{-2},\label{eq:ConstraintVFinal}\\
\partial_{u}(\Omega^{-2}\partial_{u}r)= & -4\pi rT_{uu}\Omega^{-2},\label{eq:ConstraintUFinal}\\
p^{u}\partial_{u}f+p^{v}\partial_{v}f= & \Big(\partial_{u}\log(\Omega^{2})(p^{u})^{2}+\frac{2}{r}\Omega^{-2}\partial_{v}r\frac{l^{2}}{r^{2}}\Big)\partial_{p^{u}}f+\label{eq:VlasovFinal}\\
 & \hphantom{\Big(\partial_{u}}+\Big(\partial_{v}\log(\Omega^{2})(p^{v})^{2}+\frac{2}{r}\Omega^{-2}\partial_{u}r\frac{l^{2}}{r^{2}}\Big)\partial_{p^{v}}f,\nonumber \\
supp(f)\subseteq & \big\{\Omega^{2}p^{u}p^{v}-\frac{l^{2}}{r^{2}}=0,\text{ }p^{u}\ge0\big\}.\label{NullShellFinal}
\end{align}
Note that $(r,\Omega^{2},f)=(r_{AdS},\Omega_{AdS}^{2},0)$ (where
$r_{AdS}$, $\Omega_{AdS}^{2}$ are given by (\ref{eq:AdSMetricValues-1}))
is a trivial solution for the system (\ref{eq:RequationFinal})\textendash (\ref{NullShellFinal}).

Defining the renormalised Hawking mass as 
\begin{equation}
\tilde{m}\doteq m-\frac{1}{6}\Lambda r^{3},\label{eq:RenormalisedHawkingMass}
\end{equation}
 and using the relation (\ref{eq:DefinitionHawkingMass}), the constraint
equations (\ref{eq:ConstraintVFinal})\textendash (\ref{eq:ConstraintUFinal})
are equivalent (in the region of $\mathcal{M}$ where $\partial_{v}r>0$,
$\partial_{u}r<0$ and $1-\frac{2m}{r}>0$) to 
\begin{align}
\partial_{u}\log\big(\frac{\partial_{v}r}{1-\frac{2m}{r}}\big)= & -4\pi r^{-1}\frac{r^{2}T_{uu}}{-\partial_{u}r},\label{eq:DerivativeInUDirectionKappa}\\
\partial_{v}\log\big(\frac{-\partial_{u}r}{1-\frac{2m}{r}}\big)= & 4\pi r^{-1}\frac{r^{2}T_{vv}}{\partial_{v}r}.\label{eq:DerivativeInVDirectionKappaBar}
\end{align}
Equations (\ref{eq:RequationFinal})\textendash (\ref{eq:VlasovFinal})
also formally give rise to the following set of equations for $r$,
$\tilde{m}$:

\begin{align}
\partial_{u}\partial_{v}r= & -\frac{2\tilde{m}-\frac{2}{3}\Lambda r^{3}}{r^{2}}\frac{(-\partial_{u}r)\partial_{v}r}{1-\frac{2m}{r}}+4\pi rT_{uv},\label{eq:EquationRForProof}\\
\partial_{u}\tilde{m}= & -2\pi\big(1-\frac{2m}{r}\big)\Big(\frac{r^{2}T_{uu}}{-\partial_{u}r}+\frac{r^{2}T_{uv}}{\partial_{v}r}\Big),\label{eq:DerivativeTildeUMass}\\
\partial_{v}\tilde{m}= & 2\pi\big(1-\frac{2m}{r}\big)\Big(\frac{r^{2}T_{vv}}{\partial_{v}r}+\frac{r^{2}T_{uv}}{-\partial_{u}r}\Big),\label{eq:DerivativeTildeVMass}
\end{align}
Note that, in view of the relation $4\Omega^{-2}T_{uv}=g^{AB}T_{AB}$,
(\ref{eq:OmegaEquationFinal}) is equivalent to 
\begin{equation}
\partial_{u}\partial_{v}\log(\Omega^{2})=4\frac{m}{r^{3}}\frac{(-\partial_{u}r)\partial_{v}r}{1-\frac{2m}{r}}-16\pi T_{uv}.\label{eq:EquationOmegaForProof}
\end{equation}

When considering asymptotically AdS solutions of the system (\ref{eq:RequationFinal})\textendash (\ref{NullShellFinal}),
it will be useful to consider the following renormalised quantities
near $\mathcal{I}$:
\begin{align}
\widetilde{\Omega}^{2} & \doteq\frac{\Omega^{2}}{1-\frac{1}{3}\Lambda r^{2}},\label{eq:RenormalisedQuantities}\\
\rho & \doteq\tan^{-1}\big(\sqrt{-\frac{\Lambda}{3}}r\big),\nonumber \\
\tau_{\mu\nu} & \doteq r^{2}T_{\mu\nu}.\nonumber 
\end{align}
The quantities $(\widetilde{\Omega}^{2},\rho,\tau_{\mu\nu})$ satisfy
the following renormalised equations (readily obtained from (\ref{eq:EquationRForProof})
and (\ref{eq:EquationOmegaForProof})):
\begin{align}
\partial_{u}\partial_{v}\log(\widetilde{\Omega}^{2}) & =\frac{\tilde{m}}{r}\Big(\frac{1}{r^{2}}+\frac{1}{3}\Lambda\frac{\Lambda r^{2}-1}{1-\frac{1}{3}\Lambda r^{2}}\Big)\widetilde{\Omega}^{2}-16\pi\frac{1-\frac{1}{2}\Lambda r^{2}}{1-\frac{1}{3}\Lambda r^{2}}r^{-2}\tau_{uv},\label{eq:RenormalisedEquations}\\
\partial_{u}\partial_{v}\rho & =-\frac{1}{2}\sqrt{-\frac{\Lambda}{3}}\frac{\tilde{m}}{r^{2}}\frac{1-\frac{2}{3}\Lambda r^{2}}{1-\frac{1}{3}\Lambda r^{2}}\widetilde{\Omega}^{2}+4\pi\sqrt{-\frac{\Lambda}{3}}\frac{1}{r-\frac{1}{3}\Lambda r^{3}}\tau_{uv}.\nonumber 
\end{align}

The following relation will be useful throughout this paper: Let $u_{1}(v)$
be a given function of $v$ and let $\gamma:[0,a)\rightarrow\mathcal{M}$
be a null geodesic contained in the region $\{u\ge u_{1}(v)\}$ and
having non-zero angular momentum $l$, such that $\gamma(0)$ lies
on the curve $\{u=u_{1}(v)\}$. Then, the projection of $\gamma$
on $\mathcal{U}$ will be a strictly timelike curve with respect to
the reference metric 
\begin{equation}
g_{ref}=-dudv\label{eq:ReferenceMetric}
\end{equation}
on $\mathcal{U}$, and the equations of motion (\ref{eq:NullGeodesicsSphericalSymmetry})
(combined with the relation (\ref{eq:NullShellAngularMomentum}) for
null vectors) imply that, for any $s\in[0,a)$:

\begin{align}
\log\big(\Omega^{2}\dot{\gamma}^{u}\big)(s)-\log\big(\Omega^{2}\dot{\gamma}^{u}\big)(0)= & \int_{\text{\textgreek{g}}([0,s]))}\Big(\partial_{v}\log(\Omega^{2})-2\frac{\partial_{v}r}{r}\Big)\,dv=\label{eq:UsefulRelationForGeodesic1}\\
= & \int_{v(\gamma(0))}^{v(\gamma(s))}\int_{u_{1}(v)}^{u(\gamma(s_{v}))}\Big(\partial_{u}\partial_{v}\log(\Omega^{2})-2\partial_{u}\frac{\partial_{v}r}{r}\Big)\,dudv+\nonumber \\
 & +\int_{v(\gamma(0))}^{v(\gamma(s))}\big(\partial_{v}\log(\Omega^{2})-2\frac{\partial_{v}r}{r}\big)(u_{1}(v),v)\,dv,\nonumber 
\end{align}
where, for any $\bar{v}\in(v(\gamma(0)),v(\gamma(s)))$, $s_{\bar{v}}$
denotes the value of the parameter $s$ for which 
\begin{equation}
v(\gamma(s_{\bar{v}}))=\bar{v}.\label{eq:s_v}
\end{equation}
In view of the evolution equations (\ref{eq:EquationRForProof}),
(\ref{eq:EquationOmegaForProof}) for $r,\Omega^{2}$ and the definition
(\ref{eq:RenormalisedHawkingMass}) of $\tilde{m}$, the relation
(\ref{eq:UsefulRelationForGeodesic1}) can be expressed as 
\begin{align}
\log\big(\Omega^{2}\dot{\gamma}^{u}\big)(s)-\log\big(\Omega^{2}\dot{\gamma}^{u}\big)(0)= & \int_{v(\gamma(0))}^{v(\gamma(s))}\int_{u_{1}(v)}^{u(\gamma(s_{v}))}\Big(\frac{1}{2}\frac{\frac{6\tilde{m}}{r}-1}{r^{2}}\Omega^{2}-24\pi T_{uv}\Big)\,dudv+\label{eq:UsefulRelationForGeodesicWithMu-U}\\
 & +\int_{v(\gamma(0))}^{v(\gamma(s))}\big(\partial_{v}\log(\Omega^{2})-2\frac{\partial_{v}r}{r}\big)(u_{1}(v),v)\,dv\nonumber 
\end{align}
(see also Figure \ref{fig:Geodesic_Integration}).

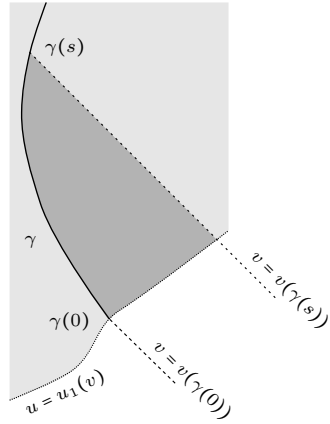
\begin{figure}[h] 
\centering 
\scriptsize
\begingroup%
  \makeatletter%
  \providecommand\color[2][]{%
    \errmessage{(Inkscape) Color is used for the text in Inkscape, but the package 'color.sty' is not loaded}%
    \renewcommand\color[2][]{}%
  }%
  \providecommand\transparent[1]{%
    \errmessage{(Inkscape) Transparency is used (non-zero) for the text in Inkscape, but the package 'transparent.sty' is not loaded}%
    \renewcommand\transparent[1]{}%
  }%
  \providecommand\rotatebox[2]{#2}%
  \newcommand*\fsize{\dimexpr\f@size pt\relax}%
  \newcommand*\lineheight[1]{\fontsize{\fsize}{#1\fsize}\selectfont}%
  \ifx\svgwidth\undefined%
    \setlength{\unitlength}{150bp}%
    \ifx\svgscale\undefined%
      \relax%
    \else%
      \setlength{\unitlength}{\unitlength * \real{\svgscale}}%
    \fi%
  \else%
    \setlength{\unitlength}{\svgwidth}%
  \fi%
  \global\let\svgwidth\undefined%
  \global\let\svgscale\undefined%
  \makeatother%
  \begin{picture}(1,1.1)%
    \lineheight{1}%
    \setlength\tabcolsep{0pt}%
    \put(0,0){\includegraphics[width=\unitlength,page=1]{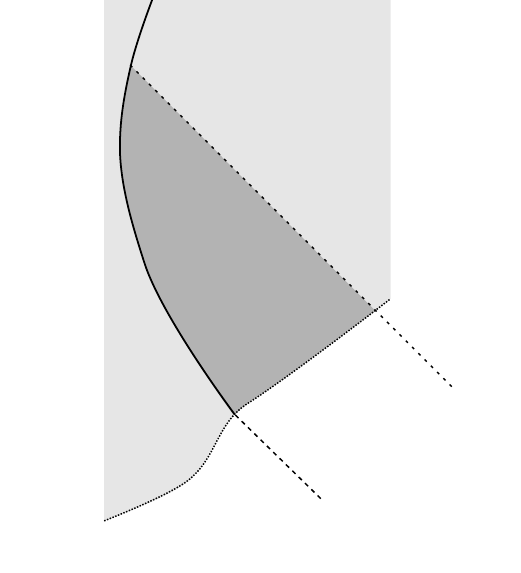}}%
    \put(0.23762377,0.49381178){\color[rgb]{0,0,0}\makebox(0,0)[lt]{\lineheight{1.25}\smash{\begin{tabular}[t]{l}$\gamma$\end{tabular}}}}%
    \put(0.28836633,0.97524738){\color[rgb]{0,0,0}\makebox(0,0)[lt]{\lineheight{1.25}\smash{\begin{tabular}[t]{l}$\gamma (s)$\end{tabular}}}}%
    \put(0.29950493,0.28589105){\color[rgb]{0,0,0}\makebox(0,0)[lt]{\lineheight{1.25}\smash{\begin{tabular}[t]{l}$\gamma (0)$\end{tabular}}}}%
    \put(0.24679985,0.06206183){\color[rgb]{0,0,0}\rotatebox{25}{\makebox(0,0)[lt]{\lineheight{1.25}\smash{\begin{tabular}[t]{l}$u=u_1(v)$\end{tabular}}}}}%
    \put(0.80420483,0.46035699){\color[rgb]{0,0,0}\rotatebox{-45}{\makebox(0,0)[lt]{\lineheight{1.25}\smash{\begin{tabular}[t]{l}$v=v(\gamma (s))$\end{tabular}}}}}%
    \put(0.55915533,0.23263432){\color[rgb]{0,0,0}\rotatebox{-45}{\makebox(0,0)[lt]{\lineheight{1.25}\smash{\begin{tabular}[t]{l}$v=v(\gamma (0))$\end{tabular}}}}}%
  \end{picture}%
\endgroup%
 
\caption{For a null geodesic $\gamma$, the domain of integration appearing in the right hand side of (\ref{eq:UsefulRelationForGeodesicWithMu-U}) is as depicted above. \label{fig:Geodesic_Integration}}
\end{figure}
\begin{rem*}
In this paper, we will be mainly interested in the case when $u_{1}(v)$
is of the form 
\[
u_{1}(v)=\begin{cases}
u_{0}, & v\le v_{c}\\
v-v_{\mathcal{I}}, & v\ge v_{c}
\end{cases}
\]
for some constants $u_{0}$, $v_{\mathcal{I}}$ and $v_{c}$.
\end{rem*}
Similarly, inverting the roles of the $u,v$ variables, for any null
geodesic $\gamma:[0,a)\rightarrow\{v\ge v_{1}(u)\}$ with $\gamma(0)\in\{v=v_{1}(u)\}$
we calculate: 
\begin{align}
\log\big(\Omega^{2}\dot{\gamma}^{v}\big)(s)-\log\big(\Omega^{2}\dot{\gamma}^{v}\big)(0)= & \int_{u(\gamma(0))}^{u(\gamma(s))}\int_{v_{1}(u)}^{v(\gamma(s_{u}))}\Big(\frac{1}{2}\frac{\frac{6\tilde{m}}{r}-1}{r^{2}}\Omega^{2}-24\pi T_{uv}\Big)\,dvdu+\label{eq:UsefulRelationForGeodesicWithMu-V}\\
 & +\int_{u(\gamma(0))}^{u(\gamma(s))}\big(\partial_{u}\log(\Omega^{2})-2\frac{\partial_{u}r}{r}\big)(u,v_{1}(u))\,du,\nonumber 
\end{align}
where $s_{u}$ is defined by 
\begin{equation}
u(\gamma(s_{\bar{u}}))=\bar{u}.\label{eq:s_u}
\end{equation}

\subsection{\label{subsec:Reflective-boundary-conditions} Reflection of null
geodesics and the reflecting boundary condition on $\mathcal{I}$}

Let $(\mathcal{M},g)$ be a spacetime as in Section \ref{subsec:Asymptotically-AdS-spacetimes}.
Let $\gamma:[0,+\infty)\rightarrow\mathcal{M}$ be a future directed
null geodesic with respect to $g$, which is future inextendible and
satisfies 
\begin{equation}
\lim_{s\rightarrow+\infty}r\circ\gamma(s)=+\infty\label{eq:InfiniteRLimit}
\end{equation}
and 
\begin{equation}
\limsup_{s\rightarrow+\infty}u\circ\gamma(s)<\sup_{\mathcal{I}}u.\label{eq:BoundedULimit}
\end{equation}
From the expression (\ref{eq:ChristoffelSymbols}) of the Christoffel
symbols and the fact that $r^{-2}\text{\textgreek{W}}^{2}$ extends
smoothly on $\mathcal{I}$ and satisifes (\ref{eq:ConformalExtensionI}),
it can be readily deduced that the curve 
\[
(\gamma,r^{2}\dot{\gamma}):[0,+\infty)\rightarrow T\widetilde{\mathcal{M}}
\]
has a regular limit on $\mathcal{I}^{(2+1)}$:
\begin{equation}
\lim_{s\rightarrow+\infty}(\gamma(s),r^{2}\dot{\gamma}(s))=(q,v)\in T|_{\mathcal{I}^{(2+1)}}\widetilde{\mathcal{M}},\label{eq:LimitOfGeodesic}
\end{equation}
with $v\in T_{q}\widetilde{\mathcal{M}}$ being a \emph{non-zero }null
vector with respect to $\tilde{g}$, satisfying 
\begin{equation}
\tilde{g}(v,n_{\mathcal{I}})>0,\label{eq:PositiveSignNormal}
\end{equation}
where $n_{\mathcal{I}}\in\text{\textgreek{G}}\big(T|_{\mathcal{I}^{(2+1)}}\widetilde{\mathcal{M}}\big)$
is the (spacelike) unit normal to $\mathcal{I}^{(2+1)}$ pointing
outwards. An analogous statement also holds for past inxtendible,
future directed null geodesics $\gamma:(-\infty,0]\rightarrow\mathcal{M}$
with past limit on $\mathcal{I}^{(2+1)}$, in which case (\ref{eq:PositiveSignNormal})
holds with the opposite sign. Conversely, for any point $q\in\mathcal{I}^{(2+1)}$
and any non-zero vector $v\in T_{q}\widetilde{\mathcal{M}}$ which
is null and future directed with respect to $\tilde{g}$:

\begin{itemize}

\item In the case $\tilde{g}(v,n_{\mathcal{I}})>0$, there exists
a unique future directed, \emph{future} inextendible null geodesic
\[
\gamma:(\alpha,+\infty)\rightarrow(\mathcal{M},g)
\]
 for which (\ref{eq:LimitOfGeodesic}) holds. 

\item In the case $\tilde{g}(v,n_{\mathcal{I}})<0$, there exists
a unique future directed, \emph{past} inextendible null geodesic 
\[
\gamma:(-\infty,a)\rightarrow(\mathcal{M},g)
\]
 for which (\ref{eq:LimitOfGeodesic}) holds with $-\infty$ in place
of $+\infty$. 

\end{itemize}

We can therefore define the reflection of null geodesics on conformal
infinity as follows:
\begin{defn}
\label{def:Reflection} Let $(\mathcal{M},g)$ be as in Definition
\ref{def:AsADS} and let $\gamma:(a,+\infty)\rightarrow(\mathcal{M},g)$
be a future directed, future inextendible null geodesic, satisfying
(\ref{eq:InfiniteRLimit}) and (\ref{eq:BoundedULimit}), and let
$q\in\mathcal{I}^{(2+1)}$, $v\in T_{q}\widetilde{\mathcal{M}}$ be
defined by (\ref{eq:LimitOfGeodesic}). Let also $v_{\models}\in T_{q}\widetilde{\mathcal{M}}$
be the reflection of $v$ across $T_{q}\mathcal{I}^{(2+1)}$, i.~e.~the
unique future directed, null vector satisfying $v_{\models}\neq v$
and 
\[
(v_{\models}-v)\parallel n_{\mathcal{I}}.
\]
We define the \emph{reflection} of $\gamma$ off $\mathcal{I}^{(2+1)}$
to be the unique future directed, past inextendible null geodesic
\[
\gamma_{\models}:(-\infty,b)\rightarrow(\mathcal{M},g)
\]
 such that 
\begin{equation}
\lim_{s\rightarrow-\infty}(\gamma_{\models}(s),r^{2}\dot{\gamma}_{\models}(s))=(q,v_{\models}).\label{eq:ReflectingConditionGeodesics}
\end{equation}
\end{defn}
\begin{rem*}
Note that the angular momenta of $\gamma$ and $\gamma_{\models}$
(defined by (\ref{eq:AngularMomentum})) are necessarily the same,
since $n_{\mathcal{I}}$ points in the radial direction. In particular,
the projections of $\gamma$, $\gamma_{\models}$ on $\mathcal{U}$
will satisfy on $\mathcal{I}$ (in the notation of Section (\ref{subsec:VlasovEquations})):
\begin{equation}
\begin{cases}
r^{2}\dot{\gamma}_{\models}^{u}\big|_{\mathcal{I}}=r^{2}\dot{\gamma}^{v}\big|_{\mathcal{I}},\\
r^{2}\dot{\gamma}_{\models}^{v}\big|_{\mathcal{I}}=r^{2}\dot{\gamma}^{u}\big|_{\mathcal{I}},\\
r^{2}\dot{\gamma}_{\models}^{A}\big|_{\mathcal{I}}=r^{2}\dot{\gamma}^{A}\big|_{\mathcal{I}}.
\end{cases}\label{eq:ReflectionConditionSphericallySymmetric}
\end{equation}
\end{rem*}
By successively extending a null geodesic $\gamma$ through its reflections
off $\mathcal{I}^{(2+1)}$, we can construct its maximal extension
through reflections:
\begin{defn}
\label{def:MaximalExtensionReflections} Let $(\mathcal{M},g)$ be
as in Definition \ref{def:AsADS} and let $\gamma=\cup_{n=0}^{N}\gamma_{n}$,
$N\in\mathbb{N}\cup\{\infty\}$, be the union of future directed,
future inextendible, affinely parametrised null geodesics $\gamma_{n}:(a_{n},b_{n})\rightarrow\mathcal{M}$,
$-\infty\le a_{n}<b_{n}\le+\infty$. We will say that $\gamma$ is
an \emph{affinely parametrised, maximally extended geodesic through
reflections }if all of the following conditions hold:

\begin{enumerate}

\item $a_{n}=-\infty$ for any $n>0$ and $b_{n}=+\infty$ for any
$n<N$,

\item For any $0<n\le N$, $\gamma_{n}$ is the reflection of $\gamma_{n-1}$
off $\mathcal{I}^{(2+1)}$, in accordance with Definition \ref{def:Reflection}.

\item If $N\neq\infty$, then $\lim_{s\rightarrow b_{N}^{-}}\gamma_{N}(s)$
does not exist in $\widetilde{\mathcal{M}}$.

\end{enumerate}
\end{defn}
\begin{rem*}
The reflecting condition (\ref{eq:ReflectingConditionGeodesics})
uniquely determines the affine parametrisation of $\gamma_{n}$ in
terms of the parametrisation of $\gamma_{n-1}$. Thus, the affine
parametrisation of $\dot{\gamma}_{0}$ uniquely determines the parametrisation
of all the $\gamma_{n}$'s.
\end{rem*}
\begin{figure}[h] 
\centering 
\begingroup%
  \makeatletter%
  \providecommand\color[2][]{%
    \errmessage{(Inkscape) Color is used for the text in Inkscape, but the package 'color.sty' is not loaded}%
    \renewcommand\color[2][]{}%
  }%
  \providecommand\transparent[1]{%
    \errmessage{(Inkscape) Transparency is used (non-zero) for the text in Inkscape, but the package 'transparent.sty' is not loaded}%
    \renewcommand\transparent[1]{}%
  }%
  \providecommand\rotatebox[2]{#2}%
  \newcommand*\fsize{\dimexpr\f@size pt\relax}%
  \newcommand*\lineheight[1]{\fontsize{\fsize}{#1\fsize}\selectfont}%
  \ifx\svgwidth\undefined%
    \setlength{\unitlength}{150bp}%
    \ifx\svgscale\undefined%
      \relax%
    \else%
      \setlength{\unitlength}{\unitlength * \real{\svgscale}}%
    \fi%
  \else%
    \setlength{\unitlength}{\svgwidth}%
  \fi%
  \global\let\svgwidth\undefined%
  \global\let\svgscale\undefined%
  \makeatother%
  \begin{picture}(1,1.75)%
    \lineheight{1}%
    \setlength\tabcolsep{0pt}%
    \put(0,0){\includegraphics[width=\unitlength,page=1]{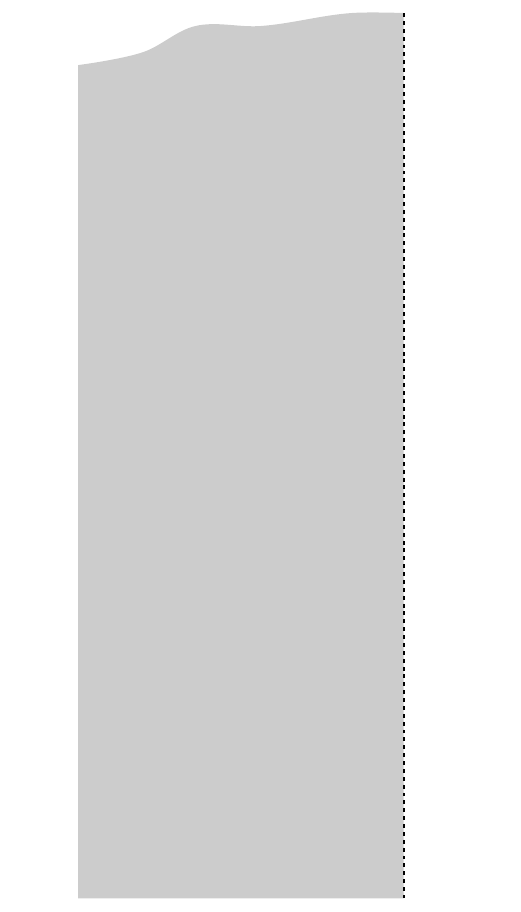}}%
    \put(0.83582917,0.75705444){\color[rgb]{0,0,0}\makebox(0,0)[lt]{\lineheight{1.25}\smash{\begin{tabular}[t]{l}$\mathcal{I}$\end{tabular}}}}%
    \put(0,0){\includegraphics[width=\unitlength,page=2]{Maximal_geodesic.pdf}}%
    \put(0.41367573,1.41862625){\color[rgb]{0,0,0}\makebox(0,0)[lt]{\lineheight{1.25}\smash{\begin{tabular}[t]{l}$\gamma_N$\end{tabular}}}}%
    \put(0.37252476,0.75804458){\color[rgb]{0,0,0}\makebox(0,0)[lt]{\lineheight{1.25}\smash{\begin{tabular}[t]{l}$\gamma_{N-1}$\end{tabular}}}}%
    \put(0.38551977,0.05414621){\color[rgb]{0,0,0}\makebox(0,0)[lt]{\lineheight{1.25}\smash{\begin{tabular}[t]{l}$\gamma_{N-2}$\end{tabular}}}}%
  \end{picture}%
\endgroup%
 
\caption{Schematic depiction of the components $\gamma_n$ of a maximally  extended geodesic through reflections (as projected on the $(u,v)$-plane).}
\end{figure}

Having defined the reflection of null geodesics on $\mathcal{I}^{(2+1)}$,
we can now introduce the notion of the reflecting boundary condition
for the massless Vlasov equation on $\mathcal{I}^{(2+1)}$:
\begin{defn}
\label{def:ReflectingBoundaryCondition} Let $(\mathcal{M},g)$ be
as in Definition \ref{def:AsADS}, and let $f$ be a smooth massless
Vlasov field on $T\mathcal{M}$ (see Section \ref{subsec:VlasovEquations}
for the relevant definition). We will say that $f$ satisfies the
\emph{reflecting} boundary condition on conformal infinity if, for
any pair of future directd null geodesics $\gamma:(a,+\infty)\rightarrow\mathcal{M}$
and $\gamma_{\models}:(-\infty,b)\rightarrow\mathcal{M}$ such that
$\gamma_{\models}$ is the reflection of $\gamma$ on $\mathcal{I}^{(2+1)}$
according to Definition \ref{def:Reflection}, $f$ satisfies 
\begin{equation}
f|_{(\gamma,\dot{\gamma})}=f|_{(\gamma_{\models},\dot{\gamma}_{\models})},\label{eq:ReflectingCondition}
\end{equation}
where $f|_{(\gamma,\dot{\gamma})}$ is the (constant) value of $f$
along the curve $(\gamma,\dot{\gamma})$ in $T\mathcal{M}$. 
\end{defn}
\begin{rem*}
Equivalently, $f$ satisfies the reflecting condition on $\mathcal{I}^{(2+1)}$
if $f$ is constant along the trajectory of $(\gamma,\dot{\gamma})$
for any future directed, affinely paramterised null geodesic $\gamma$
which is maximally extended through reflections.
\end{rem*}
The following Lemma is a straightforward consequence of the relations
(\ref{eq:DerivativeTildeUMass})\textendash (\ref{eq:DerivativeTildeVMass}),
the condition (\ref{eq:1/r0}) on $\mathcal{I}$ and the reflecting
boundary condition (\ref{eq:ReflectionConditionSphericallySymmetric}):
\begin{lem}
Let $(r,\Omega^{2},f)$ be an asymptotically AdS solution of the spherically
symmetric Einstein\textendash massless Vlasov system (\ref{eq:RequationFinal})\textendash (\ref{NullShellFinal}),
such that $f$ satisfies the reflecting boundary condition on conformal
infinity. Then, 
\begin{equation}
(\partial_{u}+\partial_{v})\tilde{m}|_{\mathcal{I}}=0,\label{eq:ConservedMassI}
\end{equation}
i.\,e.~$\tilde{m}$ is conserved along $\mathcal{I}$. In particular,
if $\tilde{m}$ has a finite limit on some point $q\in\mathcal{I}$,
then it has a finite limit everywhere on $\mathcal{I}$.
\end{lem}

\section{\label{sec:InitialValueProblem}The initial-boudary value problem
for the spherically symmetric Einstein\textendash massless Vlasov
system}

The aim of this section is to introduce the characteristic initial-boundary
value problem for (\ref{eq:RequationFinal})\textendash (\ref{NullShellFinal}),
expressed in terms of characteristic initial data sets at $u=0$.
We will focus on the class of \emph{smoothly compatible} initial data
sets, which are precisely those initial data sets which, formally
at least, are induced on $u=0$ by smooth solutions $(\mathcal{M},g;f)$
of (\ref{eq:EinsteinVlasovEquations}) satisfying the reflecting boundary
condition on $\mathcal{I}$ (the fact that any such initial data set
indeed admits a smooth development will be established in Section
\ref{sec:Cauchy_Stability_Low_Regularity}). To this end, we will
first introduce the notion of a smooth, asymptotically AdS solution
$(r,\Omega^{2},f)$ to (\ref{eq:RequationFinal})\textendash (\ref{NullShellFinal})
with a regular axis of symmetry. 

\subsection{\label{subsec:SmoothSolutions}The class of smooth solutions of (\ref{eq:RequationFinal})\textendash (\ref{NullShellFinal}) }

Before formulating the characteristic initial-boundary value problem
for (\ref{eq:RequationFinal})\textendash (\ref{NullShellFinal}),
we will introduce the regularity conditions that will define the class
of solutions to (\ref{eq:RequationFinal})\textendash (\ref{NullShellFinal})
which will be of interest to us in the following sections. This class
will consist of precisely those solutions $(r,\Omega^{2},f)$ to (\ref{eq:RequationFinal})\textendash (\ref{NullShellFinal})
which arise from \emph{smooth }spherically symmetric solutions $(\mathcal{M},g;f)$
of (\ref{eq:EinsteinVlasovEquations}). 

Checking whether a solution of (\ref{eq:RequationFinal})\textendash (\ref{NullShellFinal})
corresponds to a smooth solution of (\ref{eq:EinsteinVlasovEquations})
requires taking into consideration the coordinate singularities introduced
along the axis when switching to a spherically symmetric double null
coordinate chart. We will therefore adopt the following definition
for the smoothness of a solution $(r,\Omega^{2},f)$ of (\ref{eq:RequationFinal})\textendash (\ref{NullShellFinal})
in the presence of an axis-type boundary, which guarantees the existence
of a corresponding smooth solution $(\mathcal{M},g;f)$ of (\ref{eq:EinsteinVlasovEquations})
with a non-trivial axis:
\begin{defn}
\label{def:SmoothnessAxis} Let $\mathcal{U}$ be a domain in the
$(u,v)$-plane satisfying $\mathcal{U}\subset\{u<v\}$, such that
\[
\bar{\gamma}_{\mathcal{Z}}\doteq\{u=v\}\cap\partial\mathcal{U}
\]
is a non-empty connected curve of the form $\{u=v\}\cap\{u_{1}\le u\le u_{2}\}$
for some $u_{1}<u_{2}$. Let us also set 
\[
\gamma_{\mathcal{Z}}\doteq\{u=v\}\cap\{u_{1}<u<u_{2}\}.
\]
A solution $(r,\Omega^{2},f)$ of (\ref{eq:RequationFinal})\textendash (\ref{NullShellFinal})
on $\mathcal{U}$ will possess a\emph{ smooth axis }$\gamma_{\mathcal{Z}}$
if

\begin{itemize}

\item The functions $(r,\Omega^{2})$ are smooth and positive on
$\mathcal{U}$, and $r$ extends continuously to $0$ on $\gamma_{\mathcal{Z}}$.

\item The Vlasov field $f$ is of the form (\ref{eq:FAsASmoothDeltaFunction}),
where $\bar{f}$ is smooth in its variables.

\item The coefficients of the metric (\ref{eq:SphericallySymmetricMetric})
on $\mathcal{U}\times\mathbb{S}^{2}$, when expressed in the Cartesian
coordinate chart 
\[
(x^{0},x^{1},x^{2},x^{3}):\mathcal{U}\times\mathbb{S}^{2}\rightarrow\mathbb{R}\times\big(\mathbb{R}^{3}\backslash\{0\}\big)
\]
 defined by the relations
\begin{align}
x^{0} & =\frac{1}{2}(u+v),\label{eq:CartesianCoordinates}\\
\sqrt{(x^{1})^{2}+(x^{2})^{2}+(x^{3})^{2}} & =r,\nonumber \\
\cos^{-1}\big(x^{3}/r\big) & =\theta,\nonumber \\
x^{2}/x^{1} & =\tan\varphi\nonumber 
\end{align}
(where $(\varphi,\theta)$ are the standard angular coordinates on
$\mathbb{S}^{2}$), can be smoothly extended on $(x^{1})^{2}+(x^{2})^{2}+(x^{3})^{2}=0$.
Similarly, expressed as a function of the Cartesian coordinates $x^{a}$
and their conjugate momentum coordinates $p^{a}$, the function $\bar{f}(x^{a},p^{a})$
can be smoothly extended on $(x^{1})^{2}+(x^{2})^{2}+(x^{3})^{2}=0$.

\end{itemize}
\end{defn}
\begin{figure}[h] 
\centering 
\begingroup%
  \makeatletter%
  \providecommand\color[2][]{%
    \errmessage{(Inkscape) Color is used for the text in Inkscape, but the package 'color.sty' is not loaded}%
    \renewcommand\color[2][]{}%
  }%
  \providecommand\transparent[1]{%
    \errmessage{(Inkscape) Transparency is used (non-zero) for the text in Inkscape, but the package 'transparent.sty' is not loaded}%
    \renewcommand\transparent[1]{}%
  }%
  \providecommand\rotatebox[2]{#2}%
  \newcommand*\fsize{\dimexpr\f@size pt\relax}%
  \newcommand*\lineheight[1]{\fontsize{\fsize}{#1\fsize}\selectfont}%
  \ifx\svgwidth\undefined%
    \setlength{\unitlength}{150bp}%
    \ifx\svgscale\undefined%
      \relax%
    \else%
      \setlength{\unitlength}{\unitlength * \real{\svgscale}}%
    \fi%
  \else%
    \setlength{\unitlength}{\svgwidth}%
  \fi%
  \global\let\svgwidth\undefined%
  \global\let\svgscale\undefined%
  \makeatother%
  \begin{picture}(1,1)%
    \lineheight{1}%
    \setlength\tabcolsep{0pt}%
    \put(0,0){\includegraphics[width=\unitlength,page=1]{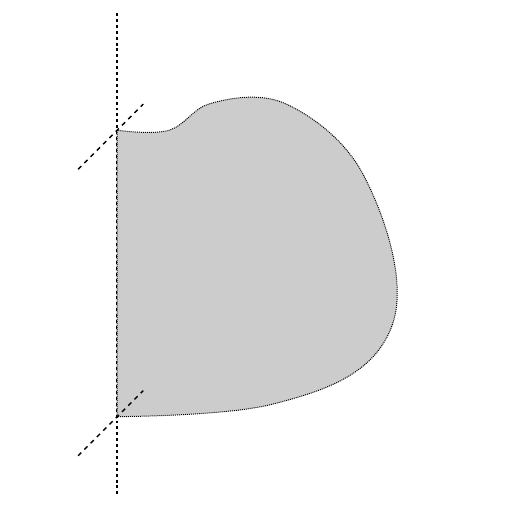}}%
    \put(0.39999999,0.47500013){\color[rgb]{0,0,0}\makebox(0,0)[lt]{\lineheight{1.25}\smash{\begin{tabular}[t]{l}$\mathcal{U}$\end{tabular}}}}%
    \put(0.19326242,0.83590303){\color[rgb]{0,0,0}\rotatebox{90}{\makebox(0,0)[lt]{\lineheight{1.25}\smash{\begin{tabular}[t]{l}$u=v$\end{tabular}}}}}%
    \put(0.04269498,0.05978632){\color[rgb]{0,0,0}\rotatebox{45}{\makebox(0,0)[lt]{\lineheight{1.25}\smash{\begin{tabular}[t]{l}$u=u_1$\end{tabular}}}}}%
    \put(0.03279391,0.60681584){\color[rgb]{0,0,0}\rotatebox{45}{\makebox(0,0)[lt]{\lineheight{1.25}\smash{\begin{tabular}[t]{l}$u=u_2$\end{tabular}}}}}%
    \put(0.11856435,0.47623775){\color[rgb]{0,0,0}\makebox(0,0)[lt]{\lineheight{1.25}\smash{\begin{tabular}[t]{l}$\gamma_{\mathcal{Z}}$\end{tabular}}}}%
    \put(0,0){\includegraphics[width=\unitlength,page=2]{Axis_Domain_Smooth.pdf}}%
  \end{picture}%
\endgroup%
 
\caption{Schematic depiction of a domain $\mathcal{U}$ satisfying the conditions of Definition \ref{def:SmoothnessAxis}.}
\end{figure}
\begin{rem*}
In the Cartesian coordinates $x^{a}$ defined by \ref{eq:CartesianCoordinates},
the metric (\ref{eq:SphericallySymmetricMetric}) is expressed as:
\begin{equation}
g=-\Omega^{2}\Big(1-\frac{\dot{r}^{2}}{(r^{\prime})^{2}}\Big)(dx^{0})^{2}-2\frac{\dot{r}\Omega^{2}}{(r^{\prime})^{2}}\frac{x^{k}\delta_{ki}}{r}dx^{i}dx^{0}+\Big(\delta_{ij}+\big(\frac{\frac{2m}{r}+\lambda-1}{1-\frac{2m}{r}}\big)\frac{x^{k}x^{l}}{r^{2}}\delta_{ki}\delta_{lj}\Big)dx^{i}dx^{j},\label{eq:CartesianMetric}
\end{equation}
where latin indices range over the index set $\{1,2,3\}$ and 
\begin{gather*}
\dot{r}=\partial_{v}r+\partial_{u}r,\\
r^{\prime}=\partial_{v}r-\partial_{u}r,\\
\lambda=\frac{-4\partial_{u}r\partial_{v}r}{(\partial_{v}r-\partial_{u}r)^{2}}.
\end{gather*}
The fact that $r|_{\gamma_{\mathcal{Z}}}=0$ also implies that 
\begin{equation}
\partial_{v}r|_{\gamma_{\mathcal{Z}}}=-\partial_{u}r|_{\gamma_{\mathcal{Z}}}.\label{eq:BoundaryConditionRAxis}
\end{equation}
\end{rem*}
For asymptotically AdS solutions, smoothness along conformal infinity
will be defined in accordance to Definition \ref{def:AsADS}:
\begin{defn}
\label{def:SmoothnessConformalInfinity}Let $\mathcal{U}$ be a domain
in the $(u,v)$-plane satisfying $\mathcal{U}\subset\{v<u+v_{\mathcal{I}}\}$
for some $v_{\mathcal{I}}\in\mathbb{R}$, such that 
\[
\overline{\mathcal{I}}\doteq\{v=u+v_{\mathcal{I}}\}\cap\partial\mathcal{U}
\]
is a non-empty connected curve of the form $\{v=u+u_{\mathcal{I}}\}\cap\{u_{1}\le u\le u_{2}\}$
for some $u_{1}<u_{2}$. Let us also set 
\begin{equation}
\mathcal{I}\doteq\{v=u+v_{\mathcal{I}}\}\cap\{u_{1}<u<u_{2}\}.
\end{equation}
A solution $(r,\Omega^{2},f)$ of (\ref{eq:RequationFinal})\textendash (\ref{NullShellFinal})
on $\mathcal{U}$ will be said to possess a\emph{ smooth conformal
infinity }$\mathcal{I}$ if

\begin{itemize}

\item The functions $(r,\Omega^{2})$ extend on the whole of the
domain
\[
\mathcal{U}^{*}\doteq\{u+v_{*}(u)\le v<u+v_{\mathcal{I}}\}\cap\{u_{1}<u<u_{2}\}
\]
 (for some smooth $v_{*}(u)<v_{\mathcal{I}}$), so that $(\mathcal{U}^{*};r,\Omega^{2})$
is asymptotically AdS in accordance with Definition \ref{def:AsADS}, 

\item The distribution $f$ is of the form (\ref{eq:FAsASmoothDeltaFunction}),
for some smooth function $\bar{f}$ which extends smoothly on $\{v=u+v_{\mathcal{I}}\}$
as a function of $\bar{p}^{u}=\Omega^{2}p^{u}$, $\bar{p}^{v}=\Omega^{2}p^{v}$
and $l$.

\end{itemize}
\end{defn}
\begin{figure}[h] 
\centering 
\begingroup%
  \makeatletter%
  \providecommand\color[2][]{%
    \errmessage{(Inkscape) Color is used for the text in Inkscape, but the package 'color.sty' is not loaded}%
    \renewcommand\color[2][]{}%
  }%
  \providecommand\transparent[1]{%
    \errmessage{(Inkscape) Transparency is used (non-zero) for the text in Inkscape, but the package 'transparent.sty' is not loaded}%
    \renewcommand\transparent[1]{}%
  }%
  \providecommand\rotatebox[2]{#2}%
  \newcommand*\fsize{\dimexpr\f@size pt\relax}%
  \newcommand*\lineheight[1]{\fontsize{\fsize}{#1\fsize}\selectfont}%
  \ifx\svgwidth\undefined%
    \setlength{\unitlength}{150bp}%
    \ifx\svgscale\undefined%
      \relax%
    \else%
      \setlength{\unitlength}{\unitlength * \real{\svgscale}}%
    \fi%
  \else%
    \setlength{\unitlength}{\svgwidth}%
  \fi%
  \global\let\svgwidth\undefined%
  \global\let\svgscale\undefined%
  \makeatother%
  \begin{picture}(1,1)%
    \lineheight{1}%
    \setlength\tabcolsep{0pt}%
    \put(0,0){\includegraphics[width=\unitlength,page=1]{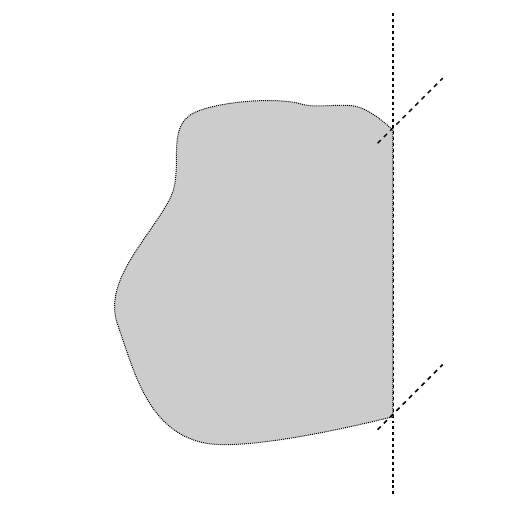}}%
    \put(0.48044554,0.46633679){\color[rgb]{0,0,0}\makebox(0,0)[lt]{\lineheight{1.25}\smash{\begin{tabular}[t]{l}$\mathcal{U}$\end{tabular}}}}%
    \put(0.71677724,0.8099133){\color[rgb]{0,0,0}\rotatebox{90}{\makebox(0,0)[lt]{\lineheight{1.25}\smash{\begin{tabular}[t]{l}$v=u+v_{\mathcal{I}}$\end{tabular}}}}}%
    \put(0.82982357,0.18107352){\color[rgb]{0,0,0}\rotatebox{45}{\makebox(0,0)[lt]{\lineheight{1.25}\smash{\begin{tabular}[t]{l}$u=u_1$\end{tabular}}}}}%
    \put(0.82363568,0.73305347){\color[rgb]{0,0,0}\rotatebox{45}{\makebox(0,0)[lt]{\lineheight{1.25}\smash{\begin{tabular}[t]{l}$u=u_2$\end{tabular}}}}}%
    \put(0.80049508,0.46633679){\color[rgb]{0,0,0}\makebox(0,0)[lt]{\lineheight{1.25}\smash{\begin{tabular}[t]{l}$\mathcal{I}$\end{tabular}}}}%
    \put(0,0){\includegraphics[width=\unitlength,page=2]{Infinity_Domain_Smooth.pdf}}%
  \end{picture}%
\endgroup%
 
\caption{Schematic depiction of a domain $\mathcal{U}$ satisfying the conditions of Definition \ref{def:SmoothnessConformalInfinity}.}
\end{figure}

Away from the axis and conformal infinity, smoothness for solutions
of (\ref{eq:RequationFinal})\textendash (\ref{NullShellFinal}) will
be defined in the usual way:
\begin{defn}
\label{def:SmoothnessGenerally}Let $\mathcal{U}$ be a domain in
the $(u,v)$-plane. A solution $(r,\Omega^{2},f)$ of (\ref{eq:RequationFinal})\textendash (\ref{NullShellFinal})
on $\mathcal{U}$ will be called \emph{smooth} if the functions $(r,\Omega^{2})$
are smooth and positive on $\mathcal{U}$, while $f$ is of the form
(\ref{eq:FAsASmoothDeltaFunction}), for some smooth function $\bar{f}$.
\end{defn}
\begin{rem*}
More generally, a solution $(r,\Omega^{2},f)$ of (\ref{eq:RequationFinal})\textendash (\ref{NullShellFinal})
will be called smooth on a subset $\mathcal{F}$ of the $(u,v)$ plane
of $(r,\Omega^{2},f)$ is smooth on an open domain $\mathcal{U}\supset\mathcal{F}$
(with a similar generalization applying to the definition of a smooth
axis and a smooth conformal infinity).
\end{rem*}

\subsection{\label{subsec:Asymptotically-AdS-Initial-Data}Asymptotically AdS
characteristic initial data sets for (\ref{eq:RequationFinal})\textendash (\ref{NullShellFinal})}

In this paper, as well as in the companion paper \cite{MoschidisVlasov},
we will be mainly interested in solutions to (\ref{eq:RequationFinal})\textendash (\ref{NullShellFinal})
arising from asymptotically AdS characteristic initial data sets prescribed
at $u=0$. In this section, we will introduce the class of initial
data under consideration. We will also present a gauge normalisation
condition, that will later allow us to compare different initial data
sets with the trivial one, as well as a condition related to the compatibility
of initial data sets with the reflecting boundary condition at $\mathcal{I}$.

We will adopt the following fundamental definition:
\begin{defn}
\label{def:AsymptoticallyAdSData} Let $v_{\mathcal{I}}>0$ and let
$r_{/},\Omega_{/}:[0,v_{\mathcal{I}})\rightarrow[0,+\infty)$, $\bar{f}_{/}:(0,v_{\mathcal{I}})\times[0,+\infty)^{2}\rightarrow[0,+\infty)$
be smooth functions. We will call $(r_{/},\Omega_{/}^{2},\bar{f}_{/};v_{\mathcal{I}})$
a \emph{smooth, asymptotically AdS initial data set }for (\ref{eq:RequationFinal})\textendash (\ref{NullShellFinal})
if: 

\begin{enumerate}

\item On $(0,v_{\mathcal{I}})$, the functions $(r_{/},\Omega_{/}^{2},\bar{f}_{/})$
satisfy the constraint equation (\ref{eq:ConstraintVFinal}), where
the energy momentum component $T_{vv}$ is defined on $(0,v_{\mathcal{I}})$
using the relation (\ref{eq:SphericallySymmetricComponentsEnergyMomentum})
(with $\bar{f}_{/}(v;p^{u},l)$ in place of $\bar{f}(u,v;p^{u},p^{v},l)\Big|_{\Omega^{2}p^{u}p^{v}=\frac{l^{2}}{r^{2}}}$).

\item At $v=0$, the functions $r_{/}$, $\Omega_{/}^{2}$ satisfy
\begin{equation}
r_{/}(0)=0,
\end{equation}
\begin{equation}
\Omega_{/}^{2}(0)>0.
\end{equation}

\item At $v=v_{\mathcal{I}}$, the functions $1/r_{/}$, $r_{/}^{-2}\Omega_{/}^{2}$
extend smoothly and satisfy
\begin{equation}
1/r_{/}(v_{\mathcal{I}})=0,
\end{equation}
\begin{equation}
\partial_{v}(1/r_{/})(v_{\mathcal{I}})<0,\label{eq:PositiveDvRInitially}
\end{equation}
\begin{equation}
\frac{\Omega_{/}^{2}}{r_{/}^{2}}(v_{\mathcal{I}})>0.
\end{equation}
Furthermore, for any $\bar{p}\ge0$ and $l\ge0$, $\bar{f}(v,\Omega_{/}^{-2}(v)p,l)$
extends smoothly on $v=v_{\mathcal{I}}$.

\end{enumerate}

We will also say that an initial data set $(r_{/},\Omega_{/}^{2},\bar{f}_{/};v_{\mathcal{I}})$
is of \emph{bounded support in phase space }if there exists some $C>0$
such that, for every $v\in(0,v_{\mathcal{I}})$ and $l\ge0$:
\begin{equation}
\sup_{p^{u}\in supp(f_{/}(v;\cdot,l))}\Big(\Omega_{/}^{2}\big(p^{u}+\frac{l^{2}}{\Omega_{/}^{2}r_{/}^{2}p^{u}}\big)\Big)\le C.\label{eq:BoundedSupportDefinition}
\end{equation}
\end{defn}
We should make the following remarks regarding the class of asymptotically
AdS initial data sets:

\medskip{}

\noindent \emph{Remark 1.} The inequality $\partial_{v}(\Omega_{/}^{-2}\partial_{v}r_{/})\le0$
(following from the constraint equation (\ref{eq:ConstraintVFinal}))
and the fact that 
\[
r_{/}(v_{\mathcal{I}})=+\infty,\text{ }\frac{\Omega_{/}^{2}}{r_{/}^{2}}(v_{\mathcal{I}})>0\text{ and }\Omega_{/}^{2}(0)>0,
\]
 imply that, for any smooth asymptotically AdS initial data set $(r_{/},\Omega_{/}^{2},\bar{f}_{/};v_{\mathcal{I}})$,
\begin{equation}
\inf_{v\in[0,v_{\mathcal{I}})}\partial_{v}r_{/}(v)>0\text{ and }\inf_{v\in[0,v_{\mathcal{I}})}\Omega_{/}^{2}(v)>0.\label{eq:NonTrappingForAsymptoticallyAdSInitialData}
\end{equation}
In particular, a smooth asymptotically AdS initial data set does \emph{not}
contain trapped spheres. 

\medskip{}

\noindent \emph{Remark 2.} For any smooth asymptotically AdS initial
data set $(r_{/},\Omega_{/}^{2},\bar{f}_{/};v_{\mathcal{I}})$, we
can formally define the initial renormalised Hawking mass $\tilde{m}_{/}$
in terms of $(r_{/},\Omega_{/}^{2},\bar{f}_{/})$ by the relation
(\ref{eq:DerivativeTildeVMass}), i.\,e.: 
\begin{equation}
\begin{cases}
\partial_{v}\tilde{m}_{/}=2\pi\Big(\big(1-\frac{2\tilde{m}_{/}}{r_{/}}-\frac{1}{3}\Lambda r_{/}^{2}\big)\frac{r_{/}^{2}(T_{/})_{vv}}{\partial_{v}r_{/}}+4\frac{r_{/}^{2}(T_{/})_{uv}}{\Omega_{/}^{2}}\Big),\\
\tilde{m}_{/}(0)=0,\\
r_{/}^{2}(T_{/})_{vv}(v)\doteq\frac{\pi}{2}\int_{0}^{+\infty}\int_{0}^{+\infty}\big(\Omega_{/}^{2}(v)p\big)^{2}\bar{f}_{/}(v;p,l)\,\frac{dp}{p}ldl,\\
r_{/}^{2}(T_{/})_{uv}(v)\doteq\frac{\pi}{2}\Omega_{/}^{2}(v)\int_{0}^{+\infty}\int_{0}^{+\infty}\frac{l^{2}}{r_{/}^{2}(v)}\bar{f}_{/}(v;p,l)\,\frac{dp}{p}ldl.
\end{cases}\label{eq:DefinitionMassInitially}
\end{equation}
Arguing as in the proof of Proposition \ref{prop:ExtensionPrincipleInfinity},
the condition (\ref{eq:BoundedSupportDefinition}) implies that, for
any smooth asymptotically AdS initial data set $(r_{/},\Omega_{/}^{2},\bar{f}_{/};v_{\mathcal{I}})$
with bounded support in phase space, we have:
\begin{equation}
\lim_{v\rightarrow v_{\mathcal{I}}^{-}}\tilde{m}(v)<+\infty.\label{eq:FiniteTotalMassInitialData}
\end{equation}

\medskip{}

\noindent \emph{Remark 3.} Let $(r_{/},\Omega_{/}^{2},\bar{f}_{/};v_{\mathcal{I}})$
be a smooth asymptotically AdS initial data set as in Definition \ref{def:AsymptoticallyAdSData},
satisfying, moreover, the gauge condition condition 
\begin{equation}
\frac{\partial_{v}r_{/}}{1-\frac{1}{3}\Lambda r_{/}^{2}}(0)=\frac{\Omega_{/}^{2}}{4\partial_{v}r_{/}}(0)\label{eq:AxisNormalization}
\end{equation}
at the axis, as well as the regularity condition 
\begin{equation}
\limsup_{v\rightarrow0}\Big|\partial_{v}^{n}\Big(\frac{\tilde{m}_{/}}{r_{/}^{3}}\Big)\Big|<+\infty\text{ for all }n\in\mathbb{N},\label{eq:SmoothnessInitialData}
\end{equation}
where $\tilde{m}_{/}$ is defined in terms of $(r_{/},\Omega_{/}^{2},\bar{f}_{/})$
by (\ref{eq:DefinitionMassInitially}). Let also $(r,\Omega^{2},f)$
be a solution of the Einstein\textendash massless Vlasov (\ref{eq:RequationFinal})\textendash (\ref{NullShellFinal})
which satisfies at $u=0$
\begin{equation}
(r,\Omega^{2})(0,v)=(r_{/},\Omega_{/}^{2})(v)
\end{equation}
and 
\begin{equation}
f(0,v;p^{u},p^{v},l)=\bar{f}_{/}(v;p^{u},l)\cdot\delta\big(\Omega_{/}^{2}p^{u}p^{v}-\frac{l^{2}}{r_{/}^{2}}\big),\label{eq:VlasovFieldInitially}
\end{equation}
as well as the boundary condition 
\begin{equation}
(\partial_{u}+\partial_{v})^{n}r(0,0)=0\text{ for all }n\in\mathbb{N}\label{eq:BoundaryConditionInitialDataForAxisCompatibility}
\end{equation}
(note that \ref{eq:BoundaryConditionInitialDataForAxisCompatibility}
is necessary for $(r,\Omega^{2},f)$ to satisfy $r=0$ along $u=v$;
the condition (\ref{eq:AxisNormalization}) is a necessary condition
for (\ref{eq:BoundaryConditionInitialDataForAxisCompatibility}) to
hold when $n=1$). 

For any $k\in\mathbb{N}$, we can formally determine the values of
the higher order transversal derivatives $\partial_{u}^{k}r|_{u=0}$,
$\partial_{u}^{k}\Omega|_{u=0}$ and $\partial_{u}^{k}f|_{u=0}$ in
terms of $(r_{/},\Omega_{/}^{2},\bar{f}_{/})$ using equations (\ref{eq:RequationFinal}),
(\ref{eq:OmegaEquationFinal}) and (\ref{eq:VlasovFinal}). In particular,
by defining $\bar{f}$ in terms of $f$ through the relation 
\begin{equation}
f(u,v;p^{u},p^{v},l)=\bar{f}(v;p^{u},l)\cdot\delta\big(\Omega^{2}(u,v)p^{u}p^{v}-\frac{l^{2}}{r^{2}(u,v)}\big)\label{eq:RelationVlasovFieldSmooth}
\end{equation}
differentiating (\ref{eq:RequationFinal}), (\ref{eq:OmegaEquationFinal})
and (\ref{eq:VlasovEquationSphericalSymmetry}), using also (\ref{NullShellFinal}),
we formally obtain: 
\begin{align}
\partial_{u}^{k+1} & r(0,v)=\partial_{u}^{k+1}r(0,0)+\label{eq:FormulaForInitialR}\\
 & +\int_{0}^{\bar{v}}\partial_{u}^{k}\Big(-\frac{1}{2}\frac{\tilde{m}-\frac{2}{3}\Lambda r^{3}}{r^{2}}\Omega^{2}+2\pi^{2}\Omega^{2}r^{-3}\int_{0}^{+\infty}\int_{0}^{+\infty}l^{2}\bar{f}(\cdot;p,l)\,ldl\frac{dp}{p}\Big)(0,\bar{v})\,d\bar{v},\nonumber 
\end{align}
\begin{align}
\partial_{u}^{k+1} & \log\Omega^{2}(0,v)=\partial_{u}^{k+1}\log\Omega^{2}(0,0)+\label{eq:FormulaForInitialOmega}\\
 & +\int_{0}^{\bar{v}}\partial_{u}^{k}\Big(\Omega^{2}\big(\frac{2\tilde{m}}{r^{3}}+\frac{1}{3}\Lambda\big)-8\pi^{2}\Omega^{2}r^{-4}\int_{0}^{+\infty}\int_{0}^{+\infty}l^{2}\bar{f}(\cdot;p,l)\,ldl\frac{dp}{p}\Big)(0,\bar{v})\,d\bar{v}\nonumber 
\end{align}
and 
\begin{equation}
\partial_{u}^{k+1}\bar{f}(0,v;p,l)=\partial_{u}^{k}\Big\{-\frac{l^{2}}{\Omega^{2}r^{2}p^{2}}\partial_{v}\bar{f}+\frac{1}{p}\Big(\partial_{u}\log(\Omega^{2})p^{2}+\frac{2}{r}\Omega^{-2}\partial_{v}r\frac{l^{2}}{r^{2}}\Big)\partial_{p}\bar{f}\Big\}(0,v;p,l).\label{eq:FormulaForInitialF}
\end{equation}
Arguing inductively in $k\ge0$ and assuming that $\partial_{u}^{\bar{k}}r|_{u=0}$,
$\partial_{u}^{\bar{k}}\Omega|_{u=0}$ and $\partial_{u}^{\bar{k}}\bar{f}|_{u=0}$
have been computed for $0\le\bar{k}\le k$ in terms of $(r_{/},\Omega_{/}^{2},\bar{f}_{/})$,
the relations (\ref{eq:FormulaForInitialR})\textendash (\ref{eq:FormulaForInitialF})
(combined with the expressions (\ref{eq:DerivativeTildeUMass})\textendash (\ref{eq:DerivativeTildeVMass})
for the derivatives of $\tilde{m}$) uniquely determine, successively,
$\partial_{u}^{k+1}r|_{u=0}$, $\partial_{u}^{k+1}\Omega|_{u=0}$
and $\partial_{u}^{k+1}\bar{f}|_{u=0}$ in terms of $(r_{/},\Omega_{/}^{2},\bar{f}_{/})$,
using at $v=0$ the boundary conditions 
\begin{equation}
\partial_{u}^{k+1}r(0,0)=\sum_{\bar{k}=0}^{k}(-1)^{k-\bar{k}+1}{k+1 \choose \bar{k}}\partial_{v}^{k-\bar{k}+1}(\partial_{u}^{\bar{k}}r)(0,0)\label{eq:HigherOrderInitialAxisR}
\end{equation}
and 
\begin{align}
\partial_{u}^{k+1}\Omega^{2}(0,0)=4\Big(\sum_{\bar{k}=0}^{k} & \Big\{(-1)^{k-\bar{k}}{k+1 \choose \bar{k}}\partial_{v}^{k-\bar{k}+1}(\partial_{u}^{\bar{k}+1}r)(0,0)\Big\}\Big)\cdot\frac{\partial_{v}r}{1-\frac{2\tilde{m}}{r}-\frac{1}{3}\Lambda r^{2}}(0,0)-\label{eq:HigherOrderInitialAxisOmega}\\
 & -4\sum_{\bar{k}=0}^{k}\Big\{\partial_{u}^{k+1-\bar{k}}r(0,0)\cdot\partial_{u}^{\bar{k}+1}\Big(\frac{\partial_{v}r}{1-\frac{2\tilde{m}}{r}-\frac{1}{3}\Lambda r^{2}}\Big)(0,0)\Big\}\nonumber 
\end{align}
 (following from (\ref{eq:BoundaryConditionInitialDataForAxisCompatibility})
and the relation (\ref{eq:DefinitionHawkingMass})). 

\medskip{}

It is natural to consider asymptotically AdS initial data sets related
by a \emph{gauge transformation} as equivalent. For characteristic
initial data as those introduced in Definition \ref{def:AsymptoticallyAdSData},
a general gauge transformation will consist of a change of coordinates
$v\rightarrow V(v)$ and a parameter $\frac{dU}{du}(0)$ at $u=0$,
satisfying
\begin{equation}
v\rightarrow V(v),\text{ }\frac{dV}{dv}(v)>0\text{ for all }v,\text{ }V(0)=0,\text{ }\frac{dU}{du}(0)>0.\label{eq:GaugeCoordinateChange}
\end{equation}
At a spacetime level, this corresponds to a coordinate transformation
of the form 
\begin{equation}
(u,v)\rightarrow(u',v')=\big(U(u),V(v)\big),\label{eq:SpacetimeCoordinateTransformation}
\end{equation}
where $U(u)$ is a function with $U(0)=0$ and $\frac{dU}{du}(0)$
equal to the given parameter. Under such a gauge transformation, an
asymptotically AdS initial data set $(r_{/},\Omega_{/}^{2},\bar{f}_{/};v_{\mathcal{I}})$
transforms as $(r_{/},\Omega_{/}^{2},\bar{f}_{/};v_{\mathcal{I}})\rightarrow(r_{/}^{\prime},(\Omega_{/}^{\prime})^{2},\bar{f}_{/}^{\prime};v'(v_{\mathcal{I}}))$,
where 
\begin{align}
r_{/}^{\prime}(v') & \doteq r_{/}(v),\label{eq:GeneralGaugeTransformationInitialData}\\
(\Omega_{/}^{\prime})^{2}(v') & \doteq\frac{1}{\frac{dU}{du}(0)\cdot\frac{dV}{dv}(v)}\Omega_{/}^{2}(v),\nonumber \\
\bar{f}_{/}^{\prime}(v';\frac{dU}{du}(0)\cdot p,l) & \doteq\bar{f}_{/}(v;p,l).\nonumber 
\end{align}
\begin{rem*}
Under the gauge transformation (\ref{eq:GeneralGaugeTransformationInitialData}),
the higher order transversal derivatives computed along $u=0$ through
the relations (\ref{eq:FormulaForInitialR})\textendash (\ref{eq:FormulaForInitialF})
transform analogously, i.\,e.~by assuming that $(r,\Omega^{2},f)$
is transformed under the gauge transformation 
\begin{align}
r^{\prime}(u',v') & \doteq r(u,v),\label{eq:GeneralGaugeTransformationSpacetime}\\
(\Omega^{\prime})^{2}(u',v') & \doteq\frac{1}{\frac{dU}{dv}(u)\cdot\frac{dV}{dv}(v)}\Omega^{2}(u,v),\nonumber \\
f^{\prime}(u',v';\frac{dU}{du}(u)p^{u^{\prime}},\frac{dV}{dv}(v)p^{v^{\prime}}l) & \doteq f(u,v;p^{u},p^{v},l).\nonumber 
\end{align}
 provided $(u,v)\rightarrow(u',v')$ fixes the straight line $\{u=v\}$
(recall that we have assumed that the axis $r=0$ lies on $\{u=v\}$
in the computations (\ref{eq:FormulaForInitialR})\textendash (\ref{eq:FormulaForInitialF}));
the latter condition necessitates that $U(x)=V(x)$ in a neighborhood
of $x=0$, hence fixing germ of $U$ at $u=0$ by the condition 
\[
\frac{d^{k}U}{(du)^{k}}(0)=\frac{d^{k}V}{(du)^{k}}(0)\text{ for all }k\in\mathbb{N}.
\]
\end{rem*}
In this paper, we will be mainly interested in initial data sets $(r_{/},\Omega_{/}^{2},\bar{f}_{/};v_{\mathcal{I}})$
on $\{u=0\}$ which give rise to solutions $(r,\Omega^{2},f)$ of
(\ref{eq:RequationFinal})\textendash (\ref{NullShellFinal}) which
are smooth, with smooth axis $\{u=v\}$ and smooth conformal infinity
$\{u=v-v_{\mathcal{I}}\}$, in accordance with Definitions \ref{def:SmoothnessAxis}\textendash \ref{def:SmoothnessGenerally}.
To this end, we will have to ensure that the class of initial data
under consideration is compatible with smoothness both at the axis
and at conformal infinity; in the latter case, additional compatibility
issues arise in regard to the reflecting boundary condition (\ref{eq:ReflectingCondition}).
The class of \emph{smoothly compatible} initial data, consisting of
data which are consistent with the aforementioned regularity conditions,
will be defined as follows:
\begin{defn}
\label{def:CompatibilityCondition} Let $(r_{/},\Omega_{/}^{2},\bar{f}_{/};v_{\mathcal{I}})$
be a smooth asymptotically AdS initial data set\emph{ }for (\ref{eq:RequationFinal})\textendash (\ref{NullShellFinal}),
in accordance with Definition \ref{def:AsymptoticallyAdSData}. We
will say that $(r_{/},\Omega_{/}^{2},\bar{f}_{/};v_{\mathcal{I}})$
is \emph{smoothly compatible }if there exists a $T>0$, a pair of
smooth functions $r,\Omega^{2}>0$ on the domain
\[
\mathcal{U}_{T;v_{\mathcal{I}}}=\big\{0\le u<T\big\}\cap\big\{ u<v<u-v_{\mathcal{I}}\big\}
\]
 in the $(u,v)$ plane and a smooth function $\tilde{f}:\mathcal{U}_{T;v_{\mathcal{I}}}\times[0,+\infty)^{3}\rightarrow[0,+\infty)$,
such that, defining 
\[
f(u,v;p^{u},p^{v},l)\doteq\tilde{f}(u,v;\Omega^{2}(u,v)p^{u},\Omega^{2}(u,v)p^{v},l)\cdot\delta\Big(\Omega^{2}(u,v)p^{u}p^{v}-\frac{l^{2}}{r^{2}(u,v)}\Big),
\]
the following conditions are satisfied by the triplet $(r,\Omega^{2},f)$:

\begin{enumerate}

\item \emph{Smooth extendibility of the data.} Along $u=0$ 
\[
r(0,v)=r_{/}(v),\text{ }\Omega^{2}(0,v)=\Omega_{/}^{2}(0,v),\text{ }\tilde{f}(0,v;\Omega^{2}(0,v)p^{u},\frac{l^{2}}{r^{2}(0,v)p^{u}},l)=\bar{f}_{/}(v,p^{u},l)
\]
 and, for every integer $k\ge1$, the quantities $\partial_{u}^{k}r|_{u=0}$,
$\partial_{u}^{k}\Omega^{2}|_{u=0}$ and $\partial_{u}^{k}f|_{u=0}$
are equal to the values determined by the initial data set $(r_{/},\Omega_{/}^{2},\bar{f}_{/};v_{\mathcal{I}})$
according to the process decscribed in Remark 3 below Definition \ref{def:AsymptoticallyAdSData}.

\item \emph{Axis regularity.} The functions $r$, $\Omega^{2}$ and
$\tilde{f}$ extend smoothly on $\{u=v\}\cap\{0\le u<T\}$ and satisfy
\[
r|_{u=v}=0\text{, }\Omega^{2}|_{u=v}>0.
\]
Moreover, the distribution $f$ satisfies the Vlasov equation (\ref{eq:VlasovEquationSphericalSymmetry})
with respect to $(r,\Omega^{2})$ in a neighborhood of $(0,0)$. Furthermore,
switching to the Cartesian coordinate chart $(x^{0},x^{1},x^{2},x^{3})$
defined by (\ref{eq:CartesianCoordinates}), the Cartesian components
(\ref{eq:CartesianMetric}) of the metric $g_{\alpha\beta}$ extend
smoothly on the axis $\sum_{i=1}^{3}(x^{i})^{2}=0$. 

\item \emph{Compatibility at infinity.} The functions $\frac{1}{r}$
and $\frac{\Omega^{2}}{r^{2}}$ extend smoothly on $\{u=v-v_{\mathcal{I}}\}\cap\{0\le u<T\}$
and satisfy 
\[
\frac{1}{r}|_{u=v-v_{\mathcal{I}}}=0,\text{ }\partial_{v}\big(\frac{1}{r}\big)|_{u=v-v_{\mathcal{I}}}>0,
\]
and 
\[
\frac{\Omega^{2}}{r^{2}}|_{u=v-v_{\mathcal{I}}}>0.
\]
The functions $\frac{1}{r}$ and $\frac{\Omega^{2}}{r^{2}}$ extend
smoothly in an open neighborhood of the point $(0,v_{\mathcal{I}})$
in the $(u,v)$-plane, while $\tilde{f}$ extends smoothly in a neighborhood
of $\{(0,v_{\mathcal{I}})\}\times[0,+\infty)^{3}$ in $\mathbb{R}^{2}\times[0,+\infty)^{3}$.
Finally, the distribution $f$ solves the Vlasov equation (\ref{eq:VlasovEquationSphericalSymmetry})
with respect to $(r,\Omega^{2})$ in a left neighborhood of $\{u=v-v_{\mathcal{I}}\}\cap\{0\le u<T\}$
in $\mathcal{U}_{T;v_{\mathcal{I}}}$, satisfying moreover the reflecting
boundary condition (\ref{eq:ReflectingCondition}) on $\{u=v-v_{\mathcal{I}}\}$.

\end{enumerate}

Given a smoothly compatible, asymptotically AdS initial data set $(r_{/},\Omega_{/}^{2},\bar{f}_{/};v_{\mathcal{I}})$,
any gauge transformation of the form (\ref{eq:GaugeCoordinateChange})\textendash (\ref{eq:GeneralGaugeTransformationInitialData})
that preserves the Conditions 1\textendash 3 above will be called
a \emph{smoothly compatible} transformation for $(r_{/},\Omega_{/}^{2},\bar{f}_{/};v_{\mathcal{I}})$.
\end{defn}
The following remarks should be noted regarding Definition \ref{def:CompatibilityCondition}:
\begin{itemize}
\item For a smoothly compatible, asymptotically AdS initial data set $(r_{/},\Omega_{/}^{2},\bar{f}_{/};v_{\mathcal{I}})$
of bounded support in phase space, the function $\tilde{m}_{/}$ extends
smoothly on $v=0$ and $v=v_{\mathcal{I}}$, with $\tilde{m}_{/}=O(r_{/}^{3})$
as $v\rightarrow0$. 
\item In view of Condition 2 of Definition \ref{def:CompatibilityCondition}
(using, in particular, the relation (\ref{eq:DefinitionHawkingMass})
on $u=0$ and the fact that $r$ is assumed to extend smoothly on
$\{u=v\}$ with $r|_{u=v}=0$), a smoothly compatible asymptotically
AdS initial data set necessarily satisfies 
\begin{equation}
4(\partial_{v}r_{/})^{2}(0)=\Omega_{/}^{2}(0).\label{eq:SmoothlyCompatibleAxisCondition}
\end{equation}
\item While the triplet $(r,\Omega^{2},f)$ in Definition \ref{def:CompatibilityCondition}
is not assumed to satisfy any particular set of equations (apart from
the Vlasov equation satisfied by $f$ near $v=0$ and $v=v_{\mathcal{I}}$),
the condition that the higher order transversal derivatives $\partial_{u}^{k}r|_{u=0}$
and $\partial_{u}^{k}\Omega^{2}|_{u=0}$ coincide with the values
determined by the process described in Remark 3 below Definition \ref{def:AsymptoticallyAdSData}
is equivalent to the statement that, along $u=0$, $(r,\Omega^{2},f)$
satisfies the system (\ref{eq:RequationFinal})\textendash (\ref{NullShellFinal})
at all orders.
\item Let $(r,\Omega^{2},f)$ be a smooth, asymptotically AdS solution of
(\ref{eq:RequationFinal})\textendash (\ref{NullShellFinal}), such
that $(r,\Omega^{2},f)$ has smooth axis $\gamma_{\mathcal{Z}}$ and
smooth conformal infinity $\mathcal{I}$ and $f$ satisfies the reflecting
boundary condition on $\mathcal{I}$. In this case, it follows trivially
from Definition \ref{def:CompatibilityCondition} that the initial
data set induced by $(r,\Omega^{2},f)$ on any slice of the form $\{u=u_{*}\}$
is smoothly compatible; in particular, Condition 1 of Definition \ref{def:CompatibilityCondition}
follows trivially from the fact that $(r,\Omega^{2},f)$ satisfies
the Einstein-equations (\ref{eq:RequationFinal})\textendash (\ref{NullShellFinal})
in a whole neighborhood of $\{u=u_{*}\}$, while Conditions 2 and
3 are a consequence of the smoothness of $(r,\Omega^{2},f)$ on $\gamma_{\mathcal{Z}}$
and $\mathcal{I}$.
\item Given a smoothly compatible, asymptotically AdS initial data set $(r_{/},\Omega_{/}^{2},\bar{f}_{/};v_{\mathcal{I}})$,
for a gauge transformation of the form (\ref{eq:GeneralGaugeTransformationInitialData})
to be smoothly compatible, it is necessary that 
\begin{equation}
\frac{dU}{du}(0)=\frac{dV}{dv}(0)\text{ and }V\in C^{\infty}([0,v_{\mathcal{I}}]).\label{eq:NecessaryForSmoothCompatibility}
\end{equation}
However, (\ref{eq:NecessaryForSmoothCompatibility}) is not sufficient
for a gauge transformation to be smoothly compatible, since Conditions
1\textendash 3 of Definition \ref{def:CompatibilityCondition} imply
an additional relation between $\frac{d^{k}V}{(dv)^{k}}(0)$ and $\frac{d^{k}V}{(dv)^{k}}(v_{\mathcal{I}})$
for all $k\in\mathbb{N}$. 
\end{itemize}
It will be useful for us in this paper, as well as in our companion
paper \cite{MoschidisVlasov}, to fix a suitable gauge condition on
initial data sets introduced by Definition \ref{def:AsymptoticallyAdSData}
that uniquely chooses a representative in each equivalence class under
gauge transformations (at least up to rescalings $(u,v)\rightarrow(\lambda u,\lambda v)$),
with the additional property that initial data sets can be uniquely
constructed by freely prescribing the initial Vlasov field $\bar{f}_{/}$
in this gauge. This will be achieved at the expense of choosing a
gauge in which smooth compatibility is, in general, not preserved.

In particular, we will introduce the following condition:
\begin{defn}
\label{def:GaugeConditionNormalise} Let $(r_{/},\Omega_{/}^{2},\bar{f}_{/};v_{\mathcal{I}})$
be a smooth, asymptotically AdS initial data set as in Definition
\ref{def:AsymptoticallyAdSData}. We will say that $(r_{/},\Omega_{/}^{2},\bar{f}_{/};v_{\mathcal{I}})$
is \emph{gauge normalised} if it satisfies the condition 
\begin{equation}
\frac{\partial_{v}r_{/}}{1-\frac{1}{3}\Lambda r_{/}^{2}}=\frac{\Omega_{/}^{2}}{4\partial_{v}r_{/}}\text{on }[0,v_{\mathcal{I}}).\label{eq:GaugeConditionNormalisedData}
\end{equation}
The gauge condition (\ref{eq:GaugeConditionNormalisedData}) is equivalent
to 
\begin{equation}
\frac{\partial_{v}r_{/}}{1-\frac{1}{3}\Lambda r_{/}^{2}}=-\frac{(\partial_{u}r)_{/}}{1-\frac{2\tilde{m}_{/}}{r_{/}}-\frac{1}{3}\Lambda r_{/}^{2}},\label{eq:EqualDuRDvRInitially}
\end{equation}
where $(\partial_{u}r)_{/}=\partial_{u}r|_{u=0}$ is determined in
terms of $(r_{/},\Omega_{/}^{2},\bar{f}_{/};v_{\mathcal{I}})$ by
(\ref{eq:FormulaForInitialR}) and (\ref{eq:HigherOrderInitialAxisR})
for $k=0$.
\end{defn}
\begin{rem*}
The condition (\ref{eq:GaugeConditionNormalisedData}) fixes a unique
representation of the trivial initial data (i.\,e.~for $\bar{f}_{/}\equiv0$)
for each value of the endpoint parameter $v_{\mathcal{I}}$. For the
standard choice $v_{\mathcal{I}}=\sqrt{-\frac{3}{\Lambda}}\pi$, the
trivial initial data set $(r_{AdS/},\Omega_{AdS/}^{2},0;\sqrt{-\frac{3}{\Lambda}}\pi)$
is expressed as: 
\begin{align}
r_{AdS/}(v) & =\sqrt{-\frac{3}{\Lambda}}\tan\Big(\frac{1}{2}\sqrt{-\frac{\Lambda}{3}}v\Big),\label{eq:AdSMetricValues}\\
\Omega_{AdS/}^{2}(v) & =1-\frac{1}{3}\Lambda r_{AdS}^{2}(v).\nonumber 
\end{align}
For different values of $v_{\mathcal{I}}$, we obtain by rescaling:
\begin{align}
r_{AdS/}^{(v_{\mathcal{I}})}(v) & =r_{AdS/}\Big(\sqrt{-\frac{3}{\Lambda}}\pi\frac{v}{v_{\mathcal{I}}}\Big),\label{eq:AdSMetricValuesRescaled1}\\
\big(\Omega_{AdS/}^{(v_{\mathcal{I}})}\big)^{2}(v) & =-\frac{3}{\Lambda}\frac{\pi^{2}}{v_{\mathcal{I}}^{2}}\Omega_{AdS/}^{2}\Big(\sqrt{-\frac{3}{\Lambda}}\pi\frac{v}{v_{\mathcal{I}}}\Big).\nonumber 
\end{align}

Let us also point out that, in view of the constraint equation (\ref{eq:ConstraintVFinal}),
the gauge condition (\ref{eq:GaugeConditionNormalisedData}) can be
alternatively expressed as 
\begin{equation}
\frac{\partial_{v}r_{/}}{1-\frac{1}{3}\Lambda r_{/}^{2}}(v)=\frac{\Omega_{/}^{2}}{4\partial_{v}r_{/}}(0)\cdot\exp\Big(4\pi\int_{0}^{v}\frac{r_{/}(T_{/})_{vv}}{(\partial_{v}r_{/})^{2}}(\bar{v})\,(\partial_{v}r_{/})d\bar{v}\Big)\label{eq:FormulaWithOmegaInitial}
\end{equation}
or, equivalently (by noting that the left hand side of (\ref{eq:FormulaWithOmegaInitial})
integrates to $\sqrt{-\frac{3}{\Lambda}}\frac{\pi}{2}$ over $[0,v_{\mathcal{I}})$):
\begin{equation}
\frac{\partial_{v}r_{/}}{1-\frac{1}{3}\Lambda r_{/}^{2}}(v)=\frac{1}{2a}\exp\Big(4\pi\int_{0}^{v}\frac{r_{/}(T_{/})_{vv}}{(\partial_{v}r_{/})^{2}}(\bar{v})\,(\partial_{v}r_{/})d\bar{v}\Big),\label{eq:FormulaForDvR/2}
\end{equation}
where the constant
\begin{equation}
a\doteq\sqrt{-\frac{\Lambda}{3}}\frac{1}{\pi}\int_{0}^{v_{\mathcal{I}}}\exp\Big(4\pi\int_{0}^{v}\frac{r_{/}(T_{/})_{vv}}{(\partial_{v}r_{/})^{2}}(\bar{v})\,(\partial_{v}r_{/})d\bar{v}\Big)\,dv\label{eq:AlphaReGavGav}
\end{equation}
is determined by the condition that $r_{/}(0)=0$, $r_{/}(v_{\mathcal{I}})=\infty$.

Finally, we should highlight that, in general, an initial data set
$(r_{/},\Omega_{/}^{2},\bar{f}_{/};v_{\mathcal{I}})$ satisfying the
gauge condition (\ref{eq:GaugeConditionNormalisedData}) will \emph{not}
be smoothly compatible (although smooth compatibility might be feasible
through a suitable gauge transformation). This follows from the observation
that, at $v=v_{\mathcal{I}}$, (\ref{eq:GaugeConditionNormalisedData})
is not consistent with $(\partial_{u}+\partial_{v})^{k}\frac{1}{r}$
vanishing at $u=v-v_{\mathcal{I}}$ for $k\ge3$ when $\tilde{m}\neq0$.
In the case of the trivial AdS initial data set (\ref{eq:RescaledAdSVlasov}),
the renormalised gauge is trivially a smoothly compatible gauge.
\end{rem*}
The gauge condition (\ref{eq:GaugeConditionNormalisedData}) allows
us to construct initial data sets $(r_{/},\Omega_{/}^{2},\bar{f}_{/};v_{\mathcal{I}})$
by freely prescribing the initial Vlasov field $\bar{f}_{/}$ under
a few regularity conditions. In particular, the following lemma will
be useful for the constructions in \cite{MoschidisVlasov}:
\begin{lem}
\label{lem:ConstructingInitialData} Let $v_{\mathcal{I}}>0$ and
let $F:[0,v_{\mathcal{I}})\times[0,+\infty)^{2}\rightarrow[0,+\infty)$
be a smooth function which is compactly supported in $(0,v_{\mathcal{I}})\times(0,+\infty)^{2}$.
Then there exists a unique smooth, asymptotically AdS initial data
set $(r_{/},\Omega_{/}^{2},\bar{f}_{/};v_{\mathcal{I}})$ for (\ref{eq:RequationFinal})\textendash (\ref{NullShellFinal})
(as in Definition \ref{def:AsymptoticallyAdSData}) satisfying the
gauge condition (\ref{eq:GaugeConditionNormalisedData}), such that
\begin{equation}
\bar{f}_{/}(v;p^{u},l)=F\big(v;\text{ }\partial_{v}r_{/}(v)p^{u},\text{ }l\big).\label{eq:AlmostVlasovFieldInvariant Form}
\end{equation}
Furthermore, in the case when $F$ satisfies the smallness condition
\begin{equation}
\mathcal{M}[F]\doteq\int_{0}^{v_{\mathcal{I}}}\frac{r_{AdS/}^{(v_{\mathcal{I}})}(T_{AdS}^{(v_{\mathcal{I}})}[F])_{vv}}{\partial_{v}r_{AdS/}^{(v_{\mathcal{I}})}}(\bar{v})\,d\bar{v}<c_{0}\ll1,\label{eq:SmallnessConditionForConstruction}
\end{equation}
where $c_{0}>0$ is an absolute constant, $(T_{AdS}^{(v_{\mathcal{I}})}[F])_{vv}$
is defined by
\[
(T_{AdS}^{(v_{\mathcal{I}})}[F])_{vv}(v)\doteq\frac{\pi}{2}\frac{1}{(r_{AdS/}^{(v_{\mathcal{I}})})^{2}(v)}\int_{0}^{+\infty}\int_{0}^{+\infty}p^{2}F(v;p,l)\,\frac{dp}{p}ldl
\]
 and $r_{AdS/}^{(v_{\mathcal{I}})},\text{ }(\Omega_{AdS/}^{(v_{\mathcal{I}})})^{2}$
are given by (\ref{eq:AdSMetricValuesRescaled1}), the following estimates
hold for some absolute constant $C>0$: 
\begin{equation}
\sup_{v\in(0,v_{\mathcal{I}})}\Big|\frac{\partial_{v}r_{/}}{1-\frac{1}{3}\Lambda r_{/}^{2}}(v)-\frac{\partial_{v}r_{AdS/}^{(v_{\mathcal{I}})}}{1-\frac{1}{3}\Lambda(r_{AdS/}^{(v_{\mathcal{I}})})^{2}}(v)\Big|\le C\mathcal{M}[F]\label{eq:SmallDifferenceInRs}
\end{equation}
and 
\begin{equation}
\int_{0}^{v_{\mathcal{I}}}\frac{r_{/}(T_{/})_{vv}}{\partial_{v}r_{/}}(\bar{v})\,d\bar{v}\le C\mathcal{M}[F].\label{eq:SmallEnergyInitial}
\end{equation}
\end{lem}
\begin{rem*}
As noted below Definition \ref{def:GaugeConditionNormalise}, the
initial data set $(r_{/},\Omega_{/}^{2},\bar{f}_{/};v_{\mathcal{I}})$
constructed in Lemma \ref{lem:ConstructingInitialData} will not,
in general, satisfy the last condition of Definition \ref{def:CompatibilityCondition}.
\end{rem*}
\begin{proof}
Let $v_{\mathcal{I}}>0$ and $F:(0,v_{\mathcal{I}})\times(0,+\infty)^{2}\rightarrow[0,+\infty)$
be a smooth function of compact support. Let us define the function
$G\in C_{c}^{\infty}\big((0,v_{\mathcal{I}})\big)$ by the relation
\begin{equation}
G(v)\doteq32\pi^{2}\int_{0}^{+\infty}\int_{0}^{+\infty}p^{2}F(v;p,l)\,\frac{dp}{p}ldl\,dv.
\end{equation}

For any $\bar{a}\ge0$, we will consider the following integral equation
for the function $\bar{\rho}(v)$, $v\ge0$: 
\begin{equation}
\bar{\rho}(v)=\frac{1}{2\bar{a}}\exp\Big(\int_{0}^{v}\frac{1}{\bar{r}(x)\big(1-\frac{1}{3}\Lambda\bar{r}^{2}(x)\big)}\bar{\rho}(x)G(x)\,dx\Big),\label{eq:ODERhoBar}
\end{equation}
where 
\begin{equation}
\bar{r}(x)\doteq\sqrt{-\frac{3}{\Lambda}}\tan\Big(\sqrt{-\frac{\Lambda}{3}}\int_{0}^{x}\bar{\rho}(y)\,dy\Big).\label{eq:RBar}
\end{equation}
The standard theory of ordinary differential equations (and the fact
that $G(x)$ is compactly supported away from $x=0$, as a consequence
of our assumptions for $F$) implies that, for any $\bar{a}>0$, there
exists a maximal $v^{(\bar{a})}\in(0,+\infty]$ such that (\ref{eq:ODERhoBar})
determines a unique pair $\bar{\rho},\bar{r}\in C^{\infty}([0,v^{(\bar{a})}))$
on the interval $[0,v^{(\bar{a})})$ and, if $v^{(\bar{a})}<+\infty$,
then 
\begin{equation}
\limsup_{v\rightarrow v^{(\bar{a})-}}|\bar{\rho}(v)|=+\infty,\label{eq:MaybeRinfinity}
\end{equation}
\begin{equation}
\liminf_{v\rightarrow v^{(\bar{a})-}}|\bar{r}(v)|=0\text{ and }v^{(\bar{a})}\in\text{supp}(G)\label{eq:TrivialRZero}
\end{equation}
or 
\begin{equation}
\limsup_{v\rightarrow v^{(\bar{a})-}}|\bar{r}(v)|=+\infty.\label{eq:MaybeRhoInfinity}
\end{equation}
 We will show that, for any $\bar{a}>0$, the endpoint $v^{(\bar{a})}$
is finite, i.\,e.
\begin{equation}
v^{(\bar{a})}<+\infty\label{eq:FiniteVa}
\end{equation}
 and 
\begin{equation}
\lim_{v\rightarrow v^{(\bar{a})-}}\bar{r}(v)=+\infty.\label{eq:BoundaryForRInfinity}
\end{equation}

\medskip{}

\noindent \emph{Proof of (\ref{eq:FiniteVa}) and (\ref{eq:BoundaryForRInfinity}).}
The fact that (\ref{eq:FiniteVa}) holds follows readily from the
assumption that $F$ is compactly supported, and hence there exists
a $v_{0}>0$ such that $G(v)=0$ for $v\ge v_{0}$: Assuming, for
the sake of contradiction that 
\[
v^{(\bar{a})}=+\infty
\]
from (\ref{eq:ODERhoBar}) we infer that 
\begin{equation}
\bar{\rho}(v)=\text{const}=c_{1}>0\text{ for }v\ge v_{0}.
\end{equation}
Thus, there exists a $v_{1}>v_{0}$ such that 
\begin{equation}
\sqrt{-\frac{\Lambda}{3}}\int_{0}^{v_{1}}\bar{\rho}(y)\,dy=\frac{\pi}{2}.
\end{equation}
From (\ref{eq:RBar}), we therefore infer that 
\[
\bar{r}(v_{1})=+\infty,
\]
which implies that $v^{(\bar{a})}<v_{1}$; hence, assuming that (\ref{eq:FiniteVa})
does not hold, we reach a contradiction.

We will now proceed to establish (\ref{eq:BoundaryForRInfinity}).
The relation (\ref{eq:ODERhoBar}) implies that $\bar{\rho}>0$ and,
hence, $\bar{r}(v)$ defined by (\ref{eq:RBar}) is strictly increasing
and positive for $v>0$; hence, if (\ref{eq:MaybeRhoInfinity}) holds,
then (\ref{eq:BoundaryForRInfinity}) would follow. Moreover, since
$G(v)$ is compactly supported away from $v=0$ and $\bar{r}(v)\ge0$
is strictly increasing in $v$, (\ref{eq:TrivialRZero}) cannot hold.
In order to establish (\ref{eq:BoundaryForRInfinity}), it thus suffices
to rule the case where
\begin{equation}
\limsup_{v\rightarrow v^{(\bar{a})-}}|\bar{\rho}(v)|=+\infty\text{ and }\limsup_{v\rightarrow v^{(\bar{a})-}}|\bar{r}(v)|<+\infty.\label{eq:ForContradiction}
\end{equation}

Assume, for the sake of contradiction, that (\ref{eq:ForContradiction})
holds. Since $\bar{\rho}>0$, the upper bound (\ref{eq:ForContradiction})
for $\bar{r}$ and the relation (\ref{eq:RBar}) between $\bar{r}$
and $\bar{\rho}$ imply that 
\begin{equation}
\int_{0}^{v^{(\bar{a})}}\bar{\rho}(v)\,dv<+\infty.\label{eq:UpperBoundRho}
\end{equation}
From (\ref{eq:ODERhoBar}), using the bound (\ref{eq:ForContradiction})
for $\bar{r}$, the fact that $G\in C_{c}^{\infty}\big((0,v_{\mathcal{I}}))$
and that $\bar{r}>0$ on the support of $G$, we can readily estimate
for some $C_{1}>0$ depending on $\bar{r}$, $G$:
\begin{align*}
\limsup_{v\rightarrow v^{(\bar{a})-}}|\bar{\rho}(v)| & =\frac{1}{2\bar{a}}\exp\Big(\int_{0}^{v^{(\bar{a})}}\frac{1}{\bar{r}(x)\big(1-\frac{1}{3}\Lambda\bar{r}^{2}(x)\big)}\bar{\rho}(x)G(x)\,dx\Big)\le\\
 & \le\frac{1}{2\bar{a}}\exp\Big(C_{1}\int_{0}^{v^{(\bar{a})}}\bar{\rho}(v)\,dv\Big)<+\infty.
\end{align*}
Therefore, we deduce that (\ref{eq:ForContradiction}) does not hold,
reaching a conrtradiction. As a consequence, (\ref{eq:BoundaryForRInfinity})
holds.

\medskip{}

It can be readily shown that $v^{(\bar{a})}$ depends continuously
on $\bar{a}>0$. Furthermore, the relations (\ref{eq:ODERhoBar})
and (\ref{eq:RBar}) imply that, for any fixed value $v_{0}>0$, the
functions $\bar{\rho}(v_{0})$ and $\bar{r}(v_{0})$ are strictly
decreasing in $\bar{a}$, provided $v_{0}<v^{(\bar{a})}$, and we
have
\[
\lim_{\bar{a}\rightarrow0^{+}}v^{(\bar{a})}=0\text{ and }\lim_{\bar{a}\rightarrow+\infty}v^{(\bar{a})}=+\infty.
\]
Therefore, there exists a unique $\bar{a}_{0}>0$ such that 
\begin{equation}
v^{(\bar{a}_{0})}=v_{\mathcal{I}}.\label{eq:Vabar0}
\end{equation}

Let us fix $\bar{a}=\bar{a}_{0}$ and let us define the function $r_{/}:[0,v_{\mathcal{I}})\rightarrow[0,+\infty)$
by the relation 
\begin{equation}
r_{/}(v)=\bar{r}(v)\label{eq:Fixingrslash}
\end{equation}
(note that this is possible in view of (\ref{eq:Vabar0})). The relations
(\ref{eq:ODERhoBar})\textendash (\ref{eq:RBar}) and (\ref{eq:BoundaryForRInfinity})
then imply that $\frac{1}{r_{/}}$ extends smoothly at $v=v_{\mathcal{I}}$,
and we have
\begin{align}
r_{/}(0) & =0\text{, }\frac{1}{r_{/}(v_{\mathcal{I}})}=0,\label{eq:BoundaryConditionsR}\\
\partial_{v}r_{/}(0) & >0,\text{ }\partial_{v}\Big(\frac{1}{r_{/}}\Big)(v_{\mathcal{I}})<0.\nonumber 
\end{align}
Let us also define $\bar{f}_{/}:(0,v_{\mathcal{I}})\times[0,+\infty)\rightarrow[0,+\infty)$
by (\ref{eq:AlmostVlasovFieldInvariant Form}). The relations (\ref{eq:ODERhoBar})\textendash (\ref{eq:RBar})
can be reexpressed in terms of $r_{/}$ and $\bar{f}_{/}$ as follows:
\begin{equation}
\frac{\partial_{v}r_{/}}{1-\frac{1}{3}\Lambda r_{/}^{2}}(v)=\frac{1}{2\bar{a}_{0}}\exp\Bigg\{32\pi^{2}\int_{0}^{v}\frac{(\partial_{v}r_{/}(x))^{3}}{\bar{r}(x)\big(1-\frac{1}{3}\Lambda\bar{r}^{2}(x)\big)^{2}}\Big(\int_{0}^{+\infty}\int_{0}^{+\infty}(p^{u})^{2}\bar{f}_{/}(x;p^{u},l)\,\frac{dp^{u}}{p^{u}}ldl\,dv\Big)\,dx\Bigg\}.\label{eq:FormulaForDvR/2-1-1}
\end{equation}
Note that (\ref{eq:BoundaryForRInfinity}) and (\ref{eq:FormulaForDvR/2-1-1})
imply that 
\begin{equation}
\bar{a}_{0}=\sqrt{-\frac{\Lambda}{3}}\frac{1}{\pi}\int_{0}^{v_{\mathcal{I}}}\exp\Bigg\{32\pi^{2}\int_{0}^{v}\frac{(\partial_{v}r_{/}(x))^{3}}{r_{/}(x)\big(1-\frac{1}{3}\Lambda r_{/}^{2}(x)\big)^{2}}\Big(\int_{0}^{+\infty}\int_{0}^{+\infty}(p^{u})^{2}\bar{f}_{/}(x;p^{u},l)\,\frac{dp^{u}}{p^{u}}ldl\,dv\Big)\,dx\Bigg\}.\label{eq:AlphaBar0}
\end{equation}

Using (\ref{eq:GaugeConditionNormalisedData}) to define $\Omega_{/}^{2}$,
i.\,e.
\begin{equation}
\Omega_{/}^{2}(v)\doteq\frac{4(\partial_{v}r_{/})^{2}}{1-\frac{1}{3}\Lambda r_{/}^{2}}(v),\label{eq:OmegaDef}
\end{equation}
we can readily calculate that the right hand side of (\ref{eq:FormulaForDvR/2-1-1})
is the same as the right hand side of (\ref{eq:FormulaForDvR/2}).
Furthermore, (\ref{eq:OmegaDef}) and (\ref{eq:BoundaryConditionsR})
imply that $r_{/}^{-2}\Omega_{/}^{2}$ extends smoothly on $v=v_{\mathcal{I}}$
and satisfies 
\[
\Omega_{/}^{2}(0)>0\text{, }r_{/}^{-2}\Omega_{/}^{2}(v_{\mathcal{I}})>0.
\]
As a result (in view also of the fact that $\bar{f}_{/}\in C_{c}^{\infty}\big((0,v_{\mathcal{I}})\times(0,+\infty)\big)$,
the quadruplet $(r_{/},\Omega_{/}^{2},\bar{f}_{/};v_{\mathcal{I}})$
satisfies all conditions of Definition \ref{def:AsymptoticallyAdSData}. 

It remains to show that, when $F$ satisfies the smallness condition
(\ref{eq:SmallnessConditionForConstruction}), then the bounds (\ref{eq:SmallDifferenceInRs})
and (\ref{eq:SmallEnergyInitial}) hold. Let us introduce the continuity
parameter $\varepsilon\in[0,1]$, and let us consider the family of
initial data sets $(r_{/}^{(\varepsilon)},(\Omega_{/}^{(\varepsilon)})^{2},\bar{f}_{/}^{(\varepsilon)};v_{\mathcal{I}})$
corresponding to 
\[
F^{(\varepsilon)}\doteq\varepsilon F.
\]
Note that the standard theory for odes applied to (\ref{eq:ODERhoBar})\textendash (\ref{eq:RBar})
implies that, for any $v\in(0,v_{\mathcal{I}})$, the quantities $r_{/}^{(\varepsilon)}(v)$,
$(\Omega_{/}^{(\varepsilon)})^{2}(v)$ depend continuously on $\varepsilon$.
Note also that, when $\varepsilon=0$, (\ref{eq:ODERhoBar})\textendash (\ref{eq:RBar})
imply that 
\[
(r_{/}^{(0)},(\Omega_{/}^{(0)})^{2},\bar{f}_{/}^{(0)};v_{\mathcal{I}})=(r_{AdS/}^{(v_{\mathcal{I}})},(\Omega_{AdS/}^{(v_{\mathcal{I}})})^{2},0;v_{\mathcal{I}}).
\]
We will argue by continuity: We will assume that, if, for some $\varepsilon_{0}\in(0,1]$
we can estimate for all $\varepsilon\in[0,\varepsilon_{0}]$
\begin{equation}
\sup_{v\in(0,v_{\mathcal{I}})}\Big|\frac{\partial_{v}r_{/}^{(\varepsilon)}}{1-\frac{1}{3}\Lambda r_{/}^{(\varepsilon)}}(v)-\frac{\partial_{v}r_{AdS/}^{(v_{\mathcal{I}})}}{1-\frac{1}{3}\Lambda(r_{AdS/}^{(v_{\mathcal{I}})})^{2}}(v)\Big|\le2C_{1}\varepsilon\mathcal{M}[F]\label{eq:SmallDifferenceInRs-1}
\end{equation}
and 
\begin{equation}
\int_{0}^{v_{\mathcal{I}}}\frac{r_{/}^{(\varepsilon)}(T_{/}^{(\varepsilon)})_{vv}}{\partial_{v}r_{/}^{(\varepsilon)}}(\bar{v})\,d\bar{v}\le2C_{1}\varepsilon\mathcal{M}[F]\label{eq:SmallEnergyInitial-1}
\end{equation}
then, in fact, the following stronger estimates hold for all $\varepsilon\in[0,\varepsilon_{0}]$:
\begin{equation}
\sup_{v\in(0,v_{\mathcal{I}})}\Big|\frac{\partial_{v}r_{/}^{(\varepsilon)}}{1-\frac{1}{3}\Lambda r_{/}^{(\varepsilon)}}(v)-\frac{\partial_{v}r_{AdS/}^{(v_{\mathcal{I}})}}{1-\frac{1}{3}\Lambda(r_{AdS/}^{(v_{\mathcal{I}})})^{2}}(v)\Big|\le C_{1}\varepsilon\mathcal{M}[F]\label{eq:SmallDifferenceInRs-1-1}
\end{equation}
and 
\begin{equation}
\int_{0}^{v_{\mathcal{I}}}\frac{r_{/}^{(\varepsilon)}(T_{/}^{(\varepsilon)})_{vv}}{\partial_{v}r_{/}^{(\varepsilon)}}(\bar{v})\,d\bar{v}\le C_{1}\varepsilon\mathcal{M}[F]\label{eq:SmallEnergyInitial-1-1}
\end{equation}

Differentiating (\ref{eq:ODERhoBar})\textendash (\ref{eq:RBar})
for $F^{(\varepsilon)}$ with respect to $\varepsilon$ (noting that
$\frac{1}{2\bar{a}_{0}}=\bar{\rho}^{(\varepsilon)}(0)$), we obtain
\begin{align}
\frac{d}{d\varepsilon}\Big(\bar{\rho}^{(\varepsilon)}(v)-\bar{\rho}^{(\varepsilon)}(0)\Big)= & \bar{\rho}^{(\varepsilon)}(v)\cdot\int_{0}^{v}\frac{d}{d\varepsilon}\big(\frac{1}{\bar{r}^{(\varepsilon)}\big(1-\frac{1}{3}\Lambda(\bar{r}^{(\varepsilon)})^{2}\big)}\bar{\rho}^{(\varepsilon)}\big)(x)G^{(\varepsilon)}(x)\,dx+\label{eq:ODERhoBar-1}\\
 & +\bar{\rho}^{(\varepsilon)}(v)\cdot\int_{0}^{v}\frac{1}{\bar{r}^{(\varepsilon)}\big(1-\frac{1}{3}\Lambda(\bar{r}^{(\varepsilon)})^{2}\big)}\bar{\rho}^{(\varepsilon)}(x)G(x)\,dx\nonumber 
\end{align}
and
\begin{equation}
\frac{d}{d\varepsilon}\bar{r}^{(\varepsilon)}(v)\doteq\big(1-\frac{1}{3}\Lambda(\bar{r}^{(\varepsilon)})^{2}\big)\int_{0}^{v}\frac{d}{d\varepsilon}\bar{\rho}^{(\varepsilon)}(y)\,dy.\label{eq:RBar-1}
\end{equation}
Setting for convenience 
\[
R_{1}^{(\varepsilon)}\doteq\frac{d}{d\varepsilon}\log\bar{\rho}^{(\varepsilon)},\text{ }R_{2}^{(\varepsilon)}\doteq\frac{\frac{d}{d\varepsilon}\bar{r}^{(\varepsilon)}(v)}{1-\frac{1}{3}\Lambda(\bar{r}^{(\varepsilon)})^{2}}
\]
and using the fact that 
\[
\bar{r}^{(\varepsilon)}=r_{/}^{(\varepsilon)},\text{ }\bar{\rho}^{(\varepsilon)}=\frac{\partial_{v}r_{/}^{(\varepsilon)}}{1-\frac{1}{3}\Lambda r_{/}^{(\varepsilon)}}(v),
\]
(see (\ref{eq:Fixingrslash})), the relations (\ref{eq:ODERhoBar-1})\textendash (\ref{eq:RBar-1})
and the bounds (\ref{eq:SmallnessConditionForConstruction}) and (\ref{eq:SmallDifferenceInRs-1})\textendash (\ref{eq:SmallEnergyInitial-1})
readily yield: 
\begin{equation}
|R_{1}^{(\varepsilon)}(v)-R_{1}^{(\varepsilon)}(0)|\le C_{0}\int_{0}^{v}H_{1}^{(\varepsilon)}(\bar{v})\cdot\big(|R_{1}^{(\varepsilon)}(\bar{v})|+\frac{|R_{2}^{(\varepsilon)}(\bar{v})|}{\bar{v}}\big)\,d\bar{v}+C_{0}\mathcal{M}[F]\label{eq:ForGronwalR1}
\end{equation}
 and 
\begin{equation}
|R_{2}^{(\varepsilon)}(v)|\le C_{0}\int_{0}^{v}|R_{1}^{(\varepsilon)}(\bar{v})|\,d\bar{v}.\label{eq:R2}
\end{equation}
for some absolute constant $C_{0}>0$, where $H_{1}^{(\varepsilon)}(v)\ge0$
satisfies the bound 
\[
\int_{0}^{v}H_{1}^{(\varepsilon)}(\bar{v})\,d\bar{v}\le C(C_{1})\varepsilon\mathcal{M}[F]
\]
 for some $C(C_{1})>0$ depending only on $C_{1}$. 

Plugging (\ref{eq:R2}) in (\ref{eq:ForGronwalR1}) and applying Gronwall's
inequality for the function 
\[
\bar{R}_{1}^{(\varepsilon)}(v)\doteq\max_{\bar{v}\in[0,v]}|R_{1}^{(\varepsilon)}(\bar{v})-R_{1}^{(\varepsilon)}(0)|,
\]
we readily obtain that 
\begin{equation}
\sup_{v\in[0,v_{\mathcal{I}})}|R_{1}^{(\varepsilon)}(v)-R_{1}^{(\varepsilon)}(0)|\le C_{0}\exp\big(C(C_{1})\mathcal{M}[F]\big)\mathcal{M}[F]\cdot\Big(1+|R_{1}^{(\varepsilon)}(0)|\Big).\label{eq:UpperBoundR1}
\end{equation}

The identity 
\[
\frac{d}{d\varepsilon}\Big(\int_{0}^{v_{\mathcal{I}}}\bar{\rho}^{(\varepsilon)}(v)\,dv\Big)=\frac{d}{d\varepsilon}\Big(\sqrt{-\frac{3}{\Lambda}}\frac{\pi}{2}\Big)=0,
\]
which holds for all $\varepsilon\in[0,\varepsilon_{0}]$ (as a consequence
of (\ref{eq:RBar}) and the fact that $\bar{r}^{(\varepsilon)}(v_{\mathcal{I}})=+\infty$,
in view of (\ref{eq:Vabar0})), implies, in view of the bound (\ref{eq:UpperBoundR1})
and the estimate (\ref{eq:SmallDifferenceInRs-1-1}) that 
\begin{align}
0 & =\int_{0}^{v_{\mathcal{I}}}\bar{\rho}^{(\varepsilon)}(v)R_{1}^{(\varepsilon)}(v)\,dv=\label{eq:IdentityForR0}\\
 & =\int_{0}^{v_{\mathcal{I}}}\Big(\frac{\partial_{v}r_{AdS/}^{(v_{\mathcal{I}})}}{1-\frac{1}{3}\Lambda(r_{AdS/}^{(v_{\mathcal{I}})})^{2}}(v)+O\big(C_{1}\mathcal{M}[F]\big)\Big)\cdot\Big(R_{1}^{(\varepsilon)}(0)+O\big(\exp\big(C(C_{1})\mathcal{M}[F]\big)\mathcal{M}[F]\big)\Big)R_{1}^{(\varepsilon)}(v)\,dv.\nonumber
\end{align}
The identity (\ref{eq:IdentityForR0}) implies (since $R_{1}^{(\varepsilon)}(0)$
is independent of $v$ and hence can be moved outside of the integral)
that, after possibly choosing a larger absolute constant $C_{0}>0$,
and assuming that the constant $c_{0}$ in (\ref{eq:SmallnessConditionForConstruction})
has been chosen small enough in terms of $C_{1}$: 
\begin{equation}
|R_{1}^{(\varepsilon)}(0)|\le C_{0}\exp\big(C(C_{1})\mathcal{M}[F]\big)\mathcal{M}[F]\le C_{0}\mathcal{M}[F].\label{eq:UpperBoundR0}
\end{equation}
From (\ref{eq:UpperBoundR1}) and (\ref{eq:UpperBoundR0}) we therefore
infer that, provided $C_{1}\gg C_{0}$ and $c_{0}$ is small enough
in terms of $C_{1}$, we can estimate for for any $\bar{\varepsilon}\in[0,\varepsilon_{0}]$:
\begin{equation}
\sup_{v\in(0,v_{\mathcal{I}})}\Bigg|\frac{d}{d\varepsilon}\Big(\frac{\partial_{v}r_{/}^{(\varepsilon)}}{1-\frac{1}{3}\Lambda r_{/}^{(\varepsilon)}}\Big)(v)\Big|_{\varepsilon=\bar{\varepsilon}}\Bigg|=\sup_{v\in(0,v_{\mathcal{I}})}\Bigg|\frac{d}{d\varepsilon}\bar{\rho}^{(\varepsilon)}(v)\Big|_{\varepsilon=\bar{\varepsilon}}\Bigg|\le2C_{0}\mathcal{M}[F]\le C_{1}\mathcal{M}[F].\label{eq:DerivativeInREpsilon}
\end{equation}

Integrating (\ref{eq:DerivativeInREpsilon}) in $\bar{\varepsilon}\in[0,\varepsilon_{0}]$,
we readily infer (\ref{eq:SmallDifferenceInRs-1-1}). Using (\ref{eq:SmallDifferenceInRs-1-1}),
the bound (\ref{eq:SmallEnergyInitial-1-1}) follows readily from
(\ref{eq:SmallnessConditionForConstruction}) and the relation (\ref{eq:AlmostVlasovFieldInvariant Form})
between $\bar{f}_{/}$ and $F$. Therefore, the proof of the lemma
is complete.

\end{proof}

\subsection{\label{subsec:GaugeFixing} Transformations between gauge normalised
and smoothly compatible initial data sets}

In this section, we will investigate the relationship between smoothly
compatible and gauge normalised initial data sets, as introduced by
Definitions \ref{def:CompatibilityCondition} and \ref{def:GaugeConditionNormalise},
respectively.

The following lemma shows that, for every smoothly compatible, asymptotically
AdS initial data set $(r_{/},\Omega_{/}^{2},\bar{f}_{/};v_{\mathcal{I}})$,
there exists a unique gauge transformation fixing $v=0,v_{\mathcal{I}}$,
for which the gauge normalisation condition (\ref{eq:GaugeConditionNormalisedData})
is achieved:
\begin{lem}
\label{lem:SmoothToNorm} Let $(r_{/},\Omega_{/}^{2},\bar{f}_{/};v_{\mathcal{I}})$
be a smoothly compatible asymptotically AdS initial data set, in accordance
with Definition \ref{def:CompatibilityCondition}, with bounded support
in phase space (i.\,e.~satisfying the condition (\ref{eq:BoundedSupportDefinition})).
Then, there exists a gauge transformation $v\rightarrow v'=V(v)$
of the form (\ref{eq:GaugeCoordinateChange}) satisfying (\ref{eq:NecessaryForSmoothCompatibility})
and
\[
V(v_{\mathcal{I}})=v_{\mathcal{I}},
\]
 such that the transformed initial data set $(r_{/}^{\prime},(\Omega_{/}^{\prime})^{2},\bar{f}_{/}^{\prime};v_{\mathcal{I}})$
(see the relation (\ref{eq:GeneralGaugeTransformationInitialData}))
satisfies the gauge condition (\ref{eq:GaugeConditionNormalisedData}).
\end{lem}
\begin{rem*}
As noted below Definition \ref{def:GaugeConditionNormalise} (see
also the remark below Definition \ref{def:CompatibilityCondition}),
the aforementioned gauge transformation will not be, in general, smoothly
compatible, despite satisfying (\ref{eq:NecessaryForSmoothCompatibility}).
See also Lemma \ref{lem:TransformationForNormalisationDevelopment}.
\end{rem*}
\begin{proof}
In order to construct the function $v'=V(v)$, we will make use of
the equivalent form (\ref{eq:FormulaWithOmegaInitial}) of the gauge
condition (\ref{eq:GaugeConditionNormalisedData}). To this end, let
us note that the function 
\[
F(v)\doteq\exp\Big(4\pi\int_{0}^{v}\frac{r_{/}(T_{/})_{vv}}{(\partial_{v}r_{/})^{2}}(\bar{v})\,(\partial_{v}r_{/})d\bar{v}\Big)
\]
(appearing in the right hand side of (\ref{eq:FormulaWithOmegaInitial}))
is gauge independent, i.\,e.~transforms under a gauge transformation
of the form (\ref{eq:GeneralGaugeTransformationInitialData}) as 
\[
F'(V(v))=F(v).
\]
As a result, the function 
\[
G(v)\doteq\frac{\Omega_{/}^{2}}{4\partial_{v}r_{/}}(0)\cdot F(v)
\]
transforms under such a gauge transformation with $\frac{dU}{du}(0)\doteq\frac{dV}{dv}(0)$
as:
\[
G'(V(v))=\frac{1}{\frac{dV}{dv}(0)}\cdot G(v).
\]

The gauge condition (\ref{eq:GaugeCoordinateChange}) for the gauge
transformed initial data set $(r_{/}^{\prime},(\Omega_{/}^{\prime})^{2},\bar{f}_{/}^{\prime};v_{\mathcal{I}})$
is equivalent to the relation 
\begin{equation}
\frac{dV}{dv}(v)=\frac{dV}{dv}(0)\cdot\frac{\partial_{v}r_{/}}{1-\frac{1}{3}\Lambda r_{/}^{2}}(v)\cdot\frac{1}{G(v)}\text{ for all }v\in[0,v_{\mathcal{I}}).\label{eq:ODEForV}
\end{equation}
Note that the relation (\ref{eq:ODEForV}) is trivially satisfied
at $v=0$ for any choice of the function $V(v)$, as a consequence
of the fact that (\ref{eq:SmoothlyCompatibleAxisCondition}) holds
for any smoothly compatible initial data set. 

Considering 
\[
b\doteq\frac{dV}{dv}(0)
\]
 as a positive parameter, we infer that the proof of the lemma will
conclude by showing that there exists a unique $b>0$ such that the
smooth function $V(v)$ determined by (\ref{eq:ODEForV}) and the
condition $V(0)=0$, i.\,e.
\[
V(v)\doteq b\int_{0}^{v}\frac{\partial_{v}r_{/}}{1-\frac{1}{3}\Lambda r_{/}^{2}}(\bar{v})\cdot\frac{1}{G(\bar{v})}\,d\bar{v},
\]
satisfies in addition
\begin{equation}
V(v_{\mathcal{I}})=v_{\mathcal{I}}.\label{eq:Aman}
\end{equation}
It can be readily verified that (\ref{eq:Aman}) uniquely fixes $b$
by the relation
\begin{equation}
b=\frac{v_{\mathcal{I}}}{\int_{0}^{v_{\mathcal{I}}}\frac{\partial_{v}r_{/}}{1-\frac{1}{3}\Lambda r_{/}^{2}}(\bar{v})\cdot\frac{1}{G(\bar{v})}\,d\bar{v}},\label{eq:B}
\end{equation}
noting that the finiteness of the denominator in (\ref{eq:B}) follows
from the fact that $\bar{f}_{/}$ was assumed to be of bounded support
in phase space.
\end{proof}
Let us now turn to the opposite question of that addressed by Lemma
\ref{lem:SmoothToNorm}, namely that of determining whether an initial
data set $(r_{/},\Omega_{/}^{2},\bar{f}_{/};v_{\mathcal{I}})$ given
in the normalised gauge of Definition \ref{def:GaugeConditionNormalise}
is gauge-equivalent to a smoothly compatible initial data set, as
in Definition \ref{def:CompatibilityCondition}.\footnote{The advantage of working in the normalised gauge condition of Definition
\ref{def:GaugeConditionNormalise} lies in the flexibility it provides
to uniquely determine an initial data set $(r_{/},\Omega_{/}^{2},\bar{f}_{/};v_{\mathcal{I}})$
by freely prescribing the value of $\bar{f}_{/}$. } The following lemma (which will not be used again in this paper,
but which will be useful for our companion paper \cite{MoschidisVlasov})
provides a broad class of normalised initial data sets for which such
a transformation always exists; this class contains, in particular,
the initial data sets considered in our companion paper \cite{MoschidisVlasov}.
The proof of this result will in make use of Proposition \ref{prop:MihalisWellPosedness}
for double characterisitc initial value problems for (\ref{eq:RequationFinal})\textendash (\ref{NullShellFinal}),
established later in Section \ref{sec:Well-posedness-of-the}. 
\begin{lem}
\label{lem:SufficientForCompatibility} Let $(r_{/},\Omega_{/}^{2},\bar{f}_{/};v_{\mathcal{I}})$
be a smooth asymptotically AdS initial data set with bounded support
in phase space, in accordance with Definition \ref{def:AsymptoticallyAdSData},
satisfying the normalised gauge condition (\ref{eq:GaugeConditionNormalisedData}).
Assume that $\bar{f}_{/}$ is supported away from $v=0,v_{\mathcal{I}}$
and $l=0$, i.\,e.~there exists some $\varepsilon>0$, such that
$\bar{f}_{/}$ satisfies 
\begin{equation}
\bar{f}_{/}(v;p,l)=0\text{ for }v\in(0,\varepsilon]\cup[v_{\mathcal{I}}-\varepsilon,v_{\mathcal{I}})\label{eq:ZeroFNearBoundary}
\end{equation}
and 
\begin{equation}
\bar{f}_{/}(v;p,l)=0\text{ for }l\in[0,\varepsilon].\label{eq:ZeroFSmallL}
\end{equation}
Then, there exists a gauge transformation of the form (\ref{eq:GaugeCoordinateChange})\textendash (\ref{eq:GeneralGaugeTransformationInitialData})
with $\frac{dU}{du}(0)=1$ and
\begin{equation}
v'(v)=v\text{ for }v\le v_{\mathcal{I}}-\frac{1}{2}\varepsilon\text{ and }v'(v_{\mathcal{I}})=v_{\mathcal{I}}\label{eq:ConditionGaugeOnlAway}
\end{equation}
such that the transformed initial data set $(r_{/}^{\prime},(\Omega_{/}^{\prime})^{2},\bar{f}_{/}^{\prime};v_{\mathcal{I}})$
is smoothly compatible, in accordance with Definition \ref{def:CompatibilityCondition}.
Furthermore, for any $\eta_{0}\in(0,1)$, the gauge transformation
can be chosen so that 
\begin{equation}
1-\eta_{0}\le\frac{dv'}{dv}(v)\le1+\eta_{0}\text{ for }v\in[0,v_{\mathcal{I}}]\label{eq:FirstDerivativeTransformation}
\end{equation}
and 
\begin{equation}
\max_{v\in[0,v_{\mathcal{I}}]}\Big|\frac{d^{2}v'}{(dv)^{2}}(v)\Big|\le\int_{0}^{v_{\mathcal{I}}}\Big(\frac{1-\Lambda r_{/}^{2}}{1-\frac{1}{3}\Lambda r_{/}^{2}}(T_{/})_{vv}+3(T_{/})_{uv}\Big)(v)\,dv+\frac{\eta_{0}}{v_{\mathcal{I}}},\label{eq:SecondDerivativeTransformation}
\end{equation}
where $(T_{/})_{vv}$, $(T_{/})_{uv}$ are defined in terms of $(r_{/},\Omega_{/}^{2},\bar{f}_{/})$
by the corresponding relations in (\ref{eq:DefinitionMassInitially}).
\end{lem}
\begin{rem*}
At a spacetime level, the gauge transformation in the statement of
Lemma \ref{lem:SufficientForCompatibility} is of the form $(u,v)\rightarrow(u,v')$. 
\end{rem*}
\begin{proof}
The proof of Lemma \ref{lem:SufficientForCompatibility} will proceed
by constructing an asymptotically AdS solution $(r,\Omega^{2},f)$
of (\ref{eq:RequationFinal})\textendash (\ref{NullShellFinal}) on
the domain $\mathcal{U}_{T;v_{\mathcal{I}}}$ for some $T>0$, satisfying
the reflecting boundary condition on $\{u=v-v_{\mathcal{I}}\}$, such
that: 

\begin{itemize}

\item $(r,\Omega^{2},f)$ induces on $u=0$ the initial data set
$(r_{/},\Omega_{/}^{2},\bar{f}_{/};v_{\mathcal{I}})$,

\item $(r,\Omega^{2},f)$ can be transformed into a solution with
smooth axis $\{u=v\}$ and smooth conformal infinity $\{u=v-v_{\mathcal{I}}\}$
after applying a gauge transformation $(u,v)\rightarrow(u,v')$ with
$v'$ satisfying (\ref{eq:ConditionGaugeOnlAway}).

\end{itemize} 

The construction of $(r,\Omega^{2},f)$ will be performed in three
steps.

\begin{figure}[h] 
\centering 
\scriptsize
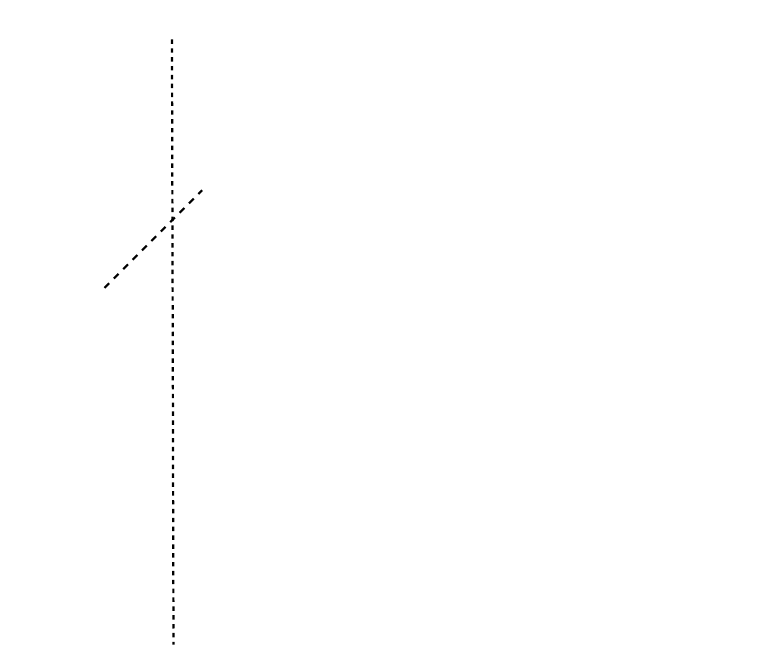 
\caption{Schematic depiction of the (partially overlapping) domains $\mathcal{D}_{0}^{\epsilon}$ and $\mathcal{W}_{*}$  in the $(u,v)$-plane. The fact that the Vlasov field $f$ is initially supported on $v\le v_\mathcal{I}-\epsilon$ and its phase space support contains only null geodesics with non-zero angular momentum, projecting as timelike curves on the $(u,v)$-plane, implies that the development $(r,\Omega^2,f)$ is vacuum (and hence isometric to Schwarzschild--AdS) in a neighborhood of the point ${0}\times {v_{\mathcal{I}}-\frac{1}{4}\epsilon}$ of the form $\{v\ge v_{\mathcal{I}}-\epsilon+Cu\}\cap \mathcal{W}_*$. \label{fig:Data_Compatibility_First}}
\end{figure}

\begin{enumerate}

\item In view of the assumption (\ref{eq:ZeroFNearBoundary}) for
$\bar{f}_{/}$ and the fact that $r_{/}$, $\Omega_{/}^{2}$ satisfy
the normalised gauge condition (\ref{eq:GaugeConditionNormalisedData}),
it follows readily that $(r_{/},\Omega_{/}^{2},\bar{f}_{/})$ coincides,
for $v\in(0,\varepsilon]$, with the normalised (rescaled) trivial
initial data set $(r_{AdS/}^{(v_{0})},\big(\Omega_{AdS/}^{(v_{0})}\big)^{2},0)$
, given by (\ref{eq:AdSMetricValuesRescaled1}) for some $v_{0}>0$.
Therefore, if we define the triplet $(r,\Omega^{2},f)$ on the domain
\[
\mathcal{D}_{0}^{\varepsilon}=\big\{0\le u\le\varepsilon\big\}\cap\big\{ u<v<\varepsilon\big\}
\]
 (see Figure \ref{fig:Data_Compatibility_First}) as 
\begin{equation}
(r,\Omega^{2},f)|_{\mathcal{D}_{0}^{\varepsilon}}\doteq(r_{AdS}^{(v_{0})},(\Omega_{AdS}^{(v_{0})})^{2},0),\label{eq:SolutionFirstStep}
\end{equation}
where 
\begin{align}
r_{AdS}^{(v_{0})}(u,v) & =r_{AdS}\Big(\sqrt{-\frac{3}{\Lambda}}\pi\frac{u}{v_{0}},\sqrt{-\frac{3}{\Lambda}}\pi\frac{v}{v_{0}}\Big),\label{eq:AdSMetricValuesRescaled-1-1}\\
\big(\Omega_{AdS}^{(v_{0})}\big)^{2}(u,v) & =-\frac{3}{\Lambda}\frac{\pi^{2}}{v_{0}^{2}}\Omega_{AdS}^{2}\Big(\sqrt{-\frac{3}{\Lambda}}\pi\frac{u}{v_{0}},\sqrt{-\frac{3}{\Lambda}}\pi\frac{v}{v_{0}}\Big),\nonumber 
\end{align}
then $(r,\Omega^{2},f)|_{\mathcal{D}_{0}^{\varepsilon}}$ is (trivially)
a smooth solution of (\ref{eq:RequationFinal})\textendash (\ref{NullShellFinal}),
with smooth axis $\{u=v\}$, inducing on $\{u=0\}\cap\mathcal{D}_{0}^{\varepsilon}$
the initial data $(r_{/},\Omega_{/}^{2},\bar{f}_{/})$.

\item For any $u_{*}\in(0,\frac{1}{4}\varepsilon)$, let us consider
the double characteristic initial value problem for the system (\ref{eq:RequationFinal})\textendash (\ref{NullShellFinal})
on the domain 
\[
\mathcal{W}_{*}\doteq[0,u_{*}]\times[\frac{1}{2}\varepsilon,v_{\mathcal{I}}-\frac{1}{4}\varepsilon]
\]
 (see Figure \ref{fig:Data_Compatibility_First}) with characteristic
initial data on $[0,u_{*}]\times\{\frac{1}{2}\varepsilon\}$ and $\{0\}\times[\frac{1}{2}\varepsilon,v_{\mathcal{I}}-\frac{1}{4}\varepsilon]$
given, respectively, by 
\begin{equation}
(r_{\backslash},\Omega_{\backslash}^{2},\bar{f}_{\backslash})(u)=(r_{AdS}^{(v_{0})}(u,\frac{1}{2}\varepsilon),\big(\Omega_{AdS}^{(v_{0})}\big)^{2}(u,\frac{1}{2}\varepsilon),0)\text{ for }u\in[0,u_{*}]\label{eq:LeftIntialData}
\end{equation}
 and 
\begin{equation}
(r_{/},\Omega_{/}^{2},\bar{f}_{/})|_{[\frac{1}{2}\varepsilon,v_{\mathcal{I}}-\frac{1}{4}\varepsilon]}.\label{eq:RightInitialData}
\end{equation}
In view of the fact that $u_{*}<\frac{1}{4}\varepsilon$, there exists
a constant $c_{0}$ (depending on $\varepsilon$) such that 
\[
\min_{u\in[0,u_{*}]}r_{\backslash}(u),\text{ }\min_{v\in[\frac{1}{2}\varepsilon,v_{\mathcal{I}}-\frac{1}{4}\varepsilon]}r_{/}(v)>c_{0}
\]
 and 
\[
\max_{u\in[0,u_{*}]}r_{\backslash}(u),\text{ }\max_{v\in[\frac{1}{2}\varepsilon,v_{\mathcal{I}}-\frac{1}{4}\varepsilon]}r_{/}(v)<c_{0}^{-1}.
\]
 As a result, by applying Proposition \ref{prop:MihalisWellPosedness},
we infer that, provided $u_{*}$ is chosen small enough in terms of
$(r_{/},\Omega_{/}^{2},\bar{f}_{/})$ and $\varepsilon$, there exists
a unique smooth solution $(r,\Omega^{2},f)$ of (\ref{eq:RequationFinal})\textendash (\ref{NullShellFinal})
on $\mathcal{W}_{*}$ inducing on $[0,u_{*}]\times\{\frac{1}{2}\varepsilon\}$
and $\{0\}\times[\frac{1}{2}\varepsilon,v_{\mathcal{I}}-\frac{1}{4}\varepsilon]$
the initial data (\ref{eq:LeftIntialData}) and (\ref{eq:RightInitialData}),
respectively.

Furthermore, since $(r_{/},\Omega_{/}^{2},\bar{f}_{/})\equiv(r_{AdS/}^{(v_{0})},\big(\Omega_{AdS/}^{(v_{0})}\big)^{2},0)$
for $v\in[\frac{1}{2}\varepsilon,\varepsilon]$, the solution $(r,\Omega^{2},f)$
satisfies 
\begin{equation}
(r,\Omega^{2},f)|_{[0,u_{*}]\times[\frac{1}{2}\varepsilon,\varepsilon]}\equiv(r_{AdS}^{(v_{0})},(\Omega_{AdS}^{(v_{0})})^{2},0).
\end{equation}
Therefore, the solution $(r,\Omega^{2},f)$ constructed in this step
on $\mathcal{W}_{*}$ coincides on $\mathcal{D}_{0}^{\varepsilon}\cap\mathcal{W}_{*}$
with the solution $(r,\Omega^{2},f)$ constructed in the previous
step on $\mathcal{D}_{0}^{\varepsilon}$ (see Figure \ref{fig:Data_Compatibility_First});
as a result, by gluing the two solutions, we obtain a single smooth
solution $(r,\Omega^{2},f)$ of (\ref{eq:RequationFinal})\textendash (\ref{NullShellFinal})
on $\mathcal{D}_{0}^{\varepsilon}\cup\mathcal{W}_{*}$.

\item In view of the assumption that $\bar{f}_{/}$ has bounded support
in phase space and satisfies (\ref{eq:ZeroFSmallL}), we infer that
there exists constant $c>0$ such that, for all $v\in[0,v_{\mathcal{I}}-\frac{1}{4}\varepsilon]$.
\begin{equation}
\bar{f}(v;p,l)=0\text{ }\text{when }p\le c\text{ or }\frac{1}{p}\le c\text{ or }l\ge c^{-1}v_{\mathcal{I}}.\label{eq:InitialBoundfbar}
\end{equation}
The condition (\ref{eq:InitialBoundfbar}) and the smoothness of the
solution $(r,\Omega^{2},f)$ constructed in the previous step on $\mathcal{D}_{0}^{\varepsilon}\cup\mathcal{W}_{*}$
implies there exists a $u^{\prime}>0$ sufficiently small and a $C>0$
sufficiently large such that $f$ satisfies on $\{0\le u\le u^{\prime}\}$:
\begin{equation}
supp(f)\cap\{0\le u\le u^{\prime}\}\subset\Big\{\frac{p^{v}}{p^{u}}\le C\Big\},
\end{equation}
i.\,e.~that the phase space support of $f$ restricted on the physical
space domain $\{0\le u\le u^{\prime}\}\cap\mathcal{W}_{*}$ contains
null geodesics which project on the $(u,v)$-plane as timelike curves
$\gamma$ of slope satisfying 
\[
\frac{\dot{\gamma}^{v}}{\dot{\gamma}^{u}}\le C.
\]
 In particular, since $\bar{f}_{/}$ satisfies (\ref{eq:ZeroFNearBoundary})
(and hence $\bar{f}_{/}\equiv0$ for $v\ge v_{\mathcal{I}}-\varepsilon$),
every null geodesic $\gamma$ in the support of $f$ satisfies 
\[
\gamma\cap\{0\le u\le u^{\prime}\}\subset\{v\le v_{\mathcal{I}}-\varepsilon+Cu\},
\]
that is to say:
\begin{equation}
f\equiv0\text{ on }\mathcal{W}_{*}\cap\big([0,u^{\prime}]\times[v_{\mathcal{I}}-\varepsilon+Cu^{\prime},+\infty)\big).\label{eq:ZeroFAway}
\end{equation}
The constraint equations (\ref{eq:ConstraintVFinal})\textendash (\ref{eq:ConstraintUFinal})
then imply that, on $\mathcal{W}_{*}\cap\big([0,u^{\prime}]\times[v_{\mathcal{I}}-\varepsilon+Cu^{\prime},+\infty)\big)$,
$(r,\Omega^{2})$ is locally isometric to a member of the Schwarzschild\textendash AdS
family of metrics (this is a consequence of the extension of Birkhoff's
theorem to the case $\Lambda<0$; see \cite{Eiesland1925}). In particular,
setting 
\[
M\doteq\tilde{m}_{/}(v_{\mathcal{I}}-\frac{1}{4}\varepsilon)=\tilde{m}_{/}(v_{\mathcal{I}})
\]
(the last equality following from the assumption (\ref{eq:ZeroFNearBoundary})
on $\bar{f}_{/}$) and assuming that $u^{\prime}$ has been fixed
small enough so that 
\[
Cu^{\prime}<\frac{1}{2}\varepsilon,
\]
 the relations (\ref{eq:DefinitionHawkingMass}), (\ref{eq:ConstraintVFinal}),
(\ref{eq:ConstraintUFinal}) and (\ref{eq:ZeroFAway}) imply that,
on $\mathcal{W}_{*}\cap\big([0,u^{\prime}]\times[v_{\mathcal{I}}-\frac{1}{2}\varepsilon,+\infty)\big)=[0,u^{\prime}]\times[v_{\mathcal{I}}-\frac{1}{2}\varepsilon,v_{\mathcal{I}}-\frac{1}{4}\varepsilon]$:
\begin{equation}
\Omega^{2}=4\frac{\partial_{v}r(-\partial_{u}r)}{1-\frac{2M}{r}-\frac{1}{3}\Lambda r^{2}}\label{eq:OmegaForSchwarzschild}
\end{equation}
and
\begin{equation}
\partial_{u}\Big(\frac{\partial_{v}r}{1-\frac{2M}{r}-\frac{1}{3}\Lambda r^{2}}\Big)=\partial_{v}\Big(\frac{-\partial_{u}r}{1-\frac{2M}{r}-\frac{1}{3}\Lambda r^{2}}\Big)=0.\label{eq:AlmostConstantKappaKappaBar}
\end{equation}

\medskip{}

\noindent \emph{Remark.} By considering a possibly smaller value of
$\varepsilon$ (depending on $(r_{/},\Omega_{/}^{2},\bar{f}_{/})$),
we can arrange so that $\frac{2\tilde{m}_{/}(v_{\mathcal{I}})}{r_{/}(v)}<1$
for $v\ge v_{\mathcal{I}}-\frac{1}{2}\varepsilon$. In this case,
assuming that $u^{\prime}$ is even smaller, if necessary, we can
bound $1-\frac{2M}{r}-\frac{1}{3}\Lambda r^{2}>0$ on $[0,u^{\prime}]\times[v_{\mathcal{I}}-\frac{1}{2}\varepsilon,v_{\mathcal{I}}-\frac{1}{4}\varepsilon]$.
Hence, for the rest of the proof, dividing with $1-\frac{2M}{r}-\frac{1}{3}\Lambda r^{2}$
will not pose a concern. 

\medskip{}

Let us define the functions $K:[v_{\mathcal{I}}-\frac{1}{2}\varepsilon,v_{\mathcal{I}})\rightarrow(0,+\infty)$
and $\bar{K}:[0,u^{\prime}]\rightarrow(0,+\infty)$ by 
\[
K(v)\doteq\frac{\partial_{v}r_{/}}{1-\frac{2\tilde{m}_{/}}{r_{/}}-\frac{1}{3}\Lambda r_{/}^{2}}(v)
\]
and 
\[
\bar{K}(u)\doteq\frac{-\partial_{u}r}{1-\frac{2\tilde{m}}{r}-\frac{1}{3}\Lambda r^{2}}(u,v_{\mathcal{I}}-\frac{1}{2}\varepsilon).
\]
Since the functions $\frac{1}{r_{/}}$, $\frac{\Omega_{/}^{2}}{r_{/}^{2}}$
and $\bar{f}_{/}(\cdot;\Omega_{/}^{-2}p,l)$ extend smoothly on $v=v_{\mathcal{I}}$
for any $p,l\ge0$ (as a consequence of the conditions imposed by
Definition \ref{def:AsymptoticallyAdSData}), the function $K$ can
be extended on the whole of the interval $[v_{\mathcal{I}}-\frac{1}{2}\varepsilon,v_{\mathcal{I}}+\delta]$
for some $\delta>0$ such that 
\[
K\in C^{\infty}\big([v_{\mathcal{I}}-\frac{1}{2}\varepsilon,v_{\mathcal{I}}+\delta]\big).
\]
 From now on, we will assume that we have fixed such a smooth extension
of $K$. Note also that, in view of the fact that $\partial_{v}\tilde{m}_{/}=0$
on $[v_{\mathcal{I}}-\frac{1}{2}\varepsilon,v_{\mathcal{I}})$ (and
hence $\tilde{m}_{/}=M$ on that interval), the normalised gauge condition
(\ref{eq:GaugeConditionNormalisedData}) implies that 
\begin{equation}
\frac{\bar{K}(0)}{K(v)}=\frac{1-\frac{2M}{r_{/}(v)}-\frac{1}{3}\Lambda r_{/}^{2}(v)}{1-\frac{1}{3}\Lambda r_{/}^{2}(v)}\text{ for }v\in[v_{\mathcal{I}}-\frac{1}{2}\varepsilon,v_{\mathcal{I}}).\label{eq:InfinitySchwarzschildNorm}
\end{equation}

The relation (\ref{eq:AlmostConstantKappaKappaBar}) implies that
on $[0,u^{\prime}]\times[v_{\mathcal{I}}-\frac{1}{2}\varepsilon,v_{\mathcal{I}}-\frac{1}{4}\varepsilon]$,
we have 
\begin{equation}
\frac{\partial_{v}r}{1-\frac{2M}{r}-\frac{1}{3}\Lambda r^{2}}(u,v)=K(v)\text{ and }\frac{-\partial_{u}r}{1-\frac{2M}{r}-\frac{1}{3}\Lambda r^{2}}(u,v)=\bar{K}(u)\text{ for all }(u,v)\in[0,u^{\prime}]\times[v_{\mathcal{I}}-\frac{1}{2}\varepsilon,v_{\mathcal{I}}-\frac{1}{4}\varepsilon],\label{eq:KappaKappaBarSchw}
\end{equation}
in which case (\ref{eq:OmegaForSchwarzschild}) can be reexpressed
as 
\begin{equation}
\frac{\Omega^{2}}{1-\frac{2M}{r}-\frac{1}{3}\Lambda r^{2}}(u,v)=4K(v)\bar{K}(u)\text{ for all }(u,v)\in[0,u^{\prime}]\times[v_{\mathcal{I}}-\frac{1}{2}\varepsilon,v_{\mathcal{I}}-\frac{1}{4}\varepsilon].\label{eq:OmegaSchw}
\end{equation}

Let us fix a smooth function $v':\mathbb{R}\rightarrow\mathbb{R}$
with the following properties:

\begin{enumerate}

\item $v'(v)=v$ for $v\le v_{\mathcal{I}}-\frac{1}{2}\varepsilon$,

\item $\frac{dv'}{dv}>0$ for $v\in[0,v_{\mathcal{I}}]$,

\item $v'(v_{\mathcal{I}})=v_{\mathcal{I}}$ and

\item There exists some $\delta^{\prime}\in(0,\delta)$ such that
$\frac{dv'}{dv}(v)$ satisfies 
\begin{equation}
\frac{dv'}{dv}(v)=\frac{K(v)}{\bar{K}\big(v'(v)-v_{\mathcal{I}}\big)}\text{ for all }v\in[v_{\mathcal{I}},v_{\mathcal{I}}+\delta^{\prime}]\label{eq:dv'dv}
\end{equation}
(recall that we have extended $K(v)$ smoothly for $v\in[v_{\mathcal{I}},v_{\mathcal{I}}+\delta]$).

\item The function $v'(v)$ satisfies the following $C^{2}$ bounds
on $[0,v_{\mathcal{I}}]$: 
\begin{align}
1-\eta_{0} & \le\frac{dv'}{dv}(v)\le1+\eta_{0}\text{ for }v\in[0,v_{\mathcal{I}}]\text{ and }\label{eq:UpperBoundsV'}\\
 & \max_{v\in[0,v_{\mathcal{I}}]}\Big|\frac{d^{2}v'}{(dv)^{2}}(v)\Big|\le2\Bigg|\frac{d}{dv}\frac{K(v)}{\bar{K}\big(v'(v)-v_{\mathcal{I}}\big)}\Big|_{v=v_{\mathcal{I}}}\Bigg|+\frac{\eta_{0}}{v_{\mathcal{I}}}.\nonumber 
\end{align}

\end{enumerate}

\begin{figure}[h] 
\centering 
\scriptsize
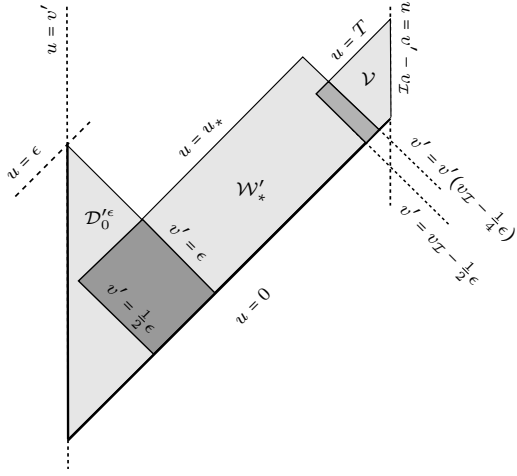 
\caption{Schematic depiction of the domains $\mathcal{D}_{0}^{\prime \epsilon}$, $\mathcal{W}_{*}^{\prime}$ and $\mathcal{V}$ in the $(u,v^{\prime})$-plane, i.\,e.~after applying the transformation $\mathcal{T}: (u,v)\rightarrow \mathcal{T}(u,v)=(u,v^{\prime}(v))$ in the $(u,v)$ plane. The assumption that $\frac{dv'}{dv}(v)=\frac{K(v)}{\bar{K}\big(v'(v)-v_{\mathcal{I}}\big)}$ for $v\in[v_{\mathcal{I}},v_{\mathcal{I}}+\delta^{\prime}]$ guarantees that, in the $(u,v')$ coordinates, the conformal infinity of the vacuum Schwarzschild--AdS region $\mathcal{V}$ lies on the vertical line $u=v'-v_{\mathcal{I}}$. \label{fig:Data_Compatibility}}
\end{figure}

\medskip{}

\noindent \emph{Remark.} It is trivial to verify that a function $v':\mathbb{R}\rightarrow\mathbb{R}$
satisfying Conditions 1-4 exists. The fact that $v'$ can be chosen
in the interval $[v_{\mathcal{I}}-\frac{1}{2}\varepsilon,v_{\mathcal{I}}]$
so that it satisfies, in addition, the bounds (\ref{eq:UpperBoundsV'})
follows from the observation that Conditions 1\textendash 4 and the
relation (\ref{eq:InfinitySchwarzschildNorm}) imply that
\[
\frac{dv'}{dv}(v)=1\text{ for }v\le v_{\mathcal{I}}-\frac{1}{2}\varepsilon\text{ and }\frac{dv'}{dv}(v_{\mathcal{I}})=\frac{K(v_{\mathcal{I}})}{\bar{K}(0)}=1
\]
and 
\[
\frac{d^{2}v'}{(dv)^{2}}(v)=0\text{ for }v\le v_{\mathcal{I}}-\frac{1}{2}\varepsilon\text{ and }\frac{d^{2}v'}{(dv)^{2}}(v_{\mathcal{I}})=\frac{d}{dv}\frac{K(v)}{\bar{K}\big(v'(v)-v_{\mathcal{I}}\big)}\Big|_{v=v_{\mathcal{I}}}.
\]

\medskip{}

Let us consider the global change of coordinates, $\mathcal{T}:\mathbb{R}^{2}\rightarrow\mathbb{R}^{2}$,
$(u,v)\rightarrow(u,v')$ (note that $\mathcal{T}$ is equal to the
identity when restricted on $\{v\le v_{\mathcal{I}}-\frac{1}{2}\varepsilon\}$)
and let us define
\[
\mathcal{D}_{0}^{\prime\varepsilon}\doteq\mathcal{T}(\mathcal{D}_{0}^{\varepsilon})=\mathcal{D}_{0}^{\varepsilon},
\]
\[
\mathcal{W}_{*}^{\prime}\doteq\mathcal{T}(\mathcal{W}_{*})=[0,u_{*}]\times[\frac{1}{2}\varepsilon,v'(v_{\mathcal{I}}-\frac{1}{4}\varepsilon)]
\]
(see Figure \ref{fig:Data_Compatibility}). Under this change of coordinates,
the intial data set $(r_{/},\Omega_{/}^{2},\bar{f}_{/};v_{\mathcal{I}})$
is mapped to a new initial data set $(r_{/}^{\prime},(\Omega_{/}^{\prime})^{2},\bar{f}_{/}^{\prime};v_{\mathcal{I}})$
determined by (\ref{eq:GeneralGaugeTransformationInitialData}), while
the solution $(r,\Omega^{2},f)$ on $\mathcal{D}_{0}^{\varepsilon}\cup\mathcal{W}_{*}$
is mapped to a new solution $(r',(\Omega^{\prime})^{2},f^{\prime})$
on $\mathcal{D}_{0}^{\prime\varepsilon}\cup\mathcal{W}_{*}^{\prime}$
under the transformation
\begin{align}
r^{\prime}(u,v'(v)) & \doteq r(u,v),\label{eq:GeneralGaugeTransformationInitialData-1}\\
(\Omega^{\prime})^{2}(u,v'(v)) & \doteq\Big(\frac{dv'}{dv}(v)\Big)^{-1}\Omega^{2}(u,v),\nonumber \\
f^{\prime}(u,v'(v);p^{u},\frac{dv'}{dv}p^{v},l) & \doteq f(u,v;p^{u},p^{v},l).\nonumber 
\end{align}
 Since $v'(v)=v$ for $v\le v_{\mathcal{I}}-\frac{1}{2}\varepsilon$,
we have 
\[
(r',(\Omega^{\prime})^{2},f^{\prime})\equiv(r,\Omega^{2},f)\text{ on }(\mathcal{D}_{0}^{\prime\varepsilon}\cup\mathcal{W}_{*}^{\prime})\cap\{v\le v_{\mathcal{I}}-\frac{1}{2}\varepsilon\}.
\]

On the domain 
\[
[0,u^{\prime}]\times[v_{\mathcal{I}}-\frac{1}{2}\varepsilon,v'(v_{\mathcal{I}}-\frac{1}{4}\varepsilon)]=\mathcal{T}\big([0,u^{\prime}]\times[v_{\mathcal{I}}-\frac{1}{2}\varepsilon,v_{\mathcal{I}}-\frac{1}{4}\varepsilon]\big),
\]
the relations (\ref{eq:KappaKappaBarSchw}) and (\ref{eq:OmegaSchw})
become: 
\begin{equation}
\frac{\partial_{v'}r'}{1-\frac{2M}{r'}-\frac{1}{3}\Lambda(r')^{2}}(u,v')=\Big(\frac{dv}{dv'}(v')\Big)K\big(v(v')\big)\text{, }\frac{-\partial_{u}r'}{1-\frac{2M}{r'}-\frac{1}{3}\Lambda(r')^{2}}(u,v')=\bar{K}(u)\label{eq:KappaKappaBarSchw-1}
\end{equation}
and 
\begin{equation}
\frac{(\Omega^{\prime})^{2}}{1-\frac{2M}{r'}-\frac{1}{3}\Lambda(r')^{2}}(u,v')=4\Big(\frac{dv}{dv'}(v')\Big)K\big(v(v')\big)\bar{K}(u).\label{eq:OmegaSchw-1}
\end{equation}
In view of the fact that $v'$ was chosen so that (\ref{eq:dv'dv})
holds, we infer that, by choosing some $T\in(0,\min\{u^{\prime},\delta^{\prime}\})$
sufficiently small, the pair $(r^{\prime},(\Omega^{\prime})^{2})$
can be extended by the formulas (\ref{eq:KappaKappaBarSchw-1}) and
(\ref{eq:OmegaSchw-1}) on the whole of the domain
\[
\mathcal{V}\doteq\{0\le u\le T\}\cap\{v_{\mathcal{I}}-\frac{1}{2}\varepsilon<v<u+v_{\mathcal{I}}\},
\]
such that $(r^{\prime},(\Omega^{\prime})^{2},0)$ is a smooth asymptotically
AdS solution of (\ref{eq:RequationFinal})\textendash (\ref{NullShellFinal})
with smooth conformal infinity $\{u=v'-v_{\mathcal{I}}\}$, inducing
on $\mathcal{V}\cap\{u=0\}$ the initial data set $(r',(\Omega^{\prime})^{2},f^{\prime})|_{[v_{\mathcal{I}}-\frac{1}{2}\varepsilon,v_{\mathcal{I}})}=(r',(\Omega^{\prime})^{2},0)$ 

\medskip{}

\noindent \emph{Remark.} The functions $r^{\prime},(\Omega^{\prime})^{2}$
are simply the metric components of the Schwarzscild\textendash AdS
metric with mass $M$ in the $(u,v')$ coordinates, with (\ref{eq:dv'dv})
ensuring that $\{r=\infty\}$ coincides with the straight line $\{u=v'-v_{\mathcal{I}}\}$. 

\end{enumerate}

By gluing the extension $(r^{\prime},(\Omega^{\prime})^{2},0)$ on
$\mathcal{V}$ with $(r',(\Omega^{\prime})^{2},f^{\prime})|_{\mathcal{D}_{0}^{\prime\varepsilon}\cup\mathcal{W}_{*}^{\prime}}$
along $\big(\mathcal{D}_{0}^{\prime\varepsilon}\cup\mathcal{W}_{*}^{\prime}\big)\cap\mathcal{V}=[0,T]\times[v_{\mathcal{I}}-\frac{1}{2}\varepsilon,v'(v_{\mathcal{I}}-\frac{1}{4}\varepsilon)]$
(see Figure \ref{fig:Data_Compatibility}), we therefore obtain a
smooth solution $(r',(\Omega^{\prime})^{2},f^{\prime})$ of the system
(\ref{eq:RequationFinal})\textendash (\ref{NullShellFinal}) on 
\[
\mathcal{D}_{0}^{\prime\varepsilon}\cup\mathcal{W}_{*}^{\prime}\cup\mathcal{V}\supset\mathcal{U}_{T;v_{\mathcal{I}}}
\]
 such that $(r',(\Omega^{\prime})^{2},f^{\prime})$ has smooth axis
$\{u=v'\}$ and smooth conformal infinity $\{u=v'-v_{\mathcal{I}}\}$
and induces on $\{u=0\}$ the initial data set $(r_{/}^{\prime},(\Omega_{/}^{\prime})^{2},\bar{f}_{/}^{\prime};v_{\mathcal{I}})$.
Therefore, we conclude that $(r_{/}^{\prime},(\Omega_{/}^{\prime})^{2},\bar{f}_{/}^{\prime};v_{\mathcal{I}})$
is smoothly compatible, in accordance with Definition \ref{def:CompatibilityCondition}. 

Furthermore, using the bounds (\ref{eq:UpperBoundsV'}) for $\frac{d^{2}v'}{dv^{2}}$
and the relations (\ref{eq:InfinitySchwarzschildNorm}), 
\begin{equation}
\frac{d\log(\bar{K})}{du}(0)=\partial_{u}\log\Big(\frac{-\partial_{u}r}{1-\frac{2M}{r}-\frac{1}{3}\Lambda r^{2}}\Big)(0,v_{\mathcal{I}})=-4\pi\int_{0}^{v_{\mathcal{I}}}\partial_{u}\Big(\frac{rT_{vv}}{\partial_{v}r}\Big)(0,v)\,dv\label{eq:DerivativeKappaBar}
\end{equation}
(which is obtained by integrating (\ref{eq:DerivativeInVDirectionKappaBar})
in $v$ for $u=0$), 
\[
\partial_{u}(rT_{vv})=-\Omega^{2}\partial_{v}(\Omega^{-2}rT_{uv})-\partial_{u}rT_{vv}-3\partial_{v}rT_{uv}
\]
(which is obtained from the conservation of energy relation (\ref{eq:ConservationEnergyMomentum})),
we readily calculate:
\begin{align}
\max\Big|\frac{d^{2}v'}{(dv)^{2}}(v)\Big| & \le2\Bigg|\frac{d}{dv}\frac{K(v)}{\bar{K}\big(v'(v)-v_{\mathcal{I}}\big)}\Big|_{v=v_{\mathcal{I}}}\Bigg|+\frac{\eta_{0}}{v_{\mathcal{I}}}=\label{eq:FirstBoundD^2V}\\
= & 2\Bigg|\frac{\frac{dK}{dv}(v_{\mathcal{I}})}{\bar{K}(0)}-\frac{K(v_{\mathcal{I}})}{\bar{K}(0)}\frac{\frac{d\bar{K}}{du}(0)}{\bar{K}(0)}\frac{dv'}{dv}(v_{\mathcal{I}})\Bigg|+\frac{\eta_{0}}{v_{\mathcal{I}}}=\nonumber \\
= & 2\Bigg|\frac{d}{du}\log(\bar{K})(0)\Bigg|+\frac{\eta_{0}}{v_{\mathcal{I}}}=\nonumber \\
= & 8\pi\Big|\int_{0}^{v_{\mathcal{I}}}\partial_{u}\Big(\frac{rT_{vv}}{\partial_{v}r}\Big)(0,v)\,dv\Big|+\frac{\eta_{0}}{v_{\mathcal{I}}}=\nonumber \\
= & 8\pi\Big|\int_{0}^{v_{\mathcal{I}}}\Big(-\frac{1}{\Omega^{-2}\partial_{v}r}\partial_{v}(\Omega^{-2}rT_{uv})-\frac{\partial_{u}r}{\partial_{v}r}T_{vv}-3T_{uv}-\frac{\partial_{u}\partial_{v}r}{\partial_{v}r}\cdot\frac{rT_{vv}}{\partial_{v}r}\Big)(0,v)\,dv\Big|+\frac{\eta_{0}}{v_{\mathcal{I}}}.\nonumber 
\end{align}
Integrating by parts in $v$ in the first term in the right hand side
of (\ref{eq:FirstBoundD^2V}) and using the fact that $T_{uv}(0,v)\equiv0$
for $v\in[0,\varepsilon]\cup[\varepsilon,v_{\mathcal{I}})$, the constraint
equation (\ref{eq:ConstraintVFinal}), the relation (\ref{eq:EquationRForProof})
for $\partial_{u}\partial_{v}r$, the gauge condition (\ref{eq:GaugeConditionNormalisedData})
and the relation (\ref{eq:DefinitionHawkingMass}) for $\Omega^{2}$,
we obtain from (\ref{eq:FirstBoundD^2V}) that: 
\begin{align}
\max\Big|\frac{d^{2}v'}{(dv)^{2}}(v)\Big| & \le\Big|\int_{0}^{v_{\mathcal{I}}}\Big(-\frac{\partial_{v}(\Omega^{-2}\partial_{v}r)}{(\Omega^{-2}\partial_{v}r)^{2}}\Omega^{-2}rT_{uv}-\frac{\partial_{u}r}{\partial_{v}r}T_{vv}-3T_{uv}-\frac{\partial_{u}\partial_{v}r}{\partial_{v}r}\cdot\frac{rT_{vv}}{\partial_{v}r}\Big)(0,v)\,dv\Big|+\frac{\eta_{0}}{v_{\mathcal{I}}}=\label{eq:SecondBoundD^2v}\\
 & =\Big|\int_{0}^{v_{\mathcal{I}}}\Big(4\pi\frac{r^{2}}{(\partial_{v}r)^{2}}T_{vv}T_{uv}+\frac{1-\frac{2\tilde{m}}{r}-\frac{1}{3}\Lambda r^{2}}{1-\frac{1}{3}\Lambda r^{2}}T_{vv}-3T_{uv}+\nonumber \\
 & \hphantom{=\Big|\int_{0}^{v_{\mathcal{I}}}\Big(4\pi}+\Big(\frac{2\tilde{m}-\frac{2}{3}\Lambda r^{3}}{r^{2}}\frac{1}{1-\frac{1}{3}\Lambda r^{2}}-4\pi r\frac{T_{uv}}{(\partial_{v}r)^{2}}\Big)rT_{vv}\Big)(0,v)\,dv\Big|+\frac{\eta_{0}}{v_{\mathcal{I}}}=\nonumber \\
 & =\Big|\int_{0}^{v_{\mathcal{I}}}\Big(\frac{1-\Lambda r_{/}^{2}}{1-\frac{1}{3}\Lambda r_{/}^{2}}(T_{/})_{vv}-3(T_{/})_{uv}\Big)(v)\,dv\Big|+\frac{\eta_{0}}{v_{\mathcal{I}}}.\nonumber 
\end{align}
The bounds (\ref{eq:FirstDerivativeTransformation}) and (\ref{eq:SecondDerivativeTransformation})
follow immediately from (\ref{eq:UpperBoundsV'}) (for $\frac{dv'}{dv}$)
and (\ref{eq:SecondBoundD^2v}).
\end{proof}

\section{\label{sec:Well-posedness-of-the}Well-posedness of the smooth characteristic
initial-boundary value problem and properties of the maximal development}

In this section, we will introduce the notion of a \emph{development}
of a smoothly compatible, asymptotically AdS initial data set for
(\ref{eq:RequationFinal})\textendash (\ref{NullShellFinal}) with
\emph{reflecting} boundary conditions on $\mathcal{I}$. We will then
present some fundamental well-posedness results related to the characteristic
initial-boundary value problem for (\ref{eq:RequationFinal})\textendash (\ref{NullShellFinal})
with reflecting boundary conditions on $\mathcal{I}$, and show the
existence and uniqueness of a \emph{maximal }smooth development for
any smoothly compatible, asymptotically AdS initial data set for (\ref{eq:RequationFinal})\textendash (\ref{NullShellFinal}). 

\subsection{\label{subsec:Developments} Developments of characteristic initial
data sets}

We will fix a class of domains in the $(u,v)$-plane which will naturally
arise as domains of definition for solutions $(r,\Omega^{2},f)$ to
the characteristic initial-boundary value problem for (\ref{eq:RequationFinal})\textendash (\ref{NullShellFinal});
this class of domains has been also considered in \cite{MoschidisMaximalDevelopment,MoschidisNullDust},
in the context of the study of the Einstein\textendash null dust system.

\begin{figure}[h] 
\centering 
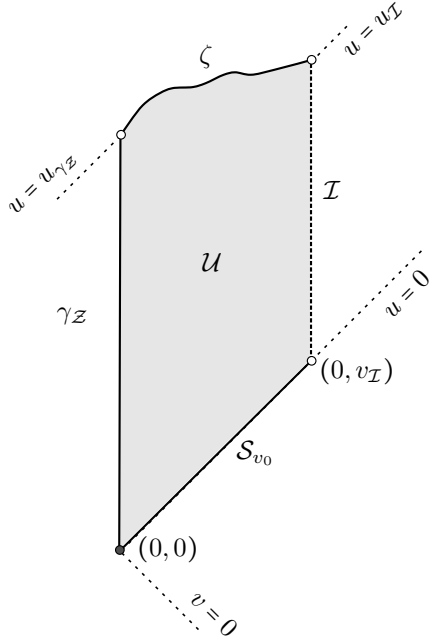 
\caption{Depicted above is a typical domain $\mathcal{U}\in\mathscr{U}_{v_{\mathcal{I}}}$. In the case when the boundary set $\zeta$ is empty, it is necessary that both $\gamma_{\mathcal{Z}}$ and $\mathcal{I}$ are unbounded (i.\,e.~$u_{\gamma_{\mathcal{Z}}}=u_{\mathcal{I}}=+\infty$).}
\end{figure}

\begin{defn}\label{def:DevelopmentSets}

For any $v_{\mathcal{I}}>0$, let $\mathscr{U}_{v_{\mathcal{I}}}$
be the set of all connected open domains $\mathcal{U}$ of the $(u,v)$-plane
with piecewise Lipschitz boundary $\partial\mathcal{U}$, having the
property that $\partial\mathcal{U}$ can be expressed as
\begin{equation}
\partial\mathcal{U}=\gamma_{\mathcal{Z}}\cup\mathcal{I}\cup\mathcal{S}_{v_{\mathcal{I}}}\cup clos(\zeta),\label{eq:BoundaryOfU}
\end{equation}
where, for some $u_{\gamma_{\mathcal{Z}}},u_{\mathcal{I}}\in(0,+\infty]$,
\begin{equation}
\gamma_{\mathcal{Z}}=\{u=v\}\cap\{0\le u<u_{\gamma_{\mathcal{Z}}}\},\label{eq:AxisForm}
\end{equation}
\begin{equation}
\mathcal{I}=\{u=v-v_{\mathcal{I}}\}\cap\{0\le u<u_{\mathcal{I}}\},\label{eq:InfinityForm}
\end{equation}
\begin{equation}
\mathcal{S}_{v_{\mathcal{I}}}=\{0\}\times[0,v_{\mathcal{I}}]
\end{equation}
 and the Lipschitz curve $\zeta\subset\mathbb{R}^{2}$ is achronal
with respect to the reference Lorentzian metric 
\begin{equation}
g_{ref}\doteq-dudv\label{eq:ComparisonUVMetric}
\end{equation}
on the $(u,v)$-plane (the closure $clos(\zeta)$ of $\zeta$ in (\ref{eq:BoundaryOfU})
is considered with respect to the standard topology of $\mathbb{R}^{2}$)
. In particular, $\zeta$ is allowed to be empty. 

\end{defn}
\begin{rem*}
It follows readily from Definition \ref{def:DevelopmentSets} that
any $\mathcal{U}\in\mathscr{U}_{v_{\mathcal{I}}}$ is necessarily
contained in the future domain of dependence of $\mathcal{S}_{v_{\mathcal{I}}}\cup\text{\textgreek{g}}_{\mathcal{Z}}\cup\mathcal{I}$
(with respect to the comparison metric (\ref{eq:ComparisonUVMetric})).
In the case when $\zeta=\emptyset$ in (\ref{eq:BoundaryOfU}), it
is necessary that both $\text{\textgreek{g}}_{\mathcal{Z}}$ and $\mathcal{I}$
extend all the way to $u+v=+\infty$. 
\end{rem*}
A development of an asymptotically AdS initial data set for (\ref{eq:RequationFinal})\textendash (\ref{NullShellFinal})
with reflecting boundary conditions on $\mathcal{I}$ can be naturally
defined as follows:
\begin{defn}
\label{def:Development} For any $v_{\mathcal{I}}>0$, let $(r_{/},\Omega_{/}^{2},\bar{f}_{/};v_{\mathcal{I}})$
be a smoothly compatible, asymptotically AdS initial data set for
the system (\ref{eq:RequationFinal})\textendash (\ref{NullShellFinal}),
according to Definition \ref{def:CompatibilityCondition}. A \underline{future development}
of $(r_{/},\Omega_{/}^{2},\bar{f}_{/};v_{\mathcal{I}})$ for (\ref{eq:RequationFinal})\textendash (\ref{NullShellFinal})
with reflecting boundary conditions on $\mathcal{I}$ consists of
an open set $\mathcal{U}\in\mathscr{U}_{v_{\mathcal{I}}}$ (see Definition
\ref{def:DevelopmentSets}) and a smooth solution $(r,\Omega^{2},f)$
of (\ref{eq:RequationFinal})\textendash (\ref{NullShellFinal}) on
$\mathcal{U}$ satisfying the following conditions:

\begin{enumerate}

\item Using the notations of Definition \ref{def:DevelopmentSets},
$(\mathcal{U};r,\Omega^{2},f)$ has smooth axis $\gamma_{\mathcal{Z}}$
and smooth conformal infinity $\mathcal{I}$, in accordance with Definitions
\ref{def:SmoothnessAxis}\textendash \ref{def:SmoothnessConformalInfinity}.

\item The solution $(r,\Omega^{2},f)$ coincides with $(r_{/},\Omega_{/}^{2},\bar{f}_{/};v_{\mathcal{I}})$
at $u=0$, i.\,e.
\begin{equation}
(r,\Omega^{2})(0,v)=(r_{/},\Omega_{/}^{2})(v)\label{eq:InitialROmegaForExistence}
\end{equation}
 and 
\begin{equation}
f(0,v;p^{u},p^{v},l)=\bar{f}_{/}(v;p^{u},l)\cdot\delta\Big(\Omega_{/}^{2}(v)p^{u}p^{v}-\frac{l^{2}}{r_{/}(v)}\Big).\label{eq:InitialFForExistence}
\end{equation}

\item The Vlasov field $f$ satisfies the reflecting boundary condition
(\ref{eq:ReflectingCondition}) on $\mathcal{I}$.

\end{enumerate}

If $\mathscr{D}=(\mathcal{U};r,\Omega^{2},f)$ and $\mathscr{D}^{\prime}=(\mathcal{U}^{\prime};r^{\prime},(\Omega^{\prime})^{2},f^{\prime})$
are two future developments of the same initial data $(r_{/},\Omega_{/}^{2},\bar{f}_{/};v_{\mathcal{I}})$,
we will say that $\mathscr{D}^{\prime}$ is an extension of $\mathscr{D}$,
writing $\mathscr{D}\subseteq\mathscr{D}^{\prime}$, if $\mathcal{U}\subseteq\mathcal{U}^{\prime}$
and the restriction of $(r^{\prime},(\Omega^{\prime})^{2},f^{\prime})$
on $\mathcal{U}$ coincides with $(r,\Omega^{2},f)$.
\end{defn}
In Section \ref{sec:Cauchy_Stability_Low_Regularity}, we will need
to perform gauge transformations on developments $(\mathcal{U};r,\Omega^{2},f)$
of smoothly compatible initial data sets $(r_{/},\Omega_{/}^{2},\bar{f}_{/};v_{\mathcal{I}})$
that normalise the initial data according to the gauge condition (\ref{def:GaugeConditionNormalise}).
The following lemma will be useful for this procedure:
\begin{lem}
\label{lem:TransformationForNormalisationDevelopment} Let $(r_{/},\Omega_{/}^{2},\bar{f}_{/};v_{\mathcal{I}})$
be a smoothly compatible, asymptotically AdS initial data set for
the system (\ref{eq:RequationFinal})\textendash (\ref{NullShellFinal})
of bounded support in phase space and let $\mathscr{D}=(\mathcal{U};r,\Omega^{2},f)$
be a development of $(r_{/},\Omega_{/}^{2},\bar{f}_{/};v_{\mathcal{I}})$,
as in Definition \ref{def:Development}. Let also $(r_{/},\Omega_{/}^{2},\bar{f}_{/};v_{\mathcal{I}})\rightarrow(r_{/}^{\prime},(\Omega_{/}^{\prime})^{2},\bar{f}_{/}^{\prime};v_{\mathcal{I}})$
be the gauge normalising transformation provided by Lemma \ref{lem:SmoothToNorm},
with associated coordinate transformation $\mathcal{T}_{/}:[0,v_{\mathcal{I}}]\rightarrow[0,v_{\mathcal{I}}]$,
$v\rightarrow\bar{V}(v)$.

There exists a unique spacetime gauge transformation $\mathcal{T}:\mathscr{D}\rightarrow\mathbb{R}^{2}$,
$(u,v)\rightarrow(u',v')=(U(u),V(v))$, $(\mathcal{U};r,\Omega^{2},f)\rightarrow(\mathcal{T}(\mathcal{U});r^{\prime},(\Omega^{\prime})^{2},f^{\prime})$
(see (\ref{eq:GeneralGaugeTransformationSpacetime})), such that

\begin{itemize}

\item The lines $\{u=v\}$ and $\{u=v-v_{\mathcal{I}}\}$ remain
invariant under $\mathcal{T}$, i.\,e.
\begin{align}
U(v) & =V(v)\text{ for all }v\in[0,\sup_{\gamma_{\mathcal{Z}}}v),\label{eq:FixingTheLines}\\
U(v-v_{\mathcal{I}}) & =V(v)-v_{\mathcal{I}}\text{ for all }v\in[0,\sup_{\mathcal{I}}v).\nonumber 
\end{align}

\item $\mathcal{T}$ is an extension of the initial data transformation
$\mathcal{T}_{/}$, i.\,e.
\begin{equation}
(U(0),V(v))=(0,\bar{V}(v))\text{ for all \ensuremath{v\in[0,v_{\mathcal{I}}).}}\label{eq:FixingTheInitialData}
\end{equation}

\end{itemize}

Furthermore, $\mathcal{T}$ is piecewise smooth on $\mathscr{D}$
and smooth on $\mathscr{D}\backslash\cup_{k=1}^{\infty}\big(\{u=kv_{\mathcal{I}}\}\cup\{v=kv_{\mathcal{I}}\}\big)$
and satisfies the Lipschitz estimate
\begin{equation}
\Big|\log\Big(\frac{dU}{du}\Big)\Big|+\Big|\log\Big(\frac{dV}{dv}\Big)\Big|\le C_{v_{\mathcal{I}}}\Big(\sup_{v\in[0,v_{\mathcal{I}})}\Big|\log\Big(\frac{\partial_{v}r_{/}}{1-\frac{1}{3}\Lambda r_{/}^{2}}\Big)-\log\big(\partial_{v}r_{/}(0)\big)\Big|+4\pi\int_{0}^{v_{\mathcal{I}}}\frac{r_{/}(T_{/})_{vv}}{(\partial_{v}r_{/})^{2}}(\bar{v})\,(\partial_{v}r_{/})d\bar{v}\Big)\label{eq:C1EstimateSpacetimeTransformation}
\end{equation}
 for some constant $C_{v_{\mathcal{I}}}>0$ depending only on $v_{\mathcal{I}}$.
\end{lem}
\begin{proof}
The conditions (\ref{eq:FixingTheLines}) and (\ref{eq:FixingTheInitialData})
provide an explicit formula for $\mathcal{T}(u,v)=\big(U(u),V(v)\big)$:
\begin{align}
U(u) & =\bar{V}\big(u-\lfloor\frac{u}{v_{\mathcal{I}}}\rfloor\cdot v_{\mathcal{I}}\big)+\lfloor\frac{u}{v_{\mathcal{I}}}\rfloor\cdot v_{\mathcal{I}}\label{eq:FormulaT}\\
V(v) & =\bar{V}\big(v-\lfloor\frac{v}{v_{\mathcal{I}}}\rfloor\cdot v_{\mathcal{I}}\big)+\lfloor\frac{v}{v_{\mathcal{I}}}\rfloor\cdot v_{\mathcal{I}},\nonumber 
\end{align}
where $\lfloor x\rfloor$ denotes the integral part of $x$. The formula
(\ref{eq:FormulaT}) and the fact that $\bar{V}\in C^{\infty}([0,+\infty))$
readily imply that $\mathcal{T}$ is piecewise smooth on $\mathscr{D}$
and smooth on $\mathscr{D}\backslash\cup_{k=1}^{\infty}\big(\{u=kv_{\mathcal{I}}\}\cup\{v=kv_{\mathcal{I}}\}\big)$
(the continuity of $\mathcal{T}$ along the lines $u=nv_{\mathcal{I}}$
and $v=nv_{\mathcal{I}}$, $n\in\mathbb{N}$, follows from the fact
that $\bar{V}(v_{\mathcal{I}})=v_{\mathcal{I}}$). Furthermore, in
view of the relation (\ref{eq:ODEForV}) for $\bar{V}$, we readily
infer that 
\begin{equation}
\Big|\log\Big(\frac{dU}{du}\Big)\Big|+\Big|\log\Big(\frac{dV}{dv}\Big)\Big|\le2\Big|\log\Big(\frac{d\bar{V}}{dv}(0)\Big)\Big|+2\Big(\sup_{v\in[0,v_{\mathcal{I}})}\Big|\log\Big(\frac{\partial_{v}r_{/}}{1-\frac{1}{3}\Lambda r_{/}^{2}}\Big)-\log\big(\frac{\Omega_{/}^{2}}{4\partial_{v}r_{/}}(0)\big)\Big|\Big)+8\pi\int_{0}^{v_{\mathcal{I}}}\frac{r_{/}(T_{/})_{vv}}{(\partial_{v}r_{/})^{2}}(\bar{v})\,(\partial_{v}r_{/})d\bar{v}.\label{eq:C1EstimateSpacetimeTransformation-1}
\end{equation}
Using the expression (\ref{eq:B}) for $b=\frac{d\bar{V}}{dv}(0)$
and the relation (\ref{eq:SmoothlyCompatibleAxisCondition}) for $\frac{\Omega_{/}^{2}}{4\partial_{v}r_{/}}$,
we readily obtain the bound (\ref{eq:C1EstimateSpacetimeTransformation})
from (\ref{eq:C1EstimateSpacetimeTransformation-1}).
\end{proof}

\subsection{\label{subsec:Well_posedness} Well-posedness results for characteristic
initial data sets and the maximal future development}

The main result of this section is the proof of the well-posedness
of the initial-boundary value problem for (\ref{eq:RequationFinal})\textendash (\ref{NullShellFinal})
with reflecting boundary conditions on $\mathcal{I}$ in the class
of smoothly compatible, asymptotically AdS initial data sets with
bounded support in phase space, introduced by Definition \ref{def:CompatibilityCondition}.
As a byproduct of the proof of the main result, we will also establish
the well-posedness of the characteristic initial value problem restricted
in a neighborhood of $r=0$, as well as of the double-characteristic
initial value problem. These well-posedness results will then allow
us to define the \emph{maximal }future development of a smooth, asymptotically
AdS initial data sets with bounded support in phase space and show
that it is unique (see Corollary \ref{cor:MaximalDevelopment}).

Let us define for any $U,\text{ }v_{\mathcal{I}}>0$ the domain
\begin{equation}
\mathcal{U}_{U;v_{\mathcal{I}}}\doteq\big\{0\le u<U\big\}\cap\big\{ u<v<u+v_{\mathcal{I}}\big\}\label{eq:GeneralDomain}
\end{equation}
and the boundary curves 
\begin{equation}
\gamma_{u_{1}}\doteq\big\{ u=v\big\}\cap\big\{0\le u<u_{1}\big\}\label{eq:AxisGeneral}
\end{equation}
and
\begin{equation}
\mathcal{I}_{u_{1}}\doteq\big\{ v=u+v_{\mathcal{I}}\big\}\cap\big\{0\le u<u_{1}\big\}.\label{eq:ConformalInfinityGeneral}
\end{equation}
The main result of this section is the following:
\begin{thm}
\label{thm:LocalExistenceUniqueness} Let $(r_{/},\Omega_{/}^{2},\bar{f}_{/};v_{\mathcal{I}})$
be any smoothly compatible asymptotically AdS initial data set for
(\ref{eq:RequationFinal})\textendash (\ref{NullShellFinal}), according
to Definition \ref{def:CompatibilityCondition}, satisfying (\ref{eq:BoundedSupportDefinition}).
Then, for some $u_{*}>0$ sufficiently small in terms of $(r_{/},\Omega_{/}^{2},\bar{f}_{/};v_{\mathcal{I}})$
, there exists a unique smooth solution $(r,\Omega^{2}f)$ of (\ref{eq:RequationFinal})\textendash (\ref{NullShellFinal})
on the domain $\mathcal{U}_{u_{*},v_{\mathcal{I}}}$ (defined by (\ref{eq:GeneralDomain}))
such that $(\mathcal{U}_{u_{*},v_{\mathcal{I}}};r,\Omega^{2}f)$ is
a future development of $(r_{/},\Omega_{/}^{2},\bar{f}_{/};v_{\mathcal{I}})$
with reflecting boundary conditions on $\mathcal{I}$ (see Definition
\ref{def:Development}). 
\end{thm}
For the proof of Theorem \ref{thm:LocalExistenceUniqueness}, see
Section \ref{subsec:Local-existence-and}.

As a corollary of the proof of Theorem \ref{thm:LocalExistenceUniqueness},
we can also obtain the following well-posedness result for characteristic
initial data sets with a smooth axis, restricted to the region where
$r\le R<+\infty$: 
\begin{prop}
\label{prop:LocalExistenceSpecialCase} For any $v_{0}>0$, let $r_{/},\Omega_{/}:[0,v_{0}]\rightarrow(0,+\infty)$
and $\bar{f}_{/}:(0,v_{0}]\times[0,+\infty)^{2}\rightarrow[0,+\infty)$
be smooth functions satisfying Conditions 1 and 2 of Definition \ref{def:AsymptoticallyAdSData},
as well as Conditions 1 and 2 of Definition \ref{def:CompatibilityCondition}.
Assume, moreover, that the support of $\bar{f}_{/}$ satisfies the
bound (\ref{eq:BoundedSupportDefinition}) for some $C>0$. Then,
for some $u_{*}>0$ sufficiently small in terms of $(r_{/},\Omega_{/}^{2},\bar{f}_{/})$,
there exists a unique smooth solution $(r,\Omega^{2}f)$ of (\ref{eq:RequationFinal})\textendash (\ref{NullShellFinal})
on the domain 
\[
\mathcal{D}_{u_{*}}\doteq\big((0,u_{*})\times(0,v_{0})\big)\cap\{u<v\}
\]
 such that $(\mathcal{D}_{u_{*}};r,\Omega^{2}f)$ has smooth axis
$\gamma_{\mathcal{Z}}=\{u=v\}\cap\{0<u<u_{*}\}$ (see Definition \ref{def:SmoothnessAxis})
and $(r,\Omega^{2}f)$ satisfy the initial conditions (\ref{eq:InitialROmegaForExistence})\textendash (\ref{eq:InitialFForExistence})
at $u=0$.
\end{prop}
The proof of Proposition \ref{prop:LocalExistenceSpecialCase} follows
by exactly the same arguments as the proof of Theorem \ref{thm:LocalExistenceUniqueness}
(see Steps 1 and 2 in Section \ref{subsec:Local-existence-and});
it will thus be omitted.

The following well-posedness result for the regular, double characteristic
initial value problem for (\ref{eq:RequationFinal})\textendash (\ref{NullShellFinal})
is also obtained as a corollary of the proof of Theorem \ref{thm:LocalExistenceUniqueness}:
\begin{prop}
\label{prop:MihalisWellPosedness}For any $u_{1}<u_{2}$ and $v_{1}<v_{2}$,
let $r_{\backslash},\Omega_{\backslash}:[u_{1},u_{2}]\rightarrow(0,+\infty)$,
$\bar{f}_{\backslash}:[u_{1},u_{2}]\times[0,+\infty)^{2}\rightarrow[0,+\infty)$,
$r_{/},\Omega_{/}:[v_{1},v_{2}]\rightarrow(0,+\infty)$ and $\bar{f}_{/}:[v_{1},v_{2}]\times[0,+\infty)^{2}\rightarrow[0,+\infty)$
be smooth functions, such that $(r_{\backslash},\Omega_{\backslash}^{2},\bar{f}_{\backslash})$
and $(r_{/},\Omega_{/}^{2},\bar{f}_{/})$ satisfy the constraint equations
(\ref{eq:ConstraintUFinal}) and (\ref{eq:ConstraintVFinal}), respectively,
as well as the compatibility conditions 
\begin{align}
r_{\backslash}(u_{1}) & =r_{/}(v_{1}),\\
\Omega_{\backslash}^{2}(u_{1}) & =\Omega_{/}^{2}(v_{1}),\nonumber \\
\bar{f}_{\backslash}(u_{1},p^{v},l) & =\bar{f}_{/}(v_{1},\frac{l^{2}}{r_{/}^{2}\Omega_{/}^{2}\big|_{v=v_{1}}p^{u}},l).\nonumber 
\end{align}
Assume, moreover, that the supports of $\bar{f}_{\backslash}$, $\bar{f}_{/}$
satisfy the bound (\ref{eq:BoundedSupportDefinition}). Then, for
some $\delta>0$ sufficiently small in terms of$(r_{\backslash},\Omega_{\backslash}^{2},\bar{f}_{\backslash})$
and $(r_{/},\Omega_{/}^{2},\bar{f}_{/})$, there exists a unique smooth
solution $(r,\Omega^{2}f)$ of (\ref{eq:RequationFinal})\textendash (\ref{NullShellFinal})
on the domain $[u_{1},u_{1}+\delta]\times[v_{1},v_{2}]$, such that
\begin{align}
(r,\Omega^{2})(u,v_{1}) & =(r_{\backslash},\Omega_{\backslash}^{2})(u),\label{eq:InitialROmegaForExistence-1}\\
(r,\Omega^{2})(u_{1},v) & =(r_{/},\Omega_{/}^{2})(v)\nonumber 
\end{align}
 and 
\begin{align}
f(u,v_{1};p^{u},p^{v},l) & =\bar{f}_{\backslash}(u;p^{v},l)\cdot\delta\Big(\Omega_{\backslash}^{2}(u)p^{u}p^{v}-\frac{l^{2}}{r_{\backslash}(u)}\Big),\label{eq:InitialFForExistence-1}\\
f(u_{1},v;p^{u},p^{v},l) & =\bar{f}_{/}(v;p^{u},l)\cdot\delta\Big(\Omega_{/}^{2}(v)p^{u}p^{v}-\frac{l^{2}}{r_{/}(v)}\Big).\nonumber 
\end{align}
The analogous statement for the domain $[u_{1},u_{2}]\times[v_{1},v_{1}+\delta]$
also holds.
\end{prop}
The proof of Proposition \ref{prop:MihalisWellPosedness} follows
by the same arguments as the proof of Theorem \ref{thm:LocalExistenceUniqueness}
(see Step 2 in Section \ref{subsec:Local-existence-and}); it will
thus be omitted. For a similar result in the case of the Einstein\textendash massive
Vlasov system on domains of the form $[u_{1},u_{1}+\delta]\times[v_{1},v_{1}+\delta]$,
see \cite{DafermosRendall2005}.

As a straightforward corollary of the local uniqueness statements
of Propositions \ref{prop:LocalExistenceSpecialCase}\textendash \ref{prop:MihalisWellPosedness},
we infer the following global uniqueness result
\begin{cor}
\label{cor:UniquenessDevelopments} If $\mathscr{D}=(\mathcal{U};r,\Omega^{2},f)$
and $\mathscr{D}^{\prime}=(\mathcal{U}^{\prime};r^{\prime},(\Omega^{\prime})^{2},f^{\prime})$
are two future developments with reflecting boundary conditions on
$\mathcal{I}$ of the same smoothly compatible, asymptotically AdS
initial data set $(r_{/},\Omega_{/}^{2},\bar{f}_{/};v_{\mathcal{I}})$
with bounded support in phase space, then 
\begin{equation}
(r,\Omega^{2},f)|_{\mathcal{U}\cap\mathcal{U}^{\prime}}=(r^{\prime},(\Omega^{\prime})^{2},f^{\prime})|_{\mathcal{U}\cap\mathcal{U}^{\prime}}.
\end{equation}
\end{cor}
\begin{rem*}
Our choice of gauge (fixing the axis $\gamma_{\mathcal{Z}}$ and conformal
infinity $\mathcal{I}$ of any development to be straight, vertical
lines) does not provide any freedom in performing gauge transformations
on the development that fix the gauge of the initial data at $u=0$.
This is the reason why Corollary \ref{cor:UniquenessDevelopments}
can be stated as a uniqueness statement without appealing to equivalence
classes of developments under gauge transformations.
\end{rem*}
Using standard arguments, the well-posedness results of Theorem \ref{thm:LocalExistenceUniqueness}
and Propositions \ref{prop:LocalExistenceSpecialCase}\textendash \ref{prop:MihalisWellPosedness}
also allow us to assign to each smoothly compatible, asymptotically
AdS initial data set for (\ref{eq:RequationFinal})\textendash (\ref{NullShellFinal})
a unique \emph{maximal }future development : 
\begin{cor}
\label{cor:MaximalDevelopment} Let $(r_{/},\Omega_{/}^{2},\bar{f}_{/};v_{\mathcal{I}})$
be any smoothly compatible, asymptotically AdS initial data set for
(\ref{eq:RequationFinal})\textendash (\ref{NullShellFinal}) with
bounded support in phase space. Then there exists a unique future
development $(\mathcal{U}_{max};r,\Omega^{2},f)$ of $(r_{/},\Omega_{/}^{2},\bar{f}_{/};v_{\mathcal{I}})$
with reflecting boundary conditions on $\mathcal{I}$ having the following
property: If $(\mathcal{U}_{*};r_{*},\Omega_{*}^{2},f_{*})$ is any
other future development of $(r_{/},\Omega_{/}^{2},\bar{f}_{/};v_{\mathcal{I}})$
with reflecting boundary conditions on $\mathcal{I}$, then 
\begin{equation}
\mathcal{U}_{*}\subseteq\mathcal{U}_{max}
\end{equation}
 and 
\begin{equation}
(r,\Omega^{2},f)|_{\mathcal{U}_{*}}=(r_{*},\Omega_{*}^{2},f_{*}).
\end{equation}
We will call $(\mathcal{U}_{max};r,\Omega^{2},f)$ the maximal future
development of $(r_{/},\Omega_{/}^{2},\bar{f}_{/};v_{\mathcal{I}})$
with reflecting boundary conditions on $\mathcal{I}$.
\end{cor}
The maximal future development of an initial data set $(r_{/},\Omega_{/}^{2},\bar{f}_{/};v_{\mathcal{I}})$
is the main object of study related to questions on the global dynamics
of (\ref{eq:RequationFinal})\textendash (\ref{NullShellFinal}),
such that the ones addressed in our companion paper \cite{MoschidisVlasov}. 

\subsection{\label{subsec:Local-existence-and} Local existence and uniqueness:
Proof of Theorem \ref{thm:LocalExistenceUniqueness}}

In this section, we will establish Theorem \ref{thm:LocalExistenceUniqueness}.
To this end, let $(r_{/},\Omega_{/}^{2},\bar{f}_{/};v_{\mathcal{I}})$
be as in the statement of Theorem \ref{thm:LocalExistenceUniqueness},
and let $C_{0}>0$ be a constant for which the bounded support condition
(\ref{eq:BoundedSupportDefinition}) is satisfied by $(r_{/},\Omega_{/}^{2},\bar{f}_{/};v_{\mathcal{I}})$
with $C_{0}$ in place of $C$. The construction of a (unique) smooth
development $(\mathcal{U}_{u_{*};v_{\mathcal{I}}};r,\Omega^{2},f)$
of $(r_{/},\Omega_{/}^{2},\bar{f}_{/};v_{\mathcal{I}})$ will be separated
into three steps, covering, successively, the region near the axis,
the intermediate region and the region near conformal infinity. 

\paragraph*{Step 1: The region $\mathcal{D}_{0}^{u_{0}}$}

Let us fix a constant $R_{0}$ satisfying 
\begin{equation}
\sqrt{-\Lambda}R_{0}\ll1,
\end{equation}
and let $V_{*}\in(0,v_{\mathcal{I}})$ be such that 
\begin{equation}
r_{/}(V_{*})=R_{0}.
\end{equation}
We will assume that $R_{0}$ is sufficiently small, so that 
\begin{equation}
\sup_{v\in(0,V_{*})}\partial_{u}r_{/}<0\label{eq:NegativeDurNearAxis}
\end{equation}
(this is possible, in view of (\ref{eq:EqualDuRDvRInitially}) and
the fact that $\tilde{m}_{/}=O(r_{/}^{3})$). We will show that, for
some $u_{0}>0$ sufficiently small in terms of $R_{0}$ and $V_{*}$,
there exists a unique smooth solution $(r,\Omega^{2};f)$ of (\ref{eq:RequationFinal})\textendash (\ref{NullShellFinal})
on the domain 
\begin{equation}
\mathcal{D}_{0}^{u_{0}}\doteq\big([0,u_{0}]\times[0,u_{0}]\big)\cap\big\{ u<v\big\}\subset\mathbb{R}^{2}\label{eq:DomainNearAxis-1}
\end{equation}
with smooth axis
\begin{equation}
\gamma_{0}^{u_{0}}\doteq\big([0,u_{0}]\times[0,u_{0}]\big)\cap\big\{ u=v\big\}\subset\partial\mathcal{D}_{u_{1}}^{u_{2}},\label{eq:SmallPartOfAxis-1}
\end{equation}
 satisfying (\ref{eq:InitialROmegaForExistence})\textendash (\ref{eq:InitialFForExistence}). 

\begin{figure}[h] 
\centering 
\scriptsize
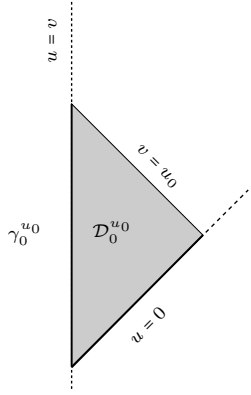 
\caption{Schematic depiction of the domain $\mathcal{D}_{0}^{u_0}$. \label{fig:AxisExistenceDomain}}
\end{figure}

\noindent \emph{Existence.} The existence of a smooth solution $(r,\Omega^{2};f)$
to (\ref{eq:RequationFinal})\textendash (\ref{NullShellFinal}) on
$\mathcal{D}_{0}^{u_{0}}$ will follow by applying a standard iteration
procedure. For any integer $n\in\mathbb{N}$, let us define the functions
$r_{n}:\mathcal{D}_{0}^{u_{0}}\rightarrow[0,+\infty)$, $\Omega_{n}^{2}:\mathcal{D}_{0}^{u_{0}}\rightarrow(0,+\infty)$
and $\bar{f}_{n}^{\prime}:\mathcal{D}_{0}^{u_{0}}\times[0,+\infty)^{2}\rightarrow[0,+\infty)$
by solving recursively (for a definition of the relevant parameters
and functions, see below):
\begin{align}
\partial_{u}\partial_{v}r_{n+1} & =-\frac{1}{2}\frac{\tilde{m}_{n}-\frac{1}{3}\Lambda r_{n}^{3}}{r_{n}^{2}}\Omega_{n}^{2}+4\pi r_{n}(T_{uv})_{n},\label{eq:RecursiveR}\\
\partial_{u}\partial_{v}\log(\Omega_{n+1}^{2}) & =\Big(\frac{\tilde{m}_{n}}{r_{n}^{3}}+\frac{\Lambda}{6}\Big)\Omega_{n}^{2}-16\pi(T_{uv})_{n},\label{eq:RecursiveOmega}\\
\Big[p^{a}\partial_{x^{a}}-(\text{\textgreek{G}}_{\alpha\beta}^{\gamma})_{n} & p^{\alpha}p^{\beta}\partial_{p^{\gamma}}\Big]f_{n+1}=0,\label{eq:RecursiveF}
\end{align}
with initial conditions on $u=0$: 
\begin{equation}
(r_{n+1},\Omega_{n+1}^{2};\bar{f}_{n+1}^{\prime})|_{u=0}=(r_{/},\Omega_{/}^{2};\bar{f}_{/})\label{eq:initialconditionsiteration}
\end{equation}
and boundary conditions on the axis $\gamma_{0}^{u_{0}}$: 
\begin{align}
r_{n+1}(u,u) & =0,\label{eq:BoundaryConditionsIteration}\\
\partial_{v}\Omega_{n+1}^{2}(u,u) & =\partial_{u}\Omega_{n+1}^{2}(u,u).\nonumber 
\end{align}
In the above, we have adopted the following notational conventions
for $n\ge1$: 

\begin{itemize}

\item The Vlasov field $f_{n}$ in (\ref{eq:RecursiveF}) is defined
in terms of $\bar{f}_{n}^{\prime}$ for any $l>0$ by
\begin{equation}
f_{n+1}(u,v;p^{u},p^{v},l)\doteq\bar{f}_{n+1}^{\prime}(u,v;p^{u},l)\cdot\delta\Big(\Omega_{n+1}^{2}(u,v)p^{u}p^{v}-\frac{l^{2}}{r_{n+1}(u,v)}\Big),\label{eq:DefinitionVlasovIteration}
\end{equation}
while its extension to $l=0$ is uniquely fixed by the condition that
$f_{n+1}$ is a continuous multiple of \\ $\delta\Big(\Omega_{n+1}^{2}(u,v)p^{u}p^{v}-\frac{l^{2}}{r_{n+1}(u,v)}\Big).$

\item The terms $\tilde{m}_{n}$ in (\ref{eq:RecursiveR})\textendash (\ref{eq:RecursiveOmega})
are defined by 
\begin{equation}
\tilde{m}_{n}(u,v)\doteq8\pi\int_{u}^{v}r_{n}^{2}\Omega_{n}^{-2}\Big(-\partial_{u}r_{n}\cdot(T_{vv})_{n}+\partial_{v}r_{n}\cdot(T_{uv})_{n}\Big)(u,\bar{v})\,d\bar{v},\label{eq:DefinitionMassIteration}
\end{equation}
where $(T_{\alpha\beta})_{n}$ are defined in terms of $r_{n}$, $\Omega_{n}^{2}$
and $f_{n}$ by (\ref{eq:SphericallySymmetricComponentsEnergyMomentum}).

\item  The terms $(\text{\textgreek{G}}_{\alpha\beta}^{\gamma})_{n}$
in (\ref{eq:RecursiveF}) are the Christoffel symbols of the metric
$g_{n}$ on $(\mathcal{D}_{0}^{u_{0}}\backslash\gamma_{0}^{u_{0}})\times\mathbb{S}^{2}$
in the $(u,v,\theta,\varphi)$ coordinate chart, where
\begin{equation}
g_{n}\doteq-\Omega_{n}^{2}dudv+r_{n}^{2}g_{\mathbb{S}^{2}}.\label{eq:MetricIteration}
\end{equation}

\end{itemize}

When $n=0$, we will set 
\begin{equation}
(r_{0},\Omega_{0}^{2};f_{0})\doteq(r_{AdS},\Omega_{AdS}^{2};0).
\end{equation}
\begin{rem*}
Note that the boundary condition (\ref{eq:BoundaryConditionsIteration})
implies that 
\begin{equation}
\partial_{v}r_{n+1}(u,u)=-\partial_{u}r_{n+1}(u,u).\label{eq:DerivativeRAxisIteration}
\end{equation}
\end{rem*}
Let $C_{1}\gg1$ be a large constant depending only on the initial
data $(r_{/},\Omega_{/}^{2},\bar{f}_{/})$. We will establish the
following inductive bounds for $(r_{n},\Omega_{n}^{2};f_{n})$: Assuming
that, for any $0\le k\le n$: 
\begin{align}
\sup_{\mathcal{D}_{0}^{u_{0}}}\Big(\sum_{j=0}^{2}\sum_{j_{1}+j_{2}=j}\big\{ R_{0}^{j}\big(\big|\partial_{u}^{j_{1}}\partial_{v}^{j_{2}}\log(\Omega_{k}^{2})\big|+\big|\partial_{u}^{j_{1}}\partial_{v}^{j_{2}}\log(\partial_{v}r_{k})\big|+\big|\partial_{u}^{j_{1}}\partial_{v}^{j_{2}}\log(-\partial_{u}r_{k})\big|\big)\big\}+\label{eq:UniformBoundPrevious}\\
+\sum_{j=0}^{1}\sum_{j_{1}+j_{2}=j}\big\{ R_{0}^{2+j}\big(|\partial_{u}^{j_{1}}\partial_{v}^{j_{2}}(T_{uu})_{k}|+|\partial_{u}^{j_{1}}\partial_{v}^{j_{2}}(T_{uv})_{k}|+|\partial_{u}^{j_{1}}\partial_{v}^{j_{2}}( & T_{vv})_{k}|\big)\big\}\Big)<C_{1},\nonumber 
\end{align}
we will show that (\ref{eq:UniformBoundPrevious}) also holds for
$k=n+1$ and, moreover (in the case $n\ge2$): 
\begin{equation}
\mathfrak{D}_{n+1}\le\frac{1}{2}\mathfrak{D}_{n},\label{eq:DifferenceBoundToShow}
\end{equation}
where 
\begin{align}
\mathfrak{D}_{k}\doteq\sup_{\mathcal{D}_{0}^{u_{0}}}\Bigg\{ & \sum_{j=1}^{2}\sum_{j_{1}+j_{2}=j}R_{0}^{j-1}\big|\partial_{u}^{j_{1}}\partial_{v}^{j_{2}}r_{k}-\partial_{u}^{j_{1}}\partial_{v}^{j_{2}}r_{k-1}\big|+\sum_{j=0}^{1}\sum_{j_{1}+j_{2}=j}R_{0}^{j}\big|\partial_{u}^{j_{1}}\partial_{v}^{j_{2}}\log\Omega_{k}^{2}-\partial_{u}^{j_{1}}\partial_{v}^{j_{2}}\log\Omega_{k-1}^{2}\big|\label{eq:DifferenceIterationDefinition}\\
 & +R_{0}^{2}r^{-2}\int_{0}^{+\infty}\int_{0}^{+\infty}\big(\Omega_{k-1}^{2}(p^{u}+p^{v})\big)^{2}\big|\bar{f}_{k}-\bar{f}_{k-1}\big|(\cdot;p^{u},l)\,\frac{dp^{u}}{p^{u}}ldl\Bigg\}.\nonumber 
\end{align}
\begin{rem*}
The fact that the left hand side of (\ref{eq:UniformBoundPrevious})
is indeed bounded on $\{0\}\times[0,u_{0})$ when $u_{0}$ is sufficiently
small follows from (\ref{eq:NonTrappingForAsymptoticallyAdSInitialData})
and (\ref{eq:NegativeDurNearAxis}). 
\end{rem*}
Let us first infer from the inductive bound (\ref{eq:UniformBoundPrevious})
a useful estimate for $\frac{\tilde{m}_{k}}{r_{k}^{3}}$. For any
$0\le k\le n$ and any $(u,v)\in\mathcal{D}_{0}^{u_{0}}$, the definition
(\ref{eq:DefinitionMassIteration}) of $\tilde{m}_{k}$ yields
\begin{equation}
\frac{\tilde{m}_{k}}{r_{k}^{3}}(u,v)=8\pi\frac{\int_{u}^{v}r_{k}^{2}\Omega_{k}^{-2}\Big(-\partial_{u}r_{k}\cdot(T_{vv})_{k}+\partial_{v}r_{k}\cdot(T_{uv})_{k}\Big)(u,\bar{v})\,d\bar{v}}{\big(r_{k}(u,v)\big)^{3}}.\label{eq:FirstFormulaToDifferentiate}
\end{equation}
Differentiating (\ref{eq:FirstFormulaToDifferentiate}) with respect
to $\partial_{v}$, we readily calculate that: 
\begin{align}
\partial_{v}\Big(\frac{\tilde{m}_{k}}{r_{k}^{3}}\Big)(u,v) & =8\pi\partial_{v}\Bigg\{ r_{k}^{-3}(u,v)\int_{u}^{v}r_{k}^{2}\Omega_{k}^{-2}\Big(-\partial_{u}r_{k}\cdot(T_{vv})_{k}+\partial_{v}r_{k}\cdot(T_{uv})_{k}\Big)(u,\bar{v})\,d\bar{v}\Bigg\}=\label{eq:DvOfMr^3}\\
 & =8\pi\Bigg\{-3\big(r_{k}^{-4}\partial_{v}r\big)(u,v)\cdot\int_{u}^{v}r_{k}^{2}\Omega_{k}^{-2}\Big(-\partial_{u}r_{k}\cdot(T_{vv})_{k}+\partial_{v}r_{k}\cdot(T_{uv})_{k}\Big)(u,\bar{v})\,d\bar{v}+\nonumber \\
 & \hphantom{=8\pi\Bigg\{}+r_{k}^{-1}\Omega_{k}^{-2}\Big(-\partial_{u}r_{k}\cdot(T_{vv})_{k}+\partial_{v}r_{k}\cdot(T_{uv})_{k}\Big)(u,v)\Bigg\}=\nonumber \\
 & =8\pi\Bigg\{-3\big(r_{k}^{-4}\partial_{v}r\big)(u,v)\cdot\int_{u}^{v}\frac{1}{3}\partial_{v}(r_{k}^{3})\Omega_{k}^{-2}\Big(-\frac{\partial_{u}r_{k}}{\partial_{v}r_{k}}\cdot(T_{vv})_{k}+(T_{uv})_{k}\Big)(u,\bar{v})\,d\bar{v}+\nonumber \\
 & \hphantom{=8\pi\Bigg\{}+r_{k}^{-1}\Omega_{k}^{-2}\Big(-\partial_{u}r_{k}\cdot(T_{vv})_{k}+\partial_{v}r_{k}\cdot(T_{uv})_{k}\Big)(u,v)\Bigg\}=\nonumber \\
 & =8\pi\big(r_{k}^{-4}\partial_{v}r\big)(u,v)\cdot\Bigg\{\int_{u}^{v}r_{k}^{3}\Big[\partial_{v}\Big(\Omega_{k}^{-2}\frac{-\partial_{u}r_{k}}{\partial_{v}r_{k}}\Big)\cdot(T_{vv})_{k}+(\partial_{v}\Omega_{k}^{-2})\cdot(T_{uv})_{k}\Big](u,\bar{v})\,d\bar{v}+\nonumber \\
 & \hphantom{=8\pi\big(r_{k}^{-4}\partial_{v}r\big)(u,v)\cdot\Bigg\{}+\int_{u}^{v}r_{k}^{3}\Big[\Omega_{k}^{-2}\frac{-\partial_{u}r_{k}}{\partial_{v}r_{k}}\cdot\partial_{v}(T_{vv})_{k}+\Omega_{k}^{-2}\cdot\partial_{v}(T_{uv})_{k}\Big](u,\bar{v})\,d\bar{v}\Bigg\}.\nonumber 
\end{align}
where, in passing from the fifth to the sixth line of (\ref{eq:DvOfMr^3}),
we integrated by parts once. Differentiating (\ref{eq:FirstFormulaToDifferentiate})
(and using the fact that $r_{k}(u,u)=0$), we also calculate that:
\begin{align}
\partial_{u}\Big(\frac{\tilde{m}_{k}}{r_{k}^{3}}\Big)(u,v) & =8\pi\Bigg\{-3\big(r_{k}^{-4}\partial_{u}r\big)(u,v)\cdot\int_{u}^{v}r_{k}^{2}\Omega_{k}^{-2}\Big(-\partial_{u}r_{k}\cdot(T_{vv})_{k}+\partial_{v}r_{k}\cdot(T_{uv})_{k}\Big)(u,\bar{v})\,d\bar{v}+\label{eq:DuOfMr^3Almost}\\
 & \hphantom{=8\pi\Bigg\{}+r_{k}^{-3}(u,v)\int_{u}^{v}\partial_{u}\Big[r_{k}^{2}\Omega_{k}^{-2}\Big(-\partial_{u}r_{k}\cdot(T_{vv})_{k}+\partial_{v}r_{k}\cdot(T_{uv})_{k}\Big)\Big](u,\bar{v})\,d\bar{v}\Bigg\}.\nonumber 
\end{align}
Using the relations 
\begin{align}
\big(r_{k}^{-4}\partial_{u}r & \big)(u,v)\cdot\int_{u}^{v}r_{k}^{2}\Omega_{k}^{-2}\Big(-\partial_{u}r_{k}\cdot(T_{vv})_{k}+\partial_{v}r_{k}\cdot(T_{uv})_{k}\Big)(u,\bar{v})\,d\bar{v}=\label{eq:Aaaah}\\
= & \frac{1}{3}\big(r_{k}^{-4}\partial_{u}r\big)(u,v)\cdot\int_{u}^{v}(\partial_{v}r_{k}^{3})\Omega_{k}^{-2}\Big(\frac{-\partial_{u}r_{k}}{\partial_{v}r_{k}}\cdot(T_{vv})_{k}+(T_{uv})_{k}\Big)(u,\bar{v})\,d\bar{v}=\nonumber \\
= & \frac{1}{3}\big(r_{k}^{-4}\partial_{u}r\big)(u,v)\cdot\Bigg\{ r_{k}^{3}\Omega_{k}^{-2}\Big(\frac{-\partial_{u}r_{k}}{\partial_{v}r_{k}}\cdot(T_{vv})_{k}+(T_{uv})_{k}\Big)(u,v)-\nonumber \\
 & \hphantom{\frac{1}{3}\big(r_{k}^{-4}\partial_{u}r\big)(u,v)\cdot}-\int_{u}^{v}r_{k}^{3}\Big[\partial_{v}\Big(\Omega_{k}^{-2}\frac{-\partial_{u}r_{k}}{\partial_{v}r_{k}}\Big)\cdot(T_{vv})_{k}+(\partial_{v}\Omega_{k}^{-2})\cdot(T_{uv})_{k}\Big](u,\bar{v})\,d\bar{v}-\nonumber \\
 & \hphantom{\frac{1}{3}\big(r_{k}^{-4}\partial_{u}r\big)(u,v)\cdot}-\int_{u}^{v}r_{k}^{3}\Big[\Omega_{k}^{-2}\frac{-\partial_{u}r_{k}}{\partial_{v}r_{k}}\cdot\partial_{v}(T_{vv})_{k}+\Omega_{k}^{-2}\cdot\partial_{v}(T_{uv})_{k}\Big](u,\bar{v})\,d\bar{v}\Bigg\}\nonumber 
\end{align}
and 
\begin{align}
\int_{u}^{v}\partial_{u}\Big[r_{k}^{2}\Omega_{k}^{-2} & \Big(-\partial_{u}r_{k}\cdot(T_{vv})_{k}+\partial_{v}r_{k}\cdot(T_{uv})_{k}\Big)\Big](u,\bar{v})\,d\bar{v}=\label{eq:Aaaaah}\\
= & 2\int_{u}^{v}r_{k}\partial_{u}r_{k}\Omega_{k}^{-2}\Big(-\partial_{u}r_{k}\cdot(T_{vv})_{k}+\partial_{v}r_{k}\cdot(T_{uv})_{k}\Big)(u,\bar{v})\,d\bar{v}-\nonumber \\
 & \hphantom{bla}-\int_{u}^{v}r_{k}^{2}\Big[-\partial_{u}\Big(\Omega_{k}^{-2}\partial_{u}r_{k}\Big)\cdot(T_{vv})_{k}+\partial_{u}\Big(\Omega_{k}^{-2}\partial_{v}r_{k}\Big)\cdot(T_{uv})_{k}\Big](u,\bar{v})\,d\bar{v}-\nonumber \\
 & \hphantom{bla}-\int_{u}^{v}r_{k}^{2}\Big[-\Omega_{k}^{-2}\partial_{u}r_{k}\cdot\partial_{u}(T_{vv})_{k}+\Omega_{k}^{-2}\partial_{v}r_{k}\cdot\partial_{u}(T_{uv})_{k}\Big](u,\bar{v})\,d\bar{v}\Bigg\}=\nonumber \\
= & \int_{u}^{v}\partial_{v}(r_{k}^{2})\cdot\partial_{u}r_{k}\Omega_{k}^{-2}\Big(\frac{-\partial_{u}r_{k}}{\partial_{v}r_{k}}\cdot(T_{vv})_{k}+(T_{uv})_{k}\Big)(u,\bar{v})\,d\bar{v}-\nonumber \\
 & \hphantom{bla}-\int_{u}^{v}r_{k}^{2}\Big[-\partial_{u}\Big(\Omega_{k}^{-2}\partial_{u}r_{k}\Big)\cdot(T_{vv})_{k}+\partial_{u}\Big(\Omega_{k}^{-2}\partial_{v}r_{k}\Big)\cdot(T_{uv})_{k}\Big](u,\bar{v})\,d\bar{v}-\nonumber \\
 & \hphantom{bla}-\int_{u}^{v}r_{k}^{2}\Big[-\Omega_{k}^{-2}\partial_{u}r_{k}\cdot\partial_{u}(T_{vv})_{k}+\Omega_{k}^{-2}\partial_{v}r_{k}\cdot\partial_{u}(T_{uv})_{k}\Big](u,\bar{v})\,d\bar{v}\Bigg\}=\nonumber \\
= & r_{k}^{2}\partial_{u}r_{k}\Omega_{k}^{-2}\Big(\frac{-\partial_{u}r_{k}}{\partial_{v}r_{k}}\cdot(T_{vv})_{k}+(T_{uv})_{k}\Big)(u,v)-\nonumber \\
 & \hphantom{bla}-\int_{u}^{v}r_{k}^{2}\Big[-\Big\{\partial_{u}\Big(\Omega_{k}^{-2}\partial_{u}r_{k}\Big)+\partial_{v}\Big(\Omega_{k}^{-2}\frac{(\partial_{u}r_{k})^{2}}{\partial_{v}r_{k}}\Big)\Big\}\cdot(T_{vv})_{k}+\nonumber \\
 & \hphantom{\hphantom{bla}-\int_{u}^{v}r_{k}^{2}\Big[-}+\Big\{\partial_{u}\Big(\Omega_{k}^{-2}\partial_{v}r_{k}\Big)+\partial_{v}\Big(\Omega_{k}^{-2}\partial_{u}r_{k}\Big)\Big\}\cdot(T_{uv})_{k}\Big](u,\bar{v})\,d\bar{v}-\nonumber\\
 & \hphantom{bla}-\int_{u}^{v}r_{k}^{2}\partial_{u}r_{k}\Omega_{k}^{-2}\Big(\frac{-\partial_{u}r_{k}}{\partial_{v}r_{k}}\cdot\partial_{v}(T_{vv})_{k}+\partial_{v}(T_{uv})_{k}\Big)(u,\bar{v})\,d\bar{v}-\nonumber \\
 & \hphantom{bla}-\int_{u}^{v}r_{k}^{2}\Big[-\Omega_{k}^{-2}\partial_{u}r_{k}\cdot\partial_{u}(T_{vv})_{k}+\Omega_{k}^{-2}\partial_{v}r_{k}\cdot\partial_{u}(T_{uv})_{k}\Big](u,\bar{v})\,d\bar{v}\Bigg\},\nonumber 
\end{align}
we deduce from (\ref{eq:DuOfMr^3Almost}) that 
\begin{align}
\partial_{u}\Big(\frac{\tilde{m}_{k}}{r_{k}^{3}}\Big)(u,v) & =-8\pi\Bigg\{\frac{1}{r_{k}^{3}(u,v)}\cdot\Big(\int_{u}^{v}r^{2}\Big(\mathcal{A}_{1}^{(v)}(T_{vv})_{k}+\mathcal{A}_{1}^{(u)}(T_{uv})_{k}+\mathcal{A}_{2}^{(v)}\partial_{v}(T_{vv})_{k}+\label{eq:DuOfMr^3}\\
 & \hphantom{=-8\pi\Bigg\{\frac{1}{r_{k}^{3}(u,v)}\cdot\Big(\int_{u}^{v}r^{2}\Big(}+\mathcal{A}_{2}^{(u)}\partial_{v}(T_{uv})_{k}+\mathcal{A}_{3}^{(v)}\partial_{u}(T_{vv})_{k}+\mathcal{A}_{3}^{(u)}\partial_{u}(T_{uv})_{k}\Big)(u,\bar{v})\,d\bar{v}\Big)+\nonumber\\
 & \hphantom{=8\pi\Bigg\{}+\partial_{u}r(u,v)\frac{1}{r_{k}^{4}(u,v)}\cdot\Big(\int_{u}^{v}r^{3}\Big(\mathcal{B}_{1}^{(v)}(T_{vv})_{k}+\mathcal{B}_{1}^{(u)}(T_{uv})_{k}+\nonumber \\
 & \hphantom{\hphantom{=8\pi\Bigg\{}+\partial_{u}r(u,v)\frac{1}{r_{k}^{4}(u,v)}\cdot\Big(\int_{u}^{v}}+\mathcal{B}_{2}^{(v)}\partial_{v}(T_{vv})_{k}+\mathcal{B}_{2}^{(u)}\partial_{v}(T_{uv})_{k}\Big)(u,\bar{v})\,d\bar{v}\Big)\Bigg\},\nonumber
\end{align}
where 
\begin{align*}
\mathcal{A}_{1}^{(v)}\doteq & -\partial_{u}\Big(\Omega_{k}^{-2}\partial_{u}r_{k}\Big)-\partial_{v}\Big(\Omega_{k}^{-2}\frac{(\partial_{u}r_{k})^{2}}{\partial_{v}r_{k}}\Big),\\
\mathcal{A}_{1}^{(u)}\doteq & \partial_{u}\Big(\Omega_{k}^{-2}\partial_{v}r_{k}\Big)+\partial_{v}\Big(\Omega_{k}^{-2}\partial_{u}r_{k}\Big),\\
\mathcal{A}_{2}^{(v)}\doteq & -\Omega_{k}^{-2}\frac{(\partial_{u}r_{k})^{2}}{\partial_{v}r_{k}},\\
\mathcal{A}_{2}^{(u)}\doteq & \Omega_{k}^{-2}\partial_{u}r_{k},\\
\mathcal{A}_{3}^{(v)}\doteq & -\Omega_{k}^{-2}\partial_{u}r_{k},\\
\mathcal{A}_{3}^{(u)}\doteq & \Omega_{k}^{-2}\partial_{v}r_{k}
\end{align*}
and 
\begin{align*}
\mathcal{B}_{1}^{(v)}\doteq & \partial_{v}\Big(\Omega_{k}^{-2}\frac{\partial_{u}r_{k}}{\partial_{v}r_{k}}\Big),\\
\mathcal{B}_{1}^{(u)}\doteq & -(\partial_{v}\Omega_{k}^{-2}),\\
\mathcal{B}_{2}^{(v)}\doteq & \Omega_{k}^{-2}\frac{\partial_{u}r_{k}}{\partial_{v}r_{k}},\\
\mathcal{B}_{2}^{(u)}\doteq & -\Omega_{k}^{-2}
\end{align*}
(note the cancellation of the ``bare'' terms $r_{k}^{2}\partial_{u}r_{k}\Omega_{k}^{-2}\Big(\frac{-\partial_{u}r_{k}}{\partial_{v}r_{k}}\cdot(T_{vv})_{k}+(T_{uv})_{k}\Big)(u,v)$
from (\ref{eq:Aaaah}) and (\ref{eq:Aaaaah}) in (\ref{eq:DuOfMr^3})). 

Using the inductive bound (\ref{eq:UniformBoundPrevious}) (and the
fact that $v-u\le e^{C_{1}}r_{k}(u,v)$, as a consequence of (\ref{eq:UniformBoundPrevious})),
we immediately infer from the relations (\ref{eq:DvOfMr^3}) and (\ref{eq:DuOfMr^3})
that there exists some constant $C(C_{1})>0$ depending only on $C_{1}$
such that, for any $0\le k\le n$, we can estimate: 
\begin{equation}
\sup_{\mathcal{D}_{0}^{u_{0}}}\sum_{j=0}^{1}\sum_{j_{1}+j_{2}=j}\Big|R_{0}^{2+j}\partial_{u}^{j_{1}}\partial_{v}^{j_{2}}(\frac{\tilde{m}_{k}}{r_{k}^{3}})\Big|\le C(C_{1}).\label{eq:InductiveM/R}
\end{equation}

Differentiating the boundary condition (\ref{eq:BoundaryConditionsIteration})
on $\gamma_{0}^{u_{0}}$ in the direction $\partial_{u}+\partial_{v}$
(which is tangential to $\gamma_{0}^{u_{0}}$) and using the relations
(\ref{eq:DefinitionMassIteration}) and (\ref{eq:RecursiveR}) (combined
with the bounds (\ref{eq:UniformBoundPrevious}) and (\ref{eq:InductiveM/R})),
we infer the following higher order relations for $r_{n+1}$ and $\Omega_{n+1}^{2}$
on $\gamma_{0}^{u_{0}}$: 
\begin{align}
\partial_{u}^{2}r_{n+1}(u,u)+\partial_{v}^{2}r_{n+1}(u,u) & =0,\label{eq:HigherOrderBoundaryConditionsRAxisIteration}\\
\partial_{u}^{3}r_{n+1}(u,u)+\partial_{v}^{3}r_{n+1}(u,u) & =0\nonumber 
\end{align}
and 
\begin{equation}
\partial_{v}^{2}\Omega_{n+1}^{2}(u,u)=\partial_{u}^{2}\Omega_{n+1}^{2}(u,u).\label{eq:HigherOrderBoundaryConditionsOmegaAxisIteration}
\end{equation}

Integrating (\ref{eq:RecursiveR})\textendash (\ref{eq:RecursiveOmega})
in $u,v$ and using the initial condition (\ref{eq:initialconditionsiteration})
at $u=0$, the boundary conditions (\ref{eq:BoundaryConditionsIteration})
(note also (\ref{eq:DerivativeRAxisIteration})) on the axis $\gamma_{0}^{u_{0}}$,
as well as the bounds (\ref{eq:UniformBoundPrevious}) and (\ref{eq:InductiveM/R})
for $k=n$, we can readily bound: 
\begin{align}
\sup_{\mathcal{D}_{0}^{u_{0}}}\Big(\big|\log(-\partial_{u}r_{n+1})\big| & +\big|\log\partial_{v}r_{n+1}\big|+\log\Omega_{n+1}^{2}+R_{0}|\partial_{u}\log\Omega_{n+1}^{2}|+R_{0}|\partial_{v}\log\Omega_{n+1}^{2}|\Big)\label{eq:BoundNextSepIteration}\\
 & \le C(C_{1})R_{0}^{-1}u_{0}+10\sup_{v\in[0,u_{0})}\Big(\big|\log\partial_{v}r_{/}\big|(v)+R_{0}|\partial_{v}\log\Omega_{/}^{2}|(v)+|\log\Omega_{/}^{2}|(v)\Big)\nonumber 
\end{align}
for some $C(C_{1})>0$ depending only on $C_{1}$. Repeating the same
procedure after commuting (\ref{eq:RecursiveR})\textendash (\ref{eq:RecursiveOmega})
with $\partial_{u},\partial_{v}$ and using (\ref{eq:HigherOrderBoundaryConditionsRAxisIteration})\textendash (\ref{eq:HigherOrderBoundaryConditionsOmegaAxisIteration}),
we also infer:
\begin{align}
\sup_{\mathcal{D}_{0}^{u_{0}}}\Big(\sum_{j_{1}+j_{2}=1}^{2} & R_{0}^{-1+j_{1}+j_{2}}|\partial_{u}^{j_{1}}\partial_{v}^{j_{2}}r_{n+1}\big|+\sum_{j_{1}+j_{2}=1}^{2}R_{0}^{j_{1}+j_{2}}\big|\partial_{u}^{j_{1}}\partial_{v}^{j_{2}}\Omega_{k}^{2}|\Big)\label{eq:BoundNextSepIterationHigher}\\
 & \le C(C_{1})R_{0}^{-1}u_{0}+10\sup_{v\in[0,u_{0})}\Big(\sum_{j=0}^{1}R_{0}^{j}\big|\partial_{v}^{j}\log\partial_{v}r_{/}\big|(v)+\sum_{j=0}^{2}R_{0}^{j}|\partial_{v}^{j}\log\Omega_{/}^{2}|(v)\Big).\nonumber 
\end{align}
On the other hand, after subtracting from (\ref{eq:RecursiveR})\textendash (\ref{eq:RecursiveOmega})
the same equations with $n-1$ in place of $n$ (assuming that $n\ge2$),
and similarly integrating in $u,v$ using (\ref{eq:initialconditionsiteration}),
(\ref{eq:BoundaryConditionsIteration}) and the bounds (\ref{eq:UniformBoundPrevious})
and (\ref{eq:InductiveM/R}) for $k=n,n-1$, we infer (for a possibly
larger $C(C_{1})$): 
\begin{equation}
\overline{\mathfrak{D}}_{n+1}\le C(C_{1})\frac{u_{0}}{R_{0}}\overline{\mathfrak{D}}_{n},\label{eq:EstimateForDifferenceEasy}
\end{equation}
wher we have set: 
\begin{align}
\overline{\mathfrak{D}}_{k}\doteq\sup_{\mathcal{D}_{0}^{u_{0}}}\Big( & \big|\log(-\partial_{u}r_{k})-\log(-\partial_{u}r_{k-1})\big|+\big|\log\partial_{v}r_{k}-\log\partial_{v}r_{k-1}\big|+\big|\log\Omega_{k}^{2}-\log\Omega_{k-1}^{2}\big|+\\
 & +R_{0}|\partial_{u}\log\Omega_{k}^{2}-\partial_{u}\log\Omega_{k-1}^{2}|+R_{0}|\partial_{v}\log\Omega_{k}^{2}-\partial_{v}\log\Omega_{k-1}^{2}|\Big).\nonumber 
\end{align}

Since $f_{n}$ solves the Vlasov equation (\ref{eq:DefinitionVlasovIteration}),
the conservation of the energy momentum tensor $(T_{\alpha\beta})_{n}$
of $f_{n}$ (i.\,e.~(\ref{eq:ConservationEnergyMomentum})) implies
that 
\begin{align}
\partial_{u}(r_{n}(T_{uv})_{n}) & =-\partial_{v}(r_{n}(T_{uu})_{n})-(\partial_{v}r_{n})(T_{uu})_{n}-(\partial_{u}r_{n})(T_{uv})_{n}+\big(r_{n}\partial_{u}\log\Omega_{n}^{2}-2\partial_{u}r_{n}\big)(T_{uv})_{n},\label{eq:DuTFromConservation}\\
\partial_{v}(r_{n}(T_{uv})_{n}) & =-\partial_{u}(r_{n}(T_{vv})_{n})-(\partial_{u}r_{n})(T_{vv})_{n}-(\partial_{v}r_{n})(T_{uv})_{n}+\big(r_{n}\partial_{v}\log\Omega_{n}^{2}-2\partial_{v}r_{n}\big)(T_{uv})_{n}.\label{eq:DvTFromConservation}
\end{align}
After commuting (\ref{eq:RecursiveR}) with $\partial_{v}^{2}$ and
$\partial_{u}\partial_{v}$, replacing the derivatives of $\partial_{v}(rT_{uv})$
with the corresponding derivatives of the expression (\ref{eq:DvTFromConservation})
and integrating in $u$, we infer, using (\ref{eq:UniformBoundPrevious}),
(\ref{eq:InductiveM/R}) and the trivial bounds
\[
\sup_{\mathcal{D}_{0}^{u_{0}}}r_{n}\le\sup_{u\in[0,u_{0}]}\int_{u}^{u_{0}}\partial_{v}r_{n}(u,v)\,dv\le C(C_{1})u_{0},
\]
(following from (\ref{eq:UniformBoundPrevious})), noting also that
$\partial_{u}[r_{n}\partial(T_{vv})_{n}]$ gives only boundary terms
after integration in $u$, we can readily bound: 
\begin{equation}
\sup_{\mathcal{D}_{0}^{u_{0}}}\Big(\sum_{j_{1}+j_{2}=2}R_{0}^{j_{1}+j_{2}}\big|\partial_{u}^{j_{1}}\partial_{v}^{j_{2}+1}r_{n+1}\big|\Big)\le C(C_{1})R_{0}^{-1}u_{0}+10\sup_{v\in[0,u_{0})}\sum_{j=0}^{2}R_{0}^{j}\Big(\big|\partial_{v}^{j}\log\partial_{v}r_{/}\big|(v)+|\partial_{v}^{j}\log\Omega_{/}^{2}|(v)\Big).\label{eq:FirstHigherBound}
\end{equation}
Performing the same procedure after commuting (\ref{eq:RecursiveR})
with $\partial_{u}^{2}$ and $\partial_{u}\partial_{v}$ and integrating
in $v$ starting from the axis $\gamma_{0}^{u_{0}}$ (using the boundary
conditions (\ref{eq:HigherOrderBoundaryConditionsRAxisIteration})\textendash (\ref{eq:HigherOrderBoundaryConditionsOmegaAxisIteration})
and the bounds (\ref{eq:BoundNextSepIterationHigher})\textendash (\ref{eq:FirstHigherBound})),
adding the resulting estimate to (\ref{eq:FirstHigherBound}), we
finally obtain:

\begin{equation}
\sup_{\mathcal{D}_{0}^{u_{0}}}\Big(\sum_{j_{1}+j_{2}=3}R_{0}^{j_{1}+j_{2}-1}\big|\partial_{u}^{j_{1}}\partial_{v}^{j_{2}}r_{n+1}\big|\Big)\le C(C_{1})R_{0}^{-1}u_{0}+10\sup_{v\in[0,u_{0})}\sum_{j=0}^{2}R_{0}^{j}\Big(\big|\partial_{v}^{j}\log\partial_{v}r_{/}\big|(v)+|\partial_{v}^{j}\log\Omega_{/}^{2}|(v)\Big).\label{eq:BoundNextSepIterationHigher-1}
\end{equation}

Similarly, subtracting from (\ref{eq:RecursiveR}) the same equation
with $n-1$ in place of $n$ (assuming that $n\ge2$) after commuting
once with $\partial_{u}$, $\partial_{v}$, we infer after integrating
in $u,v$ (using the bound (\ref{eq:UniformBoundPrevious}) and the
boundary condition (\ref{eq:HigherOrderBoundaryConditionsRAxisIteration})):
\begin{equation}
\sup_{\mathcal{D}_{0}^{u_{0}}}\sum_{j_{1}+j_{2}=2}R_{0}^{j_{1}+j_{2}-1}\big|\partial_{u}^{j_{1}}\partial_{v}^{j_{2}}r_{n+1}-\partial_{u}^{j_{1}}\partial_{v}^{j_{2}}r_{n}\big|\le C(C_{1})R_{0}^{-1}u_{0}\mathfrak{D}_{n}.\label{eq:DifferenceHigherOrderDerivativesInR}
\end{equation}

Let $\gamma:[0,a)\rightarrow\mathcal{D}_{0}^{u_{0}}$, $\gamma(0)\in\{u=0\}$,
be a future directed, null geodesic of the metric $g_{n}$ (defined
by \ref{eq:MetricIteration}), such that $(\gamma;\dot{\gamma})$
lies in the support of the Vlasov field $f_{n+1}$ and has angular
momentum $l>0$. The relations (\ref{eq:NullGeodesicsSphericalSymmetry})
for $\gamma$ yield that, with respect to the affine parameter $s$:
\begin{equation}
\frac{d}{ds}\Big(\Omega_{n}^{2}(\gamma(s))(\dot{\gamma}^{u}(s)+\dot{\gamma}^{v}(s))\Big)=\Big((\partial_{u}+\partial_{v})\log\Omega_{n}^{2}-2\frac{\partial_{v}r_{n}+\partial_{u}r_{n}}{r_{n}}\Big)\frac{l^{2}}{r_{n}^{2}}\Bigg|_{\gamma(s)}.\label{eq:NullGeodesicForEnergy}
\end{equation}
In view of the relation (\ref{eq:NullShellAngularMomentum}) for $\dot{\gamma}$
, we can bound: 
\begin{equation}
\frac{l^{2}}{r_{n}^{2}(\gamma(s))}=\Omega_{n}^{2}(\gamma(s))\dot{\gamma}^{u}(s)\dot{\gamma}^{v}(s)\le\frac{1}{2}\Omega_{n}^{-2}(\gamma(s))\cdot\Big(\Omega_{n}^{2}(\gamma(s))(\dot{\gamma}^{u}(s)+\dot{\gamma}^{v}(s))\Big)^{2}.\label{eq:BoundFromNullness}
\end{equation}
Moreover, in view of the boundary condition (\ref{eq:DerivativeRAxisIteration})
on $\gamma_{0}^{u_{0}}$, we can bound for some constant $C$ depending
only on $V_{*}$
\begin{equation}
\sup_{\mathcal{D}_{0}^{u_{0}}}\Big(\Big|\frac{\partial_{v}r_{n}+\partial_{u}r_{n}}{r_{n}}\Big|\Big)\le C\Big|\frac{\sup_{\mathcal{D}_{0}^{u_{0}}}(\partial_{v}-\partial_{u})(\partial_{v}r_{n}+\partial_{u}r_{n})}{\inf_{\mathcal{D}_{0}^{u_{0}}}(\partial_{v}-\partial_{u})r_{n}}\Big|.\label{eq:BoundOFd_r/r}
\end{equation}
Therefore, using (\ref{eq:UniformBoundPrevious}), (\ref{eq:BoundFromNullness})
and (\ref{eq:BoundOFd_r/r}), the relation (\ref{eq:NullGeodesicForEnergy})
yields the estimate
\begin{equation}
\frac{d}{ds}\Big(\Omega_{n}^{2}(\gamma(s))(\dot{\gamma}^{u}(s)+\dot{\gamma}^{v}(s))\Big)\le C(C_{1})\Big(\Omega_{n}^{2}(\gamma(s))(\dot{\gamma}^{u}(s)+\dot{\gamma}^{v}(s))\Big)^{2}.\label{eq:BeforeGronwalIteration}
\end{equation}
Using the fact that, along $\gamma$, we have 
\[
(\dot{\gamma}_{u}(s)+\dot{\gamma}_{v}(s))ds\sim(du+dv)|_{\gamma},
\]
we deduce from (\ref{eq:BeforeGronwalIteration}), after applying
Gronwall's inequality and using (\ref{eq:UniformBoundPrevious}),
that:
\begin{equation}
\sup_{\gamma}\Big(\Omega_{n}^{2}(\dot{\gamma}^{u}+\dot{\gamma}^{v})\Big)\le\Omega_{n}^{2}(\dot{\gamma}^{u}+\dot{\gamma}^{v})(0)\cdot\exp\Big(C(C_{1})R_{0}^{-1}u_{0}\Big).\label{eq:BoundForEnergy}
\end{equation}

The bound (\ref{eq:BoundForEnergy}) implies, in view of the bound
(\ref{eq:BoundedSupportDefinition}) (with $C_{0}$ in place of $C$)
for the support of the initial data $\bar{f}_{/}$ (using also (\ref{eq:NullShellAngularMomentum}))
that 
\begin{equation}
supp(f_{n+1})\subset\Big\{\Omega_{n}^{2}(p^{u}+p^{v})\le C_{0}\exp\Big(C(C_{1})R_{0}^{-1}u_{0}\Big)\Big\}\cap\Big\{\frac{l}{r_{n}}\le C_{0}\exp\Big(C(C_{1})R_{0}^{-1}u_{0}\Big)\Big\}.\label{eq:BoundSuppF_n}
\end{equation}
Thus, from (\ref{eq:SphericallySymmetricComponentsEnergyMomentum})
and (\ref{eq:BoundSuppF_n}) (using also the fact that $||\bar{f}_{n+1}||_{\infty}=||\bar{f}_{/}||_{\infty}$)
we deduce that, for some constant $C>0$ depending only on $V_{*}$:
\begin{equation}
\sup_{\mathcal{D}_{0}^{u_{0}}}\Big((T_{uu})_{n+1}+(T_{uv})_{n+1}+(T_{vv})_{n+1}\Big)\le C\cdot C_{0}^{5}||\bar{f}_{/}||_{\infty}\exp\Big(C(C_{1})R_{0}^{-1}u_{0}\Big).\label{eq:BoundEnergyMomentumTensorsIteration}
\end{equation}

Let $(x_{n}^{0},\ldots,x_{n}^{3})$ be the Cartesian coordinates associated
to the metric $g_{n}$, defined by the relations (\ref{eq:CartesianCoordinates}).
The bound (\ref{eq:UniformBoundPrevious}) implies that the coordinate
transformation $(u,v,\theta,\varphi)\rightarrow(x_{n}^{0},\ldots,x_{n}^{3})$
is of $C^{2,1}$ regularity. Furthermore, in view of the expression
(\ref{eq:CartesianMetric}) of the metric $g_{n}$ in the Cartesian
coordinates, the bound (\ref{eq:UniformBoundPrevious}) implies that
the Cartesian Christoffel symbols of $g_{n}$ satisfy 
\begin{equation}
||(\text{\textgreek{G}}_{n}^{Cart})_{\alpha\beta}^{\gamma}||_{C^{0,1}}\le C(C_{1}).\label{eq:BoundCartesianChristoffelSymbols}
\end{equation}
Differentiating the Vlasov equation (\ref{eq:RecursiveF}) for $f_{n+1}$
with respect to the Cartesian coordinates and using (\ref{eq:BoundCartesianChristoffelSymbols})
and (\ref{eq:BoundSuppF_n}), we can therefore readily estimate: 
\begin{equation}
||\partial_{x_{n}^{a}}\bar{f}_{n+1}||_{\infty}\le||\partial_{x_{n}^{a}}\bar{f}_{/}||_{\infty}+C(C_{1})R_{0}^{-2}u_{0}.\label{eq:BoundDerivativeF_n}
\end{equation}
Using (\ref{eq:BoundSuppF_n}) and (\ref{eq:BoundDerivativeF_n}),
as well as (\ref{eq:BoundNextSepIterationHigher}) and (\ref{eq:BoundEnergyMomentumTensorsIteration})for
$n$ in place of $n+1$, we can estimate for the Cartesian components
of the energy momentum tensor of $f_{n+1}$ for some constant $C>0$
depending only on $\Lambda$: 
\begin{align}
R_{0}||\partial_{x_{n}^{a}}(T_{\beta\gamma})_{n+1}||_{\infty}\le CC_{0}^{5}\Big(1+\Big\{\sum_{j=0}^{1}R_{0}^{j}||\partial_{v}^{j}\log\partial_{v}r_{/}||_{\infty}+ & \sum_{j=0}^{2}R_{0}^{j}||\partial_{v}^{j}\log\Omega_{/}^{2}||_{\infty}\Big\}+C(C_{1})R_{0}^{-1}u_{0}\Big)\label{eq:BoundCartesianDerivativesT}\\
 & \times\Big(||\bar{f}_{/}||_{\infty}+R_{0}||\partial_{x_{n}^{a}}\bar{f}_{/}||_{\infty}+C(C_{1})R_{0}^{-1}u_{0}\Big)\exp\Big(C(C_{1})R_{0}^{-1}u_{0}\Big).\nonumber 
\end{align}
Using and (\ref{eq:BoundEnergyMomentumTensorsIteration})for $n$
in place of $n+1$, as well as the relations (\ref{eq:CartesianCoordinates})
defining the Cartesian coordinates, from (\ref{eq:BoundCartesianDerivativesT})
we obtain the folowing bound for the first order derivatives of $T_{n+1}$
in the double null coordinate system (for a possibly larger constant
$C$):
\begin{align}
\sup_{\mathcal{D}_{0}^{u_{0}}}\sum_{j=0}^{1}\sum_{j_{1}+j_{2}=j}R_{0}^{j+2} & \big(|\partial_{u}^{j_{1}}\partial_{v}^{j_{2}}(T_{uu})_{n+1}|+|\partial_{u}^{j_{1}}\partial_{v}^{j_{2}}(T_{uv})_{n+1}|+|\partial_{u}^{j_{1}}\partial_{v}^{j_{2}}(T_{vv})_{n+1}|\big)\le\label{eq:BoundUVDerivativesT}\\
\le & CC_{0}^{5}\Big(1+\Big\{\sum_{j=0}^{1}R_{0}^{j}||\partial_{v}^{j}\log\partial_{v}r_{/}||_{\infty}+\sum_{j=0}^{2}R_{0}^{j}||\partial_{v}^{j}\log\Omega_{/}^{2}||_{\infty}\Big\}+C(C_{1})R_{0}^{-1}u_{0}\Big)\times\nonumber \\
 & \hphantom{CC_{0}^{5}\Big(}\times\Big(||\bar{f}_{/}||_{\infty}+R_{0}||\partial_{x_{n}^{a}}\bar{f}_{/}||_{\infty}+C(C_{1})R_{0}^{-1}u_{0}\Big)\exp\Big(C(C_{1})R_{0}^{-1}u_{0}\Big).\nonumber
\end{align}

For $n\ge2$, let $\gamma_{(n)},\gamma_{(n-1)}:(-s_{1},s_{1})\rightarrow\mathcal{D}_{0}^{u_{0}}$
be two future directed curves such that: 

\begin{itemize}

\item $\gamma_{(i)}$ is a null geodesic for $g_{i}$, $i=n,n-1$,

\item $\gamma_{(n)}(0)=\gamma_{(n-1)}(0)$ and $\dot{\gamma}_{(n)}(0)=\dot{\gamma}_{(n-1)}(0)$,

\item $(\gamma_{(i)};\dot{\gamma}_{(i)})$ lies in the support of
the Vlasov field $f_{i+1}$, $i=n,n-1$.

\end{itemize}

By subtracting from equation (\ref{eq:GeodesicFlow}) for $\gamma_{(n)}$
the same equation for $\gamma_{(n-1)}$,working in the Cartesian coordinate
systems $(x_{n}^{0},\ldots,x_{n}^{3})$, $(x_{n-1}^{0},\ldots,x_{n-1}^{3})$
for $g_{n}$, $g_{n-1}$, we can estimate for some absolute constant
$C>0$: 
\begin{align}
\frac{d}{ds}\big|\gamma_{(n)}-\gamma_{(n-1)}\big|(s) & \le\big|\dot{\gamma}_{(n)}-\dot{\gamma}_{(n-1)}\big|(s),\label{eq:DifferenceGamma}\\
\frac{d}{ds}\big|\dot{\gamma}_{(n)}-\dot{\gamma}_{(n-1)}\big|(s)\le & CC\big(1+||\dot{\gamma}_{(n)}||_{\infty}+||\dot{\gamma}_{(n-1)}||_{\infty}\big)^{2}||(\text{\textgreek{G}}_{n}^{Cart})_{\beta\gamma}^{\delta}-(\text{\textgreek{G}}_{n-1}^{Cart})_{\beta\gamma}^{\delta}||_{\infty}+\label{eq:DifferenceGammaDot}\\
 & +C\big(1+||\dot{\gamma}_{(n)}||_{\infty}+||\dot{\gamma}_{(n-1)}||_{\infty}\big)^{2}||\partial_{x^{\alpha}}(\text{\textgreek{G}}_{n}^{Cart})_{\beta\gamma}^{\delta}||_{\infty}\cdot|\gamma_{(n)}-\gamma_{(n-1)}|(s)+\\
 & +C\big(1+||\dot{\gamma}_{(n)}||_{\infty}+||\dot{\gamma}_{(n-1)}||_{\infty}\big)\big(||(\text{\textgreek{G}}_{n}^{Cart})_{\beta\gamma}^{\delta}||_{\infty}+||(\text{\textgreek{G}}_{n-1}^{Cart})_{\beta\gamma}^{\delta}||_{\infty}\big)\cdot|\dot{\gamma}_{(n)}-\dot{\gamma}_{(n-1)}|(s),\nonumber 
\end{align}
where $|\cdot|$ denotes the Euclidean norm in the Cartesian coordinates.
Using (\ref{eq:UniformBoundPrevious}) and (\ref{eq:BoundSuppF_n})
for $n,n-1$ in place of $n+1$, as well as (\ref{eq:BoundForEnergy})
and the expression (\ref{eq:CartesianMetric}) for $g_{n},g_{n-1}$
in Cartesian coordinates, from (\ref{eq:DifferenceGamma})\textendash (\ref{eq:DifferenceGammaDot})
(and the definition (\ref{eq:DifferenceIterationDefinition}) of $\mathfrak{D}_{n}$)
we infer that: 
\begin{equation}
\frac{||\dot{\gamma}_{(n)}-\dot{\gamma}_{(n-1)}||_{\infty}+R_{0}^{-1}||\gamma_{(n)}-\gamma_{(n-1)}||_{\infty}}{|\dot{\gamma}_{n}(0)|+|\dot{\gamma}_{n-1}(0)|}\le C(C_{1})R_{0}^{-1}u_{0}\mathfrak{D}_{n}.\label{eq:BoundForDifferenceGeodesicFlow}
\end{equation}
The bound (\ref{eq:BoundForDifferenceGeodesicFlow}) readily implies
the following estimate for the Vlasov fields $f_{n+1},f_{n}$ (using
also the mean value theorem for $\bar{f}_{/}$, as well as the condition
(\ref{eq:BoundedSupportDefinition})):
\begin{equation}
R_{0}^{2}\sup_{\mathcal{D}_{0}^{u_{0}}}r^{-2}\int_{0}^{+\infty}\int_{0}^{+\infty}\big(\Omega_{n}^{2}(p^{u}+p^{v})\big)^{2}\big|\bar{f}_{n+1}-\bar{f}_{n}\big|(\cdot;p^{u},l)\,\frac{dp^{u}}{p^{u}}ldl\le C(C_{1})R_{0}^{-1}u_{0}\mathfrak{D}_{n}\cdot\big(||\partial_{x^{a}}\bar{f}_{/}||_{\infty}+||\partial_{p^{a}}\bar{f}_{/}||_{\infty}\big).\label{eq:BoundDifferenceVlasovFields}
\end{equation}
Note that the boundedness of the right hand side of \ref{eq:BoundDifferenceVlasovFields}
follows from the assumption that $(r_{/},\Omega_{/}^{2},\bar{f}_{/})$
satisfy Condition 2 of Definition \ref{def:CompatibilityCondition}.

Provided $C_{1}>0$ has been fixed large in terms of the initial data
$(r_{/},\Omega_{/}^{2},\bar{f}_{/})$ and $u_{0}$ has been chosen
small enough in terms of $C_{1}$, the bounds (\ref{eq:BoundNextSepIteration}),
(\ref{eq:BoundNextSepIterationHigher}), (\ref{eq:BoundNextSepIterationHigher-1}),
(\ref{eq:BoundEnergyMomentumTensorsIteration}) and (\ref{eq:BoundUVDerivativesT})
readily yield (\ref{eq:UniformBoundPrevious}) for $k=n+1$. Similarly,
the bounds (\ref{eq:EstimateForDifferenceEasy}), (\ref{eq:DifferenceHigherOrderDerivativesInR})
and (\ref{eq:BoundDifferenceVlasovFields}) yield (\ref{eq:DifferenceBoundToShow}).

Having established that (\ref{eq:UniformBoundPrevious}) and (\ref{eq:DifferenceBoundToShow})
hold for all $n\ge2$, it readily follows (using standard arguments)
that $(r_{n},\Omega_{n}^{2};\bar{f}_{n})$ converges in the $C^{1}(\mathcal{D}_{0}^{u_{0}})\times C^{1}(\mathcal{D}_{0}^{u_{0}})\times C^{1}(\mathcal{D}_{0}^{u_{0}},L_{p,l}^{1})$
topology (where $L_{p,l}^{1}$ is a weighted $L^{1}$ norm in the
momentum variables for $\bar{f}$) to a solution $(r,\Omega^{2};f)$
of (\ref{eq:RequationFinal})\textendash (\ref{NullShellFinal}) on
$\mathcal{D}_{0}^{u_{0}}$ (the fact that the constraint equations
(\ref{eq:ConstraintVFinal})\textendash (\ref{eq:ConstraintUFinal})
are also satisfied follows readily from the fact that $(r_{/},\Omega_{/}^{2},\bar{f}_{/})$
was assumed to satisfy (\ref{eq:ConstraintVFinal})). In view of (\ref{eq:UniformBoundPrevious}),
$(r,\Omega^{2},f)$ satisfies the estimate 
\begin{equation}
||\log\Omega^{2}||_{C^{1,1}(\mathcal{D}_{0}^{u_{0}})}+||\log(\partial_{v}r)||_{C^{1,1}(\mathcal{D}_{0}^{u_{0}})}+||\log(-\partial_{u}r)||_{C^{1,1}(\mathcal{D}_{0}^{u_{0}})}+||T_{\alpha\beta}||_{C^{0,1}(\mathcal{D}_{0}^{u_{0}})}\le C_{1}.\label{eq:UniformBoundPrevious-1}
\end{equation}
Let us also remark that the bound (\ref{eq:BoundSuppF_n}) also implies
the following bound for the support of the Vlasov field $f$: 
\begin{equation}
supp(f)\subset\Big\{\Omega^{2}(p^{u}+p^{v})\le C_{0}\exp\Big(C(C_{1})R_{0}^{-1}u_{0}\Big)\Big\}\cap\Big\{\frac{l}{r}\le C_{0}\exp\Big(C(C_{1})R_{0}^{-1}u_{0}\Big)\Big\}.\label{eq:BoundSuppFNearAxis}
\end{equation}

Using the relations (\ref{eq:DerivativeTildeUMass}) and (\ref{eq:DerivativeTildeVMass})
for $\partial_{u}\tilde{m}$ and $\partial_{v}\tilde{m}$, respectively,
as well as the bound (\ref{eq:UniformBoundPrevious-1}), we can readily
estimate for any $(u,v)\in\mathcal{D}_{0}^{u_{0}}$: 
\begin{align}
\Big|\partial_{v}(\frac{\tilde{m}}{r^{3}})\Big|(u,v) & =8\pi\Big|\partial_{v}\Big(\frac{\int_{u}^{v}r^{2}\Omega^{-2}\big((-\partial_{u}r)T_{vv}+\partial_{v}rT_{uv}\big)(u,\bar{v})\,d\bar{v}}{r^{3}(u,v)}\Big)\Big|=\label{eq:DvDerivativeM.r^3-1}\\
 & =8\pi\Bigg|\frac{r^{2}\Omega^{-2}\big((-\partial_{u}r)T_{vv}+\partial_{v}rT_{uv}\big)(u,v)}{r^{3}(u,v)}-\nonumber \\
 & \hphantom{=8\pi\Bigg|}-3\frac{\int_{u}^{v}r^{2}\Omega^{-2}\big((-\partial_{u}r)T_{vv}+\partial_{v}rT_{uv}\big)(u,\bar{v})\,d\bar{v}}{r^{4}(u,v)}\partial_{v}r(u,v)\Bigg|=\nonumber \\
 & =8\pi\Bigg|\frac{r^{2}\Omega^{-2}\big((-\partial_{u}r)T_{vv}+\partial_{v}rT_{uv}\big)(u,v)}{r^{3}(u,v)}-\nonumber \\
 & \hphantom{=8\pi\Bigg|}-\frac{\int_{u}^{v}\partial_{v}(r^{3})\Omega^{-2}\big(\frac{(-\partial_{u}r)}{\partial_{v}r}T_{vv}+T_{uv}\big)(u,\bar{v})\,d\bar{v}}{r^{4}(u,v)}\partial_{v}r(u,v)\Bigg|=\nonumber \\
 & =8\pi\Bigg|\frac{\int_{u}^{v}r^{3}\partial_{v}\Big[\Omega^{-2}\big(\frac{(-\partial_{u}r)}{\partial_{v}r}T_{vv}+T_{uv}\big)\Big](u,\bar{v})\,d\bar{v}}{r^{4}(u,v)}\partial_{v}r(u,v)\Bigg|\le\nonumber \\
 & \le C(C_{1})\nonumber 
\end{align}
Similarly, 
\begin{equation}
\Big|\partial_{u}(\frac{\tilde{m}}{r^{3}})\Big|(u,v)\le C(C_{1}).\label{eq:DuDerivativeM.r^3-1}
\end{equation}

The boundary condition $r|_{\gamma_{0}^{u_{0}}}=0$ on the axis $\gamma_{0}^{u_{0}}$
implies that
\begin{equation}
(\partial_{u}+\partial_{v})r|_{\gamma_{0}^{u_{0}}}=0.\label{eq:ZeroTangentialDerivativeR-1}
\end{equation}
In view of (\ref{eq:ZeroTangentialDerivativeR-1}), the bound (\ref{eq:UniformBoundPrevious-1})
implies, through an application of the mean value theorem, that, for
any $(u,v)\in\mathcal{D}_{0}^{u_{0}}$: 
\begin{equation}
\Big|\Big|\frac{\partial_{u}r+\partial_{v}r}{r}\Big|\Big|_{C^{0,1}(\mathcal{D}_{0}^{u_{0}})}\le C(C_{1})\big(1+\sup_{\mathcal{D}_{0}^{u_{0}}}\big|\partial_{u,v}^{3}r|(u,v)\big)\le C(C_{1})\label{eq:FirstBoundForAxisTerm-1}
\end{equation}
and 
\begin{equation}
\sup_{\mathcal{D}_{0}^{u_{0}}}\Big|r^{-2}\Big(\frac{-4\partial_{u}r\partial_{v}r}{(\partial_{v}r-\partial_{u}r)^{2}}-1\Big)\Big|=\sup_{\mathcal{D}_{0}^{u_{0}}}\Bigg|\frac{1}{\big(\partial_{v}r-\partial_{u}r\big)^{2}}\Bigg(\frac{\partial_{v}r+\partial_{u}r}{r}\Bigg)^{2}\Bigg|\le C(C_{1}).\label{eq:SecondBoundAxisTerm-1}
\end{equation}
Combining (\ref{eq:UniformBoundPrevious-1}), (\ref{eq:DvDerivativeM.r^3-1}),
(\ref{eq:DuDerivativeM.r^3-1}), (\ref{eq:FirstBoundForAxisTerm-1})
and (\ref{eq:SecondBoundAxisTerm-1}), we infer that:
\begin{align}
||\log\Omega^{2}||_{C^{1,1}(\mathcal{D}_{0}^{u_{0}})}+||\log(\partial_{v}r)||_{C^{1,1}(\mathcal{D}_{0}^{u_{0}})}+||\log(-\partial_{u}r)||_{C^{1,1}(\mathcal{D}_{0}^{u_{0}})}+||T_{\alpha\beta}||_{C^{0,1}(\mathcal{D}_{0}^{u_{0}})} & +\label{eq:UniformBoundPreviousFinal}\\
+||\frac{\tilde{m}}{r^{3}}||_{C^{0,1}(\mathcal{D}_{0}^{u_{0}})}+\Big|\Big|\frac{\partial_{u}r+\partial_{v}r}{r}\Big|\Big|_{C^{0,1}(\mathcal{D}_{0}^{u_{0}})}+\Big|\Big|r^{-2}\Big(\frac{-4\partial_{u}r\partial_{v}r}{(\partial_{v}r-\partial_{u}r)^{2}}-1\Big)\Big|\Big|_{C^{0}(\mathcal{D}_{0}^{u_{0}})} & \le C(C_{1}).\nonumber 
\end{align}

Let us switch to the Cartesian coordinate system (\ref{eq:CartesianCoordinates})
on $\mathcal{N}_{u_{0}}=(\mathcal{D}_{0}^{u_{0}}\times\mathbb{S}^{2})\cup\mathcal{Z}$.
The expression (\ref{eq:CartesianMetric}) for the Cartesian components
$g_{\alpha\beta}$ of $g$ implies, in view of (\ref{eq:UniformBoundPreviousFinal}),
that 
\begin{equation}
||g_{\alpha\beta}||_{C^{1,1}(\mathcal{N}_{u_{0}})}<+\infty.\label{eq:LowerOrderRegularitycartesian}
\end{equation}
Commuting (\ref{eq:RequationFinal})\textendash (\ref{NullShellFinal})
with $\partial_{u}$, $\partial_{v}$ and arguing inductively in the
number of commutations, using also the fact that the initial data
set $(r_{/},\Omega_{/}^{2},\bar{f}_{/})$ satisfies Condition 1 of
Definition \ref{def:CompatibilityCondition}, we can readily infer
that, for any $k\in\mathbb{N}$, the Cartesian components $g_{\alpha\beta}$
in fact satisfy 
\[
||g_{\alpha\beta}||_{C^{k,1}(\mathcal{N}_{u_{0}})}<+\infty.
\]
Thus, using the fact that $f$ satisfies the Vlasov equation on the
background $(\mathcal{N}_{u_{0}},g)$, we readily infer the smoothness
of $(r,\Omega^{2},f)$ in accordance with Definition \ref{def:SmoothnessAxis}.

\medskip{}

\noindent \emph{Uniqueness.} We will establish that a smooth solution
$(r,\Omega^{2};f)$ of (\ref{eq:RequationFinal})\textendash (\ref{NullShellFinal})
on $\mathcal{D}_{0}^{u_{0}}$ with initial data (\ref{eq:InitialROmegaForExistence})\textendash (\ref{eq:InitialFForExistence})
on $u=0$ is unique through a contradiction argument. Let us assume
that $(r_{*},\Omega_{*}^{2};f_{*})$ is another smooth solution of
(\ref{eq:RequationFinal})\textendash (\ref{NullShellFinal}) on $\mathcal{D}_{0}^{u_{0}}$
with the same initial data. Since $(r,\Omega^{2};f)$, $(r_{*},\Omega_{*}^{2};f_{*})$
are smooth, there exists some $C_{1}>0$ so that the bound (\ref{eq:UniformBoundPrevious})
is satisfied with $(r,\Omega^{2};f)$, $(r_{*},\Omega_{*}^{2};f_{*})$
in place of $(r_{k},\Omega_{k}^{2};f_{k})$. Then, by repeating exactly
the same arguments that led to the proof of (\ref{eq:DifferenceBoundToShow})
(with equations (\ref{eq:EquationRForProof}) and (\ref{eq:EquationOmegaForProof})
in place of (\ref{eq:RecursiveR}) and (\ref{eq:RecursiveOmega})),
we infer that, provided $u_{0}$ is sufficiently small with respect
to $C_{1}$:
\begin{equation}
\mathfrak{D}\le\frac{1}{2}\mathfrak{D},\label{eq:DifferenceFOrUniqueness}
\end{equation}
where 
\begin{align}
\mathfrak{D}\doteq\sup_{\mathcal{D}_{0}^{u_{0}}}\Bigg\{ & \sum_{j=1}^{2}\sum_{j_{1}+j_{2}=j}R_{0}^{j-1}\partial_{u}^{j_{1}}\partial_{v}^{j_{2}}r-\partial_{u}^{j_{1}}\partial_{v}^{j_{2}}r_{*}\big|+\sum_{j=0}^{1}\sum_{j_{1}+j_{2}=j}R_{0}^{j}\big|\partial_{u}^{j_{1}}\partial_{v}^{j_{2}}\log\Omega^{2}-\partial_{u}^{j_{1}}\partial_{v}^{j_{2}}\log\Omega_{*}^{2}\big|\label{eq:DifferenceIterationDefinition-1}\\
 & +r^{-2}\int_{0}^{+\infty}\int_{0}^{+\infty}\big(\Omega^{2}(p^{u}+p^{v})\big)^{2}\big|\bar{f}-\bar{f}_{*}\big|(\cdot;p^{u},l)\,\frac{dp^{u}}{p^{u}}ldl\Bigg\}.\nonumber 
\end{align}
The bound (\ref{eq:DifferenceFOrUniqueness}) implies that $\mathfrak{D}=0$,
and thus $(r,\Omega^{2};f)$ and $(r_{*},\Omega_{*}^{2};f_{*})$ coincide.

\paragraph*{Step 2: The region $[0,u_{1}]\times[u_{0},v_{\mathcal{I}})$.}

For some $u_{1}\in(0,\frac{1}{2}u_{0}]$ to be determined later, let
$(r_{\backslash},\Omega_{\backslash}^{2},\bar{f}_{\backslash})$ be
the data induced by the solution $(\mathcal{D}_{0}^{u_{0}};r,\Omega^{2}f)$
(constructed in the previous step) on $\{v=u_{0}\}\cap\{0\le u\le u_{1}\}$,
i.\,e., for any $u\in[0,u_{1}]$: 
\begin{equation}
(r_{\backslash},\Omega_{\backslash}^{2})(u)=(r,\Omega^{2})(u,u_{0})\label{eq:InitialROmegaForExistence-2}
\end{equation}
 and 
\begin{equation}
\bar{f}_{\backslash}(u;p^{v},l)\cdot\delta\Big(\Omega_{\backslash}^{2}(u)p^{u}p^{v}-\frac{l^{2}}{r_{\backslash}(u)}\Big)=f(u,u_{0};p^{u},p^{v},l).\label{eq:InitialFForExistence-2}
\end{equation}
We will show that, provided $u_{1}$ is sufficiently small, there
exists a unique smooth solution $(r,\Omega^{2};f)$ of (\ref{eq:RequationFinal})\textendash (\ref{NullShellFinal})
on 
\[
\mathcal{W}\doteq[0,u_{1}]\times[u_{0},v_{\mathcal{I}})
\]
 satisfying (\ref{eq:InitialROmegaForExistence})\textendash (\ref{eq:InitialFForExistence})
on $\{0\}\times[u_{0},v_{\mathcal{I}})$ and (\ref{eq:InitialROmegaForExistence-2})\textendash (\ref{eq:InitialFForExistence-2})
on $[0,u_{1}]\times\{u_{0}\}$. 

\begin{figure}[h] 
\centering 
\scriptsize
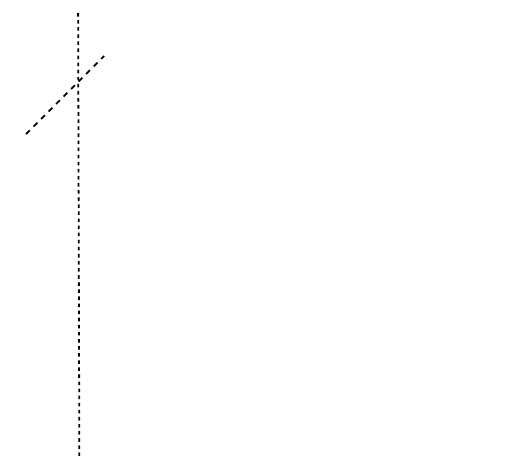 
\caption{Schematic depiction of the domain $\mathcal{W}$. \label{fig:MiddleExistenceDomain}}
\end{figure}
\begin{rem*}
As a consequence of the bound (\ref{eq:BoundSuppFNearAxis}), we infer
that 
\begin{equation}
supp(\bar{f}_{\backslash})\subset\Big\{\Omega_{\backslash}^{2}p^{v}+\frac{l^{2}}{r_{\backslash}^{2}p^{v}}\le2C_{0}\exp\Big(C(C_{1})R_{0}^{-1}u_{0}\Big)\Big\}\cap\Big\{\frac{l}{r_{\backslash}}\le C_{0}\exp\Big(C(C_{1})R_{0}^{-1}u_{0}\Big)\Big\}.\label{eq:BoundSuppFLeft}
\end{equation}
\end{rem*}
Our proof will be very similar to the one carried out in the previous
step. In particular, for any integer $n\in\mathbb{N}$, we will define
the functions $r_{n}:\mathcal{W}\rightarrow[0,+\infty)$, $\Omega_{n}^{2}:\mathcal{W}\rightarrow(0,+\infty)$
and $\bar{f}_{n}^{\prime}:\mathcal{W}\times[0,+\infty)^{2}\rightarrow[0,+\infty)$
recursively by the following conditions:

\begin{enumerate}

\item The functions $\Omega_{n}^{2}$, $r_{n}$ satisfy the following
system, which is a recursive analogue of the renormalised equations
(\ref{eq:RenormalisedEquations}) (see (\ref{eq:MuIterationU}) for
the definition of $\tilde{m}_{n}$)
\begin{align}
\partial_{u}\partial_{v}\log\Big(\frac{\Omega_{n+1}^{2}}{1-\frac{1}{3}\Lambda r_{n+1}^{2}}\Big) & =\frac{\tilde{m}_{n}}{r_{n}}\Big(\frac{1}{r_{n}^{2}}+\frac{1}{3}\Lambda\frac{\Lambda r_{n}^{2}-1}{1-\frac{1}{3}\Lambda r_{n}^{2}}\Big)\frac{\Omega_{n}^{2}}{1-\frac{1}{3}\Lambda r_{n}^{2}}-16\pi\frac{1-\frac{1}{2}\Lambda r_{n}^{2}}{1-\frac{1}{3}\Lambda r_{n}^{2}}(T_{uv})_{n},\label{eq:RenormalisedRecursiveOmegaEquation}\\
\partial_{u}\partial_{v}\Big(\tan^{-1}\big(\sqrt{-\frac{\Lambda}{3}}r_{n+1}\big)\Big) & =-\frac{1}{2}\sqrt{-\frac{\Lambda}{3}}\frac{\tilde{m}_{n}}{r_{n}^{2}}\frac{1-\frac{2}{3}\Lambda r_{n}^{2}}{1-\frac{1}{3}\Lambda r_{n}^{2}}\frac{\Omega_{n}^{2}}{1-\frac{1}{3}\Lambda r_{n}^{2}}+4\pi\sqrt{-\frac{\Lambda}{3}}\frac{r_{n}^{2}}{r_{n}-\frac{1}{3}\Lambda r_{n}^{3}}(T_{uv})_{n},\label{eq:RenormalisedRecursiveREquation}
\end{align}
with characteristic initial conditions 
\begin{equation}
(r_{n+1},\Omega_{n+1}^{2})|_{[0,u_{1}]\times\{u_{0}\}}=(r_{\backslash},\Omega_{\backslash}^{2})\label{eq:initialconditionsiterationLeft}
\end{equation}
and 
\begin{equation}
(r_{n+1},\Omega_{n+1}^{2})|_{\{0\}\times[u_{0},v_{\mathcal{I}})}=(r_{/},\Omega_{/}^{2}).\label{eq:initialconditionsiterationRight}
\end{equation}

\item   For any $l>0$, the function $\bar{f}^{\prime}$ satisfies
the following relations:

\begin{enumerate}

\item For any future directed, causal (with respect to the reference
metric (\ref{eq:ReferenceMetric})) curve $\gamma:[0,a)\rightarrow\mathcal{W}$
satisfying 
\begin{equation}
\Omega_{n}^{2}\dot{\gamma}^{u}\dot{\gamma}^{v}=\frac{l^{2}}{r_{n}^{2}}\label{eq:RecursiveNullShell}
\end{equation}
with $\gamma(0)\in\{0\}\times[u_{0},v_{\mathcal{I}})$ and solving
\begin{align}
\log\big(\Omega_{n}^{2}\dot{\gamma}^{u}\big)(s)-\log\big(\Omega_{n}^{2}\dot{\gamma}^{u}\big)(0)= & \int_{v(\gamma(0))}^{v(\gamma(s))}\int_{0}^{u(\gamma(s_{v}))}\Big(\frac{1}{2}\frac{\frac{6\tilde{m}_{n}}{r_{n}}-1}{r_{n}^{2}}\Omega_{n}^{2}-24\pi(T_{uv})_{n}\Big)\,dudv+\label{eq:RecursiveGammaFromRight}\\
 & +\int_{v(\gamma(0))}^{v(\gamma(s))}\big(\partial_{v}\log(\Omega_{/}^{2})-2\frac{\partial_{v}r_{/}}{r_{/}}\big)(v)\,dv\nonumber 
\end{align}
(where $s_{v}$ is defined by (\ref{eq:s_v})), $\bar{f}_{n+1}^{\prime}$
satisfies
\begin{equation}
\bar{f}_{n+1}^{\prime}(u(\gamma(s)),v(\gamma(s));\dot{\gamma}^{u}(s),l)=\bar{f}_{/}(v(\gamma(0));\dot{\gamma}^{u}(0),l)\label{eq:VlasovInTermsOfCharacteristicsRecursiveFromRight}
\end{equation}

\item For any future directed, causal (with respect to the reference
metric (\ref{eq:ReferenceMetric})) curve $\gamma:[0,a)\rightarrow\mathcal{W}$
satisfying (\ref{eq:RecursiveNullShell}) with $\gamma(0)\in[0,u_{1})\times\{u_{0}\}$
and solving
\begin{align}
\log\big(\Omega_{n}^{2}\dot{\gamma}^{v}\big)(s)-\log\big(\Omega_{n}^{2}\dot{\gamma}^{v}\big)(0)= & \int_{u(\gamma(0))}^{u(\gamma(s))}\int_{0}^{v(\gamma(s_{u}))}\Big(\frac{1}{2}\frac{\frac{6\tilde{m}_{n}}{r_{n}}-1}{r_{n}^{2}}\Omega_{n}^{2}-24\pi(T_{uv})_{n}\Big)\,dvdu+\label{eq:RecursiveGammaFromLeft}\\
 & +\int_{u(\gamma(0))}^{u(\gamma(s))}\big(\partial_{u}\log(\Omega_{\backslash}^{2})-2\frac{\partial_{u}r_{\backslash}}{r_{\backslash}}\big)(u)\,du\nonumber 
\end{align}
(where $s_{u}$ is defined by (\ref{eq:s_u})), the function $\bar{f}_{n+1}^{\prime}$
satisfies
\begin{equation}
\bar{f}_{n+1}^{\prime}(u(\gamma(s)),v(\gamma(s));\dot{\gamma}^{u}(s),l)=\bar{f}_{\backslash}(u(\gamma(0));\dot{\gamma}^{v}(0),l)\label{eq:VlasovInTermsOfCharacteristicsRecursiveFromLeft}
\end{equation}

\end{enumerate}

\end{enumerate}

In the above, the function $\tilde{m}_{n}$ is defined in terms of
$\Omega_{n}^{2}$, $r_{n}$ and $\bar{f}_{n}^{\prime}$ by the implicit
relation 
\begin{equation}
\tilde{m}_{n}(u,v)\doteq\tilde{m}_{/}(v)-8\pi\int_{0}^{u}r_{n-1}^{2}\Big(\frac{1-\frac{2\tilde{m}_{n}}{r_{n-1}}-\frac{1}{3}\Lambda r_{n-1}^{2}}{4\partial_{v}r_{n-1}}(T_{uv})_{n-1}+\Omega_{n-1}^{-2}\partial_{v}r_{n-1}\cdot(T_{uu})_{n-1}\Big)(\bar{u},v)\,du,\label{eq:MuIterationU}
\end{equation}
where $(T_{\mu\nu})_{n}$ is defined in terms of $\Omega_{n}^{2}$,
$r_{n}$ and $\bar{f}_{n}^{\prime}$ as in the previous step of the
proof. When $n=0$, we will adopt the convention that
\begin{equation}
(r_{0},\Omega_{0}^{2};f_{0})\doteq\big(r_{AdS}^{(v_{\mathcal{I}})}(0,\cdot),(\Omega_{/AdS}^{(v_{\mathcal{I}})})^{2}(0,\cdot);0\big)
\end{equation}
and 
\begin{equation}
\tilde{m}_{0}=0,
\end{equation}
where the rescaled AdS metric coefficients $r_{AdS}^{(v_{\mathcal{I}})},(\Omega_{AdS}^{(v_{\mathcal{I}})})^{2}$
are given by 
\begin{align}
r_{AdS}^{(v_{\mathcal{I}})}(u,v) & =r_{AdS}\Big(\sqrt{-\frac{3}{\Lambda}}\pi\frac{u}{v_{\mathcal{I}}},\sqrt{-\frac{3}{\Lambda}}\pi\frac{v}{v_{\mathcal{I}}}\Big),\label{eq:AdSMetricValuesRescaled-1}\\
\big(\Omega_{AdS}^{(v_{\mathcal{I}})}\big)^{2}(u,v) & =-\frac{3}{\Lambda}\frac{\pi^{2}}{v_{\mathcal{I}}^{2}}\Omega_{AdS}^{2}\Big(\sqrt{-\frac{3}{\Lambda}}\pi\frac{u}{v_{\mathcal{I}}},\sqrt{-\frac{3}{\Lambda}}\pi\frac{v}{v_{\mathcal{I}}}\Big).\nonumber 
\end{align}

Let $C_{1}\gg1$ be a large constant depending only on the initial
data $(r_{\backslash},\Omega_{\backslash}^{2},\bar{f}_{\backslash})$
and $(r_{/},\Omega_{/}^{2},\bar{f}_{/})$. We will establish the following
inductive bounds for $(r_{n},\Omega_{n}^{2};f_{n})$: Assuming that,
for any $0\le k\le n$,
\begin{align}
\sup_{\mathcal{W}}\Big(v_{\mathcal{I}}\Big|\partial_{v}\log\Big(\frac{\Omega_{k}^{2}}{1-\frac{1}{3}\Lambda r_{k}^{2}}\Big)\Big|+\Big|\log\Big(\frac{\Omega_{k}^{2}}{1-\frac{1}{3}\Lambda r_{k}^{2}}\Big)\Big|+\Big|\log\Big(\frac{\partial_{v}r_{k}}{1-\frac{1}{3}\Lambda r_{k}^{2}}\Big)\Big|+\label{eq:UniformBoundPrevious-2}\\
+\sqrt{-\Lambda}|\tilde{m}_{k}|++v_{\mathcal{I}}^{2}\big((T_{uu})_{k}+(T_{uv})_{k}+(T_{vv})_{k}\big) & \Big)<C_{1},\nonumber 
\end{align}
we will show that (\ref{eq:UniformBoundPrevious-2}) also holds for
$k=n+1$ and, moreover (in the case $n\ge2$): 
\begin{equation}
\mathfrak{D}_{n+1}^{\prime}\le\frac{1}{2}\mathfrak{D}_{n}^{\prime},\label{eq:DifferenceBoundToShow-1}
\end{equation}
where 
\begin{align}
\mathfrak{D}_{k}^{\prime}\doteq\sup_{\mathcal{W}}\Bigg\{ & v_{\mathcal{I}}\Big|\partial_{v}\log\Big(\frac{\Omega_{k}^{2}}{1-\frac{1}{3}\Lambda r_{k}^{2}}\Big)-\partial_{v}\log\Big(\frac{\Omega_{k-1}^{2}}{1-\frac{1}{3}\Lambda r_{k-1}^{2}}\Big)\Big|+\Big|\log\Big(\frac{\Omega_{k}^{2}}{1-\frac{1}{3}\Lambda r_{k}^{2}}\Big)-\log\Big(\frac{\Omega_{k-1}^{2}}{1-\frac{1}{3}\Lambda r_{k-1}^{2}}\Big)\Big|+\label{eq:DifferenceIterationDefinition-2}\\
 & +\Big|\log\Big(\frac{\partial_{v}r_{k}}{1-\frac{1}{3}\Lambda r_{k}^{2}}\Big)-\log\Big(\frac{\partial_{v}r_{k-1}}{1-\frac{1}{3}\Lambda r_{k-1}^{2}}\Big)\Big|+\sqrt{-\Lambda}\big|\tilde{m}_{k}-\tilde{m}_{k-1}\big|+\nonumber \\
 & +(-\Lambda)^{-2}r^{-2}\int_{0}^{+\infty}\int_{0}^{+\infty}\big(\Omega_{k-1}^{2}(p^{u}+p^{v})\big)^{2}\big|\bar{f}_{k}^{\prime}-\bar{f}_{k-1}^{\prime}\big|(\cdot;p^{u},l)\,\frac{dp^{u}}{p^{u}}ldl\Bigg\}.\nonumber 
\end{align}
\begin{rem*}
Notice that, in the definition (\ref{eq:MuIterationU}) of $\tilde{m}_{k}$
and the left hand sides of (\ref{eq:UniformBoundPrevious-2}), (\ref{eq:DifferenceIterationDefinition-2})
(\ref{eq:RecursiveGammaFromRight}) and (\ref{eq:RecursiveGammaFromLeft}),
only $\partial_{v}$ derivatives of $\Omega_{k}^{2}$, $r_{k}$ appear.
The reason that we took special care to arrange those expressions
in this way is a small parameter (necessary for the iteration procedure
to be successful) can only be obtained from the equations (\ref{eq:RenormalisedRecursiveOmegaEquation})\textendash (\ref{eq:RenormalisedRecursiveREquation})
by integrating in the $u$ direction.
\end{rem*}
Let $r_{min}$ be defined in terms of the initial data as 
\[
r_{min}\doteq r_{\backslash}(u_{1})>0.
\]
Notice that, because $\partial_{u}r_{\backslash}<0$ (following from
the properties of the solution $(r,\Omega^{2};f)$ on $\mathcal{D}_{0}^{u_{*}}$
constructed in the previous step), for any $k\in\mathbb{N}$ for which
$\inf_{\mathcal{W}}\partial_{v}r_{k}\ge0$ (which is necessarily true
if (\ref{eq:UniformBoundPrevious-2}) holds), then 
\begin{equation}
\inf_{\mathcal{W}}r_{k}=r_{min}>0.\label{eq:PositiveROnW}
\end{equation}
We will assume that $C_{1}$ in (\ref{eq:UniformBoundPrevious-2})
has been chosen large enough so that 
\begin{equation}
v_{\mathcal{I}}r_{min}^{-1}\ll C_{1}.\label{eq:LargeC1InTermsOfRmin}
\end{equation}

For any $n\ge0$, using the bound (\ref{eq:UniformBoundPrevious-2})
and (\ref{eq:PositiveROnW})\textendash (\ref{eq:LargeC1InTermsOfRmin})
for $k=n$ and integrating equations (\ref{eq:RenormalisedRecursiveOmegaEquation})\textendash (\ref{eq:RenormalisedRecursiveREquation})
in the $u$ direction, we immediately infer that 
\begin{align}
\sup_{\mathcal{W}}\Bigg\{ v_{\mathcal{I}}\Big|\partial_{v}\log\Big( & \frac{\Omega_{n+1}^{2}}{1-\frac{1}{3}\Lambda r_{n+1}^{2}}\Big)\Big|+\Big|\log\Big(\frac{\partial_{v}r_{n+1}}{1-\frac{1}{3}\Lambda r_{n+1}^{2}}\Big)\Big|\Bigg\}\le\label{eq:FirstBoundSecondIteration}\\
 & \le C(C_{1})u_{1}+\sup_{v\in[u_{0},v_{\mathcal{I}})}\Bigg\{ v_{\mathcal{I}}\Big|\partial_{v}\log\Big(\frac{\Omega_{/}^{2}}{1-\frac{1}{3}\Lambda r_{/}^{2}}\Big)\Big|+\Big|\log\Big(\frac{\partial_{v}r_{/}}{1-\frac{1}{3}\Lambda r_{/}^{2}}\Big)\Big|\Bigg\}.\nonumber 
\end{align}
Integrating in $v$ the bound for $\partial_{v}\log\Big(\frac{\Omega_{n+1}^{2}}{1-\frac{1}{3}\Lambda r_{n+1}^{2}}\Big)$
provided by (\ref{eq:FirstBoundSecondIteration}) and adding the resulting
estimate for $\log\Big(\frac{\Omega_{n+1}^{2}}{1-\frac{1}{3}\Lambda r_{n+1}^{2}}\Big)$
to (\ref{eq:FirstBoundSecondIteration}), we therefore infer (provided
$C_{1}$ has been chosen large enough in terms of $(r_{\backslash},\Omega_{\backslash}^{2})$
and $(r_{/},\Omega_{/}^{2})$): 
\begin{equation}
\sup_{\mathcal{W}}\Bigg\{ v_{\mathcal{I}}\Big|\partial_{v}\log\Big(\frac{\Omega_{n+1}^{2}}{1-\frac{1}{3}\Lambda r_{n+1}^{2}}\Big)\Big|+\Big|\log\Big(\frac{\Omega_{n+1}^{2}}{1-\frac{1}{3}\Lambda r_{n+1}^{2}}\Big)\Big|+\Big|\log\Big(\frac{\partial_{v}r_{n+1}}{1-\frac{1}{3}\Lambda r_{n+1}^{2}}\Big)\Big|\Bigg\}\le C(C_{1})v_{\mathcal{I}}^{-1}u_{1}+\frac{1}{10}C_{1}.\label{eq:SecondBoundSecondIteration}
\end{equation}

Similarly, for $n\ge1$, subtracting from (\ref{eq:RenormalisedRecursiveOmegaEquation})\textendash (\ref{eq:RenormalisedRecursiveREquation})
the same equations with $n-1$ in place of $n$ and using (\ref{eq:UniformBoundPrevious-2})
and (\ref{eq:PositiveROnW})\textendash (\ref{eq:LargeC1InTermsOfRmin})
for $k=n,n-1$, we can readily bound: 
\begin{align}
\sup_{\mathcal{W}}\Bigg\{ v_{\mathcal{I}}\Big|\partial_{v}\log\Big(\frac{\Omega_{n+1}^{2}}{1-\frac{1}{3}\Lambda r_{n+1}^{2}}\Big)-\partial_{v}\log\Big( & \frac{\Omega_{n}^{2}}{1-\frac{1}{3}\Lambda r_{n}^{2}}\Big)\Big|+\Big|\log\Big(\frac{\Omega_{n+1}^{2}}{1-\frac{1}{3}\Lambda r_{n+1}^{2}}\Big)-\log\Big(\frac{\Omega_{n}^{2}}{1-\frac{1}{3}\Lambda r_{n}^{2}}\Big)\Big|+\label{eq:FirstBoundSecondDifference}\\
 & +\Big|\log\Big(\frac{\partial_{v}r_{n+1}}{1-\frac{1}{3}\Lambda r_{n+1}^{2}}\Big)-\log\Big(\frac{\partial_{v}r_{n}}{1-\frac{1}{3}\Lambda r_{n}^{2}}\Big)\Big|\Bigg\}\le C(C_{1})v_{\mathcal{I}}^{-1}u_{1}\mathfrak{D}_{n}^{\prime}.\nonumber 
\end{align}

Integrating (\ref{eq:RenormalisedRecursiveOmegaEquation})\textendash (\ref{eq:RenormalisedRecursiveREquation})
in the $v$ direction, using the bounds (\ref{eq:UniformBoundPrevious-2})
and (\ref{eq:PositiveROnW})\textendash (\ref{eq:LargeC1InTermsOfRmin})
for $k=n$, we also infer the folowing useful estimates for the $\partial_{u}$
derivatives of $\Omega_{n+1}^{2}$, $r_{n+1}$: 
\begin{equation}
\sup_{\mathcal{W}}\Bigg\{ v_{\mathcal{I}}\Big|\partial_{u}\log\Big(\frac{\Omega_{n+1}^{2}}{1-\frac{1}{3}\Lambda r_{n+1}^{2}}\Big)\Big|+\Big|\frac{\partial_{u}r_{n+1}}{1-\frac{1}{3}\Lambda r_{n+1}^{2}}\Big|\Bigg\}\le C(C_{1}).\label{eq:BoundDUDerivativesRecursive}
\end{equation}

Using the bound (\ref{eq:UniformBoundPrevious-2}) for $k=n$, we
can readily infer from the relation (\ref{eq:MuIterationU}) for $\tilde{m}_{n+1}$
(using a simple application of Gronwall's inequality) that
\begin{equation}
\sup_{\mathcal{W}}\sqrt{-\Lambda}|\tilde{m}_{n+1}|\le\exp\big(C(C_{1})v_{\mathcal{I}}^{-1}u_{1}\big)\cdot\Big(C(C_{1})v_{\mathcal{I}}^{-1}u_{*}^{\prime}+\sup_{v\in[u_{0},v_{\mathcal{I}})}\sqrt{-\Lambda}\tilde{m}_{/}(v)\Big).\label{eq:BoundMuSecondIteration}
\end{equation}
Similarly, subtracting from (\ref{eq:MuIterationU}) for $\tilde{m}_{n+1}$
the same relation for $\tilde{m}_{n}$ and using (\ref{eq:UniformBoundPrevious-2})
for $k=n,n-1$, we can similarly estimate: 
\begin{equation}
\sup_{\mathcal{W}}\sqrt{-\Lambda}|\tilde{m}_{n+1}-\tilde{m}_{n}|\le C(C_{1})v_{\mathcal{I}}^{-1}u_{1}\mathfrak{D}_{n}^{\prime}.\label{BoundMuSecondDifference}
\end{equation}

Let $\gamma:[0,a)\rightarrow\mathcal{W}$ be a future directed, causal
curve (with respect to the reference metric (\ref{eq:ReferenceMetric}))
satisfying (\ref{eq:RecursiveNullShell}) and (\ref{eq:RecursiveGammaFromRight})
for some $l>0$, such that initially 
\begin{equation}
\gamma(0)\in\{0\}\times[u_{0},v_{\mathcal{I}}).\label{eq:GammaInitiallyRight}
\end{equation}
Using the bound (\ref{eq:UniformBoundPrevious-2}) (for $k=n$) for
the right hand side of (\ref{eq:RecursiveGammaFromRight}), as well
as the lower bound (\ref{eq:PositiveROnW}) for $r_{n}$, we trivially
infer from (\ref{eq:RecursiveNullShell}) and (\ref{eq:RecursiveGammaFromRight})
that (provided $C_{1}$ is sufficiently large in terms of $(r_{/},\Omega_{/}^{2}$)):
\begin{equation}
\frac{1}{C(C_{1})}\le\frac{\sup_{s\in[0,a)}\Big(\Omega_{n}^{2}(\gamma(s))\big(\dot{\gamma}^{u}+\dot{\gamma}^{v}\big)(s)\Big)}{\Omega_{n}^{2}(\gamma(0))\big(\dot{\gamma}^{u}+\dot{\gamma}^{v}\big)(0)}\le C(C_{1}).\label{eq:TrivialBoundEnergyCurveRecursive}
\end{equation}
Arguing in exactly the same way, we also infer that the bound (\ref{eq:TrivialBoundEnergyCurveRecursive})
also holds for future directed, causal curve $\gamma:[0,a)\rightarrow\mathcal{W}$
satisfying (\ref{eq:RecursiveNullShell}) and (\ref{eq:RecursiveGammaFromLeft})
for some $l>0$, such that initially 
\begin{equation}
\gamma(0)\in[0,u_{1})\times\{u_{0}\}.\label{eq:GammaInitiallyLeft}
\end{equation}

We will now proceed to obtain a more refined energy bound for curves
$\gamma$ as above which moreover satisfy a quantitative lower bound
on their angular momentum $l$. In particular, let $\gamma:[0,a)\rightarrow\mathcal{W}$
be a future directed, causal curve satisfying (\ref{eq:RecursiveNullShell}),
(\ref{eq:RecursiveGammaFromRight}) and (\ref{eq:GammaInitiallyRight}),
for some $l>0$ such that 
\begin{equation}
\frac{l}{\Omega_{n}^{2}(\gamma(0))\big(\dot{\gamma}^{u}+\dot{\gamma}^{v}\big)(0)}>v_{\mathcal{I}}\big(\frac{u_{1}}{v_{\mathcal{I}}}\big)^{\frac{1}{3}}.\label{eq:LowerBoundAngularMomentumIterarion}
\end{equation}
The relation (\ref{eq:GammaInitiallyRight}), combined with the bounds
(\ref{eq:UniformBoundPrevious-2}) and (\ref{eq:PositiveROnW}) (for
$k=n$) imply that 
\begin{equation}
\frac{\Omega_{n}^{2}(\gamma(0))\dot{\gamma}^{u}(0)\cdot\Omega_{n}^{2}(\gamma(0))\dot{\gamma}^{v}(0)}{\big(\Omega_{n}^{2}(\gamma(0))\big(\dot{\gamma}^{u}+\dot{\gamma}^{v}\big)(0)\big)^{2}}\ge c(C_{1})\big(\frac{u_{1}}{v_{\mathcal{I}}}\big)^{\frac{2}{3}},
\end{equation}
from which we readily infer that
\begin{equation}
\max\Big\{\frac{\dot{\gamma}^{u}}{\dot{\gamma}^{v}}(0),\frac{\dot{\gamma}^{v}}{\dot{\gamma}^{u}}(0)\Big\}\le C(C_{1})\big(\frac{u_{1}}{v_{\mathcal{I}}}\big)^{-\frac{1}{3}}.\label{eq:LowerBoundNotNullRecursive}
\end{equation}
We will show that:
\begin{equation}
\exp\big(-C(C_{1})\big(\frac{u_{1}}{v_{\mathcal{I}}}\big)^{\frac{2}{3}}\big)\le\frac{\sup_{s\in[0,a)}\Big(\Omega_{n}^{2}(\gamma(s))\big(\dot{\gamma}^{u}+\dot{\gamma}^{v}\big)(s)\Big)}{\Omega_{n}^{2}(\gamma(0))\big(\dot{\gamma}^{u}+\dot{\gamma}^{v}\big)(0)}\le\exp\big(C(C_{1})\big(\frac{u_{1}}{v_{\mathcal{I}}}\big)^{\frac{2}{3}}\big)\label{eq:RefinedEnergyBoundRecursiveGamma}
\end{equation}
and 
\begin{equation}
\sup_{s\in[0,a)}\Big(\max\Big\{\frac{\dot{\gamma}^{u}}{\dot{\gamma}^{v}}(s),\frac{\dot{\gamma}^{v}}{\dot{\gamma}^{u}}(s)\Big\}\Big)\le C(C_{1})\big(\frac{u_{1}}{v_{\mathcal{I}}}\big)^{-\frac{1}{3}}.\label{eq:NotTooNullForContinuity}
\end{equation}

\medskip{}

\noindent \emph{Proof of (\ref{eq:RefinedEnergyBoundRecursiveGamma})\textendash (\ref{eq:NotTooNullForContinuity}).
}We will establish (\ref{eq:RefinedEnergyBoundRecursiveGamma})\textendash (\ref{eq:NotTooNullForContinuity})
by continuity. Let $s_{c}\in(0,a]$ be defined as 
\begin{equation}
s_{c}\doteq\sup\Big\{ s\in[0,a):\text{ }\Big|\max\Big\{\frac{\dot{\gamma}^{u}}{\dot{\gamma}^{v}}(\bar{s}),\frac{\dot{\gamma}^{v}}{\dot{\gamma}^{u}}(\bar{s})\Big\}-\max\Big\{\frac{\dot{\gamma}^{u}}{\dot{\gamma}^{v}}(0),\frac{\dot{\gamma}^{v}}{\dot{\gamma}^{u}}(0)\Big\}\Big|\le1\text{ for all }\bar{s}\in[0,s]\Big\}.\label{eq:S_cRecursive}
\end{equation}
Note that $s_{c}>0$, in view of (\ref{eq:LowerBoundNotNullRecursive}).
We will show that (\ref{eq:RefinedEnergyBoundRecursiveGamma})\textendash (\ref{eq:NotTooNullForContinuity})
hold when the $\sup$ in (\ref{eq:RefinedEnergyBoundRecursiveGamma})
and (\ref{eq:NotTooNullForContinuity}) is considerd over $[0,s_{c})$
and that, moreover, 
\begin{equation}
\text{ }\sup_{s\in[0,s_{c})}\Big|\max\Big\{\frac{\dot{\gamma}^{u}}{\dot{\gamma}^{v}}(s),\frac{\dot{\gamma}^{v}}{\dot{\gamma}^{u}}(s)\Big\}-\max\Big\{\frac{\dot{\gamma}^{u}}{\dot{\gamma}^{v}}(0),\frac{\dot{\gamma}^{v}}{\dot{\gamma}^{u}}(0)\Big\}\Big|\le\frac{1}{2}.\label{eq:ToShowForSmallBootstrap}
\end{equation}
 In particular, (\ref{eq:ToShowForSmallBootstrap}) will imply by
continuity that 
\[
\Big|\max\Big\{\frac{\dot{\gamma}^{u}}{\dot{\gamma}^{v}}(s),\frac{\dot{\gamma}^{v}}{\dot{\gamma}^{u}}(s)\Big\}-\max\Big\{\frac{\dot{\gamma}^{u}}{\dot{\gamma}^{v}}(0),\frac{\dot{\gamma}^{v}}{\dot{\gamma}^{u}}(0)\Big\}\Big|\le1.
\]
for all $s\in[0,\min\{a,s_{c}+\delta\})$ for some small $\delta>0$,
which yields a contradiction in view of the definition of $s_{c}$,
unless $s_{c}=a$.

The relation (\ref{eq:RecursiveGammaFromRight}) for $\dot{\gamma}^{u}$
yields, in view of (\ref{eq:UniformBoundPrevious-2}) and (\ref{eq:PositiveROnW})
for $k=n$, that for any $s\in[0,s_{c})$: 
\begin{align}
\Big|\log\big(\Omega_{n}^{2}\dot{\gamma}^{u}\big)(s)-\log\big(\Omega_{n}^{2}\dot{\gamma}^{u}\big)(0)\Big| & \le C(C_{1})v_{\mathcal{I}}^{-1}u_{1}+C(C_{1})v_{\mathcal{I}}^{-1}\big|v(\gamma(s))-v(\gamma(0))\big|\le\label{eq:OneMoreUsefulBound}\\
 & \le C(C_{1})v_{\mathcal{I}}^{-1}u_{1}+C(C_{1})v_{\mathcal{I}}^{-1}\int_{0}^{s_{c}}\dot{\gamma}^{v}(\bar{s})\,d\bar{s}\le\nonumber \\
 & \le C(C_{1})v_{\mathcal{I}}^{-1}u_{1}+C(C_{1})v_{\mathcal{I}}^{-1}\int_{u(\gamma(0))}^{u(\gamma(s_{c}))}\frac{\dot{\gamma}^{v}(s_{u})}{\dot{\gamma}^{u}(s_{u})}\,du,\nonumber 
\end{align}
where $s_{u}$ is defined by (\ref{eq:s_u}). Using the upper bound
for $\frac{\dot{\gamma}^{v}(s_{u})}{\dot{\gamma}^{u}(s_{u})}$ provided
by the definition (\ref{eq:S_cRecursive}) of $s_{c}$, we therefore
infer from (\ref{eq:OneMoreUsefulBound}) that:
\begin{align}
\sup_{s\in[0,s_{c})}\Big|\log\big(\Omega_{n}^{2}\dot{\gamma}^{u}\big)(s)-\log\big(\Omega_{n}^{2}\dot{\gamma}^{u}\big)(0)\Big| & \le C(C_{1})v_{\mathcal{I}}^{-1}u_{1}+C(C_{1})v_{\mathcal{I}}^{-1}u_{*}^{\prime}\cdot\big(\frac{u_{1}}{v_{\mathcal{I}}}\big)^{-\frac{1}{3}}\le\label{eq:UsefulBoundChangeOfLogU}\\
 & \le C(C_{1})\big(\frac{u_{1}}{v_{\mathcal{I}}}\big)^{\frac{2}{3}}.\nonumber 
\end{align}
From (\ref{eq:UsefulBoundChangeOfLogU}), the relation (\ref{eq:RecursiveNullShell})
between $\dot{\gamma}^{v}$ and $\dot{\gamma}^{u}$, the bound (\ref{eq:UniformBoundPrevious-2})
for $k=n$, as well as the bound (\ref{eq:BoundDUDerivativesRecursive})
(for $n$ in place of $n+1$) on the $\partial_{u}$ derivatives of
$\Omega_{n}^{2}$, $r_{n}$ and the upper bound for $\frac{\dot{\gamma}^{v}(s_{u})}{\dot{\gamma}^{u}(s_{u})}$
provided by the definition (\ref{eq:S_cRecursive}) of $s_{c}$, we
also obtain that 
\begin{align}
\sup_{s\in[0,s_{c})}\Big| & \log\big(\Omega_{n}^{2}\dot{\gamma}^{v}\big)(s)-\log\big(\Omega_{n}^{2}\dot{\gamma}^{v}\big)(0)\Big|\le\label{eq:UsefulBoundChangeOfLogV}\\
 & \le\sup_{s\in[0,s_{c})}\Big|\log\big(\Omega_{n}^{2}\dot{\gamma}^{u}\big)(s)-\log\big(\Omega_{n}^{2}\dot{\gamma}^{u}\big)(0)\Big|+\sup_{s\in[0,s_{c})}\Big|\log\Big(\frac{l^{2}\Omega_{n}^{2}}{r_{n}^{2}}\Big)\big|_{\gamma(s)}-\log\Big(\frac{l^{2}\Omega_{n}^{2}}{r_{n}^{2}}\Big)\big|_{\gamma(0)}\Big|\le\nonumber \\
 & \le C(C_{1})\big(\frac{u_{1}}{v_{\mathcal{I}}}\big)^{\frac{2}{3}}+C(C_{1})\frac{|v(\gamma(s))-v(\gamma(0))|+|u(\gamma(s))-u(\gamma(0))|}{v_{\mathcal{I}}}\le\nonumber \\
 & \le C(C_{1})\big(\frac{u_{1}}{v_{\mathcal{I}}}\big)^{\frac{2}{3}}+C(C_{1})v_{\mathcal{I}}^{-1}\Big(\int_{u(\gamma(0))}^{u(\gamma(s))}\frac{\dot{\gamma}^{v}(s_{u})}{\dot{\gamma}^{u}(s_{u})}\,du+u_{1}\Big)\le\nonumber \\
 & \le C(C_{1})\big(\frac{u_{1}}{v_{\mathcal{I}}}\big)^{\frac{2}{3}}.\nonumber 
\end{align}

Combining the bounds (\ref{eq:UsefulBoundChangeOfLogU}) and (\ref{eq:UsefulBoundChangeOfLogV}),
we readily infer that 
\begin{equation}
\sup_{s\in[0,s_{c})}\Big|\log\Big(\frac{\Omega_{n}^{2}\dot{\gamma}^{u}(s)+\Omega_{n}^{2}\dot{\gamma}^{v}(s)}{\Omega_{n}^{2}\dot{\gamma}^{u}(0)+\Omega_{n}^{2}\dot{\gamma}^{v}(0)}\Big)\Big|\le C(C_{1})\big(\frac{u_{1}}{v_{\mathcal{I}}}\big)^{\frac{2}{3}}
\end{equation}
and 
\begin{equation}
\sup_{s\in[0,s_{c})}\Big|\log\Big(\frac{\Omega_{n}^{2}\dot{\gamma}^{u}(s)}{\Omega_{n}^{2}\dot{\gamma}^{v}(s)}\Big)-\log\Big(\frac{\Omega_{n}^{2}\dot{\gamma}^{u}(0)}{\Omega_{n}^{2}\dot{\gamma}^{v}(0)}\Big)\Big|\le C(C_{1})\big(\frac{u_{1}}{v_{\mathcal{I}}}\big)^{\frac{2}{3}},
\end{equation}
from which it follows that (\ref{eq:RefinedEnergyBoundRecursiveGamma})\textendash (\ref{eq:NotTooNullForContinuity})
(with the $\sup$ considerd over $[0,s_{c})$) and (\ref{eq:ToShowForSmallBootstrap})
hold. 

\medskip{}

In exactly the same way, it can be shown that the bounds (\ref{eq:RefinedEnergyBoundRecursiveGamma})
and (\ref{eq:NotTooNullForContinuity}) also hold for any future directed,
causal curve $\gamma:[0,a)\rightarrow\mathcal{W}$ satisfying (\ref{eq:RecursiveNullShell}),
(\ref{eq:RecursiveGammaFromLeft}) and (\ref{eq:GammaInitiallyLeft}),
for some $l>0$ such that (\ref{eq:LowerBoundAngularMomentumIterarion})
holds. 

At any point $(u,v)\in\mathcal{W}$, we can decompose the components
$(T_{\mu\nu})_{n+1}$ of the energy momentum tensor $T_{n+1}$ as
\[
(T_{\mu\nu})_{n+1}(u,v)=(T_{\mu\nu})_{n+1}^{+}(u,v)+(T_{\mu\nu})_{n+1}^{0}(u,v)
\]
with
\begin{equation}
(T_{\mu\nu})_{n+1}^{+}(u,v)\doteq2\pi r_{n}^{-2}\int_{l[u_{1}]}^{+\infty}\int_{0}^{+\infty}p_{\mu}p_{\nu}\bar{f}_{n+1}^{\prime}(u,v;p^{u},p^{v},l)\,\frac{dp^{u}}{p^{u}}ldl\label{eq:MainContributionEnergyMomentum}
\end{equation}
and 
\begin{equation}
(T_{\mu\nu})_{n+1}^{0}(u,v)\doteq2\pi r_{n}^{-2}\int_{0}^{l[u_{1}]}\int_{0}^{+\infty}p_{\mu}p_{\nu}\bar{f}_{n+1}^{\prime}(u,v;p^{u},l)\,\frac{dp^{u}}{p^{u}}ldl,\label{eq:AlmostNullContributionEnergyMomentum}
\end{equation}
where $l[u_{1}]$ is defined in terms of $u_{1}$ by 
\begin{equation}
l[u_{1}]\doteq2v_{\mathcal{I}}\big(\frac{u_{1}}{v_{\mathcal{I}}}\big)^{\frac{1}{6}}C_{0}\label{eq:LowerBoundAngularMomentumCutOff}
\end{equation}
with $C_{0}$ being the constant appearing in the right hand side
of (\ref{eq:BoundSuppFLeft}) (which is also the constant for which
$(r_{/},\Omega_{/}^{2},\bar{f}_{/})$ were assumed to satisfy (\ref{eq:BoundedSupportDefinition}))
and 
\begin{equation}
p_{\mu}=(g_{\mu\nu})_{n}p^{\nu}.
\end{equation}

As a consequence of the transport relations for (\ref{eq:VlasovInTermsOfCharacteristicsRecursiveFromRight})
and (\ref{eq:VlasovInTermsOfCharacteristicsRecursiveFromLeft}) for
$f_{n+1}$, the bound (\ref{eq:TrivialBoundEnergyCurveRecursive})
for the characteristic curves $\gamma$ on which $f_{n+1}$ is conserved,
combined with the bounds (\ref{eq:BoundSuppFLeft}) and (\ref{eq:BoundedSupportDefinition})
(with $C_{0}$ in plaace of $C$) on the initial data, readily imply
that
\begin{equation}
(T_{uu})_{n+1}^{0}+(T_{uv})_{n+1}^{0}+(T_{vv})_{n+1}^{0}\le C(C_{1})C_{0}^{2}\big(||f_{/}||_{\infty}+||f_{\backslash}||_{\infty}\big)\cdot(l[u_{1}])^{2}.\label{eq:T_0BoundSecondIteration}
\end{equation}
On the other hand, the more refined estimate (\ref{eq:RefinedEnergyBoundRecursiveGamma})
for the characteristic curves $\gamma$ which satisfy the additional
assumption (\ref{eq:LowerBoundAngularMomentumIterarion}) (which implies
that $l>l[u_{1}]$ when $u_{1}$ is sufficiently small in terms of
$C_{1}$, $C_{0}$) yields: 
\begin{equation}
(T_{uu})_{n+1}^{+}+(T_{uv})_{n+1}^{+}+(T_{vv})_{n+1}^{+}\le3||T_{data}||_{\infty}\exp\big(C(C_{1})\big(\frac{u_{1}}{v_{\mathcal{I}}}\big)^{\frac{2}{3}}\big),\label{eq:UpperBoundEnergyMomentumRecursive}
\end{equation}
where 
\begin{align*}
||T_{data}||_{\infty}\doteq & ||(T_{uu})_{/}||_{\infty}+||(T_{uv})_{/}||_{\infty}+||(T_{vv})_{/}||_{\infty}+\\
 & +||(T_{uu})_{\backslash}||_{\infty}+||(T_{uv})_{\backslash}||_{\infty}+||(T_{vv})_{\backslash}||_{\infty}.
\end{align*}

Let $\gamma_{n+1},\gamma_{n}:[0,a)\rightarrow\mathcal{W}$ be two
future directed, causal curves such that $\gamma_{n+1}$ satisfies
(\ref{eq:RecursiveGammaFromRight}) and (\ref{eq:RecursiveNullShell}),
$\gamma_{n}$ satisfies (\ref{eq:RecursiveGammaFromRight}) and (\ref{eq:RecursiveNullShell})
with $n-1$ in place of $n$, and 
\begin{align}
\gamma_{n+1}(0) & =\gamma_{n}(0)\in\{0\}\times[u_{0},v_{\mathcal{I}}),\label{eq:SameInitialDataCurves}\\
\dot{\gamma}_{n+1}(0) & =\dot{\gamma}_{n}(0).\nonumber 
\end{align}
Let us also consider the parametrization of $\gamma_{n+1},\gamma_{n}$
by 
\[
\tau=u+v-u(\gamma_{n+1}(0))-v(\gamma_{n+1}(0)),
\]
with corresponding parameter domains $[0,\tau_{n+1})$ and $[0,\tau_{n})$.
\begin{rem*}
We will only denote by $\dot{}$ differentiation with respect to $s$.
\end{rem*}
Subtracting from (\ref{eq:RecursiveGammaFromRight}) for $\gamma_{n+1}$
the same equation (with $n-1$ in place of $n$) for $\gamma_{n+1}$
and similarly for (\ref{eq:RecursiveNullShell}), we can readily estimate
using (\ref{eq:UniformBoundPrevious-2}) (and (\ref{eq:PositiveROnW}))
for $k=n,n-1$ as well as the initial conditions (\ref{eq:SameInitialDataCurves})
that, for any $\tau\in[0,\min\{\tau_{n+1},\tau_{n}\})$: 
\begin{align}
|u(\gamma_{n+1}(\tau))-u(\gamma_{n}(\tau))|+|v(\gamma_{n+1}(\tau))-v(\gamma_{n}(\tau))| & \le\label{eq:ForGronwallDifferenceGmmaRecursive}\\
\le\int_{0}^{\tau}\Big(|\Omega_{n}^{2}\dot{\gamma}_{n+1}^{u}(\bar{\tau})-\Omega_{n-1}^{2}\dot{\gamma}_{n}^{u}(\bar{\tau})| & +|\Omega_{n}^{2}\dot{\gamma}_{n+1}^{v}(\bar{\tau})-\Omega_{n-1}^{2}\dot{\gamma}_{n}^{v}(\bar{\tau})|\Big)\,d\bar{\tau},\nonumber \\
|\Omega_{n}^{2}\dot{\gamma}_{n+1}^{v}(\tau)-\Omega_{n-1}^{2}\dot{\gamma}_{n}^{v}(\tau)|+|\Omega_{n}^{2}\dot{\gamma}_{n+1}^{u}(\tau)-\Omega_{n-1}^{2}\dot{\gamma}_{n}^{u}(\tau)| & \le\nonumber \\
\le C(C_{1})v_{\mathcal{I}}^{-1}u_{1}+C(C_{1})v_{\mathcal{I}}^{-2}\int_{0}^{\tau}\Big(|u( & \gamma_{n+1}(\tau))-u(\gamma_{n}(\tau))|+|v(\gamma_{n+1}(\tau))-v(\gamma_{n}(\tau))|\Big)\,d\bar{\tau}.\nonumber 
\end{align}
Applying Gronwall's inequality on (\ref{eq:ForGronwallDifferenceGmmaRecursive}),
we therefore infer that, for any $\tau\in[0,\min\{\tau_{n+1},\tau_{n}\})$:
\begin{align}
v_{\mathcal{I}}^{-1}|u(\gamma_{n+1}(\tau))-u(\gamma_{n}(\tau))|+v_{\mathcal{I}}^{-1}|v(\gamma_{n+1}(\tau))-v(\gamma_{n}(\tau))|+\label{eq:DifferenceGeodesicsControl}\\
+|\Omega_{n}^{2}\dot{\gamma}_{n+1}^{v}(\tau)-\Omega_{n-1}^{2}\dot{\gamma}_{n}^{v}(\tau)|+|\Omega_{n}^{2} & \dot{\gamma}_{n+1}^{u}(\tau)-\Omega_{n-1}^{2}\dot{\gamma}_{n}^{u}(\tau)|\le C(C_{1})v_{\mathcal{I}}^{-1}u_{1}.\nonumber 
\end{align}
In exactly the same way, the estimate (\ref{eq:DifferenceGeodesicsControl})
can be also established in the case when $\gamma_{n+1}$ and $\gamma_{n}$
satisfy (\ref{eq:RecursiveGammaFromLeft}) in place of (\ref{eq:RecursiveGammaFromRight})
(with $n-1$ in place of $n$ in the case of $\gamma_{n}$) and 
\begin{align}
\gamma_{n+1}(0) & =\gamma_{n}(0)\in[0,u_{1})\times\{u_{0}\},\label{eq:SameInitialDataCurves-1}\\
\dot{\gamma}_{n+1}(0) & =\dot{\gamma}_{n}(0).\nonumber 
\end{align}

As a consequence of the estimate (\ref{eq:ForGronwallDifferenceGmmaRecursive})
and the bounds (\ref{eq:UniformBoundPrevious-2}) (for $k=n,n-1$),
(\ref{eq:BoundedSupportDefinition}) and (\ref{eq:BoundSuppFLeft}),
the transport relations (\ref{eq:VlasovInTermsOfCharacteristicsRecursiveFromRight})
and (\ref{eq:VlasovInTermsOfCharacteristicsRecursiveFromLeft}) for
$\bar{f}_{n+1}^{\prime}$, $\bar{f}_{n}^{\prime}$ readily imply that,
provided $C_{1}$ has been chosen large enough in terms of the initial
data (and in particular in terms of $\partial_{x^{a}}f_{/}$, $\partial_{p^{a}}f_{/}$
and $\partial_{x^{a}}f_{\backslash}$, $\partial_{p^{a}}f_{\backslash}$):
\begin{equation}
(-\Lambda)^{-2}r^{-2}\int_{0}^{+\infty}\int_{0}^{+\infty}\big(\Omega_{k-1}^{2}(p^{u}+p^{v})\big)^{2}\big|\bar{f}_{n+1}^{\prime}-\bar{f}_{n}^{\prime}\big|(\cdot;p^{u},l)\,\frac{dp^{u}}{p^{u}}ldl\le C(C_{1})()v_{\mathcal{I}}^{-1}u_{1}\mathfrak{D}_{n}^{\prime}.\label{eq:DifferenceFRecursive}
\end{equation}

The bound (\ref{eq:UniformBoundPrevious-2}) for $k=n+1$ now readily
follows from (\ref{eq:SecondBoundSecondIteration}), (\ref{eq:BoundMuSecondIteration}),
(\ref{eq:T_0BoundSecondIteration}) and (\ref{eq:UpperBoundEnergyMomentumRecursive}),
provided $u_{1}$ has been chosen small enough and $C_{1}$ large
enough, both depending only on the initial data $(r_{\backslash},\Omega_{\backslash}^{2},\bar{f}_{\backslash})$
and $(r_{/},\Omega_{/}^{2},\bar{f}_{/})$. The bound (\ref{eq:DifferenceBoundToShow-1})
for $k=n+1$ follows from (\ref{eq:FirstBoundSecondDifference}),
(\ref{BoundMuSecondDifference}) and (\ref{eq:DifferenceFRecursive}).
Therefore, by induction, (\ref{eq:UniformBoundPrevious-2}) and (\ref{eq:DifferenceBoundToShow-1})
hold for all $k\in\mathbb{N}$.

Having established tha the bounds (\ref{eq:UniformBoundPrevious-2})
and (\ref{eq:DifferenceBoundToShow-1}) hold for all $k\in\mathbb{N}$,
we infer that $(r_{k},\Omega_{k}^{2},\bar{f}_{k})$ converge in the
topology defined by the right hand side of (\ref{eq:DifferenceIterationDefinition-2})
to some limit functions $(r,\Omega^{2},\bar{f})$, for which the associated
norm defined by the left hand side of (\ref{eq:UniformBoundPrevious-2})
is finite. In particular, (\ref{eq:RecursiveNullShell}), (\ref{eq:RecursiveGammaFromRight}),
(\ref{eq:VlasovInTermsOfCharacteristicsRecursiveFromRight}), (\ref{eq:RecursiveGammaFromLeft})
and (\ref{eq:VlasovInTermsOfCharacteristicsRecursiveFromLeft}) imply
that 
\[
f(u,v;p^{u},p^{v},l)\doteq\bar{f}(u,v;p^{u},l)\cdot\delta\Big(\Omega^{2}(u,v)p^{u}p^{v}-\frac{l^{2}}{r^{2}(u,v)}\Big)
\]
 is a massless Vlasov field for the metric 
\[
g=-\Omega^{2}dudv+r^{2}g_{\mathbb{S}^{2}}
\]
 on $\mathcal{W}\times\mathbb{S}^{2}$, since it is transported along
its geodesic flow. As a consequence of (\ref{eq:RenormalisedRecursiveOmegaEquation})\textendash (\ref{eq:RenormalisedRecursiveREquation})
and the fact that (\ref{eq:DifferenceIterationDefinition-2}) controls
the $C^{0}$ norm of $r$, $\Omega^{2}$ and $T_{\mu\nu}$, we infer
that $(r,\Omega^{2},f)$ is a $C^{0}$ distributional solution of
(\ref{eq:RequationFinal})\textendash (\ref{NullShellFinal}). Commuting
(\ref{eq:RequationFinal})\textendash (\ref{NullShellFinal}) with
$\partial_{u}$, $\partial_{v}$ and arguing inductively treating
the equations as linear in the highest order terms (using the estimates
provided by (\ref{eq:UniformBoundPrevious-2}) for the lower order
terms at the first step), we infer the higher order regularity of
$(r,\Omega^{2};f)$; we will omit the details of this standard procedure. 

The uniqueness of the solution $(r,\Omega^{2},f)$ on $\mathcal{W}$
follows as in the previous step, by repeating the arguments leading
to the proof of the difference estimate (\ref{eq:DifferenceBoundToShow-1});
we will omit the details.

\paragraph*{Step 3: The region $\{u\le u_{2}\}\cap\{v_{\mathcal{I}}\le v<u+v_{\mathcal{I}}\}$}

By a slight abuse of notation, let us denote at this step by $(r_{\backslash},\Omega_{\backslash}^{2};\bar{f}_{\backslash})$
the characteristic initial data induced by the solution $(r,\Omega^{2};f)$
on $\mathcal{W}$ (constructed in the previous step) on $\{v=v_{\mathcal{I}}\}\cap\{0<u<u_{1}\}$.
We will show that, for some $u_{2}\in(0,u_{1}]$ sufficiently small
in terms of $(r_{\backslash},\Omega_{\backslash}^{2};\bar{f}_{\backslash})$,
there exists a unique smooth solution $(r,\Omega^{2};f)$ of (\ref{eq:RequationFinal})\textendash (\ref{NullShellFinal})
on 
\[
\mathcal{V}\doteq\{u\le u_{2}\}\cap\{v_{\mathcal{I}}\le v<u+v_{\mathcal{I}}\}
\]
(see Figure \ref{fig:InfinityExistenceDomain}) with characteristic
initial data $(r_{\backslash},\Omega_{\backslash}^{2};\bar{f}_{\backslash})$
on $(0,u_{2})\times\{v_{\mathcal{I}}\}$ such that $r$ satisfies
the gauge condition 
\[
\frac{1}{r}\Big|_{u=v-v_{\mathcal{I}}}=0
\]
 and $f$ satisfies the reflecting boundary condition stated in Definition
\ref{def:ReflectingBoundaryCondition}.

\begin{figure}[h] 
\centering 
\scriptsize
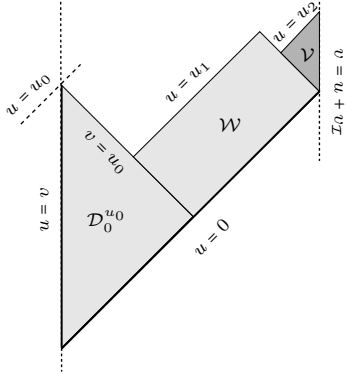 
\caption{Schematic depiction of the domain $\mathcal{V}$. \label{fig:InfinityExistenceDomain}}
\end{figure}

The proof proceeds by repeating essentially the same arguments as
in the previous step (using again an iteration scheme for the renormalised
equations \ref{eq:RenormalisedEquations} for $\Omega^{2}/(1-\frac{1}{3}\Lambda r^{2})$
and $\tan^{-1}\Big(\sqrt{-\frac{\Lambda}{3}}r\Big)$ instead of the
standard equations \ref{eq:RequationFinal}\textendash \ref{eq:OmegaEquationFinal};
the propagation of the constraint equations (\ref{eq:ConstraintVFinal})\textendash (\ref{eq:ConstraintUFinal})
under the reflecting boundary condition for $f$ on $\mathcal{I}$
is also inferred readily, using the bound on the support of $f$ in
phase space). We will therefore omit the relevant details.

\medskip{}

Let us define the solution $(r,\Omega^{2},f)$ of (\ref{eq:RequationFinal})\textendash (\ref{NullShellFinal})
on the domain $\mathcal{U}_{u_{*},v_{\mathcal{I}}}$ for $u_{*}=u_{2}$,
so that $(r,\Omega^{2},f)$ coincides with the solutions constructed
in the previous three steps on $\mathcal{D}_{0}^{u_{0}}\cap\mathcal{U}_{u_{2},v_{\mathcal{I}}}$,
$\mathcal{W}\cap\mathcal{U}_{u_{2},v_{\mathcal{I}}}$ and $\mathcal{V}$,
respectively. In order to conclude the proof of Theorem \ref{thm:LocalExistenceUniqueness},
it only remains to verify that $(r,\Omega^{2},f)$ is smooth across
the ``gluing'' boundaries $\{v=u_{0}\}\cap\mathcal{U}_{u_{*},v_{\mathcal{I}}}$
and $\{v=v_{\mathcal{I}}\}\cap\mathcal{U}_{u_{*},v_{\mathcal{I}}}$;
since $f$ solves the Vlasov equation in terms of $r,\Omega^{2}$,
it suffices to establish the smoothness of $r,\Omega^{2}$.

\begin{itemize}

\item Along $\{v=u_{0}\}\cap\mathcal{U}_{u_{*},v_{\mathcal{I}}}$,
the fact that 
\[
\lim_{\bar{v}\rightarrow u_{0}^{-}}\partial_{v}^{k}r|_{v=\bar{v}}=\lim_{\bar{v}\rightarrow u_{0}^{+}}\partial_{v}^{k}r|_{v=\bar{v}}\text{ and }\lim_{\bar{v}\rightarrow u_{0}^{-}}\partial_{v}^{k}\Omega^{2}|_{v=\bar{v}}=\lim_{\bar{v}\rightarrow u_{0}^{+}}\partial_{v}^{k}\Omega^{2}|_{v=\bar{v}}
\]
 follows by arguing inductively on $k$ and integrating equations
(\ref{eq:RequationFinal}) and (\ref{eq:OmegaEquationFinal}) in $u$
(after differentiating sufficiently many times with respect to $\partial_{v}$),
using also the smoothness of the initial data set $(r_{/},\Omega_{/}^{2},\bar{f}_{/};v_{\mathcal{I}})$
at $v=u_{0}$.

\item Similarly, along $\{v=v_{\mathcal{I}}\}\cap\mathcal{U}_{u_{*},v_{\mathcal{I}}}$,
the fact that 
\[
\lim_{\bar{v}\rightarrow v_{\mathcal{I}}^{-}}\partial_{v}^{k}r|_{v=\bar{v}}=\lim_{\bar{v}\rightarrow v_{\mathcal{I}}^{+}}\partial_{v}^{k}r|_{v=\bar{v}}\text{ and }\lim_{\bar{v}\rightarrow v_{\mathcal{I}}^{-}}\partial_{v}^{k}\Omega^{2}|_{v=\bar{v}}=\lim_{\bar{v}\rightarrow v_{\mathcal{I}}^{+}}\partial_{v}^{k}\Omega^{2}|_{v=\bar{v}}
\]
 follows by integrating the renormalised equations (\ref{eq:RenormalisedEquations})
in $u$ (again after differentiating sufficiently many times with
respect to $\partial_{v}$) and using the condition that the initial
data set $(r_{/},\Omega_{/}^{2},\bar{f}_{/};v_{\mathcal{I}})$ is
smoothly compatible, in accordance with Definition \ref{def:CompatibilityCondition}.

\end{itemize}

Thus, the proof of Theorem \ref{thm:LocalExistenceUniqueness} is
completed. \qed

\section{\label{sec:Extension-principles-for}Extension principles for smooth
solutions of (\ref{eq:RequationFinal})\textendash (\ref{NullShellFinal}) }

In this section, we will establish a number of sufficient conditions
for smooth solutions of (\ref{eq:RequationFinal})\textendash (\ref{NullShellFinal})
to admit a smooth extension beyond their original domain of definition;
in this discussion, we will adopt the notions of smoothness for solutions
to (\ref{eq:RequationFinal})\textendash (\ref{NullShellFinal}) introduced
in Section \ref{subsec:SmoothSolutions}. The extension principles
established in this section will be used in obtaining a long-time
existence resul in Section \ref{sec:Cauchy_Stability_Low_Regularity}.
The results of this section will also be useful for the proof of the
main theorem of our companion paper \cite{MoschidisVlasov}. 

\subsection{\label{subsec:ExtensionPrincipleAt_r=00003D0} Smooth extension along
$r=0$ when $\frac{2m}{r}\ll1$}

For any $u_{1}<u_{2}\in\mathbb{R}$, let us define the domain
\begin{equation}
\mathcal{D}_{u_{1}}^{u_{2}}\doteq\big([u_{1},u_{2}]\times[u_{1},u_{2}]\big)\cap\big\{ u\le v\big\}\subset\mathbb{R}^{2}\label{eq:DomainNearAxis}
\end{equation}
and the axis component $\gamma_{u_{1}}^{u_{2}}$ of the boundary of
$\mathcal{D}_{u_{1}}^{u_{2}}$:
\begin{equation}
\gamma_{u_{1}}^{u_{2}}\doteq\big([u_{1},u_{2}]\times[u_{1},u_{2}]\big)\cap\big\{ u=v\big\}\subset\partial\mathcal{D}_{u_{1}}^{u_{2}}.\label{eq:SmallPartOfAxis}
\end{equation}
Our first (and most technically involved) extension principle concerns
the smooth extendibility of solutions to (\ref{eq:RequationFinal})\textendash (\ref{NullShellFinal})
in neihgborhoods of the axis $r=0$ of the form $\mathcal{D}_{u_{1}}^{u_{2}}$:
\begin{thm}
\label{thm:ExtensionPrinciple}There exists a constant $\delta_{0}\in(0,\frac{1}{3}]$
with the following property: For any $u_{1}<u_{2}$ and any $\Lambda\in\mathbb{R}$,
let $(r,\Omega^{2},f)$ be any solution of (\ref{eq:RequationFinal})\textendash (\ref{NullShellFinal})
on $\mathcal{D}_{*}\cap\{u<v\}$, where $\mathcal{D}_{*}$ is an open
neighborhood of $\mathcal{D}_{u_{1}}^{u_{2}}\backslash\{(u_{2},u_{2})\}$,
such that $(r,\Omega^{2},f)$ is smooth with smooth axis $\gamma_{u_{1}}^{u_{2}}\backslash\{(u_{2},u_{2})\}$,
in accordance with Definition \ref{def:SmoothnessAxis} (see Figure
\ref{fig:AxisExtensionDomain}). Assume, moreover, that $(r,\Omega^{2},f)$
satisfies the following conditions:

\begin{enumerate}

\item The function $r$ satisfies at $u=u_{1}$ the one sided bound
\begin{equation}
\partial_{u}r|_{u=u_{1}}<0.\label{eq:MonotonicityInitialData}
\end{equation}

\item There exist some $C>0$, so that, at $u=u_{1}$, the support
of $f$ in the $p^{u}$ variable is bounded from above:
\begin{equation}
supp\Big(f(u_{1},\cdot;\cdot)\Big)\subseteq\big\{\partial_{v}r\cdot p^{v}-\partial_{u}r\cdot p^{u}\le C\big\}.\label{eq:BoundednessInitially}
\end{equation}

\item The solution $(r,\Omega^{2},f)$ satisfies 
\begin{equation}
\limsup_{(u,v)\rightarrow(u_{2},u_{2})}\frac{2\tilde{m}}{r}<\delta_{0}\label{eq:BoundMuLimSup}
\end{equation}
and, in the case $\Lambda>0$: 
\begin{equation}
\limsup_{(u,v)\rightarrow(u_{2},u_{2})}\frac{2m}{r}<\delta_{0}.\label{eq:BoundMassLimSup}
\end{equation}

\end{enumerate}

Then $(r,\Omega^{2},f)$ extends on a neighborhood of $\{(u_{1},u_{2})\}$
as a smooth solution with smooth axis $\gamma_{u_{1}}^{u_{2}}$, according
to Definition \ref{def:SmoothnessAxis}.

In the case when the closure of the support of $f$ does not contain
geodesics of vanishing angular momentum,i.\,e.~when $f$ satisfies
\begin{equation}
\inf_{supp\Big(f(u_{1},\cdot;\cdot)\Big)}l>0,\label{eq:NonVanishingAngularMomentum}
\end{equation}
the constant $\delta_{0}$ can be chosen to be equal to $1/3$. 
\end{thm}
\begin{figure}[h] 
\centering 
\scriptsize
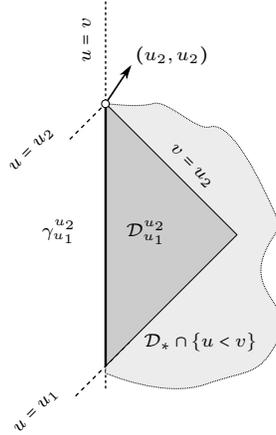 
\caption{Schematic depiction of the domains $\mathcal{D}_{u_{1}}^{u_{2}}$ and $\mathcal{D}_* \cap \{u<v \}$ appearing in the statement of Theorem~\ref{thm:ExtensionPrinciple}. \label{fig:AxisExtensionDomain}}
\end{figure}
\begin{proof}
The proof of Theorem \ref{thm:ExtensionPrinciple} will be obtained
in a number of steps, with each successive step improving the regularity
estimates for $r,\Omega^{2},f$ in a neighborhood of the point $(u_{2},u_{2})$
which were obtained in the previous step. At the final step, we will
show that
\begin{equation}
\limsup_{(u,v)\rightarrow(u_{2},u_{2})}\frac{\tilde{m}}{r^{3}}<+\infty,\label{eq:FinalBoundInMass}
\end{equation}
from which the smooth extension of $(r,\Omega^{2},f)$ on $(u_{2},u_{2})$
will then follow by a simple argument. 

\medskip{}

\noindent \emph{Gauge fixing and some initial bounds.} In view of
the initial condition (\ref{eq:MonotonicityInitialData}) and the
one sided bound 
\begin{equation}
\partial_{u}\big(\Omega^{-2}\partial_{u}r)\le0\label{eq:NegativityDuConstraint}
\end{equation}
(following from (\ref{eq:ConstraintUFinal})), we readily infer that
\begin{equation}
\partial_{u}r<0\text{ on }\mathcal{D}_{u_{1}}^{u_{2}}\backslash\{(u_{2},u_{2})\}.\label{eq:NegativeDuRInD}
\end{equation}

In view of (\ref{eq:BoundMuLimSup}), we deduce that there exists
some $u_{1}^{*}\in[u_{1},u_{2})$ such that, on
\begin{equation}
\mathcal{D}_{u_{1}^{*}}^{u_{2}}\doteq\big([u_{1}^{*},u_{2}]\times[u_{1}^{*},u_{2}]\big)\cap\big\{ u\le v\big\}\label{eq:NewDomain}
\end{equation}
we have: 
\begin{equation}
\sup_{\mathcal{D}_{u_{1}^{*}}^{u_{2}}}\frac{2\tilde{m}}{r}\le\delta_{0}.\label{eq:BoundMuTilde}
\end{equation}
In the case $\Lambda\le0$, in view of the relation (\ref{eq:RenormalisedHawkingMass})
for $\tilde{m},m$, the bound (\ref{eq:BoundMuTilde}) immediately
implies that
\begin{equation}
\sup_{\mathcal{D}_{u_{1}^{*}}^{u_{2}}}\frac{2m}{r}\le\delta_{0}.\label{eq:BoundMass}
\end{equation}
Provided $u_{1}^{*}$ is sufficiently close to $u_{2}$, the relation
(\ref{eq:BoundMass}) also holds in the case $\Lambda>0$, in view
of the assumption (\ref{eq:BoundMassLimSup}).

\medskip{}

\noindent \emph{Remark.} From now on, we will not need to distinguish
between the cases $\Lambda\le0$ and $\Lambda>0$.

\medskip{}

\noindent Since $\Omega$ is smooth on $\mathcal{D}_{u_{1}^{*}}^{u_{2}}\backslash\{(u_{2},u_{2})\}$,
from (\ref{eq:NegativeDuRInD}), (\ref{eq:DefinitionHawkingMass})
and (\ref{eq:BoundMass}) we infer that
\begin{equation}
\partial_{v}r>0\text{ on }\mathcal{D}_{u_{1}^{*}}^{u_{2}}\backslash\{(u_{2},u_{2})\}.\label{eq:PositiveDvRInD}
\end{equation}
Moreover, the smoothness of $(r,\Omega^{2},f)$ on $\mathcal{D}_{u_{1}}^{u_{2}}\cap\{u\le u_{1}^{*}\}$,
combined with the initial bound (\ref{eq:BoundednessInitially}) on
the support of $f$ imply that, for some $C_{*}>0$: 
\begin{equation}
supp\Big(f(u_{1}^{*},\cdot;\cdot)\Big)\subseteq\big\{\partial_{v}r\cdot p^{v}-\partial_{u}r\cdot p^{u}\le C_{*}\big\}.\label{eq:NewBoundPu}
\end{equation}

Let us define 
\begin{equation}
R_{*}\doteq\sup_{\mathcal{D}_{u_{1}^{*}}^{u_{2}}}r.
\end{equation}
In view of (\ref{eq:NegativeDuRInD}) and the fact that $r$ is bounded
on $\{u_{1}\}\times[u_{1},u_{2}]$, we can readily infer that 
\begin{equation}
R_{*}=\sup_{\{u_{1}\}\times[u_{1},u_{2}]}r<+\infty.
\end{equation}
Furthermore, the constraint equation (\ref{eq:DerivativeInUDirectionKappa})
readily implies that
\begin{equation}
\partial_{u}\Big(\frac{\partial_{v}r}{1-\frac{2m}{r}}\Big)\le0.\label{eq:AgainMonotonicityKappa}
\end{equation}
By integrating (\ref{eq:AgainMonotonicityKappa}) on rectangles of
the form $[u_{1}^{*},u]\times[u,v]$ for any $u\in[u_{1}^{*},u_{2})$
and $v\in[u,u_{2}]$ and using (\ref{eq:BoundMass}), (\ref{eq:PositiveDvRInD})
and the fact that 
\[
r(u,u)=0,
\]
we readily infer that 
\begin{equation}
r(u,v)\le\frac{1+\frac{1}{3}|\Lambda|R_{*}^{2}}{1-\delta_{0}}\Big(r(u_{1},u_{2})-r(u_{1},v)\Big).\label{eq:ForZeroRLimit}
\end{equation}
In view of the smoothness of $r$ on $\{u_{1}\}\times[u_{1},u_{2}]$,
from (\ref{eq:ForZeroRLimit}) we obtain that 
\begin{equation}
\lim_{(u,v)\rightarrow(u_{2},u_{2})}r(u,v)=0.\label{eq:ZeroLimitR}
\end{equation}
In view of (\ref{eq:ZeroLimitR}), we can assume without loss of generality
that $u_{1}^{*}$ has been fixed sufficiently close to $u_{2}$, so
that

\begin{equation}
\sqrt{|\Lambda|}\sup_{\mathcal{D}_{u_{1}^{*}}^{u_{2}}}r=\sqrt{|\Lambda|}R_{*}\le\delta_{0}.\label{eq:BoundR}
\end{equation}

By possibly applying a smooth coordinate transformation of the form
(\ref{eq:GeneralCoordinateTransformation}), we will also assume without
loss of generality that 
\begin{equation}
\Omega^{2}\big|_{u=u_{1}^{*}}=1.\label{eq:OmegaInitially}
\end{equation}
Note that such a tranformation does not affect the relations (\ref{eq:NegativeDuRInD}),
(\ref{eq:PositiveDvRInD}) and (\ref{eq:NewBoundPu}), which are gauge
invariant.

In view of (\ref{eq:NegativeDuRInD}), (\ref{eq:NegativityDuConstraint})
and the fact that $\partial_{u}r$ is continuous on $\{u_{1}^{*}\}\times[u_{1}^{*},u_{2}]$,
we readily obtain that, for some $c_{0}>0$ sufficiently small 
\begin{equation}
\inf_{\sup_{\mathcal{D}_{u_{1}^{*}}^{u_{2}}\backslash\{(u_{2},u_{2})\}}}\Omega^{-2}(-\partial_{u}r)\ge c_{0}>0.\label{eq:BoundForKappaD}
\end{equation}
From (\ref{eq:DefinitionHawkingMass}), (\ref{eq:BoundMass}) and
(\ref{eq:BoundForKappaD}), we therefore deduce that, for some $C_{0}>0$:
\begin{equation}
\sup_{\sup_{\mathcal{D}_{u_{1}}^{u_{2}}\backslash\{(u_{2},u_{2})\}}}\partial_{v}r\le C_{0}.\label{eq:UpperBoundDvR}
\end{equation}

\medskip{}

\noindent \emph{A first bound for $r$ on geodesics in the support
of $f$.} We will show that there exists some $c_{1}>0$ such that,
for any null geodesic $\gamma$ lying in the support of the Vlasov
field $f$ and having angular momentum $l>0$, 
\begin{equation}
\frac{\inf_{\gamma}r}{l}\ge c_{1}\min\Big\{\big(\frac{l}{R_{*}}\big)^{\frac{\delta_{0}}{1-2\delta_{0}}},1\Big\}.\label{eq:LowerBoundForR}
\end{equation}

\medskip{}

\noindent \emph{Remark.} In the case when (\ref{eq:NonVanishingAngularMomentum})
is satisfied, i.\,e.~when there exists some $l_{0}>0$ such that
$f$ is supported on $\{l\ge l_{0}\}$, the bound (\ref{eq:LowerBoundForR})
would immediately imply the statement of Theorem \ref{thm:ExtensionPrinciple}:
In view of (\ref{eq:UpperBoundDvR}), the bound (\ref{eq:LowerBoundForR})
would yield in this case that there exists some $v_{0}>0$, so that
no geodesic in the support of $f$ reaches the region $\{v\ge u_{2}-v_{0}\}\cap\mathcal{D}_{u_{1}}^{u_{2}}\backslash\{(u_{2},u_{2})\}$
and, therefore, $(r,\Omega^{2},f)$ is isometric to the trivial solution
on $\{v\ge u_{2}-v_{0}\}\cap\mathcal{D}_{u_{1}}^{u_{2}}\backslash\{(u_{2},u_{2})\}$.
The smooth extension of $(r,\Omega^{2},f)$ on $\{(u_{2},u_{2})\}$
would then follow trivially. The argument for the proof of (\ref{eq:LowerBoundForR})
only requires that $\delta_{0}\le\frac{1}{3}$, and thus one can choose
$\delta_{0}=\frac{1}{3}$ in the case when (\ref{eq:NonVanishingAngularMomentum})
holds.

\medskip{}

\noindent Let $\gamma:[0,b)\rightarrow\mathcal{D}_{u_{1}^{*}}^{u_{2}}\backslash\{(u_{2},u_{2})\}$
be a future directed, null geodesic with $(\gamma,\dot{\gamma})$
lying in the support of $f$, with its parametrization normalised
so that 
\begin{equation}
\gamma(0)\in\{u=u_{1}^{*}\}.\label{eq:InitialPointGamma}
\end{equation}
Since 
\[
\delta_{0}\le\frac{1}{3},
\]
in view of (\ref{eq:BoundMuTilde}), the non-negativity of $T_{uv}$
and the initial condition (\ref{eq:OmegaInitially}), the relation
(\ref{eq:UsefulRelationForGeodesicWithMu-U}) (with $u_{1}(v)=u_{1}^{*}$)
implies the following one-sided bound for the $\partial_{u}$ component
of $\gamma$: 
\begin{equation}
\sup_{s\in[0,b)}\log\big(\Omega^{2}\dot{\gamma}^{u}\big)(s)\le\log\big(\Omega^{2}\dot{\gamma}^{u}\big)(0).\label{eq:OneSidedBoundP}
\end{equation}
Thus, in view of (\ref{eq:NewBoundPu}), the bound (\ref{eq:OneSidedBoundP})
implies (using also (\ref{eq:DefinitionHawkingMass}), (\ref{eq:NegativeDuRInD})
and (\ref{eq:PositiveDvRInD})): 
\begin{equation}
-\partial_{u}r\cdot\dot{\gamma}^{u}\le\frac{\big(1-\frac{2m}{r}\big)}{4\partial_{v}r}C_{1}\label{eq:OneSidedBetterU}
\end{equation}
for some constant $C_{1}$ depending only on $C_{*}$ in $C_{0}$
in (\ref{eq:UpperBoundDvR}). Using (\ref{eq:OneSidedBetterU}), the
null shell relation (\ref{eq:NullShellAngularMomentum}) and the formula
(\ref{eq:DefinitionHawkingMass}), we also deduce that 
\begin{equation}
\partial_{v}r\cdot\dot{\gamma}^{v}\ge\frac{l^{2}}{r^{2}}\partial_{v}r\cdot C_{1}^{-1}.\label{eq:OneSidedBetterV}
\end{equation}

Integrating (\ref{eq:DerivativeInUDirectionKappa}) in $u$ and using
(\ref{eq:DerivativeTildeUMass}) and the bounds (\ref{eq:BoundMuTilde}),
(\ref{eq:BoundMass}), we infer that, for any $(u,v)\in\mathcal{D}_{u_{1}^{*}}^{u_{2}}\backslash\{(u_{2},u_{2})\}$:
\begin{align}
\Big|\log\Big(\frac{\partial_{v}r}{1-\frac{2m}{r}}\Big)\Big|_{(u,v)}-\log\Big(\frac{\partial_{v}r}{1-\frac{2m}{r}}\Big)\Big|_{(u_{1}^{*},v)}\Big| & \le\int_{u_{1}^{*}}^{u}\frac{2}{r}\frac{(-\partial_{u}\tilde{m})}{1-\frac{2m}{r}}(\bar{u},v)\,d\bar{u}\le\label{eq:FirstBoundKappa}\\
 & \le\frac{2}{1-\delta_{0}}\int_{u_{1}^{*}}^{u}\frac{-\partial_{u}\tilde{m}}{r}(\bar{u},v)\,d\bar{u}=\nonumber \\
 & =\frac{2}{1-\delta_{0}}\Big(-\int_{u_{1}^{*}}^{u}\frac{m}{r^{2}}(\bar{u},v)\,(-\partial_{u}r)d\bar{u}+\big[\frac{m}{r}\big]\big|_{\bar{u}=u}^{u_{1}^{*}}\Big)\le\nonumber \\
 & \le\frac{\delta_{0}}{1-\delta_{0}}\log\frac{r(u_{1}^{*},v)}{r(u,v)}+1.\nonumber 
\end{align}
The bound (\ref{eq:FirstBoundKappa}) implies that for some $C_{2}>0$
depending only on $\inf_{u=u_{1}^{*}}\partial_{v}r$: 
\begin{equation}
\frac{\partial_{v}r}{1-\frac{2m}{r}}\ge C_{2}\big(\frac{r}{R_{*}}\big)^{\frac{\delta_{0}}{1-\delta_{0}}}.\label{eq:LowerBoundDvR}
\end{equation}
Thus, in view of (\ref{eq:BoundMass}), the bounds (\ref{eq:OneSidedBetterU})
and (\ref{eq:OneSidedBetterV}) yield 
\begin{equation}
\frac{-\partial_{u}r\cdot\dot{\gamma}^{u}}{\partial_{v}r\cdot\dot{\gamma}^{v}}\le C_{1}^{\prime}\frac{r^{2}}{l^{2}}\big(\frac{r}{R_{*}}\big)^{-\frac{2\delta_{0}}{1-\delta_{0}}}.\label{eq:BoundForSlope}
\end{equation}
for some $C_{1}^{\prime}>0$ depending only on $C_{1},C_{2}$.

Along the geodesic $\gamma$ we calculate 
\begin{equation}
\dot{r}=\partial_{v}r\cdot\dot{\gamma}^{v}+\partial_{u}r\cdot\dot{\gamma}^{u}.\label{eq:RDerivativeAlongGeodesic}
\end{equation}
As a consequence, the upper bound (\ref{eq:BoundForSlope}) implies
that $\dot{r}\le0$ can only be achieved in the region where 
\begin{equation}
C_{1}^{\prime}\frac{r^{2}}{l^{2}}\big(\frac{r}{R_{*}}\big)^{-\frac{2\delta_{0}}{1-\delta_{0}}}\ge1.
\end{equation}
Therefore, (\ref{eq:LowerBoundForR}) holds (for some $c_{1}$ possibly
depending on $R_{*}$).

\medskip{}

\noindent \emph{Bounds for $\dot{\gamma}$ for geodesics $\gamma$
in the support of $f$.} Let $\gamma:[0,b)\rightarrow\mathcal{D}_{u_{1}^{*}}^{u_{2}}\backslash\{(u_{2},u_{2})\}$
be a future directed, null geodesic in the support of $f$ as before,
satisfying (\ref{eq:InitialPointGamma}). For any $s\in[0,b)$ such
that 
\begin{equation}
\partial_{v}r\cdot\dot{\gamma}^{v}(s)\le4\delta_{0}^{-1}(-\partial_{u}r)\cdot\dot{\gamma}^{u}(s),\label{eq:NotAlmostoutgoing}
\end{equation}
we can estimate, in view of (\ref{eq:OneSidedBetterU}) and (\ref{eq:LowerBoundDvR})
(using also (\ref{eq:BoundMass})), that:
\begin{equation}
\partial_{v}r\cdot\dot{\gamma}^{v}(s)\le4\delta_{0}^{-1}(-\partial_{u}r)\cdot\dot{\gamma}^{u}(s)\le C_{1}\delta_{0}^{-1}\frac{1-\frac{2m}{r}}{\partial_{v}r}\le C_{3}\delta_{0}^{-1}\big(\frac{r}{R_{*}}\big)^{-\frac{\delta_{0}}{1-\delta_{0}}}\label{eq:EasyBoundGammaV}
\end{equation}
for some $C_{3}>0$ depending only on $C_{1},C_{2}$. 

We will now establish a bound for $\dot{\gamma}^{v}$ in the region
where (\ref{eq:NotAlmostoutgoing}) does not hold. Let $s_{1}\le s_{2}\in[0,b)$
be such that, for any $s\in[s_{1},s_{2}]$, we can bound: 
\begin{equation}
\partial_{v}r\cdot\dot{\gamma}^{v}(s)\ge2\delta_{0}^{-1}(-\partial_{u}r)\cdot\dot{\gamma}^{u}(s).\label{eq:AlmostOutgoingRelation}
\end{equation}
Note that, in view of the relation (\ref{eq:RDerivativeAlongGeodesic}),
(\ref{eq:AlmostOutgoingRelation}) implies that, along $\gamma([s_{1},s_{2}])$:
\begin{equation}
dv|_{\gamma([s_{1},s_{2}])}=\dot{\gamma}^{v}ds|_{[s_{1},s_{2}]}\le(1-\frac{1}{2}\delta_{0})^{-1}(\partial_{v}r)^{-1}dr|_{\gamma([s_{1},s_{2}])}.\label{eq:ComparisonVolumeForms}
\end{equation}
Equation (\ref{eq:NullGeodesicsSphericalSymmetry}) for $\dot{\gamma}^{u}$,
in view of (\ref{eq:NullShellAngularMomentum}), (\ref{eq:EquationOmegaForProof}),
(\ref{eq:OmegaInitially}) and (\ref{eq:AlmostOutgoingRelation}),
yields the following bound for any $s\in[s_{1},s_{2}]$: 
\begin{align}
-\frac{1}{\dot{\gamma}^{v}}\frac{d}{ds}\log\big(r^{2}\Omega^{2}\dot{\gamma}^{u}\big)(s) & =-\partial_{v}\log\Omega^{2}|_{\gamma(s)}+2\frac{\dot{\gamma}^{u}}{\dot{\gamma}^{v}}\frac{(-\partial_{u}r)}{r}|_{\gamma(s)}\le\label{eq:BoundLogarithmicDerivative}\\
 & \le-\partial_{v}\log\Omega^{2}|_{\gamma(s)}+\delta_{0}\frac{\partial_{v}r}{r}|_{\gamma(s)}=\nonumber \\
 & =-\int_{u_{1}^{*}}^{u(\gamma(s))}\partial_{u}\partial_{v}\log\Omega^{2}(u,v(\gamma(s)))\,du-\partial_{v}\log\Omega^{2}(u_{1}^{*},v(\gamma(s))+\delta_{0}\frac{\partial_{v}r}{r}|_{\gamma(s)}=\nonumber \\
 & =\int_{u_{1}^{*}}^{u(\gamma(s))}\Big(-4\big(\frac{\tilde{m}}{r^{3}}+\frac{\Lambda}{6}\big)\frac{(-\partial_{u}r)\partial_{v}r}{1-\frac{2m}{r}}+16\pi T_{uv}\Big)(u,v(\gamma(s)))\,du+\delta_{0}\frac{\partial_{v}r}{r}|_{\gamma(s)}.\nonumber 
\end{align}

Using (\ref{eq:DerivativeTildeUMass}), (\ref{eq:BoundMuTilde}),
(\ref{eq:BoundR}), (\ref{eq:BoundMass}) and the fact that the right
hand side of (\ref{eq:DerivativeInUDirectionKappa}) is non-positive,
we can estimate
\begin{align}
\int_{u_{1}^{*}}^{u(\gamma(s))}\Big(-4\big( & \frac{\tilde{m}}{r^{3}}+\frac{\Lambda}{6}\big)\frac{(-\partial_{u}r)\partial_{v}r}{1-\frac{2m}{r}}+16\pi T_{uv}\Big)(u,v)\,du\le\label{eq:BoundTuv}\\
\le & \int_{u_{1}^{*}}^{u(\gamma(s))}\Big(-4\big(\frac{\tilde{m}}{r^{3}}+\frac{\Lambda}{6}\big)\frac{(-\partial_{u}r)\partial_{v}r}{1-\frac{2m}{r}}+8\frac{\partial_{v}r}{1-\frac{2m}{r}}\frac{-\partial_{u}\tilde{m}}{r^{2}}\Big)(u,v)\,du=\nonumber \\
= & \int_{u_{1}^{*}}^{u(\gamma(s))}\Big\{-4\big(\frac{\tilde{m}}{r^{3}}+\frac{\Lambda}{6}\big)\frac{(-\partial_{u}r)\partial_{v}r}{1-\frac{2m}{r}}+8\frac{\partial_{v}r}{1-\frac{2m}{r}}\frac{\tilde{m}}{r^{2}}\Big(\partial_{u}\log\big(\frac{\partial_{v}r}{1-\frac{2m}{r}}\big)+2\frac{-\partial_{u}r}{r}\Big)\Big\}(u,v)\,du+\nonumber \\
 & \hphantom{=*}+8\frac{\partial_{v}r}{1-\frac{2m}{r}}\frac{\tilde{m}}{r^{2}}(u_{1}^{*},v)-8\frac{\partial_{v}r}{1-\frac{2m}{r}}\frac{\tilde{m}}{r^{2}}(u(\gamma(\sigma)),v)\nonumber \\
\le & \int_{u_{1}^{*}}^{u(\gamma(s))}\Big\{-4\big(\frac{\tilde{m}}{r^{3}}+\frac{\Lambda r^{2}}{6r^{2}}\big)\frac{(-\partial_{u}r)\partial_{v}r}{1-\frac{2m}{r}}+16\frac{\tilde{m}}{r^{3}}\frac{(-\partial_{u}r)\partial_{v}r}{1-\frac{2m}{r}}\Big\}(u,v)\,du+\nonumber \\
 & \phantom{=*}+8\frac{\partial_{v}r}{1-\frac{2m}{r}}\frac{\tilde{m}}{r^{2}}(u_{1}^{*},v)\le\nonumber \\
\le & 6\delta_{0}\int_{u_{1}^{*}}^{u(\gamma(s))}\frac{-\partial_{u}r}{r^{2}}\frac{\partial_{v}r}{1-\frac{2m}{r}}(u,v)\,du+\frac{4\delta_{0}}{1-\frac{2m}{r}}\frac{\partial_{v}r}{r}(u_{1}^{*},v)=\nonumber \\
= & 6\delta_{0}\int_{u_{1}^{*}}^{u(\gamma(s))}\frac{1}{r}\partial_{u}\Big(\frac{\partial_{v}r}{1-\frac{2m}{r}}\Big)(u,v)\,du+\frac{6\delta_{0}}{1-\frac{2m}{r}}\frac{\partial_{v}r}{r}(u(\gamma(s)),v)-\nonumber \\
 & \phantom{=*}-\frac{6\delta_{0}}{1-\frac{2m}{r}}\frac{\partial_{v}r}{r}(u_{1}^{*},v)+\frac{4\delta_{0}}{1-\frac{2m}{r}}\frac{\partial_{v}r}{r}(u_{1}^{*},v)\le\nonumber\\
\le & \frac{6\delta_{0}}{1-\delta_{0}}\frac{\partial_{v}r}{r}(u(\gamma(s)),v)-\delta_{0}\frac{\partial_{v}r}{r}(u_{1}^{*},v).\nonumber 
\end{align}
Substituting in (\ref{eq:BoundLogarithmicDerivative}), we therefore
obtain for any $s\in[s_{1},s_{2}]$: 
\begin{equation}
-\frac{1}{\dot{\gamma}^{v}}\frac{d}{ds}\log\big(r^{2}\Omega^{2}\dot{\gamma}^{u}\big)(s)\le\frac{6\delta_{0}}{1-\delta_{0}}\frac{\partial_{v}r}{r}\Big|_{\gamma(s)}.\label{eq:UpperBoundForPvAlmost}
\end{equation}
Integrating (\ref{eq:UpperBoundForPvAlmost}) in $s$ and using the
relation 
\[
\dot{\gamma}^{v}=\frac{l^{2}}{\big(r^{2}\Omega^{2}\dot{\gamma}^{u}\big)}
\]
(following from (\ref{eq:NullShellAngularMomentum})) and the bound
(\ref{eq:ComparisonVolumeForms}), we can finally estimate for any
$s_{1}\le s_{2}$ for which (\ref{eq:AlmostOutgoingRelation}) holds
on $[s_{1},s_{2}]$:
\begin{equation}
\frac{\dot{\gamma}^{v}(s_{2})}{\dot{\gamma}^{v}(s_{1})}\le\Big(\frac{r|_{\gamma(s_{2})}}{r|_{\gamma(s_{1})}}\Big)^{\frac{6\delta_{0}}{(1-\delta_{0})^{2}}}.\label{eq:BoundPv}
\end{equation}

Using the fact that (\ref{eq:EasyBoundGammaV}) holds for all $s$
for which (\ref{eq:NotAlmostoutgoing}) is true, while (\ref{eq:BoundPv})
holds for all $s_{1}\le s_{2}$ for which (\ref{eq:AlmostOutgoingRelation})
is true on $[s_{1},s_{2}]$, we can estimate for all $s\in[0,b)$,
in view of the bounds (\ref{eq:UpperBoundDvR}) and (\ref{eq:LowerBoundDvR})
for $\partial_{v}r$, the bound (\ref{eq:BoundMass}) for $\frac{2m}{r}$
and the lower bound (\ref{eq:LowerBoundForR}) for $r$ along $\gamma$:
\begin{equation}
\partial_{v}r\cdot\dot{\gamma}^{v}(s)\le C_{4}\delta_{0}^{-1}\max\Big\{\big(\frac{l}{R_{*}}\big)^{-\frac{7\delta_{0}}{(1-2\delta_{0})^{2}}},1\Big\}\label{eq:BoundPvFinal}
\end{equation}
for some $C_{4}>0$ depending on $c_{1},C_{1},C_{2}$. 

From (\ref{eq:OneSidedBetterU}), (\ref{eq:LowerBoundDvR}) and (\ref{eq:LowerBoundForR}),
we also obtain (using the formula (\ref{eq:DefinitionHawkingMass})
for $\Omega^{2}$ in $\Omega^{2}\dot{\gamma}^{u}$) 
\begin{equation}
-\partial_{u}r\cdot\dot{\gamma}^{u}(s)\le C_{4}\max\Big\{\big(\frac{l}{R_{*}}\big)^{-\frac{\delta_{0}}{1-2\delta_{0}}},1\Big\}.\label{eq:BoundPuFinal}
\end{equation}

\medskip{}

\noindent \emph{First improved bound for $\frac{2\tilde{m}}{r}$.}
We will now proceed to establish an improved bound for $\frac{2\tilde{m}}{r}$
in terms of $r$. 

In view of the fact that $f$ satisfies the transport equation (\ref{eq:Vlasov}),
we can trivially bound for $\bar{f}|_{supp(f)}$ (defined by the relation
(\ref{eq:FAsASmoothDeltaFunction}))
\begin{equation}
\sup_{supp(f)}\bar{f}\le\sup_{\{u=u_{1}^{*}\}\cap supp(f)}\bar{f}<+\infty.\label{eq:UpperBoundF}
\end{equation}
In view of (\ref{eq:NewBoundPu}), the definitions (\ref{eq:ReducedParticleCurrent})
and (\ref{eq:TotalReducedParticleCurrent}) imply that there exists
some $C_{5}>0$ such that:
\begin{equation}
\sup_{l\ge0,v\in[u_{1}^{*},u_{2}]}r^{2}N_{v}^{(l)}(u_{1}^{*},v)\le C_{5}.\label{eq:BoundReducedParticleCurrent}
\end{equation}
Integrating (\ref{eq:ReducedParticleConservation}) and using (\ref{eq:BoundReducedParticleCurrent}),
we can therefore bound for any $l\ge0$:
\begin{equation}
\sup_{u_{1}^{*}\le u<u_{2}}\int_{u}^{u_{2}}r^{2}N_{v}^{(l)}(u,v)\,dv\le\int_{u_{1}^{*}}^{u_{2}}r^{2}N_{v}^{(l)}(u_{1}^{*},v)\,dv\le C_{5}(u_{2}-u_{1}^{*}).\label{eq:ParitcleCurrentIntegratedBound}
\end{equation}

Integrating equation (\ref{eq:DerivativeTildeVMass}) in $v$ (using
the fact that $\tilde{m}|_{\gamma}=0$), the bound (\ref{eq:VDerivativeMassfromReducedParticleCurrent})
implies that, for any $u\ge u_{1}^{*}$:
\begin{align}
\tilde{m}(u,v) & =2\pi\int_{u}^{v}\Big(\frac{1-\frac{2m}{r}}{\partial_{v}r}r^{2}T_{vv}(u,\bar{v})+\frac{1-\frac{2m}{r}}{-\partial_{u}r}r^{2}T_{uv}(u,\bar{v})\Big)\,d\bar{v}=\label{eq:ExpressionMassImprovedR}\\
 & =2\pi\int_{0}^{+\infty}\int_{u}^{v}\Big(\frac{1-\frac{2m}{r}}{\partial_{v}r}r^{2}T_{vv}^{(l)}(u,\bar{v})+\frac{1-\frac{2m}{r}}{-\partial_{u}r}r^{2}T_{uv}^{(l)}(u,\bar{v})\Big)\,d\bar{v}ldl\le\nonumber \\
 & \le4\pi\int_{0}^{+\infty}\int_{u}^{v}\sup_{supp\big(f(u,v;\cdot,\cdot,l)\big)}\Big(\partial_{v}r(u,\bar{v})p^{v}-\partial_{u}r(u,\bar{v})p^{u}\Big)\cdot r^{2}N_{v}^{(l)}(u,\bar{v})\,d\bar{v}ldl.\nonumber 
\end{align}
 In view of the bounds (\ref{eq:BoundPvFinal}) and (\ref{eq:BoundPuFinal})
for geodesics in the support of $f$, as well as the particle current
bound (\ref{eq:ParitcleCurrentIntegratedBound}), the estimate (\ref{eq:ExpressionMassImprovedR})
yields for any $u\ge u_{1}^{*}$: 
\begin{align}
\tilde{m}(u,v) & \le4\pi C_{4}\delta_{0}^{-1}\int_{0}^{+\infty}\max\Big\{\big(\frac{l}{R_{*}}\big)^{-\frac{7\delta_{0}}{(1-2\delta_{0})^{2}}},1\Big\}\Big(\int_{u}^{v}r^{2}N_{v}^{(l)}(u,\bar{v})\,d\bar{v}\Big)\,ldl\le\label{eq:AlmostImprovedMassBound}\\
 & \le C_{\delta_{0}}\int_{0}^{l^{+}(u,v)}l\cdot\max\Big\{\big(\frac{l}{R_{*}}\big)^{-\frac{7\delta_{0}}{(1-2\delta_{0})^{2}}},1\Big\}\,dl\le\nonumber \\
 & \le C_{\delta_{0}}(l^{+}(u,v))^{2}\max\Big\{\big(\frac{l^{+}(u,v)}{R_{*}}\big)^{-\frac{7\delta_{0}}{(1-2\delta_{0})^{2}}},1\Big\},\nonumber 
\end{align}
where 
\begin{equation}
l^{+}(u,v)=\sup_{\bar{v}\in[u,v]}\Big(\sup_{supp\big(f(u,\bar{v};\cdot,\cdot,\cdot)\big)}l\Big).
\end{equation}
Notice that (\ref{eq:LowerBoundForR}) (in view also of (\ref{eq:PositiveDvRInD}))
implies that 
\begin{equation}
l^{+}(u,v)\le Cr(u,v)\max\Big\{\big(\frac{r(u,v)}{R_{*}}\big)^{-\frac{\delta_{0}}{1-\delta_{0}}},1\Big\}.\label{eq:UpperBoundL+}
\end{equation}
Hence, from (\ref{eq:AlmostImprovedMassBound}) and (\ref{eq:UpperBoundL+})
we obtain: 
\begin{equation}
\sup_{(u,v)\in\mathcal{D}_{u_{1}^{*}}^{u_{2}}\backslash\{(u_{2},u_{2})\}}\frac{\tilde{m}}{r^{2-100\delta_{0}}}<+\infty.\label{eq:FirstImprovedBoundMass}
\end{equation}

\medskip{}

\noindent \emph{Improved bounds for the geometry.} Assuming that 
\begin{equation}
\delta_{0}\le10^{-3},
\end{equation}
returning to the proof of (\ref{eq:LowerBoundForR}) and (\ref{eq:LowerBoundDvR})
and using the stronger bound (\ref{eq:FirstImprovedBoundMass}) in
place of the weaker initial bound (\ref{eq:BoundMuTilde}), we can
readily improve (\ref{eq:LowerBoundForR}) and (\ref{eq:LowerBoundDvR})
as follows: There exists some $c>0$, such that:

\begin{itemize}

\item For any $(u,v)\in\mathcal{D}_{u_{1}^{*}}^{u_{2}}\backslash\{(u_{2},u_{2})\}$
\begin{equation}
\partial_{v}r(u,v)\ge c.\label{eq:ImprovedLowerBoundDvR}
\end{equation}

\item For any future directed null geodesic $\gamma:[0,b)\rightarrow\mathcal{D}_{u_{1}^{*}}^{u_{2}}\backslash\{(u_{2},u_{2})\}$
in the support of $f$:
\begin{equation}
\inf_{\gamma}r\le c\cdot l\label{eq:ImprovedLowerBoundForR}
\end{equation}
and (in view of (\ref{eq:OneSidedBetterU}) and (\ref{eq:ImprovedLowerBoundDvR}))
\begin{equation}
\sup_{s\in[0,b)}\big(-\partial_{u}r\cdot\dot{\gamma}^{u}(s)\big)\le c^{-1}.\label{eq:ImprovedBoundGammaU}
\end{equation}

\end{itemize}

Integrating (\ref{eq:DerivativeInVDirectionKappaBar}) in $v$ and
using the boundary condition 
\begin{equation}
\partial_{v}r|_{\gamma}=-\partial_{u}r|_{\gamma},
\end{equation}
the bounds (\ref{eq:BoundMass}), (\ref{eq:FirstImprovedBoundMass}),
(\ref{eq:UpperBoundDvR}) and (\ref{eq:ImprovedLowerBoundDvR}) imply
that there exists some $C>0$ such that, for any $(u,v)\in\mathcal{D}_{u_{1}^{*}}^{u_{2}}\backslash\{(u_{2},u_{2})\}$:
\begin{equation}
C^{-1}\le-\partial_{u}r(u,v)\le C.\label{eq:ImprovedUpperBoundDuR}
\end{equation}
Note also that (\ref{eq:UpperBoundDvR}), (\ref{eq:ImprovedLowerBoundDvR})
and (\ref{eq:ImprovedUpperBoundDuR}) imply (in view of (\ref{eq:DefinitionHawkingMass}))
\begin{equation}
\sup_{\mathcal{D}_{u_{1}^{*}}^{u_{2}}\backslash\{(u_{2},u_{2})\}}\big|\log\Omega^{2}\big|<+\infty.\label{eq:BoundsForOmega}
\end{equation}

\begin{figure}[h] 
\centering 
\scriptsize
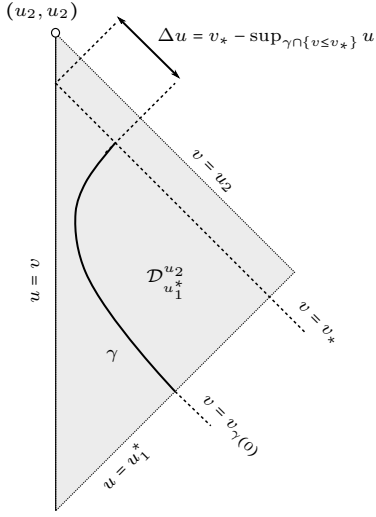 
\caption{The bound (\ref{eq:BoundedAwayFromTop}) implies that a null geodesic $\gamma$ cannot remain in the region close to the axis for a long period of retarded time compared to its initial separation $u_2-v(\gamma(0))$ from $v=u_2$. \label{fig:Gamma_Proximity_Axis}}
\end{figure}

We will now show that there exists some $0<c<1$ such that, for any
$v_{*}\in(u_{1}^{*},u_{2}]$ and any future directed null geodesic
$\gamma:[0,b)\rightarrow\mathcal{D}_{u_{1}^{*}}^{u_{2}}\backslash\{(u_{2},u_{2})\}$
in the support of $f$ satisfying $u(\gamma(0))=u_{1}^{*}$ and $v(\gamma(0))\le v_{*}$,
we have 
\begin{equation}
v_{*}-\sup_{\gamma\cap\{v\le v_{*}\}}u\ge c\big(v_{*}-v(\gamma(0))\big)\label{eq:BoundedAwayFromTop}
\end{equation}
(see also Figure \ref{fig:Gamma_Proximity_Axis}).

\medskip{}

\noindent \emph{Proof of (\ref{eq:BoundedAwayFromTop}).} By possibly
restricting the domain of $\gamma$ to a subdomain of the form $[0,b_{*})\subset[0,b)$,
if necessary, we will assume that, without loss of generality,
\begin{equation}
\sup_{s\in[0,b)}v(\gamma(s))\le v_{*}.\label{eq:supvgammas}
\end{equation}
Differentiating (\ref{eq:UsefulRelationForGeodesicWithMu-U}) (with
$u_{1}(v)=u_{1}^{*}$) and using (\ref{eq:OmegaInitially}), (\ref{eq:BoundMuTilde}),
(\ref{eq:ImprovedUpperBoundDuR}) and (\ref{eq:BoundsForOmega}),
we obtain for any $s\in[0,b)$:
\begin{align}
\frac{1}{\dot{\gamma}^{v}}\frac{d}{ds}\log\big(\Omega^{2}\dot{\gamma}^{u}\big)(s) & =\Bigg(\int_{u_{1}^{*}}^{u(\gamma(s))}\Big(\frac{1}{2}\frac{\frac{6\tilde{m}}{r}-1}{r^{2}}\Omega^{2}-24\pi T_{uv}\Big)(u,v(\gamma(s))\,du\Bigg)-2\frac{\partial_{v}r}{r}(u_{1}^{*},v(\gamma(s)))\le\label{eq:OnTheWayToImprovedBounds}\\
 & \le-c\Big(\frac{1}{r}(u(\gamma(s)),v(\gamma(s)))-\frac{1}{r}(u_{1}^{*},v(\gamma(s)))\Big).\nonumber 
\end{align}

Let $0<\delta_{1}\ll1$ be a sufficiently small parameter that will
be fixed later, and let $s_{0}\in[0,b)$ be the minimum value of $s$
for which 
\begin{equation}
v_{*}-u(\gamma(s_{0}))=\delta_{1}\cdot\big(v_{*}-v(\gamma(0))\big)\label{eq:DefinitionS_0}
\end{equation}
(note that if (\ref{eq:DefinitionS_0}) does not hold for any $s_{0}\in[0,b)$,
then (\ref{eq:BoundedAwayFromTop}) trivially holds for $c=\delta_{1}$);
see Figure \ref{fig:Gamma_Proximity_Axis_Proof}. Since $\gamma$
is future directed and null, $u(\gamma(s))$ is increasing in $s$
and, hence, (\ref{eq:DefinitionS_0}) implies that, for any $s\in[s_{0},b)$:
\begin{equation}
v_{*}-u(\gamma(s))\le\delta_{1}\cdot\big(v_{*}-v(\gamma(0))\big).\label{eq:IfClosedEnoughToTop}
\end{equation}

\begin{figure}[h] 
\centering 
\scriptsize
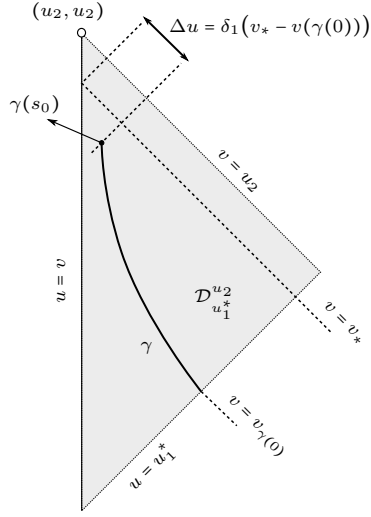 
\caption{Schematic depiction of the image of $\gamma([0,s_0])$, where $s_0$ is defined by (\ref{eq:DefinitionS_0}). \label{fig:Gamma_Proximity_Axis_Proof}}
\end{figure}

Using the bound (\ref{eq:IfClosedEnoughToTop}) for $s\ge s_{0}$
and the fact that $u\le v$ on $\mathcal{D}_{u_{1}^{*}}^{u_{2}}$,
we infer that, for any $s\in[s_{0},b)$:
\begin{equation}
v(\gamma(s))-u_{1}^{*}\ge(1-\delta_{1})(v_{*}-u_{1}^{*}).\label{eq:LowerBoundVGammaS}
\end{equation}
Hence, from (\ref{eq:ImprovedLowerBoundDvR}) and (\ref{eq:LowerBoundVGammaS})
(in view also of the fact that $r|_{\{u=v\}}=0$) we obtain that,
for any $s\in[s_{0},b)$:
\begin{equation}
r(u_{1}^{*},v(\gamma(s)))\ge c(1-\delta_{1})(v_{*}-u_{1}^{*}).\label{eq:LowerBoundRBehind}
\end{equation}
Similarly, in view of (\ref{eq:ImprovedUpperBoundDuR}) and the bounds
(\ref{eq:IfClosedEnoughToTop}) and (\ref{eq:supvgammas}), we can
readily estimate for any $s\in[s_{0},b)$:
\begin{equation}
r(u(\gamma(s)),v(\gamma(s)))\le C(v(\gamma(s))-u(\gamma(s)))\le C\delta_{1}\cdot\big(v_{*}-v(\gamma(0))\big).\label{eq:UpperboundRAbove}
\end{equation}
Note that, since $v(\gamma(0))\ge u(\gamma(0))\ge u_{1}^{*}$, the
bounds (\ref{eq:LowerBoundRBehind}) and (\ref{eq:UpperboundRAbove})
imply (provided $\delta_{1}$ is small depending on $C,c$): 
\begin{equation}
\frac{1}{r\big(u(\gamma(s)),v(\gamma(s))\big)}\ge\frac{1}{C}\delta_{1}^{-1}\frac{1}{r\big(u_{1}^{*},v(\gamma(s))\big)}.\label{eq:ComparisonOfRs}
\end{equation}

Plugging (\ref{eq:LowerBoundRBehind}) and (\ref{eq:UpperboundRAbove})
in (\ref{eq:OnTheWayToImprovedBounds}) (using also (\ref{eq:ComparisonOfRs}))
and integrating the resulting expression in $s$ starting from $s=s_{0}$,
we infer that, for any $s\in[s_{0},b)$ (provided $\delta_{1}$ is
small depending on $C,c$): 
\begin{equation}
\log\big(\Omega^{2}\dot{\gamma}^{u}\big)(s)-\log\big(\Omega^{2}\dot{\gamma}^{u}\big)(s_{0})\le-\delta_{1}^{-1/2}\big(v_{*}-v(\gamma(0))\big)^{-1}.\label{eq:FirstBoundGammaUFinal}
\end{equation}
Using (\ref{eq:BoundsForOmega}), (\ref{eq:OneSidedBoundP}) and (\ref{eq:NewBoundPu}),
and provided $\delta_{1}$ is small depending on $C,c,u_{2}-u_{1}^{*}$,
we infer that, for any $s\in[s_{0},b)$: 
\begin{equation}
\dot{\gamma}^{u}(s)\le C\exp\Big(-c\delta_{1}^{-1/2}\big(v_{*}-v(\gamma(0))\big)^{-1}\Big).\label{eq:OneMoreGammaUBound}
\end{equation}

In view of the relation (\ref{eq:NullShell}) for $\dot{\gamma}^{u},\dot{\gamma}^{v}$
and the bounds (\ref{eq:ImprovedLowerBoundForR}) and (\ref{eq:BoundsForOmega}),
we infer from (\ref{eq:OneMoreGammaUBound}) that, for any $s\in[s_{0},b)$
(provided $\delta_{1}$ is small depending on $C,c,u_{2}-u_{1}^{*}$):
\begin{equation}
\frac{\dot{\gamma}^{u}(s)}{\dot{\gamma}^{v}(s)}\le\exp\Big(-\delta_{1}^{-1/4}\big(v_{*}-v(\gamma(0))\big)^{-1}\Big)
\end{equation}
and, thus:
\begin{equation}
\dot{\gamma}^{u}(s)\le\exp\Big(-\delta_{1}^{-1/4}\big(v_{*}-v(\gamma(0))\big)^{-1}\Big)\dot{\gamma}^{v}(s).\label{eq:BoundForSlopeFinal}
\end{equation}
Integrating (\ref{eq:BoundForSlopeFinal}) in $s$ for $s\ge s_{0}$,
we therefore obtain for any $s\in[s_{0},b)$: 
\begin{align}
u(\gamma(s))-u(\gamma(s_{0})) & \le\int_{s_{0}}^{s}\exp\Big(-\delta_{1}^{-1/4}\big(v_{*}-v(\gamma(0))\big)^{-1}\Big)\dot{\gamma}^{v}(\bar{s})\,d\bar{s}=\label{eq:FinalBoundUDifference}\\
 & =\big(v(\gamma(s))-v(\gamma(s_{0}))\big)\exp\Big(-\delta_{1}^{-1/4}\big(v_{*}-v(\gamma(0))\big)^{-1}\Big)\le\nonumber \\
 & \le\big(v_{*}-v(\gamma(0))\big)\exp\Big(-\delta_{1}^{-1/4}\big(v_{*}-u_{1}^{*}\big)^{-1}\Big).\nonumber 
\end{align}

In view of (\ref{eq:DefinitionS_0}), we obtain from (\ref{eq:FinalBoundUDifference}),
provided $\delta_{1}$ is small enough depending on $C,c$ and $u_{2}-u_{1}^{*}$:
\begin{equation}
u_{2}-\sup_{s\in[s_{0},b)}u(\gamma(s))\ge\frac{1}{2}\big(v_{*}-v(\gamma(0))\big).\label{eq:1BoundU}
\end{equation}
In view of the fact that the definition (\ref{eq:DefinitionS_0})
of $s_{0}$ and the fact that $u_{2}\ge v_{*}$ imply that 
\begin{equation}
u_{2}-\sup_{s\in[0,s_{0}]}u(\gamma(s))\ge\delta_{1}\big(v_{*}-v(\gamma(0))\big),\label{eq:2BoundU}
\end{equation}
from (\ref{eq:1BoundU}) and (\ref{eq:2BoundU}) we finally infer
(\ref{eq:BoundedAwayFromTop}).

\medskip{}

\noindent \emph{Further improved bounds for $\frac{\tilde{m}}{r}$.}
For any $\bar{u}\in[u_{1}^{*},u_{2})$ and any $\bar{v}\in(\bar{u},u_{2})$,
let us define
\[
\mathcal{V}_{\bar{u}\bar{v}}^{*}=\Big\{(u,v;p^{u},p^{v},l)\in supp(f)\bigcap\big\{\big(\{\bar{u}\le u\le v\}\cap\{v\le\bar{v}\}\big)\times[0,+\infty)^{3}\big\}\Big\}
\]
and let 
\begin{equation}
\mathcal{V}_{\bar{u}\bar{v}}=\big(\cup_{s\in\mathbb{R}}\Phi_{s}(\mathcal{V}_{\bar{u}\bar{v}}^{*})\big)\cap\big(\big(\{\bar{u}\le u\le v\}\cap\{v\le\bar{v}\}\big)\times[0,+\infty)^{3}\big),
\end{equation}
 where $\Phi_{s}$ denotes the image of the geodesic flow (\ref{eq:NullGeodesicsSphericalSymmetry})
after time $s$. Note that $\mathcal{V}_{\bar{u}\bar{v}}$ is invariant
under the geodesic flow and consists exactly to the region in phase
space traced out by those geodesics in the support of $f$ that intersect
the physical space domain $\{u\ge\bar{u}\}\cap\{v\le\bar{v}\}$.

The estimate (\ref{eq:BoundedAwayFromTop}) is equivalent to the following
statement: There exists some $C>0$ such that, for any $\bar{u}\in[u_{1}^{*},u_{2})$
and any $\bar{v}\in(\bar{u},u_{2})$, every future directed null geodesic
$\gamma$ in the support of $f$ with $\gamma(0)\in\{u=u_{1}^{*}\}$
that intersects the region $\{u\ge\bar{u}\}\cap\{v\le\bar{v}\}$ satisfies
\begin{equation}
\gamma(0)\in\mathcal{C}_{\bar{u}\bar{v}}\doteq\big\{ u=u_{1}^{*}\big\}\cap\big\{\max\{u_{1}^{*},\bar{v}-C(\bar{v}-\bar{u})\}\le v\le\bar{v}\big\}.
\end{equation}
Note that, if $v(\gamma(0))>\bar{v}$, then $v(\gamma(s))>\bar{v}$
for all $s$, since $\gamma$ is causal. In view of the conservation
law (\ref{eq:ReducedParticleConservation}), the above statement implies
that, for any $\bar{u}\in[u_{1}^{*},u_{2})$, any $\bar{v}\in(\bar{u},u_{2})$
and any $l\ge0$: 
\begin{align}
\int_{\bar{u}}^{\bar{v}}r^{2}N_{v}^{(l)}(\bar{u},v)\,dv & =\int_{\bar{u}}^{\bar{v}}r^{2}N_{v}^{(\mathcal{V}_{\bar{u}\bar{v}};l)}(\bar{u},v)\,dv\le\label{eq:FirstBoundN}\\
 & \le\int_{u_{1}^{*}}^{\bar{v}}r^{2}N_{v}^{(\mathcal{V}_{\bar{u}\bar{v}};l)}(u_{1}^{*},v)\,dv=\nonumber \\
 & =\int_{\max\{u_{1}^{*},\bar{v}-C(\bar{v}-\bar{u})\}}^{\bar{v}}r^{2}N_{v}^{(\mathcal{V}_{\bar{u}\bar{v}};l)}(u_{1}^{*},v)\,dv.\nonumber 
\end{align}

Using (\ref{eq:ImprovedUpperBoundDuR}), we can estimate
\begin{equation}
\sup_{\mathcal{D}_{u_{1}^{*}}^{u_{2}}\cap\{u\ge\bar{u}\}\cap\{v\le\bar{v}\}}r\le C(\bar{v}-\bar{u}).\label{eq:UpperBoudRSmallerRegion}
\end{equation}
Therefore, in view of (\ref{eq:ImprovedLowerBoundForR}), the bound
(\ref{eq:UpperBoudRSmallerRegion}) implies that any geodesic $\gamma$
in the support of $f$ intersecting the region $\{u\ge\bar{u}\}\cap\{v\le\bar{v}\}$
must necessarily have angular momentum $l$ bounded as follows (for
a possibly larger constant $C$): 
\begin{equation}
l\le C(\bar{v}-\bar{u}).
\end{equation}
Hence, 
\begin{equation}
\sup_{\mathcal{V}_{\bar{u}\bar{v}}}l\le C(\bar{v}-\bar{u}).\label{eq:BoundForAngularMomentum}
\end{equation}

In view of (\ref{eq:ExpressionMassImprovedR}) and the bounds (\ref{eq:BoundPvFinal})
and (\ref{eq:ImprovedBoundGammaU}), we can estimate for any $\bar{u}\in[u_{1}^{*},u_{2})$,
$\bar{v}\in(\bar{u},u_{2})$: 
\begin{align}
\tilde{m}(\bar{u},\bar{v}) & \le4\pi\int_{0}^{+\infty}\int_{\bar{u}}^{\bar{v}}\sup_{supp\big(f(\bar{u},v;\cdot,\cdot,l)\big)}\Big(\partial_{v}r(\bar{u},v)p^{v}-\partial_{u}r(\bar{u},v)p^{u}\Big)\cdot r^{2}N_{v}^{(l)}(\bar{u},v)\,dvldl\le\label{eq:ExpressionMassImprovedR-1}\\
 & \le C_{\delta_{0}}\int_{0}^{+\infty}\max\Big\{\big(\frac{l}{R_{*}}\big)^{-\frac{7\delta_{0}}{(1-2\delta_{0})^{2}}},1\Big\}\Big(\int_{\bar{u}}^{\bar{v}}r^{2}N_{v}^{(l)}(\bar{u},v)\,dv\Big)\,ldl.\nonumber 
\end{align}
In view of (\ref{eq:BoundReducedParticleCurrent}), (\ref{eq:UpperBoundDvR}),
(\ref{eq:ImprovedUpperBoundDuR}) (\ref{eq:FirstBoundN}) and (\ref{eq:BoundForAngularMomentum}),
we readily deduce from (\ref{eq:ExpressionMassImprovedR-1}) that,
for any $\bar{u}\in[u_{1}^{*},u_{2})$, $\bar{v}\in(\bar{u},u_{2})$:
\begin{align}
\tilde{m}(\bar{u},\bar{v}) & \le C_{\delta_{0}}\int_{0}^{\sup_{\mathcal{V}_{\bar{u}\bar{v}}}l}\max\Big\{\big(\frac{l}{R_{*}}\big)^{-\frac{7\delta_{0}}{(1-2\delta_{0})^{2}}},1\Big\}\Big(\int_{\max\{u_{1}^{*},\bar{v}-C(\bar{v}-\bar{u})\}}^{\bar{v}}r^{2}N_{v}^{(\mathcal{V}_{\bar{u}\bar{v}};l)}(u_{1}^{*},v)\,dv\Big)\,ldl\le\label{eq:FinalBoundSmallerParticleCurrent}\\
 & \le C_{\delta_{0}}(\bar{v}-\bar{u})\int_{0}^{C(\bar{v}-\bar{u})}\max\Big\{\big(\frac{l}{R_{*}}\big)^{-\frac{7\delta_{0}}{(1-2\delta_{0})^{2}}},1\Big\}\,ldl\le\nonumber \\
 & \le C_{\delta_{0}}(\bar{v}-\bar{u})^{3}\Big(\frac{\bar{v}-\bar{u}}{R_{*}}\Big)^{-\frac{7\delta_{0}}{(1-2\delta_{0})^{2}}}.\nonumber 
\end{align}
The estimate (\ref{eq:FinalBoundSmallerParticleCurrent}) yields,
in view of (\ref{eq:UpperBoundDvR}) and (\ref{eq:ImprovedLowerBoundDvR}),
the following improvement of (\ref{eq:FirstImprovedBoundMass}): 
\begin{equation}
\sup_{\mathcal{D}_{u_{1}^{*}}^{u_{2}}\backslash(u_{2},u_{2})}\frac{\tilde{m}}{r^{3-100\delta_{0}}}<+\infty.\label{eq:AlmostFinalImprovementMass}
\end{equation}

\medskip{}

\noindent \emph{Bounds in $L^{\infty}$ for $\frac{\tilde{m}}{r^{3}}$
and $T_{\mu\nu}$.} In view of (\ref{eq:AlmostFinalImprovementMass}),
we can estimate using the formula (\ref{eq:DerivativeTildeUMass}),
the bounds (\ref{eq:UpperBoundDvR}), (\ref{eq:ImprovedLowerBoundDvR})
and (\ref{eq:ImprovedUpperBoundDuR}): 
\begin{align}
\sup_{v\in[u_{1}^{*},u_{2}]}\int_{u_{1}^{*}}^{v}T_{uv}(u,v)\,du & \le C\sup_{v\in[u_{1}^{*},u_{2}]}\int_{u_{1}^{*}}^{v}\frac{-\partial_{u}\tilde{m}}{r^{2}}(u,v)\,du\le\label{eq:BoundTuvIntU}\\
 & \le C\sup_{v\in[u_{1}^{*},u_{2}]}\Big(\int_{u_{1}^{*}}^{v}\frac{\tilde{m}}{r^{3}}(u,v)\,du+\frac{\tilde{m}}{r^{2}}(u_{1}^{*},v)\Big)<\nonumber \\
 & <+\infty.\nonumber 
\end{align}
Similarly, using (\ref{eq:DerivativeTildeVMass}) and the bounds (\ref{eq:UpperBoundDvR}),
(\ref{eq:ImprovedLowerBoundDvR}), as well as the condition 
\[
\frac{\tilde{m}}{r^{2}}\Big|_{\gamma_{u_{1}}^{u_{2}}\backslash\{(u_{2},u_{2})\}}=0
\]
 (following from the smoothness of $(r,\Omega^{2},f)$ on $\mathcal{D}_{u_{1}^{*}}^{u_{2}}\backslash(u_{2},u_{2})$),
we can estimate:
\begin{equation}
\sup_{u\in[u_{1}^{*},u_{2})}\int_{u}^{u_{2}}T_{uv}(u,v)\,dv<+\infty.\label{eq:BoundTuvIntV}
\end{equation}

In view of (\ref{eq:ConstraintVFinal}), (\ref{eq:ConstraintUFinal}),
(\ref{eq:EquationRForProof}) and the fact that $r|_{\gamma_{u_{1}}^{u_{2}}}=0$
(and that $(r,\Omega^{2},f)$ is smooth on $\mathcal{D}_{u_{1}^{*}}^{u_{2}}\backslash\{(u_{2},u_{2})\}$),
we calculate on $\gamma_{u_{1}}^{u_{2}}\backslash\{(u_{2},u_{2})\}$:
\begin{gather*}
\partial_{v}^{2}r|_{\gamma_{u_{1}}^{u_{2}}\backslash\{(u_{2},u_{2})\}}=-\partial_{v}\Omega^{-2}|_{\gamma_{u_{1}}^{u_{2}}\backslash\{(u_{2},u_{2})\}}\partial_{v}r|_{\gamma_{u_{1}}^{u_{2}}\backslash\{(u_{2},u_{2})\}},\\
\partial_{u}^{2}r|_{\gamma_{u_{1}}^{u_{2}}\backslash\{(u_{2},u_{2})\}}=-\partial_{u}\Omega^{-2}|_{\gamma_{u_{1}}^{u_{2}}\backslash\{(u_{2},u_{2})\}}\partial_{u}r|_{\gamma_{u_{1}}^{u_{2}}\backslash\{(u_{2},u_{2})\}},\\
\partial_{u}\partial_{v}r|_{\gamma_{u_{1}}^{u_{2}}\backslash\{(u_{2},u_{2})\}}=0.
\end{gather*}
Thus, in view of the fact that $r|_{\gamma_{u_{1}}^{u_{2}}}=0$, we
readily infer that 
\begin{equation}
\partial_{u}\Omega^{2}|_{\gamma_{u_{1}}^{u_{2}}\backslash\{(u_{2},u_{2})\}}=\partial_{v}\Omega^{2}|_{\gamma_{u_{1}}^{u_{2}}\backslash\{(u_{2},u_{2})\}}.\label{eq:BoundaryConditionOmegaAxis}
\end{equation}
 By plugging (\ref{eq:AlmostFinalImprovementMass}), (\ref{eq:BoundTuvIntU})
and (\ref{eq:BoundTuvIntV}) in (\ref{eq:EquationOmegaForProof})
and using (\ref{eq:UpperBoundDvR}), (\ref{eq:ImprovedLowerBoundDvR}),
(\ref{eq:ImprovedUpperBoundDuR}) and the boundary condition (\ref{eq:BoundaryConditionOmegaAxis}),we
therefore obtain after integration in $u,v$, respectively: 
\begin{equation}
\sup_{\mathcal{D}_{u_{1}^{*}}^{u_{2}}\backslash(u_{2},u_{2})}\Big(\big|\partial_{v}\log\Omega^{2}\big|+\big|\partial_{u}\log\Omega^{2}\big|\Big)<+\infty.\label{eq:BoundForOmegaDerivatives}
\end{equation}

In view of the bounds (\ref{eq:ImprovedLowerBoundDvR}) and (\ref{eq:ImprovedUpperBoundDuR}),
there exists some $\bar{C}>0$ such that, for any $(u,v)\in\mathcal{D}_{u_{1}^{*}}^{u_{2}}\backslash(u_{2},u_{2})$:
\begin{equation}
\bar{C}\partial_{v}r+\partial_{u}r>0.\label{eq:GloballyOutgoingDirection}
\end{equation}

Let $\gamma:[0,b)\rightarrow\mathcal{D}_{u_{1}^{*}}^{u_{2}}\backslash(u_{2},u_{2})$
be a future directed null geodesic in the support of $f$. We will
establish that there exists some $C>0$ independent of $\gamma$,
so that 
\begin{equation}
\sup_{s\in[0,b)}\dot{\gamma}^{v}\le C.\label{eq:FinalBoundGammaV}
\end{equation}
It suffices to establish that, for any $s_{1}<s_{2}\in[0,b)$ such
that 
\begin{equation}
\dot{\gamma}^{v}(s)\ge1\text{ for all }s\in[s_{1},s_{2}],\label{eq:BigGamma}
\end{equation}
(\ref{eq:FinalBoundGammaV}) holds on $[s_{1},s_{2}]$, i.\,e.
\begin{equation}
\sup_{s\in[s_{1},s_{2}]}\dot{\gamma}^{v}\le C.\label{eq:FinalBoundGammaV-1}
\end{equation}
(note that if no such $s_{1},s_{2}$ exist, then (\ref{eq:FinalBoundGammaV})
automatically holds). Note that (\ref{eq:BigGamma}) implies the bound
\begin{equation}
s_{2}-s_{1}\le\int_{s_{1}}^{s_{2}}\dot{\gamma}^{v}\,ds\le v(s_{2})-v(s_{1})\le u_{2}-u_{1}^{*}.\label{eq:BoundParameters}
\end{equation}

Using equations (\ref{eq:NullGeodesicsSphericalSymmetry}), we calculate
for $\bar{C}$ as in (\ref{eq:GloballyOutgoingDirection}):
\begin{equation}
\frac{d}{ds}\Big(\bar{C}\Omega^{2}\dot{\gamma}^{u}+\Omega^{2}\dot{\gamma}^{v}\Big)(s)=\Big(\bar{C}\partial_{v}\log(\Omega^{2})+\partial_{u}\log(\Omega^{2})-2\frac{\bar{C}\partial_{v}r+\partial_{u}r}{r}\Big)\frac{l^{2}}{r^{2}}\Big|_{\gamma(s)}.\label{eq:AlmostThere...}
\end{equation}
Using (\ref{eq:ImprovedLowerBoundForR}), (\ref{eq:BoundsForOmega})
and (\ref{eq:GloballyOutgoingDirection}), we deduce that there exists
some $C'>0$ independent of $\gamma$ such that 
\begin{equation}
\frac{d}{ds}\Big(\bar{C}\Omega^{2}\dot{\gamma}^{u}+\Omega^{2}\dot{\gamma}^{v}\Big)(s)\le C'.\label{eq:Onemore....}
\end{equation}
Integrating (\ref{eq:Onemore....}) and using (\ref{eq:BoundsForOmega})
and (\ref{eq:BoundParameters}), we readily infer (\ref{eq:FinalBoundGammaV-1}).

It therefore follows from (\ref{eq:UpperBoundDvR}), (\ref{eq:ImprovedUpperBoundDuR}),
(\ref{eq:ImprovedBoundGammaU}) and (\ref{eq:Onemore....}) that there
exists some $C_{b}>0$ such that, for any $(u,v)\in\mathcal{D}_{u_{1}^{*}}^{u_{2}}\backslash(u_{2},u_{2})$:
\begin{equation}
supp\Big(f(u,v;\cdot,\cdot,\cdot)\Big)\subseteq\big\{-\partial_{u}r\cdot p^{u}+\partial_{v}r\cdot p^{v}\le C_{b}\big\}.\label{eq:FirstFinalBoundSupport}
\end{equation}
Since $f$ is supported on (\ref{eq:NullShellAngularMomentum}) and
any geodesic $\gamma$ in the support of $f$ satisfies (\ref{eq:ImprovedLowerBoundForR}),
the bound (\ref{eq:FirstFinalBoundSupport}) yields in view of (\ref{eq:DefinitionHawkingMass})
and (\ref{eq:BoundMass}): 
\begin{equation}
supp\Big(f(u,v;\cdot,\cdot,\cdot)\Big)\subseteq\big\{-\partial_{u}r\cdot p^{u}+\partial_{v}r\cdot p^{v}\le C_{b}\big\}\cap\big\{\frac{l}{r}\le4C_{b}\big\}.\label{eq:SecondFinalBoundSupport}
\end{equation}

It can be readily deduced from the estimate (\ref{eq:UpperBoundF})
combined with the bound (\ref{eq:SecondFinalBoundSupport}) for the
support of $f$ that 
\begin{equation}
\sup_{\mathcal{D}_{u_{1}^{*}}^{u_{2}}\backslash(u_{2},u_{2})}\Big(T_{uu}+T_{uv}+T_{vv}\Big)<+\infty.\label{eq:LInftyT}
\end{equation}
In view of (\ref{eq:DerivativeTildeVMass}), from (\ref{eq:LInftyT})
we also infer that 
\begin{equation}
\sup_{\mathcal{D}_{u_{1}^{*}}^{u_{2}}\backslash(u_{2},u_{2})}\frac{\tilde{m}}{r^{3}}<+\infty.\label{eq:LInftyMass}
\end{equation}
Note also that, in view of the relations (\ref{eq:ConstraintVFinal})\textendash (\ref{eq:ConstraintUFinal})
and (\ref{eq:EquationRForProof}), the bounds (\ref{eq:UpperBoundDvR}),
(\ref{eq:ImprovedLowerBoundDvR}), (\ref{eq:ImprovedUpperBoundDuR}),
(\ref{eq:BoundsForOmega}), (\ref{eq:BoundForOmegaDerivatives}),
(\ref{eq:LInftyT}) and (\ref{eq:LInftyMass}) imply that 
\begin{equation}
\sup_{\mathcal{D}_{u_{1}^{*}}^{u_{2}}\backslash(u_{2},u_{2})}\Big(|\partial_{v}^{2}r|+|\partial_{u}\partial_{v}r|+|\partial_{u}^{2}r|\Big)<+\infty.\label{eq:BoundsSecondOrderDerivativesR}
\end{equation}

\noindent \medskip{}
\emph{Proof of the smooth extendibility of $(r,\Omega^{2},f)$.} Let
\[
x:(\mathcal{D}_{u_{1}^{*}}^{u_{2}}\backslash\gamma_{u_{1}^{*}}^{u_{2}})\times\mathbb{S}^{2}\rightarrow\mathcal{V}\subset\mathbb{R}\times(\mathbb{R}^{3}\backslash\{0\})
\]
 be the diffeomorphism associated to the Cartesian coordinate chart
$(x^{0},\ldots,x^{3})$ defined by (\ref{eq:CartesianCoordinates}),
and let $\mathcal{N}$ denote the closure of $\mathcal{V}$ in $\mathbb{R}\times\mathbb{R}^{3}$.
Note that, in view of (\ref{eq:ZeroLimitR}), (\ref{eq:UpperBoundDvR}),
(\ref{eq:ImprovedLowerBoundDvR}), (\ref{eq:ImprovedUpperBoundDuR}),
(\ref{eq:BoundsForOmega}), (\ref{eq:BoundsSecondOrderDerivativesR})
and the assumption that $(r,\Omega^{2},f)$ is a smooth solution of
(\ref{eq:RequationFinal})\textendash (\ref{NullShellFinal}) on $\mathcal{D}_{u_{1}}^{u_{2}}\backslash\{(u_{2},u_{2})\}$
with smooth axis $\gamma_{u_{1}}^{u_{2}}\backslash\{(u_{2},u_{2})\}$,
we infer that, for any $\sigma\in\mathbb{S}^{2}$, $x(\cdot,\sigma)$
extends as a $C^{\infty}$ embedding on $\gamma_{u_{1}^{*}}^{u_{2}}\backslash\{(u_{2},u_{2})\}$
and as a $C^{1}$ embedding on $\gamma_{u_{1}^{*}}^{u_{2}}$, with
\[
x(\gamma_{u_{1}^{*}}^{u_{2}},\sigma)=(\mathbb{R}\times\{0\})\cap\mathcal{N},
\]
and the matrix $\mathcal{T}$ of the frame transformation 
\[
\mathcal{T}:\{\partial_{u},\partial_{v},r^{-1}\partial_{\theta},r^{-1}(\sin\theta)^{-1}\partial_{\varphi}\}\rightarrow\{\partial_{x^{0}},\partial_{x^{1}},\partial_{x^{2}},\partial_{x^{3}}\}
\]
satisfies 
\begin{equation}
\sup_{\mathcal{V}\backslash\{\theta=0,\pi\}}(||\mathcal{T}||+||\mathcal{T}^{-1}||)<+\infty,\label{eq:BoundedFrameTransformation}
\end{equation}
while the matrix $\mathcal{T}_{r}$ of the frame transformation 
\[
\mathcal{T}_{r}:\{\partial_{u},\partial_{v}\}\rightarrow\{\partial_{x^{0}},\frac{x^{i}}{\sqrt{\delta_{ij}x^{i}x^{j}}}\partial_{i}\}
\]
 satisfies 
\begin{equation}
||\mathcal{T}_{r}||_{C^{1}(\mathcal{N})}+||\mathcal{T}_{r}^{-1}||_{C^{1}(\mathcal{N})}<+\infty.\label{eq:RestrictedFrameTransformationBound}
\end{equation}
By a slight abuse of notation, we will also denote $(\mathbb{R}\times\{0\})\cap\mathcal{N}$
by $\gamma_{u_{1}^{*}}^{u_{2}}$. Let also define $q\in\mathcal{N}$
to be the point corresponding to $\{(u_{2},u_{2})\}$.

Let $g$ be the metric (\ref{eq:SphericallySymmetricMetric}) on $\mathcal{N}\backslash\gamma_{u_{1}^{*}}^{u_{2}}$.
Our assumption that $(r,\Omega^{2},f)$ is a smooth solution of (\ref{eq:RequationFinal})\textendash (\ref{NullShellFinal})
on $\mathcal{D}_{u_{1}}^{u_{2}}\backslash\{(u_{2},u_{2})\}$ with
smooth axis $\gamma_{u_{1}}^{u_{2}}\backslash\{(u_{2},u_{2})\}$ implies
that $g$ extends as a smooth metric on $\mathcal{N}\backslash q$,
i.\,e. in the Cartesian coordinates
\[
g_{\alpha\beta}\in C^{\infty}(\mathcal{N}\backslash p)
\]
 Furthermore, the bounds (\ref{eq:UpperBoundDvR}), (\ref{eq:ImprovedLowerBoundDvR}),
(\ref{eq:ImprovedUpperBoundDuR}), (\ref{eq:BoundsForOmega}), (\ref{eq:BoundForOmegaDerivatives}),
(\ref{eq:LInftyT}), (\ref{eq:LInftyMass}) and (\ref{eq:BoundsSecondOrderDerivativesR})
imply, after integrating equations (\ref{eq:EquationRForProof}) and
(\ref{eq:EquationOmegaForProof}) and using the boundary condition
(\ref{eq:BoundaryConditionOmegaAxis}) (as well as the relations (\ref{eq:ConstraintVFinal}),
(\ref{eq:ConstraintUFinal}) to express $\partial_{v}^{2}r$, $\partial_{u}^{2}r$
in terms of $\partial\Omega^{2}$), that the Cartesian components
of $g$ extend as $C^{1}$ functions in a neighborhood of $q$, i.\,e.
\begin{equation}
||g_{\alpha\beta}||_{C^{1}(\mathcal{N})}<+\infty.\label{eq:C1CartesianComponents}
\end{equation}
In view of the fact that $r,\Omega^{2}$ satisfy (\ref{eq:RequationFinal})\textendash (\ref{eq:ConstraintUFinal})
and 
\[
g^{\mu\nu}T_{\mu\nu}=0
\]
(since $f$ is supported on (\ref{eq:NullShellAngularMomentum})),
we readily calculate that for the Riemann curvature tensor of $g$
in the $(u,v,\theta,\varphi)$ coordinate chart:
\begin{align}
R_{uAuB} & =4\pi g_{AB}T_{uu},\label{eq:CurvatureComponent1}\\
R_{vAvB} & =4\pi g_{AB}T_{vv},\\
R_{uvuv} & =\Omega^{2}\partial_{u}\partial_{v}\log(\Omega^{2})=\big(\frac{\tilde{m}}{r^{3}}+\frac{1}{6}\Lambda\big)\Omega^{4}-16\pi\Omega^{2}T_{uv},\\
R_{uAvB} & =\Big(-\partial_{u}\partial_{v}\log(\Omega^{2})+4\pi T_{uv}\Big)g_{AB}=\\
 & =\Big(\big(\frac{\tilde{m}}{r^{3}}+\frac{1}{6}\Lambda\big)\Omega^{2}+20\pi T_{uv}\Big)g_{AB},\nonumber \\
R_{ABCD} & =2\Omega^{-1}\partial_{u}\partial_{v}\log(\Omega^{2})\big(g_{AC}g_{BD}-g_{AD}g_{BC}\big)=\label{eq:CurvatureComponent2}\\
 & =\Big(\big(\frac{2\tilde{m}}{r^{3}}+\frac{1}{3}\Lambda\big)\Omega^{4}-32\pi\Omega^{2}T_{uv}\Big)\big(g_{AC}g_{BD}-g_{AD}g_{BC}\big),\nonumber 
\end{align}
while all the other components are either identically $0$ or can
be expressed as a suitable linear combination of (\ref{eq:CurvatureComponent1})\textendash (\ref{eq:CurvatureComponent2}).
Thus, in view of the bounds (\ref{eq:BoundsForOmega}) and (\ref{eq:LInftyT})\textendash (\ref{eq:LInftyMass}),
we infer, after switching to the Cartesian coordinates $x^{\alpha}$
and using (\ref{eq:BoundedFrameTransformation}), that the components
$R_{\alpha\beta\gamma\delta}$ of the Riemann curvature tensor satisfy
\begin{equation}
\sup_{\mathcal{N}\backslash q}|R_{\alpha\beta\gamma\delta}|<+\infty.\label{eq:C0BoundsCurvature-1}
\end{equation}

Let $\mathcal{S}\subset\mathcal{N}\backslash q$ be the smooth spacelike
hypersurface defined by 
\[
\mathcal{S}\doteq\{x^{0}=\frac{1}{2}(u_{1}^{*}+u_{2})\}\cap\mathcal{N}.
\]
Note that $\mathcal{S}$ is a Cauchy hypersurface of $(\mathcal{N},g)$.
Let us also define the map $\mathcal{P}:\mathcal{C}_{null}^{+}(\mathcal{N}\backslash q)\times\mathbb{R}^{3+1}\rightarrow\mathcal{S}\times(\mathbb{R}^{3+1})^{3}$
(where $\mathcal{C}_{null}^{+}(\mathcal{N}\backslash q)\subset T(\mathcal{N}\backslash q)$
is the subset of future directed null vectors), so that, for each
triad $(y,u^{\alpha},z^{\alpha})$ with $y\in\mathcal{N}\backslash q$,
$U=u^{\alpha}\partial_{x^{\alpha}}$ being a future directed null
vector in $T_{y}\mathcal{N}$ and $Z=z^{\alpha}\partial_{x^{\alpha}}$,
we have 
\begin{equation}
\mathcal{P}(y,u^{\alpha},z^{\alpha})=(s,\bar{u}^{\alpha},v^{\alpha},w^{\alpha}),\label{eq:PMap}
\end{equation}
where, in the standard coordinates $(x^{\alpha},p^{\alpha})$ on $T\mathcal{N}$
corresponding to the Cartesian coordinates $x^{a}$:

\begin{itemize}

\item $s\doteq\mathcal{S}\cap\gamma_{U}$, where $\gamma_{U}$ is
the unique inextendible null geodesic emanating from $y$ in the direction
of $U$.

\item $\bar{u}^{\alpha}=\dot{\gamma}_{U}^{\alpha}(t_{f})$, where,
after normalising the affine parametrization of $\gamma_{U}$ so that
$\gamma_{U}(0)=y$ and $\dot{\gamma}_{U}^{\alpha}(0)=-u^{\alpha}$,
$t_{f}$ is defined so that $\gamma_{U}(t_{f})=s$.

\item $V=v^{\alpha}\partial_{x^{\alpha}}+w^{\alpha}\partial_{p^{\alpha}}$
is the unique vector in the fiber of $TT\mathcal{N}$ over $s$ such
that 
\[
(v^{\alpha},w^{\alpha}+\Gamma_{\beta\gamma}^{\alpha}v^{\beta}\dot{\gamma}_{U}^{\gamma})=(J^{a}(t_{f}),\frac{\nabla}{dt}J^{a}(t_{f})),
\]
where the vector field $J:[0,t_{f}]\rightarrow T_{\gamma_{U}}\mathcal{N}$
is a Jacobi field along $\gamma_{U}$, i.\,e.~satisfies 
\begin{equation}
\frac{\nabla^{2}}{dt^{2}}J^{\alpha}+R_{\hphantom{\alpha}\beta\gamma\delta}^{\alpha}\dot{\gamma}_{U}^{\gamma}\dot{\gamma}_{U}^{\delta}J^{\alpha}=0,\label{eq:JacobiEquation}
\end{equation}
 with initial conditions at $s$: 
\begin{gather}
J^{\alpha}(0)=z^{\alpha},\label{eq:InitialConditionsJacobi}\\
\frac{\nabla}{dt}J^{a}(0)=0.\nonumber 
\end{gather}

\end{itemize}

The Jacobi field $J$ with initial conditions (\ref{eq:InitialConditionsJacobi})
appearing in the definition (\ref{eq:PMap}) of $\mathcal{P}$ corresponds
to the infinitesimal variation of the geodesic $\gamma_{U}$ obtained
by varying the basepoint $\gamma_{U}(0)$ in the direction of $Z$,
while parallel tanslating the initial direction vector $\dot{\gamma}_{U}(0)$.
Thus, since the Vlasov field $f$ is conserved along the geodesic
flow, the following relation holds whenever $\mathcal{P}(y,u^{\alpha},z^{\alpha})=(s,v^{\alpha},w^{\alpha})$:
\begin{equation}
\big(z^{a}\partial_{x^{a}}+\Gamma_{\beta\gamma}^{\alpha}z^{\beta}u^{\gamma}\partial_{p^{\alpha}}\big)f(y,u^{a})=\big(v^{\alpha}\partial_{x^{\alpha}}+w^{\alpha}\partial_{p^{\alpha}}\big)f(s,\bar{u}^{\alpha}).\label{eq:DerivativePropagatedForF}
\end{equation}

In view of (\ref{eq:C1CartesianComponents}) and (\ref{eq:C0BoundsCurvature-1}),
by integrating the Jacobi field equation (\ref{eq:JacobiEquation})
and using the initial conditions (\ref{eq:InitialConditionsJacobi}),
we infer that there exists some $C>0$ such that the map (\ref{eq:PMap})
satisfies for any $y\in\mathcal{N}\backslash q$:
\begin{align}
||\bar{u}|| & \le C||u||,\label{eq:BoundednessP}\\
\Big(||v||+||w||\Big) & \le C||z||\cdot||u||,\nonumber 
\end{align}
where $||\cdot||$ denotes the standard Euclidean norm in $\mathbb{R}^{3+1}$.

The relation (\ref{eq:DerivativePropagatedForF}), combined with (\ref{eq:C1CartesianComponents}),
the bounds (\ref{eq:BoundednessP}), the form (\ref{eq:FAsASmoothDeltaFunction})
of the Vlasov field $f$ and the bound (\ref{eq:SecondFinalBoundSupport})
for the support of $f$, implies that there exists some $C>0$ such
that, for any $y\in\mathcal{N}\backslash q$, we can estimate in the
$(x^{\alpha},p^{\alpha})$ coordinate chart: 
\begin{equation}
\sum_{\alpha=0}^{3}\Big|\partial_{x^{\alpha}}\Big(\int_{\pi^{-1}(x)}||p||^{2}f(x,p)\,\sqrt{-\det g}dp\Big)\big|_{x=y}\Big|\le C\sup_{\bar{x}\in\mathcal{S}}\sum_{\alpha=0}^{3}\int_{\pi^{-1}(\bar{x})}\big(|\partial_{x^{a}}\bar{f}|+|\partial_{p^{\alpha}}\bar{f}|\big)(\bar{x},p)\cdot\delta(g_{\gamma\delta}(\bar{x})p^{\gamma}p^{\delta})\,\sqrt{-\det g}dp,\label{eq:BoundForDerivativeOfF}
\end{equation}
where $\pi:T\mathcal{N}\rightarrow\mathcal{N}$ is the natural projection.
Thus, (\ref{eq:BoundForDerivativeOfF}) and the fact that $(\mathcal{N},g)$
and $\bar{f}$ are smooth across $\mathcal{S}$ yields: 
\begin{equation}
\sup_{y\in\mathcal{N}\backslash q}\Big(\sum_{\alpha,\beta,\gamma=0}^{3}\big|\partial_{x^{\alpha}}T_{\beta\gamma}(y)\big|\Big)<+\infty.\label{eq:BoundDerivativeEnergyMomentumTensor}
\end{equation}

In view of the expression (\ref{eq:CartesianMetric}) for the Cartesian
components of $g$, using (\ref{eq:DerivativeTildeUMass}), (\ref{eq:DerivativeTildeVMass}),
(\ref{eq:UpperBoundDvR}), (\ref{eq:ImprovedLowerBoundDvR}), (\ref{eq:ImprovedUpperBoundDuR}),
(\ref{eq:BoundsForOmega}), (\ref{eq:BoundForOmegaDerivatives}) and
(\ref{eq:RestrictedFrameTransformationBound}), we can readily bound
for some $C>0$ and any $y\in\mathcal{V}$:
\begin{align}
\sum_{\alpha,\beta,\gamma,\delta=0}^{3}|\partial_{x^{\alpha}}\partial_{x^{\beta}}g_{\gamma\delta}|(y) & \le C\sum_{\alpha,\beta,\gamma=0}^{3}\Big(|\partial_{x^{\alpha}}\partial_{x^{\beta}}\log\Omega^{2}|+|\partial_{x^{\alpha}}\partial_{x^{\beta}}\partial_{u}r|+|\partial_{x^{\alpha}}\partial_{x^{\beta}}\partial_{v}r|+\label{eq:TrivialBoundSecondOrderTerms}\\
 & \hphantom{C\sum_{\alpha,\beta,\gamma=0}^{3}\Big(|\partial_{x^{\alpha}}}+\Big|\partial_{x^{\alpha}}(\frac{\partial_{u}r+\partial_{v}r}{r})\Big|+\Big|r^{-2}\Big(\frac{-4\partial_{u}r\partial_{v}r}{(\partial_{v}r-\partial_{u}r)^{2}}-1\Big)\Big|+\nonumber \\
 & \hphantom{C\sum_{\alpha,\beta,\gamma=0}^{3}\Big(|\partial_{x^{\alpha}}}+\frac{\tilde{m}}{r^{3}}+|T_{\alpha\beta}|+|\partial_{x^{\gamma}}T_{\alpha\beta}|+1\Big)(y).\nonumber 
\end{align}
In view of (\ref{eq:RestrictedFrameTransformationBound}), we infer
from (\ref{eq:TrivialBoundSecondOrderTerms}) for a possibly different
constant $C>0$: 
\begin{align}
\sum_{\alpha,\beta,\gamma,\delta=0}^{3}|\partial_{x^{\alpha}}\partial_{x^{\beta}}g_{\gamma\delta}|(y) & \le C\Big(|\partial_{u,v}^{2}\Omega^{2}|+|\partial_{u,v}^{3}r|+\Big|\partial_{u,v}^{1}(\frac{\partial_{u}r+\partial_{v}r}{r})\Big|+\Big|r^{-2}\Big(\frac{-4\partial_{u}r\partial_{v}r}{(\partial_{v}r-\partial_{u}r)^{2}}-1\Big)\Big|+\label{eq:TrivialBoundSecondOrderTerms-1}\\
 & \hphantom{\le C\Big(}+\frac{\tilde{m}}{r^{3}}+\sum_{\alpha,\beta,\gamma=0}^{3}\big(|T_{\alpha\beta}|+|\partial_{x^{\gamma}}T_{\alpha\beta}|\big)+1\Big)(y)\nonumber 
\end{align}
(where $\partial_{u,v}^{k}$ denotes a sum over all combinations of
$k$ derivatives of the form $\partial_{u}$ or $\partial_{v}$). 

Using the boundary condition (\ref{eq:BoundaryConditionRAxis}) for
$r$ on the axis $\gamma_{u_{1}^{*}}^{u_{2}}$, we readily infer that
\begin{equation}
(\partial_{u}+\partial_{v})r|_{\gamma_{u_{1}^{*}}^{u_{2}}\backslash\{(u_{2},u_{2})\}}=0.\label{eq:ZeroTangentialDerivativeR}
\end{equation}
In view of (\ref{eq:ZeroTangentialDerivativeR}), the bounds (\ref{eq:UpperBoundDvR}),
(\ref{eq:ImprovedLowerBoundDvR}) for $\partial_{v}r$ and (\ref{eq:BoundsSecondOrderDerivativesR})
for $\partial^{2}r$ imply, through an application of the mean value
theorem, that there exists some $C>0$ such that, for any $(u,v)\in\mathcal{D}_{u_{1}^{*}}^{u_{2}}\backslash\{(u_{2},u_{2})\}$:
\begin{equation}
\Big|\partial_{u,v}^{1}(\frac{\partial_{u}r+\partial_{v}r}{r})\Big|(u,v)\le C\big(1+\sup_{u<\bar{v}\le v}\big|\partial_{u,v}^{3}r|(u,\bar{v})\big)\label{eq:FirstBoundForAxisTerm}
\end{equation}
and 
\begin{equation}
\Big|r^{-2}\Big(\frac{-4\partial_{u}r\partial_{v}r}{(\partial_{v}r-\partial_{u}r)^{2}}-1\Big)\Big|(u,v)=\frac{1}{\big(\partial_{v}r-\partial_{u}r\big)^{2}}\Bigg|\frac{\partial_{v}r+\partial_{u}r}{r}\Bigg|^{2}(u,v)\le C(1+\sup_{u<\bar{v}\le v}\big|\partial_{u,v}^{2}r|^{2}(u,\bar{v}))\le C.\label{eq:SecondBoundAxisTerm}
\end{equation}
In view of (\ref{eq:FirstBoundForAxisTerm}) and (\ref{eq:SecondBoundAxisTerm}),
the bound (\ref{eq:TrivialBoundSecondOrderTerms-1}) can be simplified
as follows for any $y\in\mathcal{N}\backslash q$:
\begin{equation}
\sum_{\alpha,\beta,\gamma,\delta=0}^{3}|\partial_{x^{\alpha}}\partial_{x^{\beta}}g_{\gamma\delta}|(y)\le C\sup_{\bar{y}\in J^{-}(y)}\Big(|\partial_{u,v}^{2}\Omega^{2}|+|\partial_{u,v}^{3}r|+\frac{\tilde{m}}{r^{3}}+\sum_{\alpha,\beta,\gamma=0}^{3}\big(|T_{\alpha\beta}|+|\partial_{x^{\gamma}}T_{\alpha\beta}|\big)+1\Big)(\bar{y})\label{eq:TrivialBoundSecondOrderTerms-1-1}
\end{equation}

By differentiating equations (\ref{eq:ConstraintVFinal}), (\ref{eq:ConstraintUFinal})
and (\ref{eq:EquationRForProof}) with respect to $v,u$, we infer
that: 
\begin{align}
\partial_{v}^{3}r & =\Big(-4\pi\partial_{v}(rT_{vv})+\partial_{v}\log\Omega^{2}\partial_{v}^{2}r+\partial_{v}r\partial_{v}^{2}\log\Omega^{2}\Big),\label{eq:RelationsD2Omega}\\
\partial_{u}\partial_{v}^{2}r & =\frac{\tilde{m}}{r^{3}}\partial_{v}r\Omega^{2}-\frac{1}{4}\pi\big(T_{vv}\partial_{u}r+T_{uv}\partial_{v}r\big)+\frac{1}{6}\Lambda\partial_{v}r\Omega^{2}-\frac{1}{2}\frac{\tilde{m}-\frac{1}{3}\Lambda r^{3}}{r^{2}}\partial_{v}\Omega^{2}+4\pi\partial_{v}(rT_{uv})\nonumber 
\end{align}
and similarly with $u\leftrightarrow v$. Hence, in view of the bounds
(\ref{eq:UpperBoundDvR}), (\ref{eq:ImprovedLowerBoundDvR}), (\ref{eq:ImprovedUpperBoundDuR}),
(\ref{eq:BoundsForOmega}), (\ref{eq:BoundForOmegaDerivatives}) and
(\ref{eq:RestrictedFrameTransformationBound}) and the relations (\ref{eq:RelationsD2Omega}),
the estimate (\ref{eq:TrivialBoundSecondOrderTerms-1-1}) yields for
any $y\in\mathcal{N}\backslash q$: 
\begin{equation}
\sum_{\alpha,\beta,\gamma,\delta=0}^{3}|\partial_{x^{\alpha}}\partial_{x^{\beta}}g_{\gamma\delta}|(y)\le C\sup_{\bar{y}\in J^{-}(y)}\Big(|\partial_{u,v}^{2}\Omega^{2}|+\frac{\tilde{m}}{r^{3}}+\sum_{\alpha,\beta,\gamma=0}^{3}\big(|T_{\alpha\beta}|+|\partial_{x^{\gamma}}T_{\alpha\beta}|\big)+1\Big)(\bar{y}).\label{eq:SimplerBoundSecondOrder}
\end{equation}

Using the relations (\ref{eq:DerivativeTildeUMass}) and (\ref{eq:DerivativeTildeVMass})
for $\partial_{u}\tilde{m}$ and $\partial_{v}\tilde{m}$, respectively,
as well as the bounds (\ref{eq:UpperBoundDvR}), (\ref{eq:ImprovedLowerBoundDvR}),
(\ref{eq:ImprovedUpperBoundDuR}), (\ref{eq:BoundsForOmega}) and
(\ref{eq:BoundForOmegaDerivatives}), we can readily estimate for
any $(u,v)\in\mathcal{D}_{u_{1}^{*}}^{u_{2}}\backslash\{(u_{2},u_{2})\}$:
\begin{align}
\Big|\partial_{v}(\frac{\tilde{m}}{r^{3}})\Big|(u,v) & =8\pi\Big|\partial_{v}\Big(\frac{\int_{u}^{v}r^{2}\Omega^{-2}\big((-\partial_{u}r)T_{vv}+\partial_{v}rT_{uv}\big)(u,\bar{v})\,d\bar{v}}{r^{3}(u,v)}\Big)\Big|=\label{eq:DvDerivativeM.r^3}\\
 & =8\pi\Bigg|\frac{r^{2}\Omega^{-2}\big((-\partial_{u}r)T_{vv}+\partial_{v}rT_{uv}\big)(u,v)}{r^{3}(u,v)}-\nonumber \\
 & \hphantom{=8\pi\Bigg|}-3\frac{\int_{u}^{v}r^{2}\Omega^{-2}\big((-\partial_{u}r)T_{vv}+\partial_{v}rT_{uv}\big)(u,\bar{v})\,d\bar{v}}{r^{4}(u,v)}\partial_{v}r(u,v)\Bigg|=\nonumber \\
 & =8\pi\Bigg|\frac{r^{2}\Omega^{-2}\big((-\partial_{u}r)T_{vv}+\partial_{v}rT_{uv}\big)(u,v)}{r^{3}(u,v)}-\nonumber \\
 & \hphantom{=8\pi\Bigg|}-\frac{\int_{u}^{v}\partial_{v}(r^{3})\Omega^{-2}\big(\frac{(-\partial_{u}r)}{\partial_{v}r}T_{vv}+T_{uv}\big)(u,\bar{v})\,d\bar{v}}{r^{4}(u,v)}\partial_{v}r(u,v)\Bigg|=\nonumber \\
 & =8\pi\Bigg|\frac{\int_{u}^{v}r^{3}\partial_{v}\Big[\Omega^{-2}\big(\frac{(-\partial_{u}r)}{\partial_{v}r}T_{vv}+T_{uv}\big)\Big](u,\bar{v})\,d\bar{v}}{r^{4}(u,v)}\partial_{v}r(u,v)\Bigg|\le\nonumber \\
 & \le C\sup_{u\le\bar{v}\le v}\big(|\partial_{v}T_{vv}|+|\partial_{v}T_{uv}|+1\big).\nonumber 
\end{align}
Similarly, 
\begin{equation}
\Big|\partial_{u}(\frac{\tilde{m}}{r^{3}})\Big|(u,v)\le C\sup_{u\le\bar{u}\le v}\big(|\partial_{u}T_{uv}|+|\partial_{u}T_{uu}|+1\big).\label{eq:DuDerivativeM.r^3}
\end{equation}

By differentiating (\ref{eq:EquationOmegaForProof}) in $u,v$ and
then integrating the resulting relation in $v,u$, using the boundary
relation 
\begin{equation}
\partial_{v}^{2}\Omega^{2}|_{\gamma_{u_{1}^{*}}^{u_{2}}\backslash\{(u_{2},u_{2})\}}=\partial_{v}^{2}\Omega^{2}|_{\gamma_{u_{1}^{*}}^{u_{2}}\backslash\{(u_{2},u_{2})\}}
\end{equation}
(following by differentiating (\ref{eq:BoundaryConditionOmegaAxis})
in the direction tangent to $\gamma_{u_{1}^{*}}^{u_{2}}\backslash\{(u_{2},u_{2})\}$),
as well as the estimates (\ref{eq:DvDerivativeM.r^3})\textendash (\ref{eq:DuDerivativeM.r^3})
and the bounds (\ref{eq:UpperBoundDvR}), (\ref{eq:ImprovedLowerBoundDvR}),
(\ref{eq:ImprovedUpperBoundDuR}), (\ref{eq:BoundsForOmega}), (\ref{eq:BoundForOmegaDerivatives}),
(\ref{eq:LInftyT}) and (\ref{eq:LInftyMass}), we obtain from (\ref{eq:SimplerBoundSecondOrder}):
\begin{equation}
\sup_{y\in\mathcal{N}\backslash q}\Big(\sum_{\alpha,\beta,\gamma,\delta=0}^{3}|\partial_{x^{\alpha}}\partial_{x^{\beta}}g_{\gamma\delta}|(y)\Big)<+\infty.\label{eq:AtLAstSecondOrderBoundsG}
\end{equation}
In particular, in the Cartesian coordinates $x^{\alpha}$ on $\mathcal{N}\backslash q$,
$g$ extends as a $C^{1,1}$ metric in a neighborhood of $q$, i.\,e.
the following improvement of (\ref{eq:C1CartesianComponents}) holds:
\begin{equation}
||g_{\alpha\beta}||_{C^{1,1}(\mathcal{N})}<+\infty.\label{eq:C1,1extension}
\end{equation}

Arguing inductively, we can similarly show that, for any $k\ge2$,
if 
\begin{equation}
\sup_{\mathcal{D}_{u_{1}^{*}}^{u_{2}}\backslash\{(u_{2},u_{2})\}}\sum_{j=1}^{k-1}\big(|\partial_{u,v}^{j+3}r|+|\partial_{u,v}^{j+1}\Omega^{2}|+|\partial_{u,v}^{j}T_{vv}|+|\partial_{u,v}^{j}T_{uv}|+|\partial_{u,v}^{j}T_{uu}|\big)<+\infty\label{eq:InductionCoefficients}
\end{equation}
and 
\begin{equation}
||g_{\alpha\beta}||_{C^{k-1,1}(\mathcal{N})}<+\infty,\label{eq:InductionMetric}
\end{equation}
then 
\begin{equation}
\sup_{\mathcal{D}_{u_{1}^{*}}^{u_{2}}\backslash\{(u_{2},u_{2})\}}\big(|\partial_{u,v}^{k+3}r|+|\partial_{u,v}^{k+1}\Omega^{2}|+|\partial_{u,v}^{k}T_{vv}|+|\partial_{u,v}^{k}T_{uv}|+|\partial_{u,v}^{k}T_{uu}|\big)<+\infty\label{eq:InductionCoefficients-1}
\end{equation}
and 
\begin{equation}
||g_{\alpha\beta}||_{C^{k,1}(\mathcal{N})}<+\infty\label{eq:InductionMetric-1}
\end{equation}
(note that, for $k=2$, (\ref{eq:InductionCoefficients})\textendash (\ref{eq:InductionMetric})
follow immediately from the (\ref{eq:C1,1extension}) and our proof
that the right hand side of (\ref{eq:TrivialBoundSecondOrderTerms})
is bounded). The proof of (\ref{eq:InductionCoefficients-1})\textendash (\ref{eq:InductionMetric-1})
can be achieved as follows: 

\begin{enumerate}

\item From the expressions (\ref{eq:CurvatureComponent1})\textendash (\ref{eq:CurvatureComponent2})
for the Riemann curvature components, the bounds (\ref{eq:UpperBoundDvR}),
(\ref{eq:ImprovedLowerBoundDvR}), (\ref{eq:ImprovedUpperBoundDuR}),
(\ref{eq:BoundsForOmega}), (\ref{eq:BoundForOmegaDerivatives}),
(\ref{eq:LInftyT}) and (\ref{eq:LInftyMass}), the inductive bounds
(\ref{eq:InductionCoefficients})\textendash (\ref{eq:InductionMetric})
and the estimate 
\begin{equation}
\sup_{\mathcal{D}_{u_{1}^{*}}^{u_{2}}\backslash\{(u_{2},u_{2})\}}\sum_{j=0}^{k-1}\Big|\partial_{u,v}^{j}(\frac{\tilde{m}}{r^{3}})\Big|\le C_{k}\sup_{\mathcal{D}_{u_{1}^{*}}^{u_{2}}\backslash\{(u_{2},u_{2})\}}\sum_{j=0}^{k-1}\big(|\partial_{u,v}^{j}T_{vv}|+|\partial_{u,v}^{j}T_{vu}|+|\partial_{u,v}^{j}T_{uu}|+1\big)\label{eq:Derivativem/r^3-1}
\end{equation}
(which is the higher order analogue of (\ref{eq:DvDerivativeM.r^3})\textendash (\ref{eq:DuDerivativeM.r^3}),
following similarly from (\ref{eq:DerivativeTildeUMass}), (\ref{eq:DerivativeTildeVMass}),
(\ref{eq:UpperBoundDvR}), (\ref{eq:ImprovedLowerBoundDvR}), (\ref{eq:ImprovedUpperBoundDuR}),
(\ref{eq:BoundsForOmega}), (\ref{eq:BoundForOmegaDerivatives}),
(\ref{eq:LInftyT}), (\ref{eq:LInftyMass}), (\ref{eq:InductionCoefficients}),
(\ref{eq:InductionMetric}) and the Taylor expansion formula), we
can readily bound in the Cartesian coordinates $x^{\alpha}$:
\begin{equation}
\sup_{\mathcal{N}\backslash q}\sum_{j=0}^{k-1}\sum_{|\lambda|=j}|\partial_{x}^{\lambda}R_{\alpha\beta\gamma\delta}|<+\infty.\label{eq:HigherOrderCurvature}
\end{equation}

\item By differentiating the Jacobi field equation (\ref{eq:JacobiEquation})
and using (\ref{eq:InductionMetric}) and (\ref{eq:HigherOrderCurvature}),
we can readily obtain the following higher order analogue of (\ref{eq:BoundDerivativeEnergyMomentumTensor})
(inferred by a similar argument):
\begin{equation}
\sup_{\mathcal{N}\backslash q}\Big(\sum_{|\lambda|=k}\big|\partial_{x}^{\lambda}T_{\alpha\beta}\big|\Big)<+\infty.\label{eq:BoundDerivativeEnergyMomentumTensorHigher}
\end{equation}

\item By differentiating equations (\ref{eq:ConstraintVFinal}),
(\ref{eq:ConstraintUFinal}), (\ref{eq:EquationRForProof}) and (\ref{eq:EquationOmegaForProof})
and arguing similarly as for the proof of (\ref{eq:AtLAstSecondOrderBoundsG})
(using (\ref{eq:InductionCoefficients}), (\ref{eq:InductionMetric}),
(\ref{eq:Derivativem/r^3-1}) and (\ref{eq:BoundDerivativeEnergyMomentumTensorHigher})),
we readily infer (\ref{eq:InductionCoefficients-1})\textendash (\ref{eq:InductionMetric-1}).

\end{enumerate}

We therefore deduce that the metric $g$ on $\mathcal{N}\backslash q$
admits a $C^{\infty}$ extension on $q$. The smooth extendibility
of the whole solution $(\mathcal{N}\backslash q,g;f)$ in a neighborhood
of $q$ then follows readily.
\end{proof}

\subsection{\label{subsec:Extension-principles-away}Extension principles away
from $r=0$}

In this section, we will establish two extension principles for smooth
solutions of (\ref{eq:RequationFinal})\textendash (\ref{NullShellFinal}):
One which is valid in the region where $r$ is bounded away from $0,\infty$,
and one along $\mathcal{I}$.

The next extension principle is a straightforward modification of
a more general extension principle for solutions to the Einstein\textendash \emph{massive
}Vlasov system obtained in \cite{DafermosRendall} (see Proposition
3.1 in \cite{DafermosRendall}):

\begin{customprop}{5.1}[Smooth extension away from $r=0,\infty$]\label{prop:ExtensionPrincipleMihalis}

For any $u_{1}<u_{2}$, $v_{1}<v_{2}$ and $\Lambda\in\mathbb{R}$,
let $(r,\Omega^{2},f)$ be any solution of (\ref{eq:RequationFinal})\textendash (\ref{NullShellFinal})
on an open neighborhood $\mathcal{V}$ of 
\[
\mathcal{R}\doteq[u_{1},u_{2}]\times[v_{1},v_{2}]\backslash\{(u_{2},v_{2})\}
\]
(see Figure \ref{fig:GeneralExtensionDomain}) which is smooth according
to Definition \ref{def:SmoothnessGenerally} and satisfies 
\begin{equation}
\inf_{\mathcal{V}}r>0,\label{eq:LowerBoundRGiven}
\end{equation}
\begin{equation}
\sup_{\mathcal{V}}r<+\infty,\label{eq:UpperBoundRGiven}
\end{equation}
\begin{equation}
\sup_{\mathcal{V}}\tilde{m}<+\infty,\label{eq:UpperBoundMass}
\end{equation}
\begin{equation}
\sup_{(\{u_{1}\}\times[v_{1},v_{2}])\cup([u_{1},u_{2}]\times\{v_{1}\})}\partial_{u}r<0,\label{eq:InitialNegativeDerivativeInU}
\end{equation}
and, for some $C<+\infty$:
\begin{equation}
supp\Big(f(u_{1},\cdot;\cdot)\Big),supp\Big(f(\cdot,v_{1};\cdot)\Big)\subseteq\big\{\Omega^{2}(p^{v}+p^{u})\le C\big\}.\label{eq:CompactSupportF}
\end{equation}
Then, $(r,\Omega^{2},f)$ extends smoothly in a neighborhood of $\{(u_{2},v_{2})\}$.

\end{customprop}

\begin{figure}[h] 
\centering 
\scriptsize
\begingroup%
  \makeatletter%
  \providecommand\color[2][]{%
    \errmessage{(Inkscape) Color is used for the text in Inkscape, but the package 'color.sty' is not loaded}%
    \renewcommand\color[2][]{}%
  }%
  \providecommand\transparent[1]{%
    \errmessage{(Inkscape) Transparency is used (non-zero) for the text in Inkscape, but the package 'transparent.sty' is not loaded}%
    \renewcommand\transparent[1]{}%
  }%
  \providecommand\rotatebox[2]{#2}%
  \newcommand*\fsize{\dimexpr\f@size pt\relax}%
  \newcommand*\lineheight[1]{\fontsize{\fsize}{#1\fsize}\selectfont}%
  \ifx\svgwidth\undefined%
    \setlength{\unitlength}{150bp}%
    \ifx\svgscale\undefined%
      \relax%
    \else%
      \setlength{\unitlength}{\unitlength * \real{\svgscale}}%
    \fi%
  \else%
    \setlength{\unitlength}{\svgwidth}%
  \fi%
  \global\let\svgwidth\undefined%
  \global\let\svgscale\undefined%
  \makeatother%
  \begin{picture}(1,0.8)%
    \lineheight{1}%
    \setlength\tabcolsep{0pt}%
    \put(0,0){\includegraphics[width=\unitlength,page=1]{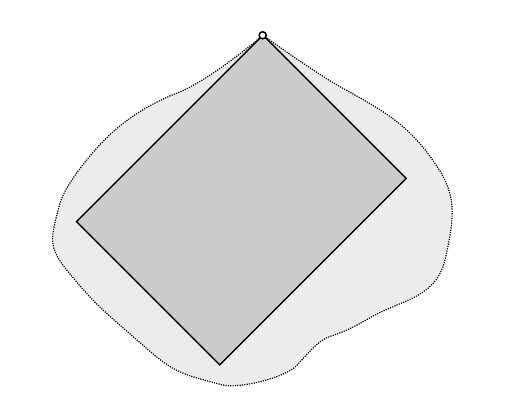}}%
    \put(0.42698015,0.40346552){\color[rgb]{0,0,0}\makebox(0,0)[lt]{\lineheight{1.25}\smash{\begin{tabular}[t]{l}$\mathcal{R}$\end{tabular}}}}%
    \put(0.74504945,0.27227665){\color[rgb]{0,0,0}\makebox(0,0)[lt]{\lineheight{1.25}\smash{\begin{tabular}[t]{l}$\mathcal{V}$\end{tabular}}}}%
    \put(0.36138611,0.76608874){\color[rgb]{0,0,0}\makebox(0,0)[lt]{\lineheight{1.25}\smash{\begin{tabular}[t]{l}$(u_2,v_2)$\end{tabular}}}}%
    \put(0.3072989,0.45186546){\color[rgb]{0,0,0}\rotatebox{45}{\makebox(0,0)[lt]{\lineheight{1.25}\smash{\begin{tabular}[t]{l}$u=u_2$\end{tabular}}}}}%
    \put(0.56843757,0.16844973){\color[rgb]{0,0,0}\rotatebox{45}{\makebox(0,0)[lt]{\lineheight{1.25}\smash{\begin{tabular}[t]{l}$u=u_1$\end{tabular}}}}}%
    \put(0.19334163,0.25170876){\color[rgb]{0,0,0}\rotatebox{-45}{\makebox(0,0)[lt]{\lineheight{1.25}\smash{\begin{tabular}[t]{l}$v=v_1$\end{tabular}}}}}%
    \put(0.5510147,0.59948097){\color[rgb]{0,0,0}\rotatebox{-45}{\makebox(0,0)[lt]{\lineheight{1.25}\smash{\begin{tabular}[t]{l}$v=v_2$\end{tabular}}}}}%
  \end{picture}%
\endgroup%
 
\caption{Schematic depiction of the domains $\mathcal{R}$ and $\mathcal{V}$ appearing in the statement of Proposition~\ref{prop:ExtensionPrincipleMihalis}. \label{fig:GeneralExtensionDomain}}
\end{figure}
\begin{proof}
In view of the smoothness of $r,\Omega^{2}$ in a neighborhood of
$\{u_{1}\}\times[v_{1},v_{2}]$, by integrating in the $u$ direction
the following inequality: 
\[
\partial_{u}\Big(\frac{\Omega^{2}}{-\partial_{u}r}\Big)\le0
\]
(which follows readily from (\ref{eq:ConstraintUFinal}) and the fact
that $T_{uu}\ge0$), we can bound 
\begin{equation}
\sup_{\mathcal{R}}\frac{\Omega^{2}}{-\partial_{u}r}<+\infty.\label{eq:TrivialBoundKappa}
\end{equation}
Note that, in particular, in view of the smoothness of $r,\Omega^{2}$
in $\mathcal{R}$ and the initial sign condition (\ref{eq:NegativeDuRInD}),
the bound (\ref{eq:TrivialBoundKappa}) implies that 
\begin{equation}
\partial_{u}r<0\text{ on }\mathcal{R}.
\end{equation}
Moreover, in view of the smoothness of $r,\Omega^{2}$ in a neighborhood
of $[u_{1},u_{2}]\times\{v_{1}\}$, the initial sign condition (\ref{eq:InitialNegativeDerivativeInU})
and the given bounds (\ref{eq:LowerBoundRGiven})\textendash (\ref{eq:UpperBoundMass})
on $\mathcal{R}$, by integrating in the $v$ direction the following
inequality: 
\begin{equation}
\partial_{v}\log(-\partial_{u}r)\le\frac{\tilde{m}-\frac{1}{3}\Lambda r^{3}}{2r^{2}}\frac{\Omega^{2}}{-\partial_{u}r}
\end{equation}
(which follows readily from (\ref{eq:RequationFinal}), (\ref{eq:DefinitionHawkingMass}),
(\ref{eq:RenormalisedHawkingMass}) and the fact that $T_{uv}\ge0$),
using also the upper bound (\ref{eq:TrivialBoundKappa}), we estimate:
\begin{equation}
\sup_{\mathcal{R}}(-\partial_{u}r)<+\infty.\label{eq:TrivialBoundDuR}
\end{equation}
From (\ref{eq:TrivialBoundKappa}) and (\ref{eq:TrivialBoundDuR}),
we therefore deduce that 
\begin{equation}
\sup_{\mathcal{R}}\Omega^{2}<+\infty.\label{eq:TrivialBoundOmega}
\end{equation}

The bounds (\ref{eq:LowerBoundRGiven}), (\ref{eq:UpperBoundRGiven}),
(\ref{eq:CompactSupportF}) and (\ref{eq:TrivialBoundOmega}) allow
us to apply the proof of Proposition 3.1 in \cite{DafermosRendall}
without any change (except for replacing the massive mass shell relation
(26) in \cite{DafermosRendall} with its massless analogue (\ref{eq:NullShellAngularMomentum}),
which does not affect the proof). As a result, we obtain the required
smooth extendibility of $(r,\Omega^{2},f)$ in a neighborhood of $\{(u_{2},v_{2})\}$
(see also the comment at the end of the proof of Proposition 3.1 in
\cite{DafermosRendall} on how to obtain upgrade $C^{2}$ bounds on
the extension into $C^{\infty}$ bounds). 
\end{proof}
For any $v_{1}<v_{2}\in\mathbb{R}$, let us set
\begin{equation}
\mathcal{V}_{v_{1}}^{v_{2}}\doteq\big([v_{1},v_{2}]\times[v_{1},v_{2}]\big)\cap\big\{ u\ge v\big\}\subset\mathbb{R}^{2}\label{eq:DomainNearInfinity}
\end{equation}
and 
\begin{equation}
\mathcal{I}_{v_{1}}^{v_{2}}\doteq\big([v_{1},v_{2}]\times[v_{1},v_{2}]\big)\cap\big\{ u=v\big\}\subset\partial\mathcal{V}_{v_{1}}^{v_{2}}.
\end{equation}
The next extension principle will concern the smooth extendibility
of solutions to (\ref{eq:RequationFinal})\textendash (\ref{NullShellFinal})
in neighborhoods of conformal infinity $\mathcal{I}$:

\begin{customprop}{5.2}[Smooth extension along $r=\infty$]\label{prop:ExtensionPrincipleInfinity}

For any $v_{1}<v_{2}$ and $\Lambda<0$, let $(r,\Omega^{2},f)$ be
any solution of (\ref{eq:RequationFinal})\textendash (\ref{NullShellFinal})
on $\mathcal{V}\cap\{u>v\}$, where $\mathcal{V}$ is an open neighborhood
of $\mathcal{V}_{v_{1}}^{v_{2}}\backslash\{(v_{2},v_{2})\}$ (see
Figure \ref{fig:InfinityExtensionDomain}), such that $(r,\Omega^{2},f)$
is smooth with smooth conformal infinity $\mathcal{I}_{v_{1}}^{v_{2}}\backslash\{(v_{2},v_{2})\}$,
according to Definition \ref{def:SmoothnessConformalInfinity}, and
in addition, $f$ satisfies the reflecting boundary condition on $\mathcal{I}_{v_{1}}^{v_{2}}\backslash\{(v_{2},v_{2})\}$,
according to Definition \ref{def:ReflectingBoundaryCondition}. Assume,
moreover, that 
\begin{equation}
\sup_{\mathcal{V}_{v_{1}}^{v_{2}}\backslash\mathcal{I}_{v_{1}}^{v_{2}}}\frac{2m}{r}<1\label{eq:AssumptionNoTrapped}
\end{equation}
and, for some $C<+\infty$,
\begin{equation}
supp\Big(f(\cdot,v_{1};\cdot)\Big)\subseteq\big\{\Omega^{2}(p^{v}+p^{u})\le C\big\}.\label{eq:CompactSupportF-1}
\end{equation}
Then, $(r,\Omega^{2},f)$ extends as a smooth solution of (\ref{eq:RequationFinal})\textendash (\ref{NullShellFinal})
with smooth conformal infinity to a larger open set $\widetilde{\mathcal{V}}\cap\{u>v\}$
such that $(v_{2},v_{2})\in\widetilde{\mathcal{V}}$. 

\end{customprop}

\begin{figure}[h] 
\centering 
\scriptsize
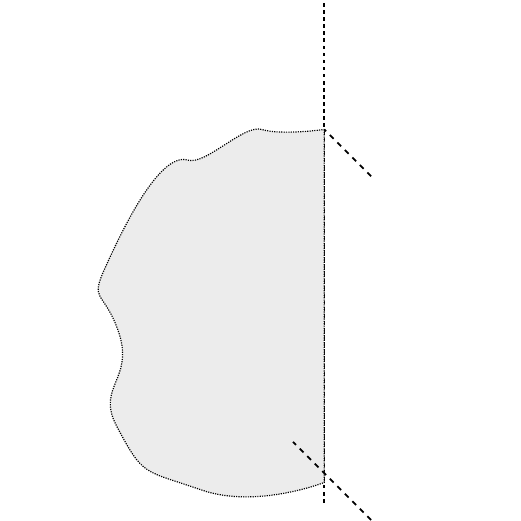 
\caption{Schematic depiction of the domains $\mathcal{V}_{u_{1}}^{u_{2}}$ and $\mathcal{V}\cap \{u>v\}$ appearing in the statement of Proposition~\ref{prop:ExtensionPrincipleInfinity}. \label{fig:InfinityExtensionDomain}}
\end{figure}
\begin{proof}
Since $(r,\Omega^{2},f)$ is smooth with smooth conformal infinity
$\mathcal{I}_{v_{1}}^{v_{2}}\backslash\{(v_{2},v_{2})\}$, it is necessary
that 
\begin{equation}
\partial_{v}(\frac{1}{r})\Big|_{\mathcal{I}_{v_{1}}^{v_{2}}\backslash\{(v_{2},v_{2})\}}<0\text{ and }\partial_{u}(\frac{1}{r})\Big|_{\mathcal{I}_{v_{1}}^{v_{2}}\backslash\{(v_{2},v_{2})\}}>0.\label{eq:SignsInfinity}
\end{equation}
Since $\Omega^{2}$ is smooth on $\mathcal{V}_{v_{1}}^{v_{2}}\backslash\mathcal{I}_{v_{1}}^{v_{2}}$,
the relation (\ref{eq:DefinitionHawkingMass}), assumption (\ref{eq:AssumptionNoTrapped})
and the inequalities (\ref{eq:SignsInfinity}) on $\mathcal{I}_{v_{1}}^{v_{2}}\backslash\{(v_{2},v_{2})\}$
imply that, everywhere on $\mathcal{V}_{v_{1}}^{v_{2}}\backslash\mathcal{I}_{v_{1}}^{v_{2}}$:
\begin{equation}
\partial_{u}r<0<\partial_{v}r.\label{eq:SignR}
\end{equation}
The sign condition (\ref{eq:SignR}) implies that $r$ is strictly
positive in the interior of $\mathcal{V}_{v_{1}}^{v_{2}}\backslash\mathcal{I}_{v_{1}}^{v_{2}}$.
By possibly choosing a slightly larger $v_{1}$, we will assume without
loss of generality that 
\begin{equation}
\inf_{\mathcal{V}_{v_{1}}^{v_{2}}\backslash\mathcal{I}_{v_{1}}^{v_{2}}}r>0.\label{eq:LowerBoundR}
\end{equation}

The bound (\ref{eq:CompactSupportF-1}) implies, in view of the relation
(\ref{eq:NullShellAngularMomentum}), the fact that $r^{-2}\Omega^{2}$
extens smoothly on $\mathcal{I}_{v_{1}}^{v_{2}}\backslash\{(v_{2},v_{2})\}$
and the conservation of angular momentum for null geodesics, that
there exists some $L_{0}>0$ such that, for any $(u,v)\in\mathcal{V}_{v_{1}}^{v_{2}}\backslash\mathcal{I}_{v_{1}}^{v_{2}}$
\begin{equation}
supp\Big(f(u,v;\cdot)\Big)\subset\big\{ l\le L_{0}\big\}.\label{eq:SupportBoundedAngularMomentum}
\end{equation}
In view of (\ref{eq:CompactSupportF-1}) and (\ref{eq:SupportBoundedAngularMomentum}),
the relations (\ref{eq:SphericallySymmetricComponentsEnergyMomentum})
imply that 
\begin{equation}
\sup_{[v_{1},v_{2}]\times\{v_{1}\}}\big(r^{2}(T_{uu}+T_{uv}+T_{vv})\big)<+\infty.\label{eq:InitialBoundEnergyMomentum}
\end{equation}
Hence, in view of (\ref{eq:DerivativeTildeUMass}), (\ref{eq:SignR})
and (\ref{eq:AssumptionNoTrapped}), we infer that 
\begin{equation}
\limsup_{u\rightarrow v_{1}^{+}}\tilde{m}(u,v_{1})<+\infty.\label{eq:UpperBoundMassInitially}
\end{equation}
The relations (\ref{eq:SignR}) and (\ref{eq:AssumptionNoTrapped})
imply, in view of (\ref{eq:DerivativeTildeUMass})\textendash (\ref{eq:DerivativeTildeVMass}),
that 
\[
\partial_{u}\tilde{m}\le0\le\partial_{v}\tilde{m}
\]
 and hence, in view of (\ref{eq:ConservedMassI}), (\ref{eq:UpperBoundMassInitially})
and the fact that $r,\Omega^{2}$ are smooth on $[v_{1},v_{2}]\times\{v_{1}\}$,
we infer that
\begin{align}
\sup_{\mathcal{V}_{v_{1}}^{v_{2}}\backslash\mathcal{I}_{v_{1}}^{v_{2}}}\tilde{m} & <+\infty,\label{eq:BoundedMass}\\
\inf_{\mathcal{V}_{v_{1}}^{v_{2}}\backslash\mathcal{I}_{v_{1}}^{v_{2}}}\tilde{m} & >-\infty.\nonumber 
\end{align}
Furthermore, in view of (\ref{eq:RenormalisedHawkingMass}), (\ref{eq:AssumptionNoTrapped})
and (\ref{eq:BoundedMass}), there exists some $C>0$ such that 
\begin{equation}
\sup_{\mathcal{V}_{v_{1}}^{v_{2}}\backslash\mathcal{I}_{v_{1}}^{v_{2}}}\Big|\log\Big(\frac{1-\frac{1}{3}\Lambda r^{2}}{1-\frac{2m}{r}}\Big)\Big|<+\infty.\label{eq:BoundForComparisonm/r}
\end{equation}

We can readily estimate in view of (\ref{eq:DerivativeTildeUMass})\textendash (\ref{eq:DerivativeTildeVMass}),
(\ref{eq:AssumptionNoTrapped}) , (\ref{eq:SignR}), (\ref{eq:LowerBoundR})
and (\ref{eq:BoundedMass}):
\begin{align}
\sup_{\bar{u}\in[v_{1},v_{2})}\int_{\{u=\bar{u}\}\cap\mathcal{V}_{v_{1}}^{v_{2}}\backslash\mathcal{I}_{v_{1}}^{v_{2}}}\frac{4\pi rT_{vv}}{\partial_{v}r}\,dv & \le\sup_{\bar{u}\in[v_{1},v_{2})}\int_{\{u=\bar{u}\}\cap\mathcal{V}_{v_{1}}^{v_{2}}\backslash\mathcal{I}_{v_{1}}^{v_{2}}}\frac{2\partial_{v}\tilde{m}}{r\cdot(1-\frac{2m}{r})}\,dv\le\label{eq:BoundForConstraintVRightHandSide}\\
 & \le C\sup_{\bar{u}\in[v_{1},v_{2})}\int_{\{u=\bar{u}\}\cap\mathcal{V}_{v_{1}}^{v_{2}}\backslash\mathcal{I}_{v_{1}}^{v_{2}}}\partial_{v}\tilde{m}\,dv<\nonumber \\
 & <+\infty\nonumber 
\end{align}
and, similarly 
\begin{equation}
\sup_{\bar{v}\in[v_{1},v_{2})}\int_{\{v=\bar{v}\}\cap\mathcal{V}_{v_{1}}^{v_{2}}\backslash\mathcal{I}_{v_{1}}^{v_{2}}}\frac{4\pi rT_{uu}}{-\partial_{u}r}\,du<+\infty.\label{eq:BoundForConstraintURightHandSide}
\end{equation}
Integrating (\ref{eq:DerivativeInUDirectionKappa}), (\ref{eq:DerivativeInVDirectionKappaBar})
in $u,v$ respectively, and using (\ref{eq:BoundForConstraintVRightHandSide})\textendash (\ref{eq:BoundForConstraintURightHandSide}),
as well as (\ref{eq:RenormalisedHawkingMass}), (\ref{eq:BoundMass})
and the boundary condition (\ref{eq:BoundaryConditionRInfinity})
on $\mathcal{I}_{v_{1}}^{v_{2}}\backslash\{(v_{2},v_{2})\}$, we infer
that 
\begin{equation}
\sup_{\mathcal{V}_{v_{1}}^{v_{2}}\backslash\mathcal{I}_{v_{1}}^{v_{2}}}\Big(\Big|\log\big(\frac{\partial_{v}r}{1-\frac{2m}{r}}\big)\Big|+\Big|\log\big(\frac{-\partial_{u}r}{1-\frac{2m}{r}}\big)\Big|\Big)<+\infty.\label{eq:BoundForKappas}
\end{equation}
The bound (\ref{eq:BoundForKappas}) implies, in view of (\ref{eq:DerivativeTildeUMass})\textendash (\ref{eq:DerivativeTildeVMass}),
that there exists some $C>0$ such that for any $(u,v)\in\mathcal{V}_{v_{1}}^{v_{2}}\backslash\mathcal{I}_{v_{1}}^{v_{2}}$:
\begin{equation}
r^{2}T_{uv}(u,v)\le C\min\big\{\partial_{v}\tilde{m},-\partial_{u}\tilde{m}\big\}
\end{equation}
and hence, after integrating over curves of the form $u=const$, $v=const$
and using (\ref{eq:BoundedMass}):
\begin{equation}
\sup_{\bar{u}\in[v_{1},v_{2})}\int_{u=\bar{u}}r^{2}T_{uv}\,dv+\sup_{\bar{v}\in[v_{1},v_{2})}\int_{v=\bar{v}}r^{2}T_{uv}\,du<+\infty.\label{eq:BoundForTuv}
\end{equation}

Let us consider the renormalised quantities $(\rho,\widetilde{\Omega}^{2},\tau_{\mu\nu})$,
defined by (\ref{eq:RenormalisedQuantities}) and satisfying (\ref{eq:RenormalisedEquations}).
In view of (\ref{eq:CompactSupportF-1}) and the fact that $(r,\Omega^{2},f)$
is smooth on $\mathcal{V}_{v_{1}}^{v_{2}}\backslash\mathcal{I}_{v_{1}}^{v_{2}}$
with smooth conformal infinity $\mathcal{I}_{v_{1}}^{v_{2}}\backslash\{(v_{2},v_{2})\}$,
the quantities $(\rho,\widetilde{\Omega}^{2},\tau_{\mu\nu})$ extend
smoothly on $\mathcal{I}_{v_{1}}^{v_{2}}\backslash\{(v_{2},v_{2})\}$,
with $\widetilde{\Omega}^{2}\Big|_{\mathcal{I}_{v_{1}}^{v_{2}}\backslash\{(v_{2},v_{2})\}}>0$.
The proof of Proposition \ref{prop:ExtensionPrincipleInfinity} will
thus follow by showing that the renormalised quantitiesalso extend
smoothly on $(v_{2},v_{2})$, with 
\begin{equation}
\widetilde{\Omega}^{2}(v_{2},v_{2})>0.\label{eq:PositiveOmegaTilde}
\end{equation}

In view of (\ref{eq:DefinitionHawkingMass}), (\ref{eq:BoundForComparisonm/r})
and (\ref{eq:BoundForKappas}), we immediately infer that 
\begin{equation}
\sup_{\mathcal{V}_{v_{1}}^{v_{2}}\backslash\{(v_{2},v_{2})\}}\Big|\log(\widetilde{\Omega}^{2})\Big|<+\infty,\label{eq:BoundOmegaTilde}
\end{equation}
 as well as 
\begin{equation}
\sup_{\mathcal{V}_{v_{1}}^{v_{2}}\backslash\{(v_{2},v_{2})\}}\Big(\big|\log\big(\partial_{v}\rho\big)\big|+\big|\log\big(-\partial_{u}\rho\big)\big|\Big)<+\infty.\label{eq:BoundDerRho}
\end{equation}
Integrating (\ref{eq:RenormalisedEquations}) in $u,v$ and using
(\ref{eq:BoundedMass}), (\ref{eq:BoundForTuv}) and (\ref{eq:BoundOmegaTilde}),
we also obtain: 
\begin{equation}
\sup_{\mathcal{V}_{v_{1}}^{v_{2}}\backslash\{(v_{2},v_{2})\}}\Big(|\partial_{v}\log(\widetilde{\Omega}^{2})|+|\partial_{u}\log(\widetilde{\Omega}^{2})|\Big).\label{eq:DerivativesOmegaTilde}
\end{equation}
The bounds (\ref{eq:BoundOmegaTilde})\textendash (\ref{eq:DerivativesOmegaTilde})
imply that $\rho,\widetilde{\Omega}^{2}$ admit a $C^{0,1}$ extension
on $(v_{2},v_{2})$, satisfying (\ref{eq:PositiveOmegaTilde}).

Using the identity appearing in the first line of (\ref{eq:UsefulRelationForGeodesic1})
(as well as the analogous identity obtained after switching the roles
of $u,v$) for any null geodesic $\gamma:I\rightarrow\mathcal{V}_{v_{1}}^{v_{2}}\backslash\mathcal{I}_{v_{1}}^{v_{2}}$
in the support of $f$ with non-vanishing angular momentum, possibly
extended to its reflection off $\mathcal{I}_{v_{1}}^{v_{2}}$ (according
to Definition \ref{def:Reflection}), we infer, in view of the assumption
(\ref{eq:CompactSupportF-1}) for the support of $f$ initially, that,
for some absolute constant $C>0$:
\begin{equation}
\sup_{s\in\mathcal{I}}\Big(\Omega^{2}\big(\dot{\gamma}^{v}+\dot{\gamma}^{u}\big)(s)\Big)\le C\exp\Big(C\sup_{\mathcal{V}_{v_{1}}^{v_{2}}\backslash\mathcal{I}_{v_{1}}^{v_{2}}}\big(\big|\partial_{v}\log\Omega^{2}-2\frac{\partial_{v}r}{r}\big|+\big|\partial_{u}\log\Omega^{2}-2\frac{\partial_{u}r}{r}\big|\big)\Big).\label{eq:BoundGeodesicsFromChristoffelPrecise}
\end{equation}
Since the definition (\ref{eq:RenormalisedQuantities}) of $\widetilde{\Omega}^{2}$
implies that
\begin{equation}
\partial_{v}\log\Omega^{2}-2\frac{\partial_{v}r}{r}=\partial_{v}\log\widetilde{\Omega}^{2}+O(r^{-3})\partial_{v}r,
\end{equation}
the estimate (\ref{eq:BoundGeodesicsFromChristoffelPrecise}) yields
(in view of the bounds (\ref{eq:BoundDerRho})\textendash (\ref{eq:DerivativesOmegaTilde}))
for a possibly different absolute constant $C>0$: 
\begin{equation}
\sup_{s\in\mathcal{I}}\Big(\Omega^{2}\big(\dot{\gamma}^{v}+\dot{\gamma}^{u}\big)(s)\Big)\le C.\label{eq:FinalBoundGeodesics}
\end{equation}
Hence, there exists some $C_{*}>0$ such that, for any $(u,v)\in\mathcal{V}_{v_{1}}^{v_{2}}\backslash\mathcal{I}_{v_{1}}^{v_{2}}$:
\begin{equation}
supp\Big(f(u,v;\cdot)\Big)\subset\Big\{\Omega^{2}(p^{u}+p^{v})\le C_{*}\Big\}.\label{eq:BoundSupportFEverywhere}
\end{equation}
Arguing similarly as for the proof of (\ref{eq:InitialBoundEnergyMomentum}),
from (\ref{eq:BoundSupportFEverywhere}) we therefore infer that:
\begin{equation}
\sup_{\mathcal{V}_{v_{1}}^{v_{2}}\backslash\{(v_{2},v_{2})\}}\big(\tau_{uu}+\tau_{uv}+\tau_{vv}\big)<+\infty.\label{eq:InitialBoundEnergyMomentum-1}
\end{equation}

Commuting equations (\ref{eq:RenormalisedEquations}) and (\ref{eq:VlasovFinal})
with $\partial_{u},\partial_{v}$ and arguing inductively (using (\ref{eq:BoundOmegaTilde}),
(\ref{eq:BoundDerRho}), (\ref{eq:DerivativesOmegaTilde}), (\ref{eq:BoundSupportFEverywhere})
and (\ref{eq:InitialBoundEnergyMomentum-1}) as a basis for the induction),
we similarly obtain the following higher order analogues of (\ref{eq:BoundOmegaTilde}),
(\ref{eq:BoundDerRho}), (\ref{eq:DerivativesOmegaTilde}) and (\ref{eq:InitialBoundEnergyMomentum-1})
for any $k\in\mathbb{N}$: 
\begin{equation}
\sup_{\mathcal{V}_{v_{1}}^{v_{2}}\backslash\{(v_{2},v_{2})\}}\Big(\big|\partial_{u,v}^{k}\log(\widetilde{\Omega}^{2})\big|+\big|\partial_{u,v}^{k}\rho\big|+\big|\partial_{u,v}^{k-1}\tau\big|\Big)<+\infty.
\end{equation}
Hence, the smooth extension of $(\rho,\widetilde{\Omega}^{2},\tau)$
on $\{(v_{2},v_{2})\}$ follows readily.
\end{proof}

\subsection{A general extension principle for domains of outer communications }

In this section, we will obtain, as a corollary of Theorem \ref{thm:ExtensionPrinciple}
and Propositions \ref{prop:ExtensionPrincipleMihalis} and \ref{prop:ExtensionPrincipleInfinity},
a general extension principle for asymptotically AdS, smooth solutions
of (\ref{eq:RequationFinal})\textendash (\ref{NullShellFinal}) which
coincide with their domain of outer communications.
\begin{cor}
\label{cor:GeneralContinuationCriterion} For any $v_{\mathcal{I}},u_{1}>0$,
let $(r,\Omega^{2},f)$ be a smooth solution of (\ref{eq:RequationFinal})\textendash (\ref{NullShellFinal})
on $\mathcal{U}_{u_{1};v_{\mathcal{I}}}$, with smooth axis $\gamma_{u_{1}}$
and smooth conformal infinity $\mathcal{I}_{u_{1}}$ (for the relevant
notation, see (\ref{eq:GeneralDomain})\textendash (\ref{eq:ConformalInfinityGeneral})
and Definitions \ref{def:SmoothnessAxis}\textendash \ref{def:SmoothnessGenerally}).
Assume, moreover, that $(r,\Omega^{2},f)$ satisfies 
\begin{equation}
\sup_{\mathcal{U}_{u_{1};v_{\mathcal{I}}}}\frac{2m}{r}<1,\label{eq:NoTrappingContinuity}
\end{equation}
\begin{equation}
\limsup_{(u,v)\rightarrow(u_{1},u_{1})}\frac{2\tilde{m}}{r}\le\delta_{0},\label{eq:NoConcentrationContinuity}
\end{equation}
where $\delta_{0}$ is the constant appearing in the statement of
Theorem \ref{thm:ExtensionPrinciple}, as well as the initial bound
\begin{equation}
supp\Big(f(0,\cdot;\cdot)\Big)\subset\Big\{\Omega^{2}(p^{u}+p^{v})\le C\Big\}\label{eq:CompactSupportFContinuity}
\end{equation}
(for some $C<+\infty$). Then, there exists some $\bar{u}_{1}>u_{1}$,
such that $(r,\Omega^{2},f)$ extends on the whole of $\mathcal{U}_{\bar{u}_{1};v_{\mathcal{I}}}$
as a smooth solution of (\ref{eq:RequationFinal})\textendash (\ref{NullShellFinal})
with smooth axis $\gamma_{\bar{u}_{1}}$ and smooth conformal infinity
$\mathcal{I}_{\bar{u}_{1}}$.
\end{cor}
\begin{proof}
Using the condition 
\[
\partial_{v}\big(\Omega^{-2}\partial_{v}r)\le0
\]
(following readily from the constraint equation (\ref{eq:ConstraintVFinal}))
and the fact that $r|_{\mathcal{I}^{+}}=+\infty$, we readily infer
that 
\begin{equation}
\partial_{v}r>0\text{ on }\mathcal{U}_{u_{1};v_{\mathcal{I}}}.
\end{equation}
As a consequence of the assumption (\ref{eq:NoTrappingContinuity})
and the relation (\ref{eq:DefinitionHawkingMass}) between $\Omega^{2}$,
$\partial_{v}r$ and $\partial_{u}r$, we therefore infer that 
\begin{equation}
\partial_{u}r<0\text{ on }\mathcal{U}_{u_{1};v_{\mathcal{I}}}\label{eq:NoWhiteHoleContinuity}
\end{equation}
and hence the condition (\ref{eq:SignR}) holds everywhere on $\mathcal{U}_{u_{1};v_{\mathcal{I}}}$.
By arguing exactly as in the proof of (\ref{eq:BoundedMass}) in Proposition
\ref{prop:ExtensionPrincipleInfinity}, using (\ref{eq:CompactSupportFContinuity})
and the fact that $(r,\Omega^{2},f)$ is smooth on the axis $\gamma_{u_{1}}$,
we infer that
\begin{equation}
0\le\inf_{\mathcal{U}_{u_{1};v_{\mathcal{I}}}}\tilde{m}\le\sup_{\mathcal{U}_{u_{1};v_{\mathcal{I}}}}\tilde{m}<+\infty.\label{eq:BoundsForMassContinuity}
\end{equation}

By applying Theorem \ref{thm:ExtensionPrinciple}, we immediately
obtain that $(r,\Omega^{2},f)$ extends as a smooth solution of (\ref{eq:RequationFinal})\textendash (\ref{NullShellFinal})
with smooth axis on a neighborhood of $(u_{1},u_{1})$. Then, applying
Propostion \ref{prop:ExtensionPrincipleMihalis} combined with a simple
continuity argument, we infer that $(r,\Omega^{2},f)$ extends as
a smooth solution of (\ref{eq:RequationFinal})\textendash (\ref{NullShellFinal})
in an open neighborhood of $\{u_{1}\}\times[u_{1},u_{1}+v_{\mathcal{I}})$.
Finally, by applying Proposition \ref{prop:ExtensionPrincipleInfinity},
we deduce that $(r,\Omega^{2},f)$ extends as a smooth solution of
(\ref{eq:RequationFinal})\textendash (\ref{NullShellFinal}) with
smooth conformal infinity on $\mathcal{V}\cap\{v<u+v_{\mathcal{I}}\}$,
where $\mathcal{V}$ is an open neighborhood of $(u_{1},u_{1}+v_{\mathcal{I}})$.
Hence, for some $\bar{u}_{1}>u_{1}$ sufficiently close to $u_{1}$,
$(r,\Omega^{2},f)$ extends as a smooth solution of (\ref{eq:RequationFinal})\textendash (\ref{NullShellFinal})
on the whole of $\mathcal{U}_{\bar{u}_{1};v_{\mathcal{I}}}$, with
smooth axis $\gamma_{\bar{u}_{1}}$ and smooth conformal infinity
$\mathcal{I}_{\bar{u}_{1}}$.
\end{proof}

\section{\label{sec:Cauchy_Stability_Low_Regularity}A Cauchy stability statement
in a low regularity topology }

In this section, we will introduce a low regularity, scale invariant
norm $||\cdot||$ on the space of smoothly compatible, asymptotically
AdS initial data sets $(r_{/},\Omega_{/}^{2},f_{/};v_{\mathcal{I}})$,
as introduced in Section \ref{subsec:Asymptotically-AdS-Initial-Data}.
This norm will measure the ``concentration of energy'' along the
evolution of a \emph{free} Vlasov field $f$ on $(\mathcal{M}_{AdS},g_{AdS})$
determined from $f_{/}$ by an explicit formula in a renormalised
gauge. We will then proceed to establish a Cauchy stability statement
for the trivial solution $(r_{AdS},\Omega_{AdS}^{2},0)$ of (\ref{eq:RequationFinal})\textendash (\ref{NullShellFinal})
in the initial data topology defined by $||\cdot||$: We will show
that, for any fixed retarded time $U_{*}\ge0$ and for any initial
data set $\mathcal{S}$ with $||\mathcal{S}||$ sufficiently small,
the geometry of the corresponding maximal development $(\mathcal{M},g;f)$
with a reflecting boundary condition on $\mathcal{I}$ exists for
sufficiently long time and remains close to that of $(\mathcal{M}_{AdS},g_{AdS};0)$
in the fixed retarded time interval $u\in[0,U_{*}]$. 

The results of this section allow addressing the AdS instability conjecture
for the system (\ref{eq:RequationFinal})\textendash (\ref{NullShellFinal})
in the low regularity setting of $||\cdot||$ and will be crucial
for the results of our companion paper \cite{MoschidisVlasov} 

\subsection{\label{subsec:TheRoughTopology} A low regularity norm on the space
of initial data}

In this section, we will introduce a low regularity norm on the space
of smoothly compatible, asymptotically AdS initial data sets for (\ref{eq:RequationFinal})\textendash (\ref{NullShellFinal}).
In order to simplify our notations, we will adopt the following definition:
\begin{defn}
\noindent \label{def:Bspace}We will denote with $\mathcal{B}$ the
set of smoothly compatible, asymptotically AdS initial data sets $(r_{/},\Omega_{/}^{2},\bar{f}_{/};v_{\mathcal{I}})$
for (\ref{eq:RequationFinal})\textendash (\ref{NullShellFinal})
which are bounded in phase space, in accordance with Definitions \ref{def:AsymptoticallyAdSData}
and \ref{def:CompatibilityCondition}. For any $(r_{/},\Omega_{/}^{2},\bar{f}_{/};v_{\mathcal{I}})\in\mathcal{B}$,
we will denote with $(V(v);\frac{dU}{du}(0))$ the parameters of the
unique gauge normalising gauge transformation provided by Lemma \ref{lem:SmoothToNorm},
i.\,e.~a gauge transformation $(r_{/},\Omega_{/}^{2},\bar{f}_{/};v_{\mathcal{I}})\rightarrow(r_{/}^{\prime},(\Omega_{/}^{\prime})^{2},\bar{f}_{/}^{\prime};v_{\mathcal{I}})$
defined by the relations \ref{eq:GeneralGaugeTransformationInitialData}
such that the transformed initial data set $(r_{/}^{\prime},(\Omega_{/}^{\prime})^{2},\bar{f}_{/}^{\prime};v_{\mathcal{I}})$
satisfies the normalisation condition (\ref{eq:GaugeConditionNormalisedData}).
\end{defn}
We will define a map from $\mathcal{B}$ to the space of smooth solution
of the (free) massless Vlasov equation (\ref{eq:Vlasov}) on AdS spacetime
as follows:
\begin{defn}
\label{def:ComparisonVlasovField} For any $(r_{/},\Omega_{/}^{2},\bar{f}_{/};v_{\mathcal{I}})\in\mathcal{B}$,
let $\bar{f}_{/}^{(AdS)}:[0,\sqrt{-\frac{3}{\Lambda}}\pi)\times[0,+\infty)^{2}\rightarrow[0,+\infty)$
be given by the expression
\begin{equation}
\bar{f}_{/}^{(AdS)}(v;\,p^{u},l)\doteq\bar{f}_{/}^{\prime}(\frac{\sqrt{-\frac{3}{\Lambda}}\pi}{v_{\mathcal{I}}}\cdot v;\,p^{u},l),\label{eq:RescaledAdSVlasov}
\end{equation}
where\textbf{ $\bar{f}_{/}^{\prime}$ }is the initial Vlasov field
in the gauge normalised expression $(r_{/}^{\prime},(\Omega_{/}^{\prime})^{2},\bar{f}_{/}^{\prime};v_{\mathcal{I}})$
of $(r_{/},\Omega_{/}^{2},\bar{f}_{/};v_{\mathcal{I}})$ provided
by Lemma \ref{lem:SmoothToNorm} (see Definition \ref{def:Bspace}). 

We will define $f^{(AdS)}:T\mathcal{M}_{AdS}\rightarrow[0,+\infty)$
to be the unique solution of the massless Vlasov equation (\ref{eq:Vlasov})
on $(\mathcal{M}_{AdS},g_{AdS})$ with initial conditions on $u=0$
corresponding to $\bar{f}_{/}^{(AdS)}$. In particular, denoting with
$\Omega_{AdS}^{2}(u,v)$, $r_{AdS}(u,v)$ the coefficients of the
AdS metric (\ref{eq:AdSMetricValues-1}), the Vlasov field  $f^{(AdS)}$
is expressed as 
\[
f^{(AdS)}(u,v;p^{u},p^{v},l)=\bar{f}^{(AdS)}(u,v;p^{u},p^{v},l)\cdot\delta\Big(\Omega_{AdS}^{2}(u,v)p^{u}p^{v}-\frac{l^{2}}{r_{AdS}^{2}(u,v)}\Big)
\]
 for some smooth function $\bar{f}^{(AdS)}$ satisfying the initial
condition
\begin{equation}
\bar{f}^{(AdS)}(0,v;p^{u},\frac{l^{2}}{p^{u}\cdot\Omega_{AdS}^{2}r_{AdS}^{2}(0,v)},l)=\bar{f}_{/}^{(AdS)}(v;\,p^{u},l).
\end{equation}
For any $\bar{u}\ge0$ and $\bar{v}\in(\bar{u},\bar{u}+\sqrt{-\frac{3}{\Lambda}}\pi)$,
we will also set
\begin{equation}
\Big[\frac{rT_{vv}}{\partial_{v}r}\Big]^{(AdS)}(\bar{u},\bar{v})\doteq\frac{r_{AdS}T_{vv}[f^{(AdS)}]}{\partial_{v}r_{AdS}}(\bar{u},\bar{v})\label{eq:VlasovEnergyConcentration}
\end{equation}
(and similarly for $\Big[\frac{rT_{uv}}{-\partial_{u}r}\Big]^{(AdS)}$,
$\Big[\frac{rT_{uv}}{\partial_{v}r}\Big]^{(AdS)}$ and $\Big[\frac{rT_{uu}}{-\partial_{u}r}\Big]^{(AdS)}$),
where the energy momentum components $T_{\alpha\beta}[f^{(AdS)}]$
are defined using the relations (\ref{eq:SphericallySymmetricComponentsEnergyMomentum})
(with $\Omega_{AdS}^{2}$, $r_{AdS}$ in place of $\Omega^{2}$, $r$).
\end{defn}
\begin{rem*}
For $u\ge0$, $v\in(u,u+\sqrt{-\frac{3}{\Lambda}}\pi)$, $p^{u},p^{v}\ge0$
and $l\ge0$, the free Vlasov field $f^{(AdS)}(u,v;p^{u},p^{v},l)$
is expressed in terms of $\bar{f}_{/}^{(AdS)}$ through the explicit
relations (\ref{eq:FAsASmoothDeltaFunction-1})\textendash (\ref{eq:ExplicitSolutionVlasovAdS})
with $\bar{f}^{(AdS)}$ in place of $F$ in (\ref{eq:ExplicitSolutionVlasovAdS})
(for various useful identities regarding the geodesic flow on AdS
spacetime, see Section \ref{sec:Geodesic-flow-AdS} of the Appendix). 
\end{rem*}
The following definition introduces a non-negative functional $||\cdot||$
on the space $\mathcal{B}$ which can be used to measure the distance
from the trivial initial data set $\mathcal{S}_{AdS}$: 
\begin{defn}
\label{def:InitialDataNorm}For any $(r_{/},\Omega_{/}^{2},\bar{f}_{/};v_{\mathcal{I}})\in\mathcal{B}$,
we will define:
\begin{align}
||(r_{/},\Omega_{/}^{2},\bar{f}_{/};v_{\mathcal{I}})||\doteq\sup_{U_{*}\ge0} & \int_{U_{*}}^{U_{*}+\sqrt{-\frac{3}{\Lambda}}\pi}\Big(\Big[\frac{rT_{vv}}{\partial_{v}r}\Big]^{(AdS)}(U_{*},v)+\Big[\frac{rT_{uv}}{-\partial_{u}r}\Big]^{(AdS)}(U_{*},v)\Big)\,dv+\label{eq:InitialDataNorm}\\
+ & \sup_{V_{*}\ge0}\int_{\max\{0,V_{*}-\sqrt{-\frac{3}{\Lambda}}\pi\}}^{V_{*}}\Big(\Big[\frac{rT_{uu}}{-\partial_{u}r}\Big]^{(AdS)}(u,V_{*})+\Big[\frac{rT_{uv}}{\partial_{v}r}\Big]^{(AdS)}(u,V_{*})\Big)\,du+\nonumber \\
 & +\sqrt{-\Lambda}\tilde{m}_{/}|_{v=v_{\mathcal{I}}}.\nonumber 
\end{align}
\end{defn}
\begin{rem*}
In view of the periodicity of the null geodesic flow on the spherical
quotient of AdS spacetime (see Section \ref{sec:Geodesic-flow-AdS}
of the Appendix), the value of (\ref{eq:InitialDataNorm}) does not
change if one restricts to the supremum over $U_{*}\in[0,\sqrt{-\frac{3}{\Lambda}}\pi]$
and $V_{*}\in[0,2\sqrt{-\frac{3}{\Lambda}}\pi]$ in the right hand
side of (\ref{eq:InitialDataNorm}). Therefore, the condition that
the elements of $\mathcal{B}$ have bounded support in phase space
(i.\,e.~satisfy (\ref{eq:BoundedSupportDefinition})) implies that
$||\mathcal{S}||$ is finite for any $\mathcal{S}\in\mathcal{B}$. 

The norm $||(r_{/},\Omega_{/}^{2},\bar{f}_{/};v_{\mathcal{I}})||$
vanishes if and only if $f^{(AdS)}\equiv0$, i.\,e.~if $\bar{f}_{/}\equiv0$.
In this case, $(r_{/},\Omega_{/}^{2},0;v_{\mathcal{I}})$ is mapped
through the gauge transformation provided by Lemma \ref{lem:SmoothToNorm}
to the rescaled normalised trivial data $(r_{AdS/}^{(v_{\mathcal{I}})},\big(\Omega_{AdS/}^{(v_{\mathcal{I}})}\big)^{2},0;v_{\mathcal{I}})$,
given by (\ref{eq:AdSMetricValuesRescaled1}).

We should also point out that the quantity (\ref{eq:InitialDataNorm})
is both gauge invariant (i.\,e.~invariant under coordinate transformations
of the form (\ref{eq:GaugeCoordinateChange})\textendash (\ref{eq:GeneralGaugeTransformationInitialData}))
and scale invariant, i.\,e.~invariant under transformations of $(r_{/},\Omega_{/}^{2};\bar{f}_{/})$
of the form 
\begin{gather*}
r_{/}(v)\rightarrow\lambda^{-1}r_{/}(\lambda v),\\
\Omega_{/}^{2}(v)\rightarrow\Omega_{/}^{2}(\lambda v),\\
\bar{f}_{/}(v;p^{v},l)\rightarrow\lambda^{2}(\lambda')^{4}\bar{f}_{/}(\lambda v;\lambda'p^{v},\lambda\cdot\lambda'l),\\
\Lambda\rightarrow\lambda^{2}\Lambda,
\end{gather*}
for any $\lambda,\lambda'>0$. The scale invariance of $||\cdot||$
is used in a fundamental way in the constructions of our companion
paper \cite{MoschidisVlasov}.
\end{rem*}
Any smooth solution $(r,\Omega^{2};f)$ of (\ref{eq:RequationFinal})\textendash (\ref{NullShellFinal})
on a domain $\mathcal{U}_{u_{1};v_{\mathcal{I}}}$ of the form (\ref{eq:GeneralDomain}),
with smooth axis $\{u=v\}$ and smooth conformal infinity $\{u=v-v_{\mathcal{I}}\}$,
induces a smoothly compatible initial data set $(r_{/u_{*}},\Omega_{/u_{*}}^{2},\bar{f}_{/u_{*}};v_{\mathcal{I}})$
on slices of the form $\{u=u_{*}\}\cap\mathcal{U}_{u_{1};v_{\mathcal{I}}}$
for any $u_{*}\in(0,u_{1})$, where
\begin{equation}
(r_{/u_{*}},\Omega_{/u_{*}}^{2})(\bar{v})\doteq(r,\Omega^{2})(u_{*},u_{*}+\bar{v})\label{eq:SliceData}
\end{equation}
and 
\begin{equation}
\bar{f}_{/u_{*}}(\bar{v};p,l)=\bar{f}(u_{*},u_{*}+\bar{v};p,\frac{l^{2}}{\Omega^{2}r^{2}|_{(u_{*},u_{*}+\bar{v})}p},l).\label{eq:SliceF}
\end{equation}
As a result, $||\cdot||$ can be used to measure the ``size'' of
a solution $(r,\Omega^{2};f)$ at time $u=u_{*}$:
\begin{defn}
\label{def:NormSlice} Let $(r,\Omega^{2};f)$, $\mathcal{U}_{u_{1};v_{\mathcal{I}}}$
and $u_{*}\in[0,u_{1})$ be as above, with $f$ of bounded support
in phase space. We will define the norm on the initial data induced
by $(r,\Omega^{2};f)$ on the slice $\{u=u_{*}\}\cap\mathcal{U}_{u_{1};v_{\mathcal{I}}}$
by the relation 
\begin{equation}
||(r,\Omega^{2};f)|_{u=u_{*}}||\doteq||\big(r_{/u_{*}},\Omega_{/u_{*}}^{2},\bar{f}_{/u_{*}};v_{\mathcal{I}}\big)||\label{eq:InitialDataNormSlice}
\end{equation}
where $||\cdot||$ is defined by (\ref{eq:InitialDataNorm}) and $r_{/u_{*}},\Omega_{/u_{*}}^{2},\bar{f}_{/u_{*}}$
are given by (\ref{eq:SliceData})\textendash (\ref{eq:SliceF}).
\end{defn}

\subsection{\label{subsec:CauchyStability} Cauchy stability of AdS in the low
regularity topology}

In this section, we will establish a Cauchy stability statement for
AdS spacetime $(\mathcal{M}_{AdS},g_{AdS})$ as a solution of (\ref{eq:RequationFinal})\textendash (\ref{NullShellFinal}),
with respect to the initial data topology defined by (\ref{eq:InitialDataNorm}).
This result will also provide us with some a priori control on the
geometry of solutions of (\ref{eq:RequationFinal})\textendash (\ref{NullShellFinal})
arising as small perturbations of $(\mathcal{M}_{AdS},g_{AdS})$ with
respect to (\ref{eq:InitialDataNorm}), which will be useful for the
results in our companion paper \cite{MoschidisVlasov}.

In particular, we will prove the following result:
\begin{thm}
\label{thm:CauchyStabilityAdS} For any $v_{\mathcal{I}}>0$, any
$U=n\cdot v_{\mathcal{I}}>0$ (where $n\in\mathbb{N}$) and any $C_{0}>0$,
there exist $\varepsilon_{0}>0$ and $C_{1}>0$ such that the following
statement holds: For any $0\le\varepsilon<\varepsilon_{0}$ and any
smooth initial data set $(r_{/},\Omega_{/}^{2},\bar{f}_{/};v_{\mathcal{I}})\in\mathcal{B}$
satisfying the smallness condition
\begin{equation}
||(r_{/},\Omega_{/}^{2},\bar{f}_{/};v_{\mathcal{I}})||<\varepsilon\label{eq:SmallnessInitialNorm}
\end{equation}
(where $||\cdot||$ is defined by (\ref{eq:InitialDataNorm})), as
well as the bound (\ref{eq:BoundedSupportDefinition}) with $C_{0}$
in place of $C$, the maximal future development $(\mathcal{U}_{max};r,\Omega^{2},f)$
of $(r_{/},\Omega_{/}^{2},\bar{f}_{/};v_{\mathcal{I}})$ with reflecting
boundary conditions on $\mathcal{I}$ (see Corollary \ref{cor:MaximalDevelopment})
is defined on the whole of the the domain $\mathcal{U}_{U;v_{\mathcal{I}}}$
(defined by (\ref{eq:GeneralDomain})), i.\,e.
\begin{equation}
\mathcal{U}_{U;v_{\mathcal{I}}}\subset\mathcal{U}_{max}.\label{eq:COntainedInDomain}
\end{equation}

In addition, $(r,\Omega^{2};f)$ satisfies the estimates: 
\begin{equation}
\sup_{u_{*}\in(0,U)}||(r,\Omega^{2};f)|_{u=u_{*}}||\le C_{1}\varepsilon\label{eq:SmallnessInNormCauchy}
\end{equation}
(where the notation $||(r,\Omega^{2};f)|_{u=u_{*}}||$ was introduced
in Definition \ref{def:NormSlice}), 
\begin{equation}
\sup_{(u,v)\in\mathcal{U}_{U;v_{\mathcal{I}}}}\Big(\sup_{p^{u},p^{v}\in supp(f(u,v;\cdot,\cdot,\cdot))}\Big(\big((-\partial_{u}r)p^{u}+(\partial_{v}r)p^{v}\big)\Big)\Big)\le(1+C_{1}\varepsilon)C_{0},\label{eq:SmallChangeBoundedSupportCauchy}
\end{equation}
\begin{equation}
\sup_{u\in(0,U)}\int_{u}^{u+v_{\mathcal{I}}}r\Big(\frac{T_{vv}}{\partial_{v}r}+\frac{T_{uv}}{-\partial_{u}r}\Big)(u,v)\,dv+\sup_{v\in(0,U+v_{\mathcal{I}})}\int_{\max\{0,v-v_{\mathcal{I}}\}}^{\min\{v,U\}}r\Big(\frac{T_{uv}}{\partial_{v}r}+\frac{T_{uu}}{-\partial_{u}r}\Big)(u,v)\,du\le C_{1}\varepsilon,\label{eq:SmallnessRightHandSideConstraintsCauchy}
\end{equation}
and
\begin{equation}
\sup_{\mathcal{U}_{U;v_{\mathcal{I}}}}\frac{2\tilde{m}}{r}<C_{1}\varepsilon.\label{eq:SmallnessTrappingCauchy}
\end{equation}
\end{thm}
\begin{rem*}
Let us note that the bounds (\ref{eq:SmallnessInNormCauchy})\textendash (\ref{eq:SmallnessTrappingCauchy})
are \emph{gauge invariant}. Furthermore, any coordinate transformation
$(u,v)\rightarrow(U(u),V(v))$ with $U(0)=0$ for which the lines
$\{u=v\}$ and $\{u=v-v_{\mathcal{I}}\}$ remain invariant necessarily
satisfies $U(nv_{\mathcal{I}})=nU(v_{\mathcal{I}})$ for any $n\in\mathbb{N}$;
for this reason, the inclusion (\ref{eq:COntainedInDomain}) is also
a gauge invariant statement.
\end{rem*}
\begin{proof}
In order to establish the gauge invariant estimates (\ref{eq:COntainedInDomain})\textendash (\ref{eq:SmallnessTrappingCauchy}),
it will be convenient for us to work in a gauge which normalises the
initial data. For this reason, we will assume that $(\mathscr{D};r,\Omega^{2},f)$
has been transformed under the normalising gauge transformation $\mathcal{T}$
provided by Lemma \ref{lem:TransformationForNormalisationDevelopment}
and, therefore, that $(r_{/},\Omega_{/}^{2},\bar{f}_{/};v_{\mathcal{I}})$
satisfy the normalisation condition (\ref{def:GaugeConditionNormalise}).
Note that, since $\mathcal{T}$ is only piecewise smooth, the transformed
functions $\Omega^{2}$ and $\partial r$ will only asumed to be bounded.
The gauge condition (\ref{def:GaugeConditionNormalise}), which is
equivalent to (\ref{eq:FormulaForDvR/2}), implies, in view of the
smallness assumption (\ref{eq:SmallnessInitialNorm}) and the definition
(\ref{eq:InitialDataNorm}) of $||\cdot||$, that $r_{/}$ is pointwise
close to the normalised, rescaled AdS radius $r_{AdS/}^{(v_{\mathcal{I}})}$
(defined by (\ref{eq:AdSMetricValuesRescaled1})), satisfying the
estimate
\begin{equation}
\sup_{v\in(0,v_{\mathcal{I}})}\Big|\frac{\partial_{v}r_{/}}{1-\frac{1}{3}\Lambda r_{/}^{2}}(v)-\frac{\partial_{v}r_{AdS/}^{(v_{\mathcal{I}})}}{1-\frac{1}{3}\Lambda(r_{AdS/}^{(v_{\mathcal{I}})})^{2}}(v)\Big|\le C\varepsilon\label{eq:FPointwiseCloseToAdSInitially}
\end{equation}
for some absolute constant $C>0$.

Let us introduce some shorthand notation for various geometric objects
that will appear in the proof of Theorem \ref{thm:CauchyStabilityAdS}:
Let $\mathcal{U}_{U_{*};v_{\mathcal{I}}}\subset\mathscr{D}$ be a
domain of the form (\ref{eq:GeneralDomain}); given any $v_{0}\in(0,v_{\mathcal{I}})$,
$p_{0}\in(0,+\infty]$ and $l\in(0,+\infty)$, we will denote with
$\gamma[\cdot;v_{0},p_{0},l]:[0,a)\rightarrow\mathcal{U}_{U_{*};v_{\mathcal{I}}}$
the affinely parametrised null geodesic of $(r,\Omega^{2})$ which
is uniquely determined by the condition that it has angular momentum
$l$ and satisfies initially: 
\begin{equation}
\begin{cases}
\gamma[0;v_{0},p_{0},l]=(0,v_{0}),\\
\dot{\gamma}^{u}[0;v_{0},p_{0},l]=p_{0}.
\end{cases}\label{eq:InitialConditionsGeodesicNotation}
\end{equation}

\medskip{}

\noindent \emph{Remark.} Since the gauge transformation $\mathcal{T}$
is piecewise smooth on $\mathscr{D}$ and smooth on $\mathscr{D}\backslash\cup_{k=1}^{\infty}\big(\{u=kv_{\mathcal{I}}\}\cup\{v=kv_{\mathcal{I}}\}\big)$
(see \ref{lem:TransformationForNormalisationDevelopment}), the initial
condition (\ref{eq:InitialConditionsGeodesicNotation}) on $u=0$
is regularly transformed under $\mathcal{T}^{-1}$, i.\,e.~$(D_{(0,v_{0})}\mathcal{T}^{-1})_{\mu}\dot{\gamma}^{\mu}$
is well defined; hence the existence and uniqueness of a null geodesic
$\gamma$ satisfying (\ref{eq:InitialConditionsGeodesicNotation})
can be shown readily by applying the inverse transformation $\mathcal{T}^{-1}$
on $(\mathscr{D};r,\Omega^{2},f)$ and appealing to the standard theory.

\medskip{}

\noindent We will denote the dervative of $\gamma$ with resect to
the affine parametrization by $\dot{}$, while $\tau$ will denote
the parameter defined by the function $u+v$. As a function of $\tau$,
we will assume that $\gamma[\cdot;v_{0},p_{0},l]$ is maximally extended
in $\mathcal{U}_{U_{*};v_{\mathcal{I}}}$ through reflections off
$\mathcal{I}$, according to Definition \ref{def:MaximalExtensionReflections}.
Furthermore, we will set 
\begin{equation}
E[\tau;v_{0},p_{0},l]\doteq\frac{1}{2}\Omega^{2}(\gamma[\tau;v_{0},p_{0},l])\Big(\dot{\gamma}^{u}[\tau;v_{0},p_{0},l]+\dot{\gamma}^{v}[\tau;v_{0},p_{0},l]\Big).
\end{equation}
We will also define $\gamma_{AdS}[\cdot;v_{0},p_{0},l]:[0,a')\rightarrow\mathcal{U}_{U_{*};v_{\mathcal{I}}}$
to be the unique maximally extended (through reflections) null geodesic
of the normalised rescaled AdS metric $(r_{AdS}^{(v_{\mathcal{I}})},(\Omega_{AdS}^{(v_{\mathcal{I}})})^{2})$
on $\mathcal{U}_{U_{*};v_{\mathcal{I}}}$ (the coefficients of which
are given by (\ref{eq:AdSMetricValuesRescaled-1})) satisfying the
initial conditions (\ref{eq:InitialConditionsGeodesicNotation}) and
having angular momentum $l$. We will adopt the same notational conventions
for $\gamma_{AdS}$ as for $\gamma$.

Let $C_{2}=C_{2}(U)\gg1$ be a fixed large constant depending only
on $U$. In order to establish Theorem \ref{thm:CauchyStabilityAdS},
we will first assume that, for some $U_{*}\in(0,U)$, 
\[
\mathcal{U}_{U_{*};v_{\mathcal{I}}}\doteq\{0<u<U_{*}\}\cap\{u<v<u+v_{\mathcal{I}}\}\subset\mathscr{D}
\]
 and that the following \emph{bootstrap assumptions }are satisfied\emph{:}
\begin{equation}
\sup_{u_{*}\in(0,U_{*})}||(r,\Omega^{2};f)|_{u=u_{*}}||\le C_{2}\varepsilon,\label{eq:SmallnessInNormCauchyBootstrap}
\end{equation}
\begin{equation}
\sup_{u\in(0,U_{*})}\int_{u}^{u+v_{\mathcal{I}}}r\Big(\frac{T_{vv}}{\partial_{v}r}+\frac{T_{uv}}{-\partial_{u}r}\Big)(u,v)\,dv+\sup_{v\in(0,U_{*}+v_{\mathcal{I}})}\int_{\max\{0,v-v_{\mathcal{I}}\}}^{\min\{v,U_{*}\}}r\Big(\frac{T_{uv}}{\partial_{v}r}+\frac{T_{uu}}{-\partial_{u}r}\Big)(u,v)\,du\le C_{2}\varepsilon,\label{eq:SmallnessRightHandSideConstraintsCauchyBootstrap}
\end{equation}
\begin{equation}
\sup_{\mathcal{U}_{U_{*};v_{\mathcal{I}}}}\frac{2\tilde{m}}{r}<C_{2}\varepsilon\label{eq:SmallnessTrappingCauchyBootstrap}
\end{equation}
and, for any $v_{0}\in(0,v_{\mathcal{I}})$, $p_{0}>0$, $l>0$:
\begin{equation}
\Big|u(\gamma[\tau;v_{0},p_{0},l])-u(\gamma_{AdS}[\tau;v_{0},p_{0},l])\Big|+\Big|v(\gamma[\tau;v_{0},p_{0},l])-v(\gamma_{AdS}[\tau;v_{0},p_{0},l])\Big|\le C_{2}\frac{l}{E[0;v_{0},p_{0},l]}\varepsilon,\label{eq:GeodesicDistanceBootstrap}
\end{equation}
\begin{equation}
\Big|\Omega^{2}\dot{\gamma}^{u}[\tau;v_{0},p_{0},l]-\Omega^{2}\dot{\gamma}_{AdS}^{u}[\tau;v_{0},p_{0},l])\Big|+\Big|\Omega^{2}\dot{\gamma}^{v}[\tau;v_{0},p_{0},l]-\Omega^{2}\dot{\gamma}_{AdS}^{v}[\tau;v_{0},p_{0},l]\Big|\le C_{2}\varepsilon E[0;v_{0},p_{0},l].\label{eq:GeodesicDistanceDerivativeBootstrap}
\end{equation}
We will then show that the following improvement of (\ref{eq:SmallnessInNormCauchyBootstrap})\textendash (\ref{eq:GeodesicDistanceDerivativeBootstrap})
actually holds on $\mathcal{U}_{U_{*};v_{\mathcal{I}}}$:
\begin{equation}
\sup_{u_{*}\in(0,U_{*})}||(r,\Omega^{2};f)|_{u=u_{*}}||\le4\varepsilon,\label{eq:SmallnessInNormCauchyBootstrap-1}
\end{equation}
\begin{equation}
\sup_{u\in(0,U_{*})}\int_{u}^{u+v_{\mathcal{I}}}r\Big(\frac{T_{vv}}{\partial_{v}r}+\frac{T_{uv}}{-\partial_{u}r}\Big)(u,v)\,dv+\sup_{v\in(0,U_{*}+v_{\mathcal{I}})}\int_{\max\{0,v-v_{\mathcal{I}}\}}^{\min\{v,U_{*}\}}r\Big(\frac{T_{uv}}{\partial_{v}r}+\frac{T_{uu}}{-\partial_{u}r}\Big)(u,v)\,du\le4\varepsilon,\label{eq:SmallnessRightHandSideConstraintsCauchyBootstrap-1}
\end{equation}
\begin{equation}
\sup_{\mathcal{U}_{U_{*};v_{\mathcal{I}}}}\frac{2\tilde{m}}{r}<\frac{1}{2}C_{2}\varepsilon\label{eq:SmallnessTrappingCauchyBootstrap-1}
\end{equation}
and, for any $v_{0}\in(0,v_{\mathcal{I}})$, $p_{0}>0$, $l>0$:
\begin{equation}
\Big|u(\gamma[\tau;v_{0},p_{0},l])-u(\gamma_{AdS}[\tau;v_{0},p_{0},l])\Big|+\Big|v(\gamma[\tau;v_{0},p_{0},l])-v(\gamma_{AdS}[\tau;v_{0},p_{0},l])\Big|\le\frac{1}{2}C_{2}\frac{l}{E[0;v_{0},p_{0},l]}\varepsilon,\label{eq:GeodesicDistanceBootstrap-1}
\end{equation}

\begin{equation}
\Big|\Omega^{2}\dot{\gamma}^{u}[\tau;v_{0},p_{0},l]-\Omega^{2}\dot{\gamma}_{AdS}^{u}[\tau;v_{0},p_{0},l])\Big|+\Big|\Omega^{2}\dot{\gamma}^{v}[\tau;v_{0},p_{0},l]-\Omega^{2}\dot{\gamma}_{AdS}^{v}[\tau;v_{0},p_{0},l]\Big|\le\frac{1}{2}C_{2}\varepsilon E[0;v_{0},p_{0},l].\label{eq:GeodesicDistanceDerivativeBootstrap-1}
\end{equation}
The proof of Theorem \ref{thm:CauchyStabilityAdS} will then follow
immediately through a standard continuity argument (using also the
gauge dependend bounds (\ref{eq:FirstBoundDvRCauchy})\textendash (\ref{eq:FirstBoundDuRCauchy})
and (\ref{eq:OmegaCauchyImp}) to compare $\Omega^{2}\dot{\gamma}^{v}$,
$\Omega^{2}\dot{\gamma}^{u}$ with $\partial_{v}r\dot{\gamma}^{v}$,
$-\partial_{u}r\dot{\gamma}^{u}$), by applying the inverse transformation
$\mathcal{T}^{-1}$ on $\mathscr{D}$ and using the extension principle
of Corollary \ref{cor:GeneralContinuationCriterion} (which guarantees
that, given a smooth solution $(r,\Omega^{2},f)$ on $\mathcal{U}_{U_{*};v_{\mathcal{I}}}$
satisfying (\ref{eq:SmallnessTrappingCauchyBootstrap-1}), it can
be smoothly extended on $\mathcal{U}_{U_{*}+c;v_{\mathcal{I}}}$ for
some $c>0$). 

We will now proceed to establish \ref{eq:SmallnessInNormCauchyBootstrap-1}\textendash \ref{eq:GeodesicDistanceDerivativeBootstrap-1}
in two steps, assuming that (\ref{eq:SmallnessInNormCauchyBootstrap})\textendash (\ref{eq:GeodesicDistanceDerivativeBootstrap})
hold. 

\medskip{}

\noindent \emph{Step 1: Proof of (\ref{eq:SmallnessInNormCauchyBootstrap-1})\textendash (\ref{eq:SmallnessTrappingCauchyBootstrap-1}).}
Integrating (\ref{eq:DerivativeInUDirectionKappa}) and (\ref{eq:DerivativeInVDirectionKappaBar})
using the initial conditions 
\begin{equation}
(r,\Omega^{2})(0,v)=(r_{/},\Omega_{/}^{2})(v),\label{eq:InitialROmega}
\end{equation}
the initial bound (\ref{eq:FPointwiseCloseToAdSInitially}), the bounds
(\ref{eq:SmallnessRightHandSideConstraintsCauchyBootstrap}) and (\ref{eq:SmallnessTrappingCauchyBootstrap}),
as well as the boundary conditions 
\begin{equation}
\begin{cases}
\partial_{v}r|_{\{u=v\}}=-\partial_{u}r|_{\{u=v\}},\\
\frac{\partial_{v}r}{1-\frac{1}{3}\Lambda r^{2}}|_{\{u=v-v_{\mathcal{I}}\}}=-\frac{\partial_{u}r}{1-\frac{1}{3}\Lambda r^{2}}|_{\{u=v-v_{\mathcal{I}}\}}
\end{cases}\label{eq:BoundaryConditionsRAxisConformalInfinity}
\end{equation}
on the axis and conformal infinity and the expression (\ref{eq:FormulaForDvR/2})
for $r_{/}$ (combined with the initial bound (\ref{eq:SmallnessInitialNorm})),
we can readily estimate for some absolute constant $C>0$, noting
that $U_{*}\le U$:
\begin{equation}
\sup_{\mathcal{U}_{U_{*};v_{\mathcal{I}}}}\Big|\frac{\partial_{v}r}{1-\frac{1}{3}\Lambda r^{2}}-\frac{\partial_{v}r_{AdS}^{(v_{\mathcal{I}})}}{1-\frac{1}{3}\Lambda(r_{AdS}^{(v_{\mathcal{I}})})^{2}}\Big|\le CC_{2}\varepsilon\cdot(1+\sqrt{-\Lambda}U)\label{eq:FirstBoundDvRCauchy}
\end{equation}
and 
\begin{equation}
\sup_{\mathcal{U}_{U_{*};v_{\mathcal{I}}}}\Big|\frac{-\partial_{u}r}{1-\frac{1}{3}\Lambda r^{2}}-\frac{-\partial_{u}r_{AdS}^{(v_{\mathcal{I}})}}{1-\frac{1}{3}\Lambda(r_{AdS}^{(v_{\mathcal{I}})})^{2}}\Big|\le CC_{2}\varepsilon\cdot(1+\sqrt{-\Lambda}U).\label{eq:FirstBoundDuRCauchy}
\end{equation}
Integrating (\ref{eq:FirstBoundDvRCauchy}) and using the boundary
conditions $r|_{\{u=v\}}=0$ and $r|_{\{u=v-v_{\mathcal{I}}\}}=+\infty$,
we infer that 
\begin{equation}
\sup_{\mathcal{U}_{U_{*};v_{\mathcal{I}}}}\frac{|r-r_{AdS}^{(v_{\mathcal{I}})}|}{r}\le CC_{2}\varepsilon\cdot(1+\sqrt{-\Lambda}U).\label{eq:FirstBoundRCauchy}
\end{equation}

As a consequence of the initial bound (\ref{eq:SmallnessInitialNorm}),
Definition (\ref{def:NormSlice}) of $||(r,\Omega^{2};f)|_{u=u_{*}}||$,
the bounds (\ref{eq:FirstBoundDvRCauchy})\textendash (\ref{eq:FirstBoundRCauchy}),
the bootstrap estimates (\ref{eq:GeodesicDistanceBootstrap})\textendash (\ref{eq:GeodesicDistanceDerivativeBootstrap})
for the difference of the dynamics of the geodesic flow of $(r,\Omega^{2})$
and $(r_{AdS}^{(v_{\mathcal{I}})},(\Omega_{AdS}^{(v_{\mathcal{I}})})^{2})$
and the explicit description (\ref{eq:AbstractFormulaGeodesicFlowAdS})\textendash (\ref{eq:ExplicitFormulaGeodesicFlow})
of the geodesic flow in AdS spacetime (in particular, the bound (\ref{eq:MinRAdS})),
we can readily obtain the following improvement of (\ref{eq:SmallnessInNormCauchyBootstrap})
and (\ref{eq:SmallnessRightHandSideConstraintsCauchyBootstrap}) (assuming
$\varepsilon_{0}$ is small enough in terms of $C_{2}$ and $U$):
\begin{equation}
\sup_{u_{*}\in(0,U_{*})}||(r,\Omega^{2};f)|_{u=u_{*}}||\le\big(1+CC_{1}\varepsilon)^{5}\varepsilon\le2\varepsilon\label{eq:SmallnessInNormCauchyImproved}
\end{equation}
and
\begin{equation}
\sup_{u\in(0,U_{*})}\int_{u}^{u+v_{\mathcal{I}}}r\Big(\frac{T_{vv}}{\partial_{v}r}+\frac{T_{uv}}{-\partial_{u}r}\Big)(u,v)\,dv+\sup_{v\in(0,U_{*}+v_{\mathcal{I}})}\int_{\max\{0,v-v_{\mathcal{I}}\}}^{\min\{v,U_{*}\}}r\Big(\frac{T_{uv}}{\partial_{v}r}+\frac{T_{uu}}{-\partial_{u}r}\Big)(u,v)\,du\le2\varepsilon.\label{eq:SmallnessRightHandSideConstraintsCauchyImproved}
\end{equation}
Therefore, we infer (\ref{eq:SmallnessInNormCauchyBootstrap-1})\textendash (\ref{eq:SmallnessRightHandSideConstraintsCauchyBootstrap-1}).
Returning to the proof of (\ref{eq:FirstBoundDvRCauchy})\textendash (\ref{eq:FirstBoundRCauchy})
and using (\ref{eq:SmallnessRightHandSideConstraintsCauchyImproved})
in place of (\ref{eq:SmallnessRightHandSideConstraintsCauchyBootstrap}),
we infer that 
\begin{equation}
\sup_{\mathcal{U}_{U_{*};v_{\mathcal{I}}}}\Big|\frac{\partial_{v}r}{1-\frac{1}{3}\Lambda r^{2}}-\frac{\partial_{v}r_{AdS}^{(v_{\mathcal{I}})}}{1-\frac{1}{3}\Lambda(r_{AdS}^{(v_{\mathcal{I}})})^{2}}\Big|\le C(1+\sqrt{-\Lambda}U)\varepsilon,\label{eq:FirstBoundDvRCauchyImp}
\end{equation}
\begin{equation}
\sup_{\mathcal{U}_{U_{*};v_{\mathcal{I}}}}\Big|\frac{-\partial_{u}r}{1-\frac{1}{3}\Lambda r^{2}}-\frac{-\partial_{u}r_{AdS}^{(v_{\mathcal{I}})}}{1-\frac{1}{3}\Lambda(r_{AdS}^{(v_{\mathcal{I}})})^{2}}\Big|\le C(1+\sqrt{-\Lambda}U)\varepsilon\label{eq:FirstBoundDuRCauchyImp}
\end{equation}
and
\begin{equation}
\sup_{\mathcal{U}_{U_{*};v_{\mathcal{I}}}}\frac{|r-r_{AdS}^{(v_{\mathcal{I}})}|}{r}\le C(1+\sqrt{-\Lambda}U)\varepsilon.\label{eq:FirstBoundRCauchyImp}
\end{equation}

Integrating the relation (\ref{eq:DerivativeTildeVMass}) for $\tilde{m}$
and using the boundary condition $\tilde{m}|_{\{u=v\}}=0$ (following
from the smoothness of $\mathcal{T}^{-1}(r,\Omega^{2};f)$ on the
axis) and using the bounds (\ref{eq:SmallnessTrappingCauchyBootstrap})
and (\ref{eq:SmallnessRightHandSideConstraintsCauchyImproved})\textendash (\ref{eq:FirstBoundRCauchyImp}),
we can readily estimate provided $\varepsilon_{0}$ is sufficiently
small with respect to $C_{2}$, $U$: 
\begin{equation}
\sup_{\mathcal{U}_{U_{*};v_{\mathcal{I}}}\cap\big\{ r\le\sqrt{-\frac{3}{\Lambda}}\big\}}\frac{2\tilde{m}}{r}\le C\varepsilon.\label{eq:BoundTrappingNearAxis}
\end{equation}
On the other hand, in view of the initial bound (\ref{eq:SmallnessInitialNorm}),
the fact that 
\[
\sqrt{-\Lambda}\tilde{m}_{/}|_{v=v_{\mathcal{I}}}\le||(r_{/},\Omega_{/}^{2},\bar{f}_{/};v_{\mathcal{I}})||
\]
and the conservation of $\tilde{m}$ along conformal infinity $\{u=v-v_{\mathcal{I}}\}$,
we can immediately bound:
\begin{equation}
\sup_{\mathcal{U}_{U_{*};v_{\mathcal{I}}}\cap\big\{ r\ge\sqrt{-\frac{3}{\Lambda}}\big\}}\frac{2\tilde{m}}{r}\le C\varepsilon.\label{eq:BoundTrappingAway}
\end{equation}
In particular, combining (\ref{eq:BoundTrappingNearAxis}) and (\ref{eq:BoundTrappingAway}),
we infer (\ref{eq:SmallnessTrappingCauchyBootstrap-1}). As a consequence
of (\ref{eq:DefinitionHawkingMass}) and (\ref{eq:RenormalisedHawkingMass}),
from (\ref{eq:FirstBoundDvRCauchyImp})\textendash (\ref{eq:BoundTrappingAway})
we deduce that 
\begin{equation}
\sup_{\mathcal{U}_{U_{*};v_{\mathcal{I}}}}\Big|\frac{\Omega^{2}}{1-\frac{1}{3}\Lambda r^{2}}-\frac{(\Omega_{AdS}^{(v_{\mathcal{I}})})^{2}}{1-\frac{1}{3}\Lambda(r_{AdS}^{(v_{\mathcal{I}})})^{2}}\Big|\le C(1+\sqrt{-\Lambda}U)\varepsilon.\label{eq:OmegaCauchyImp}
\end{equation}

\medskip{}

\noindent \emph{Step 2: Proof of (\ref{eq:GeodesicDistanceBootstrap-1})\textendash (\ref{eq:GeodesicDistanceDerivativeBootstrap-1}).}
As a consequence of the bound (\ref{eq:MinRAdS}) for geodesics on
AdS and the bootstrap assumption (\ref{eq:GeodesicDistanceBootstrap})
(using also (\ref{eq:FirstBoundDvRCauchyImp})\textendash (\ref{eq:FirstBoundRCauchyImp})),
we infer (provided $\varepsilon_{0}$ is small enough in terms of
$C_{2}(U)$) that, for any $v_{0}\in(0,v_{\mathcal{I}})$, $p_{0}>0$,
$l>0$:
\begin{equation}
\inf_{\gamma[\cdot;v_{0},p_{0},l]}r\doteq r_{min}[v_{0},p_{0},l]\ge\frac{1}{2}\Big(\frac{E^{2}[\gamma[0;v_{0},p_{0},l]]}{l^{2}}+\frac{1}{3}\Lambda\Big)^{-\frac{1}{2}}.\label{eq:LowerBoundRGeodesics}
\end{equation}

Let us set for any $k\in\mathbb{N}$ and any $v_{0}\in(0,v_{\mathcal{I}})$,
$p_{0}>0$ and $l>0$:
\begin{equation}
\tau_{k}[v_{0}]\doteq\min\big\{(2k-1)v_{\mathcal{I}}+2v_{0},\text{ }U_{*}\big\}
\end{equation}
and
\[
\tau_{k}^{\prime}[v_{0},p_{0},l]\doteq\min\big\{ U_{*},\text{ }\text{value of }\tau\text{ when }\gamma[\tau;v_{0},p_{0},l]\text{ reaches }\{u=v-v_{\mathcal{I}}\}\text{ for the }k\text{-th time}\big\}.
\]
We will also set for convenience 
\begin{equation}
\tau_{0}[v_{0}]=\tau_{0}^{\prime}[v_{0},p_{0},l]=v_{0}.
\end{equation}
Note that, as a consequence of the bootstrap assumption \ref{eq:GeodesicDistanceBootstrap}
and the explicit description (\ref{eq:AbstractFormulaGeodesicFlowAdS})\textendash (\ref{eq:ExplicitFormulaGeodesicFlow})
of the geodesic flow on AdS, we have for any $k\in\mathbb{N}$
\begin{equation}
|\tau_{k}[v_{0}]-\tau_{k}^{\prime}[v_{0},p_{0},l]|\le C_{2}\varepsilon\frac{l}{E[0;v_{0},p_{0},l]}.\label{eq:DifferenceForHittingPoints}
\end{equation}

\noindent \medskip{}
\emph{Remark}. In order to simplify our notation, we will frequently
drop the parameters $v_{0},p_{0},l$ from the notation for $\tau_{k}$,
$\tau_{k}^{\prime}$ when no confusion arises. 

\medskip{}

In view of the bootstrap assumptions \ref{eq:GeodesicDistanceBootstrap}\textendash \ref{eq:GeodesicDistanceDerivativeBootstrap},
the bounds (\ref{eq:FirstBoundDvRCauchyImp})\textendash (\ref{eq:FirstBoundRCauchyImp})
and the explicit description (\ref{eq:AbstractFormulaGeodesicFlowAdS})\textendash (\ref{eq:ExplicitFormulaGeodesicFlow})
of the geodesic flow on AdS (using, in particular, the fact that the
energy (\ref{eq:EnergyAdS}) of null geodesics in AdS is conserved),
we can readily show that, for some constant $C>0$ depending only
on $\Lambda$, we have for any $k\ge0$ (provided $\varepsilon_{0}$
is small enough in terms of $C_{2}(U)$):
\begin{equation}
\bigcup_{\tau\in[\tau_{k},\tau_{k+1}]}\gamma[\tau;v_{0},p_{0},l]\subset\Big\{ r\ge r_{min}[v_{0},p_{0},l]\Big\}\cap\Big(U_{in}^{(k)}\cup U_{out}^{(k)}\Big)\cap\mathcal{U}_{U^{*};v_{\mathcal{I}}},\label{eq:ShapeForGammaAfterReflections}
\end{equation}
where 
\begin{align}
U_{in}^{(k)} & =[(k-1)v_{\mathcal{I}}+v_{0}-C\frac{l}{E_{0}},kv_{\mathcal{I}}+v_{0}+C\frac{l}{E_{0}}]\times[kv_{\mathcal{I}}+v_{0}-C\frac{l}{E_{0}},kv_{\mathcal{I}}+v_{0}+C\frac{l}{E_{0}}],\label{eq:IngoingU}\\
U_{out}^{(k)} & =[kv_{\mathcal{I}}+v_{0}-C\frac{l}{E_{0}},kv_{\mathcal{I}}+v_{0}+C\frac{l}{E_{0}}]\times[kv_{\mathcal{I}}+v_{0}-C\frac{l}{E_{0}},(k+1)v_{\mathcal{I}}+v_{0}+C\frac{l}{E_{0}}]\label{eq:OutgoingU}
\end{align}
and we have denoted for simplicity 
\[
E_{0}\doteq E[0;v_{0},p_{0},l]
\]
(see Figure \ref{fig:Tubes_AdS}). Moreoever, 
\begin{equation}
\frac{1}{2}\le\frac{E[\tau;v_{0},p_{0},l]}{E_{0}}\le\frac{3}{2}.\label{eq:EnergyEstimate}
\end{equation}
\begin{figure}[h] 
\centering 
\scriptsize
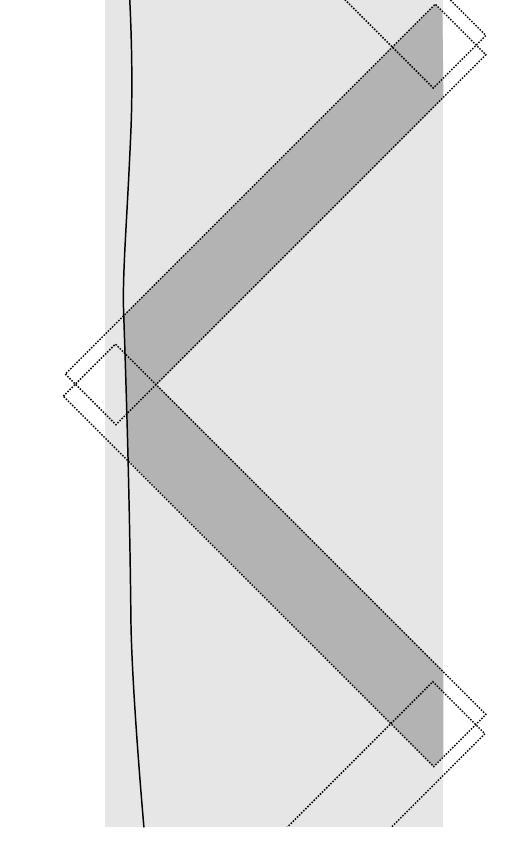 
\caption{Schematic depiction of the domains $U_{in}^{(k)}$ and $U_{out}^{(k)}$ for $k\ge1$ when $\frac{l}{E_0}\ll 1$. \label{fig:Tubes_AdS}}
\end{figure}

\medskip{}

\noindent \emph{Remark}. The bound (\ref{eq:ShapeForGammaAfterReflections})
becomes non-trivial only when $l/E_{0}\ll1$, since when $l/E_{0}\gtrsim1$
the regions in the right hand side of (\ref{eq:ShapeForGammaAfterReflections})
might contain all of $\{r\ge r_{min}[v_{0},p_{0},l]\}$.

\medskip{}

Let us define the \emph{approximately ingoing} and \emph{approximately
outgoing} intervals for $\gamma[\cdot;v_{0},p_{0},l]$ as follows:
For any integer $k\ge0$, we will set (for some fixed $C\gg1$ depending
on $\Lambda$)
\begin{equation}
I_{in}^{(k)}=[\tau_{k}^{\prime},\tau_{k+1}^{\prime}-v_{\mathcal{I}}+C\frac{l}{E_{0}}],\text{ }\hphantom{a}I_{out}^{(k)}=[\tau_{k+1}^{\prime}-v_{\mathcal{I}}-C\frac{l}{E_{0}},\tau_{k+1}^{\prime}].\label{IngoingOutgoingIntervalK}
\end{equation}
Note that $I_{in}^{(k)}\cap I_{out}^{(k)}\neq\emptyset$. Notice also
that, as a consequence of the bootstrap assumptions \ref{eq:GeodesicDistanceBootstrap}\textendash \ref{eq:GeodesicDistanceDerivativeBootstrap},
the bounds (\ref{eq:FirstBoundDvRCauchyImp})\textendash (\ref{eq:FirstBoundRCauchyImp})
and the explicit description (\ref{eq:AbstractFormulaGeodesicFlowAdS})\textendash (\ref{eq:ExplicitFormulaGeodesicFlow})
of the geodesic flow on AdS, we have for any $k\ge0$: 
\begin{align}
\bigcup_{\tau\in I_{in}^{(k)}}\gamma[\tau;v_{0},p_{0},l] & \subset\Big\{ r\ge r_{min}[v_{0},p_{0},l]\Big\}\cap U_{in}^{(k)},\label{eq:IngoingDomain}\\
\bigcup_{\tau\in I_{out}^{(k)}}\gamma[\tau;v_{0},p_{0},l] & \subset\Big\{ r\ge r_{min}[v_{0},p_{0},l]\Big\}\cap U_{out}^{(k)},\label{eq:OutgoingDOmain}
\end{align}
and, for some fixed $c>0$ depending on $\Lambda$: 
\begin{align}
\Omega^{2}\dot{\gamma}^{u}[\tau;v_{0},p_{0},l] & \ge cE_{0}\text{ for }\tau\in I_{in}^{(k)},\label{eq:IngoingToEnergy}\\
\Omega^{2}\dot{\gamma}^{v}[\tau;v_{0},p_{0},l] & \ge cE_{0}\text{ for }\tau\in I_{out}^{(k)}.\label{eq:OutgoingToEnergy}
\end{align}

\medskip{}

\noindent \emph{Remark}. In view of (\ref{eq:NullShellAngularMomentum})
and the conservation of (\ref{eq:EnergyAdS}) on pure AdS spacetime,
the bounds (\ref{eq:IngoingToEnergy})\textendash (\ref{eq:OutgoingToEnergy})
become non-trivial only when $l/E_{0}\ll1$.

\medskip{}

For any $k\in\mathbb{N}$, $v_{0}\in(0,v_{\mathcal{I}})$, $p_{0}>0$,
$l>0$, let $\mathcal{W}$ denote the region between $\gamma[\cdot;v_{0},p_{0},l_{0}]$
and $\gamma_{AdS}[\cdot;v_{0},p_{0},l_{0}]$ in $\mathcal{U}_{U_{*};v_{\mathcal{I}}}$,
i.\,e.: 
\begin{align}
\mathcal{W}=\mathcal{W}[v_{0},p_{0},l_{0}]\doteq\Big\{(\bar{u},\bar{v})\in\mathcal{U}_{U_{*};v_{\mathcal{I}}}:\text{ }\min\big\{(v-u) & |_{\gamma[\bar{u}+\bar{v};v_{0},p_{0},l_{0}]},(v-u)|_{\gamma_{AdS}[\bar{u}+\bar{v};v_{0},p_{0},l_{0}]}\big\}\le\bar{v}-\bar{u}\le\\
 & \le\max\big\{(v-u)|_{\gamma[\bar{u}+\bar{v};v_{0},p_{0},l_{0}]},(v-u)|_{\gamma_{AdS}[\bar{u}+\bar{v};v_{0},p_{0},l_{0}]}\big\}\Big\}.\nonumber 
\end{align}

\begin{figure}[h] 
\centering 
\scriptsize
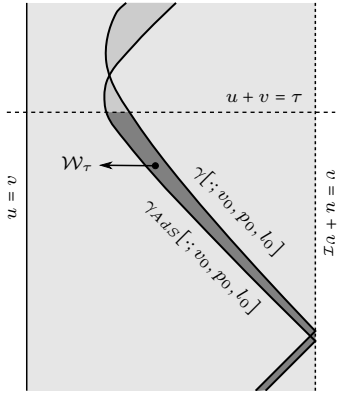 
\caption{Schematic depiction of the region $\mathcal{W}_{\tau}$ bounded by $\gamma[\cdot;v_{0},p_{0},l_{0}]$, $\gamma_{AdS}[\cdot;v_{0},p_{0},l_{0}]$ and $\{u+v\le \tau\}$. \label{fig:Region_Between}}
\end{figure}

Moreover, for any $\tau\in(v_{0},U_{*}+v_{\mathcal{I}})$, we will
set: 
\begin{equation}
\mathcal{W}_{\tau}=\mathcal{W}_{\tau}[v_{0},p_{0},l_{0}]\doteq\mathcal{W}\cap\{u+v\le\tau\}
\end{equation}
and 
\begin{equation}
|\mathcal{W}_{\tau}|\doteq\frac{1}{r_{min}[v_{0},p_{0},l]}\Big(\sup\big\{|u_{1}-u_{1}|:\text{ }(u_{1},v_{1}),(u_{2},v_{2})\in\mathcal{W}_{\tau}\text{ with }u_{1}+v_{1}=u_{2}+v_{2}\big\}\label{eq:DefinitionWidth}
\end{equation}
(see Figure \ref{fig:Region_Between}). Using the geodesic equation
(\ref{eq:NullGeodesicsSphericalSymmetry}) and the fact that $\tau$
corresponds to the parametrization of $\gamma$, $\gamma_{AdS}$ with
respect to $u+v$, we can readily estimate:
\begin{equation}
\frac{d}{d\tau}|\mathcal{W}_{\tau}|\le\frac{1}{r_{min}[v_{0},p_{0},l]}\Bigg|\frac{\Omega^{2}\dot{\gamma}^{v}-\Omega^{2}\dot{\gamma}^{u}}{\Omega^{2}\dot{\gamma}^{u}+\Omega^{2}\dot{\gamma}^{v}}[\tau;v_{0},p_{0},l]-\frac{(\Omega_{AdS}^{(v_{\mathcal{I}})})^{2}\dot{\gamma}_{AdS}^{v}-(\Omega_{AdS}^{(v_{\mathcal{I}})})^{2}\dot{\gamma}_{AdS}^{u}}{(\Omega_{AdS}^{(v_{\mathcal{I}})})^{2}\dot{\gamma}_{AdS}^{u}+(\Omega_{AdS}^{(v_{\mathcal{I}})})^{2}\dot{\gamma}_{AdS}^{v}}[\tau;v_{0},p_{0},l]\Bigg|.\label{eq:ForulaDerivativeWt}
\end{equation}

Using (\ref{eq:NullShellAngularMomentum}), (\ref{eq:EnergyEstimate})
and the fact that 
\[
\frac{1}{2}\le\frac{\Omega^{2}\dot{\gamma}^{v}}{\Omega^{2}\dot{\gamma}^{u}}\cdot\frac{(\Omega_{AdS}^{(v_{\mathcal{I}})})^{2}\dot{\gamma}_{AdS}^{u}}{(\Omega_{AdS}^{(v_{\mathcal{I}})})^{2}\dot{\gamma}_{AdS}^{v}}[\tau;v_{0},p_{0},l]\le2
\]
(following from the bootstrap assumption (\ref{eq:GeodesicDistanceDerivativeBootstrap})),
from (\ref{eq:ForulaDerivativeWt}) we readily infer for some absolute
constant $C>0$: 
\begin{equation}
\frac{d}{d\tau}|\mathcal{W}_{\tau}|\le\frac{C}{r_{min}[v_{0},p_{0},l]}\frac{l^{2}\Omega^{2}}{r^{2}}\Big|_{\gamma_{AdS}[\tau;v_{0},p_{0},l]}\cdot\frac{1}{E_{0}^{2}}E_{D}[\tau;v_{0},p_{0},l_{0}],\label{eq:BoundForWidthGrowth}
\end{equation}
where the \emph{energy difference} $E_{D}[\tau;v_{0},p_{0},l_{0}]$
is defined as 
\begin{equation}
E_{D}[\tau;v_{0},p_{0},l_{0}]\doteq\Big|\big(\Omega^{2}\dot{\gamma}^{u}+\Omega^{2}\dot{\gamma}^{v}\big)[\tau;v_{0},p_{0},l]-\big((\Omega_{AdS}^{(v_{\mathcal{I}})})^{2}\dot{\gamma}_{AdS}^{u}+(\Omega_{AdS}^{(v_{\mathcal{I}})})^{2}\dot{\gamma}_{AdS}^{v}\big)[\tau;v_{0},p_{0},l]\Big|.\label{eq:DefinitionEnergyDifference}
\end{equation}

\medskip{}

\noindent \emph{Remark.} Note that the factor in front of $E_{D}$
in the right hand side of (\ref{eq:BoundForWidthGrowth}) is expected
to behave like $\sim\frac{r_{min}}{r^{2}}$ along $\gamma$. Thus,
its integral over the geodesic is \emph{independent} of the value
of $r_{min}$ (that would not have been the case if the factor was
merely $\sim\frac{1}{r}$). This fact will be crucial for an application
of Gronwal's inequality later in the proof.

\medskip{}

In view of (\ref{eq:FirstBoundDvRCauchyImp})\textendash (\ref{eq:FirstBoundRCauchyImp}),
(\ref{eq:OmegaCauchyImp}), (\ref{eq:LowerBoundRGeodesics}) for $r_{min}$
and the definition (\ref{eq:IngoingU}), (\ref{eq:OutgoingU}) of
$U_{in}^{(k)}$, $U_{out}^{(k)}$, we can readily estimate for some
constant $C>0$ depending only on $\Lambda$: 
\begin{equation}
\int_{U_{in}^{(k)}\cap\{r\ge r_{min}[v_{0},p_{0},l]\}}\frac{\Omega^{2}}{r^{2}}\,dudv+\int_{U_{out}^{(k)}\cap\{r\ge r_{min}[v_{0},p_{0},l]\}}\frac{\Omega^{2}}{r^{2}}\,dudv\le C.\label{eq:BoundForBulkInGeodesicDifference}
\end{equation}
We can also bound for any $\tau\in(v_{0},U_{*}+v_{\mathcal{I}})$:
\begin{equation}
\int_{\mathcal{W}_{\tau}\cap U_{in}^{(k)}}\frac{\Omega^{2}}{r^{2}}\,dudv+\int_{\mathcal{W}_{\tau}\cap U_{out}^{(k)}}\frac{\Omega^{2}}{r^{2}}\,dudv\le C|\mathcal{W}_{\tau}|.\label{eq:BoundUsingDifferenceGeodesics}
\end{equation}
Setting 
\[
\mathcal{S}^{(k)}=\begin{cases}
U_{in}^{(0)}\cap\{u=0\}, & \text{ if }k=0\\
U_{in}^{(k)}\cap\{u=v-v_{\mathcal{I}}\}, & \text{ if }k>0,
\end{cases}
\]
using the relations (\ref{eq:DerivativeTildeUMass}), (\ref{eq:DerivativeTildeVMass})
and the bounds (\ref{eq:FirstBoundDvRCauchyImp})\textendash (\ref{eq:OmegaCauchyImp}),
we can also estimate 
\begin{align}
 & \int_{U_{in}^{(k)}\cap\{r\ge r_{min}[v_{0},p_{0},l]\}}\Big(\frac{\tilde{m}}{r^{3}}\Omega^{2}+T_{uv}\Big)\,dudv\le\label{eq:BoundForGeodesicFormulaInU}\\
 & \hphantom{\int_{U_{in}^{(k)}\cap\{r\ge}}\le C\int_{U_{in}^{(k)}\cap\{r\ge r_{min}[v_{0},p_{0},l]\}}\Big(\frac{\tilde{m}}{r^{3}}\Omega^{2}+\frac{-\partial_{u}\tilde{m}}{r^{2}}(1-\frac{1}{3}\Lambda r^{2})\Big)\,dudv\le\nonumber\\
 & \hphantom{\int_{U_{in}^{(k)}\cap\{r\ge}}\le C\int_{U_{in}^{(k)}\cap\{r\ge r_{min}[v_{0},p_{0},l]\}}\frac{\tilde{m}}{r^{3}}(1-\frac{1}{3}\Lambda r^{2})\,dudv+C\int_{\mathcal{S}^{(k)}}\frac{\tilde{m}}{r^{2}}(1-\frac{1}{3}\Lambda r^{2})\,dv\le\nonumber\\
 & \hphantom{\int_{U_{in}^{(k)}\cap\{r\ge}}\le C\int_{U_{in}^{(k)}\cap\{r\ge r_{min}[v_{0},p_{0},l]\}}\frac{\tilde{m}}{r^{3}}\Omega^{2}\,dudv+C\sup_{U_{in}^{(k)}\cap\{r\ge r_{min}[v_{0},p_{0},l]\}}\Big(\frac{\tilde{m}}{r^{2}}(1-\frac{1}{3}\Lambda r^{2})\Big)\cdot r_{min}[v_{0},p_{0},l]\le\nonumber \\
 & \hphantom{\int_{U_{in}^{(k)}\cap\{r\ge}}\le C\varepsilon.\nonumber 
\end{align}
where, in passing from the first to the second line in (\ref{eq:BoundForGeodesicFormulaInU}),
we have integrated by parts in $u$ and used the fact that $\frac{\tilde{m}}{r^{2}}(1-\frac{1}{3}\Lambda r^{2})$
is non-negative. Similarly, 
\begin{equation}
\int_{U_{out}^{(k)}\cap\{r\ge r_{min}[v_{0},p_{0},l]\}}\Big(\frac{\tilde{m}}{r^{3}}\Omega^{2}+T_{uv}\Big)\,dudv\le C\varepsilon.\label{eq:BoundForGeodesicFormulaInV}
\end{equation}

In view of the boundary condition $1/r|_{\{u=v-v_{\mathcal{I}}\}}=0$,
the finiteness of $\tilde{m}|_{\{u=v-v_{\mathcal{I}}\}}$, $r^{2}T_{uv}|_{\{u=v-v_{\mathcal{I}}\}}$,
formula (\ref{eq:DefinitionHawkingMass}) and the renormalised equation
(\ref{eq:RenormalisedEquations}) for $\tan^{-1}(\sqrt{-\Lambda}r)$),
the renormalised quantity $\Omega^{2}/(1-\frac{1}{3}\Lambda r^{2})$
satisfies the following boundary condition on conformal infinity:
\begin{equation}
(\partial_{v}-\partial_{u})\frac{\Omega^{2}}{1-\frac{1}{3}\Lambda r^{2}}\Big|_{\{u=v-v_{\mathcal{I}}\}}=0.\label{eq:BoundaryConditionOmegaInfinity}
\end{equation}
Setting 
\[
u_{1}(v)=\begin{cases}
0, & v\le v_{\mathcal{I}},\\
v-v_{\mathcal{I}}, & v\ge v_{\mathcal{I}}
\end{cases}
\]
and subtracting from the relation (\ref{eq:UsefulRelationForGeodesicWithMu-U})
for $\gamma[\cdot;v_{0},p_{0},l_{0}]$ the same relation for $\gamma_{AdS}[\cdot;v_{0},p_{0},l_{0}]$
(thus replacing $r,\Omega$ in (\ref{eq:UsefulRelationForGeodesicWithMu-U})
with $r_{AdS}^{(v_{\mathcal{I}})}$, $\Omega_{AdS}^{(v_{\mathcal{I}})}$
and setting $\tilde{m}=0$, $T_{uv}=0$), using also the relation
\begin{equation}
\partial_{v}\Big(\frac{\Omega^{2}}{1-\frac{1}{3}\Lambda r^{2}}\Big)\Big|_{\{u=v-v_{\mathcal{I}}\}}=\frac{1}{2}(\partial_{u}+\partial_{v})\Big(\frac{\Omega^{2}}{1-\frac{1}{3}\Lambda r^{2}}\Big)\Big|_{\{u=v-v_{\mathcal{I}}\}}\label{eq:BoundaryderivativesOmega}
\end{equation}
(and the analogous relation for $r_{AdS}^{(v_{\mathcal{I}})}$, $(\Omega_{AdS}^{(v_{\mathcal{I}})})^{2}$)
following readily from (\ref{eq:BoundaryConditionOmegaInfinity}),
we infer after using the bounds (\ref{eq:FirstBoundDvRCauchyImp})\textendash (\ref{eq:OmegaCauchyImp})
and (\ref{eq:BoundForBulkInGeodesicDifference})\textendash (\ref{eq:BoundForGeodesicFormulaInV})
that, for some $C>0$ depending only on $\Lambda$, the following
bound holds for any $\tau\in I_{in}^{(0)}$:
\begin{align}
\Big|\log\Big(\Omega^{2}\dot{\gamma}^{u} & [\tau;v_{0},p_{0},l_{0}]\Big)-\log\Big(\Omega_{AdS}^{2}\dot{\gamma}_{AdS}^{u}[\tau;v_{0},p_{0},l_{0}]\Big)\Big|\le\label{eq:AlmostFirstBoundForGronwallCauchy}\\
\le & C\int_{U_{in}^{(k)}\cap\{r\ge r_{min}[v_{0},p_{0},l]\}}\Big(\frac{\tilde{m}}{r^{3}}\Omega^{2}+T_{uv}\Big)\,dudv+C\int_{U_{in}^{(k)}\cap\{r\ge r_{min}[v_{0},p_{0},l]\}}\frac{(1+\sqrt{-\Lambda}U)\varepsilon}{r^{2}}\,dudv+\nonumber \\
 & +\int_{\mathcal{W}_{\tau}\cap U_{in}^{(k)}}\frac{\Omega^{2}}{r^{2}}\,dudv+C\Big|\int_{U_{out}^{(k)}\cap\{u=0\}}\partial_{v}\Big(\frac{\Omega^{2}}{-\Lambda r^{2}}-\frac{(\Omega_{AdS}^{(v_{\mathcal{I}})})^{2}}{-\Lambda(r_{AdS}^{(v_{\mathcal{I}})})^{2}}\Big)\,dv\Big|+\nonumber \\
 & +\frac{1}{2}\Big|\int_{U_{out}^{(k)}\cap\{u=v-v_{\mathcal{I}}\}}(\partial_{u}+\partial_{v})\Big(\frac{\Omega^{2}}{-\Lambda r^{2}}-\frac{(\Omega_{AdS}^{(v_{\mathcal{I}})})^{2}}{-\Lambda(r_{AdS}^{(v_{\mathcal{I}})})^{2}}\Big)\,dv\Big|\le\nonumber\\
\le & C(1+\sqrt{-\Lambda}U)\Big(\varepsilon+|\mathcal{W}_{\tau}|\Big).\nonumber 
\end{align}
From (\ref{eq:AlmostFirstBoundForGronwallCauchy}), using (\ref{eq:NullShellAngularMomentum}),
(\ref{eq:IngoingToEnergy}), (\ref{eq:EnergyEstimate}), (\ref{eq:BoundForWidthGrowth})
and the bootstrap assumption (\ref{eq:GeodesicDistanceBootstrap}),
we infer that, for any $\tau\in I_{in}^{(0)}$ (recalling the definition
(\ref{eq:DefinitionEnergyDifference}) of the energy difference $E_{D}$):
\begin{equation}
E_{D}[\tau;v_{0},p_{0},l_{0}]\le C(1+\sqrt{-\Lambda}U)E_{0}\Bigg(\varepsilon+C\frac{l^{2}}{r_{min}[v_{0},p_{0},l]E_{0}^{2}}\int_{v_{0}}^{\tau}\frac{(\Omega_{AdS}^{(v_{\mathcal{I}})})^{2}}{(r_{AdS}^{(v_{\mathcal{I}})})^{2}}\Big|_{\gamma_{AdS}[\bar{\tau};v_{0},p_{0},l]}E_{D}[\bar{\tau};v_{0},p_{0},l_{0}]\,d\bar{\tau}\Bigg).\label{eq:BoundForGronwallCauchy}
\end{equation}

Repeating the same procedure for $\tau\in I_{out}^{(0)}$ using (\ref{eq:UsefulRelationForGeodesicWithMu-V})
with $v_{1}(u)=v_{0}-C\frac{l}{E_{0}}$ (replacing (\ref{eq:IngoingToEnergy})
with (\ref{eq:OutgoingToEnergy}) and using (\ref{eq:BoundForGronwallCauchy})
to obtain an initial estimate for $E_{D}$ on $I_{in}^{(0)}\cap I_{out}^{(0)}$),
we deduce that (\ref{eq:AlmostFirstBoundForGronwallCauchy}) also
holds for $\tau\in I_{out}^{(0)}$. Furthermore, the same argument
yields that (\ref{eq:AlmostFirstBoundForGronwallCauchy}) also holds
for $\tau\in I_{in}^{(k)},I_{out}^{(k)}$ (for a possibly larger $C>0$)
for any $k>0$. Therefore, since, (\ref{eq:BoundForGronwallCauchy})
holds for all $\tau\in[v_{0},U_{*}+v_{\mathcal{I}})$, an application
of Gronwall's inequality, in combination with the trivial estimate
\begin{equation}
\frac{l^{2}}{r_{min}[v_{0},p_{0},l]E_{0}^{2}}\int_{v_{0}}^{\tau}\frac{(\Omega_{AdS}^{(v_{\mathcal{I}})})^{2}}{(r_{AdS}^{(v_{\mathcal{I}})})^{2}}\Big|_{\gamma_{AdS}[\bar{\tau};v_{0},p_{0},l]}\,d\bar{\tau}\le C\lceil\frac{\tau}{2v_{\mathcal{I}}}\rceil
\end{equation}
(following from the lower bound (\ref{eq:LowerBoundRGeodesics}) for
$r_{min}[v_{0},p_{0},l]$ and the explicit relations (\ref{eq:AbstractFormulaGeodesicFlowAdS})\textendash (\ref{eq:ExplicitFormulaGeodesicFlow})
for the geodesic flow on AdS), yield that 
\begin{equation}
\sup_{\tau}E_{D}[\tau;v_{0},p_{0},l_{0}]\le CE_{0}\exp(C(-\Lambda)U^{2})\cdot\varepsilon.\label{eq:BoundEnergyDifference}
\end{equation}

Using the relation (\ref{eq:NullShellAngularMomentum}) and the bounds
(\ref{eq:FirstBoundDvRCauchyImp})\textendash (\ref{eq:OmegaCauchyImp})
and (\ref{eq:LowerBoundRGeodesics}), from (\ref{eq:BoundEnergyDifference})
we infer (\ref{eq:GeodesicDistanceDerivativeBootstrap-1}), provided
$C_{2}(U)\gg\exp(C(-\Lambda)U^{2})$. Integrating the bound (\ref{eq:BoundForWidthGrowth})
and recalling the definition (\ref{eq:DefinitionWidth}) of $|\mathcal{W}_{\tau}|$,
we then obtain (\ref{eq:GeodesicDistanceBootstrap-1}). Hence, the
proof of Theorem \ref{thm:CauchyStabilityAdS} is complete. 
\end{proof}

\appendix

\section{\label{sec:Geodesic-flow-AdS}Geodesic flow on AdS spacetime}

In thi s section, we will collect a few useful relations regarding
the geodesic flow on AdS spacetime $(\mathcal{M}_{AdS}^{3+1},g_{AdS})$.
Let us fix a spherically symmetric double null coordinate pair $(u,v)$
on $\mathcal{M}_{AdS}$ by the condition 
\[
\frac{\partial_{v}r}{1-\frac{1}{3}\Lambda r^{2}}=-\frac{\partial_{u}r}{1-\frac{1}{3}\Lambda r^{2}}=\frac{1}{2}.
\]
(see Section \ref{subsec:Spherically-symmetric-spacetimes} for the
relevant definitions). In this coordinate chart, $\mathcal{M}_{AdS}\backslash\mathcal{Z}$
is mapped to the coordinate domain 
\[
\mathcal{U}_{AdS}=\big\{ u<v<u+\sqrt{-\frac{3}{\Lambda}}\pi\big\},
\]
while the metric $g_{AdS}$ is expressed as 
\[
g_{AdS}=-\Omega_{AdS}^{2}dudv+r^{2}g_{\mathbb{S}^{2}},
\]
where $\Omega_{AdS}^{2}$, $r$ are given by (\ref{eq:AdSMetricValues-1}).

In view of the fact that the vector field 
\[
T=\partial_{u}+\partial_{v}
\]
on $(\mathcal{M}_{AdS},g_{AdS})$ is Killing, the quantity 
\begin{equation}
E[\gamma]\doteq\frac{1}{2}\Omega^{2}(\dot{\gamma}^{u}+\dot{\gamma}^{v})\label{eq:EnergyAdS}
\end{equation}
 is constant along any affinely parametrised geodesic $\gamma$. Given
an initial point $\gamma(0)$ for a null geodesic, the values of the
conserved quantities $E[\gamma]$ and $l[\gamma]$ (see Section \ref{subsec:VlasovEquations}),
together with the sign of $\dot{\gamma}^{v}-\dot{\gamma}^{u}$, determine
$\dot{\gamma}(0)$ (and thus the whole geodesic $\gamma$) uniquely
up to a rotation of $\mathcal{M}_{AdS}$. For a future directed null
geodesic $\gamma$, the relation (\ref{eq:NullShellAngularMomentum})
yields (after multiplication with $\Omega_{AdS}^{2}$) that, at any
point $\gamma(s)$ on $\gamma$: 
\[
\frac{E^{2}[\gamma]}{l^{2}[\gamma]}\ge\frac{1}{r^{2}(\gamma(s))}-\frac{1}{3}\Lambda.
\]
Furthermore, using the relation (\ref{eq:NullShellAngularMomentum}),
we can determine the minimum value of $r$ along an inextendible,
future directed null geodesic $\gamma$ (attained at the point where
$\dot{\gamma}^{v}=\dot{\gamma}^{u}$) only in terms of $E[\gamma],l[\gamma]$;
in particular: 
\begin{equation}
\min_{\gamma}r=\Big(\frac{E^{2}[\gamma]}{l^{2}[\gamma]}+\frac{1}{3}\Lambda\Big)^{-\frac{1}{2}}.\label{eq:MinRAdS}
\end{equation}
Note that only radial null geodesics (i.\,e.~those with $l=0$)
pass through the center $r=0$. 

Let $\gamma$ be any future inextendible, future directed null geodesic
in $(\mathcal{M}_{AdS},g_{AdS})$ with affine parameter $s$, such
that $\gamma(0)$ lies on $u=0$ and $\gamma(\infty)$ (i.\,e.~the
intersection point of $\gamma$ with $\mathcal{I}$) lies at $u+v=\tau_{\infty}[\gamma]$.
Integrating the geodesic equation for (\ref{eq:AdSMetricValues-1}),
we obtain the following useful relations for $\gamma$: Denoting with
$sgn(\cdot)$ the sign function on $\mathbb{R}$, we calculate for
any $\tau\in[v(0),\tau_{\infty}[\gamma])$
\begin{align}
u|_{\gamma\cap\{u+v=\tau\}} & =U\big\{\tau;v(0),E[\gamma],l[\gamma],(\dot{\gamma}^{v}-\dot{\gamma}^{u})(0)\big\},\label{eq:AbstractFormulaGeodesicFlowAdS}\\
v|_{\gamma\cap\{u+v=\tau\}} & =V\big\{\tau;v(0),E[\gamma],l[\gamma],(\dot{\gamma}^{v}-\dot{\gamma}^{u})(0)\big\},\nonumber \\
\frac{d}{ds}\gamma^{u}|_{\gamma\cap\{u+v=\tau\}} & =\Omega_{AdS}^{-2}|_{\gamma\cap\{u+v=\tau\}}\cdot G^{u}\big\{\tau;v(0),E[\gamma],l[\gamma],sgn(\dot{\gamma}^{v}-\dot{\gamma}^{u})(0)\big\},\nonumber \\
\frac{d}{ds}\gamma^{v}|_{\gamma\cap\{u+v=\tau\}} & =\Omega_{AdS}^{-2}|_{\gamma\cap\{u+v=\tau\}}\cdot G^{v}\big\{\tau;v(0),E[\gamma],l[\gamma],sgn(\dot{\gamma}^{v}-\dot{\gamma}^{u})(0)\big\},\nonumber 
\end{align}
where, setting 
\begin{equation}
\rho_{min}\big\{ E,l\big\}=\tan^{-1}\Big(\Big(1-\frac{3}{\Lambda}\frac{E^{2}}{l^{2}}\Big)^{-\frac{1}{2}}\Big)
\end{equation}
and

\begin{equation}
\omega_{0}\big\{ v_{0},E,l,\sigma\big\}=-\sigma\cdot\cos^{-1}\Big(\frac{\cos\big(\frac{1}{2}\sqrt{-\frac{\Lambda}{3}}v_{0}\big)}{\cos\rho_{min}}\Big),\label{eq:omega0}
\end{equation}
the functions $U,V,G^{u},G^{v}$ appearing in (\ref{eq:AbstractFormulaGeodesicFlowAdS})
are defined by the relations
\begin{align}
U\big\{\tau;v_{0},E,l,\sigma\big\} & \doteq\frac{1}{2}\tau-\sqrt{-\frac{3}{\Lambda}}\cos^{-1}\Big\{\cos\rho_{min}\cdot\cos\Big[\omega_{0}-\frac{1}{2}\sqrt{-\frac{\Lambda}{3}}\Big(\tau-v_{0}\Big)\Big]\Big\}\label{eq:ExplicitFormulaGeodesicFlow}\\
V\big\{\tau;v_{0},E,l,\sigma\big\} & \doteq\frac{1}{2}\tau+\sqrt{-\frac{3}{\Lambda}}\cos^{-1}\Big\{\cos\rho_{min}\cdot\cos\Big[\omega_{0}-\frac{1}{2}\sqrt{-\frac{\Lambda}{3}}\Big(\tau-v_{0}\Big)\Big]\Big\},\nonumber \\
G^{u}\big\{\tau;v_{0},E,l,\sigma\big\} & \doteq2E\cdot\frac{dU}{d\tau}\big\{\tau;v_{0},E,l,\sigma\big\},\nonumber \\
G^{v}\big\{\tau;v_{0},E,l,\sigma\big\} & \doteq2E\cdot\frac{dV}{d\tau}\big\{\tau;v_{0},E,l,\sigma\big\}.\nonumber 
\end{align}
Note that $\tau_{\infty}[\gamma]$ can be explicitly expressed as
\[
\tau_{\infty}[\gamma]=v[0]+\sqrt{-\frac{3}{\Lambda}}(2\omega_{0}+\pi).
\]
\begin{rem*}
In the expressions (\ref{eq:omega0})\textendash (\ref{eq:ExplicitFormulaGeodesicFlow}),
we use the convention that $\cos^{-1}x\ge0$. In particular, for $\theta\in[-\frac{\pi}{2},\frac{\pi}{2}]$,
we have that $\cos^{-1}(\cos\theta)=|\theta|$. Note also that, when
$(\dot{\gamma}^{v}-\dot{\gamma}^{u})(0)=0$, the relation (\ref{eq:omega0})
yields $\omega_{0}=0$ independently of the convention for $sgn(0)$
(and, thus, the expressions (\ref{eq:AbstractFormulaGeodesicFlowAdS})
are smooth along variations of $(\dot{\gamma}^{v}-\dot{\gamma}^{u})(0)$
with non-constant sign).
\end{rem*}
By extending the geodesic $\gamma$ through its reflection off $\mathcal{I}$
(according to Definition \ref{def:Reflection}), we can extend the
functions $U$, $V$, $G^{u}$ and $G^{v}$ (defined by (\ref{eq:ExplicitFormulaGeodesicFlow}))
for $\tau$ belonging to the whole interval $[v_{0},+\infty)$. After
this extension, the functions $U,V$ become continuous and piecewise
smooth in $\tau$, while $G^{u},G^{v}$ have a jump discontinuity
at $\tau=\tau_{\infty}[\gamma]+2k\sqrt{-\frac{3}{\Lambda}}\pi$, $k\in\mathbb{N}$.
In particular, $U$, $V$, $G^{u}$ and $G^{v}$ satisfy for any $\tau\in[v_{0},+\infty)$
\begin{equation}
(U,V,G^{u},G^{v})\big\{\tau+2\sqrt{-\frac{3}{\Lambda}}\pi;v_{0},E,l,\sigma\big\}=(U,V,G^{u},G^{v})\big\{\tau;v_{0},E,l,\sigma\big\}\label{eq:PeriodicExtension}
\end{equation}
and 
\begin{align}
\lim_{\tau\rightarrow(\tau_{\infty}[\gamma])^{-}}G^{u}\big\{\tau;v_{0},E,l,\sigma\big\} & =\lim_{\tau\rightarrow(\tau_{\infty}[\gamma])^{+}}G^{v}\big\{\tau;v_{0},E,l,\sigma\big\},\label{eq:JumpDerivativeOnReflection}\\
\lim_{\tau\rightarrow(\tau_{\infty}[\gamma])^{-}}G^{v}\big\{\tau;v_{0},E,l,\sigma\big\} & =\lim_{\tau\rightarrow(\tau_{\infty}[\gamma])^{+}}G^{u}\big\{\tau;v_{0},E,l,\sigma\big\}.\nonumber 
\end{align}
 
\begin{rem*}
For a null geodesic $\gamma$ such that 
\[
\varepsilon\doteq\frac{l[\gamma]}{E[\gamma]}\ll\sqrt{-\frac{3}{\Lambda}},
\]
the relations (\ref{eq:ExplicitFormulaGeodesicFlow}) simplify as
follows (for $\tau\in[v(0),\tau_{\infty}[\gamma])$): 
\begin{align}
U\big\{\tau;v_{0},E,l,\sigma\big\} & =\frac{1}{2}\tau-\sqrt{\varepsilon^{2}+\frac{1}{4}\Big(\tau-v_{0}-2\sqrt{-\frac{3}{\Lambda}}\omega_{0}\Big)^{2}}+O(\varepsilon^{2})\label{eq:SimplifiedExplicitFormulaGeodesicFlow}\\
V\big\{\tau;v_{0},E,l,\sigma\big\} & =\frac{1}{2}\tau+\sqrt{\varepsilon^{2}+\frac{1}{4}\Big(\tau-v_{0}-2\sqrt{-\frac{3}{\Lambda}}\omega_{0}\Big)^{2}}+O(\varepsilon^{2}),\nonumber \\
G^{u}\big\{\tau;v_{0},E,l,\sigma\big\} & =E\cdot\Big(1-\frac{\frac{1}{2}\Big(\tau-v_{0}-2\sqrt{-\frac{3}{\Lambda}}\omega_{0}\Big)}{\sqrt{\varepsilon^{2}+\frac{1}{4}\Big(\tau-v_{0}-2\sqrt{-\frac{3}{\Lambda}}\omega_{0}\Big)^{2}}}+O(\varepsilon^{2})\Big),\nonumber \\
G^{v}\big\{\tau;v_{0},E,l,\sigma\big\} & =E\cdot\Big(1-\frac{\frac{1}{2}\Big(\tau-v_{0}-2\sqrt{-\frac{3}{\Lambda}}\omega_{0}\Big)}{\sqrt{\varepsilon^{2}+\frac{1}{4}\Big(\tau-v_{0}-2\sqrt{-\frac{3}{\Lambda}}\omega_{0}\Big)^{2}}}+O(\varepsilon^{2})\Big),\nonumber 
\end{align}
i.\,e.~$\gamma$ approximately traces a hyperbola in the $(u,v)$
plane (when $\varepsilon>0$), with its vertex at a point where $r\simeq\varepsilon$.
\end{rem*}
\begin{figure}[h] 
\centering 
\scriptsize
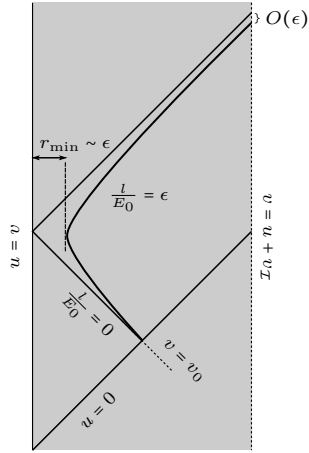 
\caption{Schematic depiction of the projection onto the $(u,v)$-plane of two null geodesics in AdS spacetime emanating from the same point on $\{ u=0 \}$, with $l=0$ and $l=\epsilon E_0$, respectively, where $\epsilon \ll \sqrt{-\frac{3}{\Lambda}}$. Due to the special form of the AdS metric, null geodesics emanating, instead, from the same point on $\mathcal{I}$ will also terminate on the same point on $\mathcal{I}$. \label{fig:Geodesics_AdS}}
\end{figure}

The relations (\ref{eq:AbstractFormulaGeodesicFlowAdS}) allow us
to obtain some simple expressions for solutions of the massless Vlasov
equation on $(\mathcal{M}_{AdS},g_{AdS})$. In particular, for any
smooth function $F:[0,\sqrt{-\frac{3}{\Lambda}}\pi)\times[0,+\infty)\times[0,+\infty)\rightarrow[0,+\infty)$,
the unique solution $f$ of the massless Vlasov equation (\ref{eq:Vlasov})
satisfying at $u=0$ the initial condition 
\begin{equation}
f(0,v;p^{u},p^{v},l)=F(v;p^{u},l)\cdot\delta\big(\Omega_{AdS}^{2}p^{u}p^{v}-\frac{l^{2}}{r^{2}}\big)
\end{equation}
can be expressed as
\begin{equation}
f(u,v;p^{u},p^{v},l)=\bar{f}(u,v;p^{u},p^{v},l)\cdot\delta\big(\Omega_{AdS}^{2}p^{u}p^{v}-\frac{l^{2}}{r^{2}}\big),\label{eq:FAsASmoothDeltaFunction-1}
\end{equation}
where the smooth function $\bar{f}(u,v;p^{u},p^{v},l)$ satisfies
for any $\tau\ge0$ and any $v\in[0,\sqrt{-\frac{3}{\Lambda}}\pi)$,
$E>0$ and $l\ge0$ the relation 
\begin{equation}
\bar{f}\Big(U\{\tau\},V\{\tau\};\Omega_{AdS}^{-2}(U\{\tau\},V\{\tau\})G^{u}\{\tau\},\Omega_{AdS}^{-2}(U\{\tau\},V\{\tau\})G^{v}\{\tau\},l\Big)=F(v,p^{u}[E,v],l).\label{eq:ExplicitSolutionVlasovAdS}
\end{equation}
where $U\{\tau\}$ is shorthand for $U\big\{\tau;v,E,l,sgn(p^{u}[E,v]-\frac{l^{2}}{r^{2}\Omega_{AdS}^{2}(0,v)p^{u}[E,v]})\big\}$
(and similarly for $V\{\tau\}$, $G^{u}\{\tau\}$ and $G^{v}\{\tau\}$),
while $p^{u}[E,v]$ (corresponding to $p^{u}$ along $u=0$ at energy
level $E$) is determined by the relation
\[
\Omega_{AdS}^{2}(0,v)p^{u}+\frac{l^{2}}{r^{2}(0,v)p^{u}}=2E.
\]

\section{\label{sec:CompleteI} Completeness of $\mathcal{I}$ in the presence
of a trapped sphere}

In this section, we will address the question of completeness of conformal
infinity $\mathcal{I}$ for the maximal future development of characteristic,
asymptotically AdS initial data sets. In general, we will not be able
to show that the maximal future development of a smoothly compatible
initial data set $(r_{/},\Omega_{/}^{2},\bar{f}_{/};v_{\mathcal{I}})$
(in accordance with Definition \ref{def:CompatibilityCondition})
has future complete conformal infinity $\mathcal{I}$,\footnote{The statement that for generic initial data, $\mathcal{I}$ is future
complete, is of course equivalent to the statement of the weak cosmic
censorship conjecture in the asymptotically AdS settings for (\ref{eq:EinsteinVlasovEquations})
in spherical symmetry. } i.\,e.~satisfies 
\begin{equation}
\int_{0}^{u_{\mathcal{I}}}\frac{\Omega}{\big(1-\frac{1}{3}\Lambda r^{2}\big)^{\frac{1}{2}}}(u,u+v_{\mathcal{I}})\,du=+\infty.\label{eq:ConditionFutureCompleteInfinity}
\end{equation}
 However, we will be able to infer future completeness for $\mathcal{I}$
in the presence of a trapped sphere; this result will be useful in
our companion paper \cite{MoschidisVlasov}.
\begin{lem}
\label{lem:CompletenessOfI} Let $(\mathcal{U}_{max};r,\Omega^{2},f)$
be the maximal future development of a smoothly compatible, asymptotically
AdS initial data set $(r_{/},\Omega_{/}^{2},\bar{f}_{/};v_{\mathcal{I}})$
for (\ref{eq:RequationFinal})\textendash (\ref{NullShellFinal})
with bounded support in phase space, in accordance with Definitions
\ref{def:AsymptoticallyAdSData} and \ref{def:CompatibilityCondition}.
Assume that there exists some $(\bar{u},\bar{v})\in\mathcal{U}_{max}$
such that 
\begin{equation}
\frac{2m}{r}(\bar{u},\bar{v})>1.\label{eq:TrappedPoint}
\end{equation}
Then the conformal infinity $\mathcal{I}$ of $(\mathcal{U}_{max};r,\Omega^{2},f)$
is future complete, i.\,e.~(\ref{eq:ConditionFutureCompleteInfinity})
is satisfied.
\end{lem}
\begin{proof}
Let $u_{\gamma_{\mathcal{Z}}},u_{\mathcal{I}}\in(0,+\infty]$ be the
endpoint parameters of $\gamma_{\mathcal{Z}},\mathcal{I}$ and let
$\zeta$ be the (possibly empty) achronal future boundary of $\mathcal{U}_{max}$,
defined as in Definition \ref{def:DevelopmentSets}. As a consequence
of the relation (\ref{eq:DefinitionHawkingMass}) and the fact that
$\partial_{u}r<0$ everywhere on $\mathcal{U}$ (following from (\ref{eq:ConstraintUFinal})
and the fact that $\partial_{u}r<0$ on $\{u=0\}\cup\mathcal{I}$),
for any point $(u_{*},v_{*})\in\mathcal{U}$ satisfying 
\begin{equation}
\frac{2m}{r}(u_{*},v_{*})\ge1,\label{eq:bbbbb}
\end{equation}
 we can bound from above
\[
\partial_{v}r(u_{*},v_{*})<0.
\]
Thus, integrating (\ref{eq:ConstraintVFinal}) along $u=u_{*}$ starting
from $(u_{*},v_{*})$, we infer that:
\begin{equation}
\sup_{v\ge v_{*}}\partial_{v}r(u_{*},v)<0.
\end{equation}
Therefore, the function $r(u_{*},\cdot)$ is bounded from above, which
implies that 
\[
\{u=u_{*}\}\cap\mathcal{I}=\emptyset
\]
(since $\mathcal{I}=\{r=\infty\}$) or, equivalently: 
\begin{equation}
u_{*}\ge u_{\mathcal{I}}.\label{eq:ccccc}
\end{equation}
Since (\ref{eq:ccccc}) holds for any $u_{*}$ for which (\ref{eq:bbbbb})
is satisfied for some $v_{*}$, we infer that:
\begin{equation}
\frac{2m}{r}(u,v)<1\text{ for all }(u,v)\in\{u<u_{\mathcal{I}}\}\cap\mathcal{U}.\label{eq:NoTrappingPastInfinity}
\end{equation}

Since $(\bar{u},\bar{v})$ satisfies (\ref{eq:bbbbb}) (in view of
the assumption (\ref{eq:TrappedPoint})), it also satisfies (\ref{eq:ccccc}),
i.\,e.
\begin{equation}
\bar{u}\ge u_{\mathcal{I}}.\label{eq:UpperBoundEndpointInfinity}
\end{equation}
In view of the fact that $\zeta$ is achronal (yielding that $u|_{\zeta}\le\max\{u_{\mathcal{I}},u_{\gamma_{\mathcal{Z}}}\}$)
and $(\bar{u},\bar{v})$ is an interior point of $\mathcal{U}_{max}$,
\ref{eq:UpperBoundEndpointInfinity} also implies that
\begin{equation}
\bar{u}<u_{\gamma_{\mathcal{Z}}}.\label{eq:LowerBoundEndpointAxis}
\end{equation}
As a consequence of (\ref{eq:UpperBoundEndpointInfinity}) and (\ref{eq:LowerBoundEndpointAxis})
and the fact that $\zeta$ is achronal, we therefore infer that, for
any $\varepsilon>0$, there exists some $\delta(\varepsilon)>0$ such
that 
\begin{equation}
\Big((u_{\mathcal{I}}-\delta(\varepsilon),u_{\mathcal{I}}+\delta(\varepsilon))\times(u_{\mathcal{I}},v_{\mathcal{I}}+u_{\mathcal{I}}-\varepsilon)\Big)\cap\{u<v\}\subset\mathcal{U}_{max}.\label{eq:NeighborhoodOfHorizon}
\end{equation}

The maximality of $(\mathcal{U}_{max};r,\Omega^{2},f)$ implies that
$(\mathcal{U}_{max};r,\Omega^{2},f)$ \emph{cannot} be extended as
a smooth solution of (\ref{eq:RequationFinal})\textendash (\ref{NullShellFinal})
with smooth conformal infinity $\mathcal{I}$ beyond $u=u_{\mathcal{I}}$
to any open set of the form $\{u\le u_{\mathcal{I}}+\delta\}\cap\{u<v<u+v_{\mathcal{I}}\}$
for $\delta>0$. This fact, together with the fact that (\ref{eq:NeighborhoodOfHorizon})
holds for any $\varepsilon>0$, implies that the conditions of the
extension principle of Proposition \ref{prop:ExtensionPrincipleInfinity}
fail to hold for $(\mathcal{U}_{max};r,\Omega^{2},f)$ at the boundary
point at $(u_{\mathcal{I}},u_{\mathcal{I}}+v_{\mathcal{I}})$, i.\,e.~there
exists a sequence $\{(u_{n},v_{n})\}_{n\in\mathbb{N}}\subset\mathcal{U}_{max}$
with $u_{n}<u_{\mathcal{I}}$ and $(u_{n},v_{n})\rightarrow(u_{\mathcal{I}},u_{\mathcal{I}}+v_{\mathcal{I}})$,
such that 
\begin{equation}
\lim_{n\rightarrow\infty}\frac{2m}{r}(u_{n},v_{n})\ge1.\label{eq:AlmostFailureExtension}
\end{equation}
In view of (\ref{eq:NoTrappingPastInfinity}), the inequality (\ref{eq:AlmostFailureExtension})
trivially reduces to the equality 
\begin{equation}
\lim_{n\rightarrow\infty}\frac{2m}{r}(u_{n},v_{n})=1.\label{eq:FailureExtension}
\end{equation}

The relation (\ref{eq:DerivativeInVDirectionKappaBar}) (which is
well defined on $\{u<u_{\mathcal{I}}\}\cap\mathcal{U}$ in view of
(\ref{eq:NoTrappingPastInfinity})) implies that 
\begin{equation}
\partial_{v}\Big(\frac{-\partial_{u}r}{1-\frac{2m}{r}}\Big)\ge0.\label{eq:IncreasingHalfOmegaOutgoing}
\end{equation}
Integrating (\ref{eq:IncreasingHalfOmegaOutgoing}) over the triangle
$\{u\le u_{n}\}\cap\{v\ge v_{n}\}\cap\{u>v-v_{\mathcal{I}}\}$, we
infer that 
\begin{equation}
\int_{\max\{v_{n}-v_{\mathcal{I}},0\}}^{u_{n}}\frac{-\partial_{u}r}{1-\frac{2m}{r}}\Big|_{\mathcal{I}}(u,u+v_{\mathcal{I}})\,du\ge\int_{\max\{v_{n}-v_{\mathcal{I}},0\}}^{u_{n}}\frac{-\partial_{u}r}{1-\frac{2m}{r}}(u,v_{n})\,du\label{eq:TrivialInequality}
\end{equation}
for all $n\in\mathbb{N}$. Using the relation (\ref{eq:RenormalisedHawkingMass})
between $m$ and $\tilde{m}$ and the fact that $\partial_{u}\tilde{m}\le0$
on $\{u<u_{\mathcal{I}}\}\cap\mathcal{U}$ (following from (\ref{eq:DerivativeTildeUMass})
and (\ref{eq:NoTrappingPastInfinity})), we can trivially estimate
from below
\begin{align}
\int_{v_{n}-v_{\mathcal{I}}}^{u_{n}}\frac{-\partial_{u}r}{1-\frac{2m}{r}}(u,v_{n})\,du & \ge\int_{v_{n}-v_{\mathcal{I}}}^{u_{n}}\frac{-\partial_{u}r}{1-\frac{2\tilde{m}(u_{n},v_{n})}{r}-\frac{1}{3}\Lambda r^{2}}(u,v_{n})\,du=\label{eq:LowerBoundRLength}\\
 & =\int_{r(u_{n},v_{n})}^{+\infty}\frac{1}{1-\frac{2\tilde{m}(u_{n},v_{n})}{r}-\frac{1}{3}\Lambda r^{2}}\,dr\ge\nonumber \\
 & \ge-\log\Big(1-\frac{2m}{r}(u_{n},v_{n})\Big)-C\nonumber 
\end{align}
for some $C>0$ independent of $n$. On the other hand, in view of
the relations (\ref{eq:DefinitionHawkingMass}), (\ref{eq:RenormalisedHawkingMass})
and the bound $\tilde{m}|_{\mathcal{I}}<\infty$ (following from the
fact that the initial data $(r_{/},\Omega_{/}^{2},\bar{f}_{/};v_{\mathcal{I}})$
were assumed to be of bounded support in phase space), we readily
infer using the boundary condition (\ref{eq:BoundaryConditionRInfinity})
that
\begin{equation}
\frac{-\partial_{u}r}{1-\frac{2m}{r}}\Big|_{\mathcal{I}}=\frac{1}{2}\frac{\Omega}{\big(1-\frac{1}{3}\Lambda r^{2}\big)^{\frac{1}{2}}}.\label{eq:HalfOmegaAtInfinity}
\end{equation}
Using (\ref{eq:LowerBoundRLength}) and (\ref{eq:HalfOmegaAtInfinity}),
the inequality (\ref{eq:TrivialInequality}) yields for any $n\in\mathbb{N}$
\begin{equation}
\int_{v_{n}-v_{\mathcal{I}}}^{u_{n}}\frac{\Omega}{\big(1-\frac{1}{3}\Lambda r^{2}\big)^{\frac{1}{2}}}(u,u+v_{\mathcal{I}})\,du\ge-\log\Big(1-\frac{2m}{r}(u_{n},v_{n})\Big)-C.\label{eq:AlmostThereCompleteInfinity=00005D}
\end{equation}
Considering the limit $n\rightarrow\infty$ for (\ref{eq:AlmostThereCompleteInfinity=00005D})
and using (\ref{eq:FailureExtension}), we therefore infer (\ref{eq:ConditionFutureCompleteInfinity}).
\end{proof}
\bibliographystyle{plain}
\bibliography{DatabaseExample}

\begin{thebibliography}{10}

\bibitem{AGMOO2000}
O.~Aharony, S.~Gubser, J.~Maldacena, H.~Ooguri, and Y.~Oz.
\newblock {Large N field theories, string theory and gravity}.
\newblock {\em Physics Reports}, 323(3--4):183--386, 2000.

\bibitem{AmmonErdmenger}
M.~Ammon and J.~Erdmenger.
\newblock {\em {Gauge/Gravity Duality}}.
\newblock Cambridge University Press, 2015.

\bibitem{DafermosRendall2005}
M.~Dafermos and A.~Rendall.
\newblock An extension principle for the {E}instein--{V}lasov system in
  spherical symmetry.
\newblock {\em Annales Henri Poincar\'{e}}, 6(6):1137--1155, 2005.

\bibitem{DafermosRendall}
M.~Dafermos and A.~Rendall.
\newblock Strong cosmic censorship for surface-symmetric cosmological
  spacetimes with collisionless matter.
\newblock {\em Communictions on Pure and Applied Mathematics}, 69:815--908,
  2016.

\bibitem{Eiesland1925}
J.~Eiesland.
\newblock {The group of motions of an Einstein space}.
\newblock {\em Transactions of the American Mathematical Society}, 27:213--245,
  1925.

\bibitem{EncisoKamran2014}
A.~Enciso and N.~Kamran.
\newblock Lorentzian {E}instein metrics with prescribed conformal infinity.
\newblock {\em arXiv preprint arXiv:1412.4376}, 2014.

\bibitem{Friedrich1995}
H.~Friedrich.
\newblock {Einstein equations and conformal structure: Existence of Anti-de
  Sitter-type space-times }.
\newblock {\em Journal of Geometry and Physics}, 17:125--184, 1995.

\bibitem{Friedrich2014}
H.~Friedrich.
\newblock {On the AdS stability problem}.
\newblock {\em Classical and Quantum Gravity}, 31(10):105001, 2014.

\bibitem{gubser1998gauge}
S.~Gubser, I.~Klebanov, and A.~Polyakov.
\newblock Gauge theory correlators from non-critical string theory.
\newblock {\em Physics Letters B}, 428(1):105--114, 1998.

\bibitem{Hartnoll2009}
S.~Hartnoll.
\newblock Lectures on holographic methods for condensed matter physics.
\newblock {\em Classical and Quantum Gravity}, 26:224022, 2009.

\bibitem{HawkingEllis1973}
S.~Hawking and G.~Ellis.
\newblock {\em The large scale structure of space-time}.
\newblock Cambridge University Press, 1973.

\bibitem{HolzegelLukSmuleviciWarnick}
G.~Holzegel, J.~Luk, J.~Smulevici, and C.~Warnick.
\newblock {Asymptotic properties of linear field equations in anti-de Sitter
  space}.
\newblock {\em arXiv preprint: arXiv:1502.04965}, 2015.

\bibitem{HolzSmul2012}
G.~Holzegel and J.~Smulevici.
\newblock Self-gravitating {K}lein--{G}ordon fields in asymptotically
  anti-de-{S}itter apacetimes.
\newblock {\em Annales Henri Poincar\'{e}}, 13(4):991--1038, 2012.

\bibitem{HolzWarn2015}
G.~Holzegel and C.~Warnick.
\newblock {The Einstein--Klein--Gordon--AdS system for general boundary
  conditions}.
\newblock {\em Journal of Hyperbolic Differential Equations}, 12(2):293--342,
  2015.

\bibitem{Maldacena}
J.~M. Maldacena.
\newblock The large {N} limit of superconformal field theories and
  supergravity.
\newblock {\em Advances in Theoretical and Mathematical Physics}, 2:231--252,
  1998.

\bibitem{MoschidisVlasov}
G.~Moschidis.
\newblock {A proof of the instability of AdS for the Einstein--massless Vlasov
  system}.
\newblock {\em preprint}.

\bibitem{MoschidisNullDust}
G.~Moschidis.
\newblock {A proof of the instability of AdS for the Einstein--null dust system
  with an inner mirror}.
\newblock {\em arXiv preprint: arXiv:1704.08681}, 2017.

\bibitem{MoschidisMaximalDevelopment}
G.~Moschidis.
\newblock {The Einstein--null dust system in spherical symmetry with an inner
  mirror: structure of the maximal development and Cauchy stability}.
\newblock {\em arXiv preprint: arXiv:1704.08685}, 2017.

\bibitem{VasyAdS}
A.~Vasy.
\newblock {The wave equation on asymptotically anti de Sitter spaces}.
\newblock {\em Analysis \& PDE}, 5(1):81--144, 2012.

\bibitem{Warnick2013}
C.~Warnick.
\newblock The massive wave equation in asymptotically {AdS} spacetimes.
\newblock {\em Communications in Mathematical Physics}, 321(1):85--111, 2013.

\bibitem{witten1998anti}
E.~Witten.
\newblock Anti de {S}itter space and holography.
\newblock {\em Advances in Theoretical and Mathematical Physics}, 2:253--291,
  1998.

\end{thebibliography}

\end{document}